\documentclass[reqno,10pt]{amsart}
\usepackage{amssymb,amsmath,amsthm,}
\usepackage{a4wide}

\usepackage{color}
\usepackage{esint}
\usepackage{calrsfs}
\usepackage{stmaryrd}
\usepackage[colorlinks,
           linkcolor=blue,
           anchorcolor=blue,
           citecolor=blue
           ]{hyperref}

\newtheorem{theorem}{Theorem}

\newtheorem{lemma}[theorem]{Lemma}
\newtheorem{proposition}[theorem]{Proposition}
\newtheorem{remark}[theorem]{Remark}
\newcommand{\dd}{\mathrm{d}}

\newcommand{\bbR}{\mathbb{R}}

\newcommand{\0}{{\delta}}
\newcommand{\1}{{\eta}}

\def\beq{\begin{equation}}
\def\eeq{\end{equation}}
\def\beqs{\begin{equation*}}
\def\eeqs{\end{equation*}}
\def\bal#1\eal{\begin{align}#1\end{align}}
\def\bals#1\eals{\begin{align*}#1\end{align*}}
\def\bsp#1\esp{\begin{split}#1\end{split}}

\def\d{{\mathrm{d}}}

\let\a=\alpha
\let\e=\varepsilon

\let\pt=\partial

\let\o=\omega

\let\th=\theta

\let\b=\beta

\numberwithin{equation}{section}
\numberwithin{theorem}{section}

\allowdisplaybreaks[3]

\begin{document}
\date{}
\title[Diffusive Limit of  the Vlasov--Poisson--Boltzmann system]{
\large{Diffusive Limit of the Vlasov--Poisson--Boltzmann System for the Full Range of Cutoff Potentials}}

\author{Weijun Wu$^*$, Fujun Zhou$^\dagger$ and Yongsheng Li$^\ddagger$}

\address[Weijun Wu$^*$]{School of Network Information Security, Guangdong Police College, Guangzhou 510230, China}
\email{scutweijunwu@qq.com}

\address[Fujun Zhou$^\dagger$, Corresponding author]{School of Mathematics, South China University of Technology, Guangzhou 510640, China}
\email{fujunht@scut.edu.cn}

\address[Yongsheng Li$^\ddagger$]{School of Mathematics, South China University of Technology, Guangzhou 510640, China}
\email{yshli@scut.edu.cn}

%
%
%
%
%
%

\begin{abstract}
Diffusive limit of the Vlasov--Poisson--Boltzmann system with cutoff soft potentials $-3<\gamma<0$ in the perturbative framework
around global Maxwellian still remains open.
By introducing a new weighted $H_{x,v}^2$--$W_{x,v}^{2, \infty}$ approach with time decay, we solve this problem
for the full range of cutoff potentials $-3<\gamma\leq 1$. The core of this approach lies in the interplay between the velocity
weighted $H_{x,v}^2$ energy estimate with time decay and the time-velocity weighted $W_{x,v}^{2,\infty}$ estimate with time decay
for the Vlasov--Poisson--Boltzmann system, which leads to the uniform estimate with respect to the Knudsen number $\e\in (0,1]$ globally in time.
As a result, global strong solution is constructed and incompressible Navier--Stokes--Fourier--Poisson limit is rigorously justified
 for both hard and soft potentials. Meanwhile, this uniform estimate with respect to $\e\in (0,1]$ also
 yields optimal $L^2$ time decay rate and $L^\infty$ time decay rate for the Vlasov--Poisson--Boltzmann
system and its incompressible Navier--Stokes--Fourier--Poisson limit. This newly introduced weighted $H_{x,v}^2$--$W_{x,v}^{2, \infty}$ approach with
 time decay is flexible and robust, as it can deal with both optimal time decay problems and hydrodynamic limit problems
 in a unified framework for the Boltzmann equation as well as the Vlasov--Poisson--Boltzmann system for the full range of cutoff
 potentials. It is also expected to shed some light on the
 more challenging hydrodynamic limit of the Landau equation and the Vlasov--Poisson--Landau system.\\[3mm]

\end{abstract}

%
%
\maketitle

\tableofcontents

\section{Introduction}

\subsection{Description of the Problem}
\hspace*{\fill}

In this paper, we intend to investigate hydrodynamic limit of the Vlasov--Poisson--Boltzmann  (briefly VPB for short)
system in the incompressible diffusive regime
\beq\label{GG1}
 \left\{
\begin{array}{ll}
\displaystyle \e\pt_tF^\e+v\cdot\nabla_xF^\e- \e\nabla_x\phi^\e\cdot\nabla_v F^\e =\frac{1}{\e}Q(F^\e,F^\e),\,\quad\qquad\qquad     \\[2mm]
\displaystyle -\e\Delta_x\phi^\e(t,x)=\int_{\mathbb{R}^3}F^\e \d v-1, \quad
\lim_{|x|\rightarrow \infty}\phi^\e(t,x)=0, \quad\qquad\qquad\\[2mm]
\displaystyle F^\e(0,x,v)=F^\e_0(x,v),\quad\qquad\qquad
\end{array}
\right.
\eeq
which describes the dynamics of an electron gas in a constant ion background.
More precisely, $F^{\e}(t,x,v)\geq 0$ represents the density of rarefied gas at time $t\geq0$, velocity $v\in \mathbb{R}^{3}$
and position $x\in \mathbb{R}^{3}$, with initial data $F^\e_0(x,v)$.
The self-consistent electric potential $\phi^\e(t, x)$ is coupled with the density function $F^{\e}(t,x,v)$ through the
Poisson equation in (\ref{GG1}), where the constant ion background charge is scaled to $1$. The Boltzmann collision operator $Q(G_1, G_2)$ describing
the binary elastic collision between particles takes the form
\begin{align*}
  Q(G_1,G_2)(v)&:=\iint_{\bbR^3\times \mathbb{S}^2}
   |v-v_*|^{\gamma}q_0(\theta) \big[G_1(v')G_2(v_*')-G_1(v)G_2(v_*)\big] \d v_*\d \omega,
\end{align*}
where $v'=v-[(v-v_*)\cdot \omega]\omega$ and $v_*'=v_*+[(v-v_*)\cdot \omega]\omega$ by the conservation of momentum and energy
$$
v'+v_*'=v+v_*, \quad |v'|^2+|v_*'|^2=|v|^2+|v_*|^2.
$$
The collision kernel $ |v-v_*|^{\gamma}q_0(\theta)$, depending only on
$|v-v_*|$ and $\cos\theta=\omega\cdot\frac{v-v_*}{|v-v_*|}$, satisfies $-3<\gamma\leq 1$ and the Grad's
angular cutoff assumption
$$
0\leq q_0(\theta)\leq C|\cos\theta|.
$$
The exponent $\gamma$ is determined by the potential of intermolecular forces, and is classified into the
soft potential case when $-3<\gamma<0$ and the hard potential case when $0\leq \gamma\leq 1$. In particular, the hard potential
case includes the hard sphere model with $\gamma=1, q_0(\theta)=C|\cos\theta|$ and the Maxwell model with $\gamma=0$.
For soft potentials, the case $-2\leq \gamma<0$ is called the moderately soft potential, while $-3<\gamma<-2$ is called the
very soft potential \cite{V}.
 The  positive parameter $\e\in (0,1]$ is the Knudsen number, which equals to the ratio of the mean free path to the macroscopic length scale.

 Global existence and long time behavior of solutions to the standard VPB system, i.e. (\ref{GG1})
 with $\e=1$, have been paid a lot of attention. In the framework of renormalized solutions, Lions \cite{Lion1993} constructed
 global large solutions for initial value problem, and Mischler \cite{Mischler2000} established global large solutions for
 initial and boundary value problem. In the perturbative framework  around global Maxwellian, Guo \cite{Guo2001}
 firstly constructed global existence of classical solutions  in the whole space $\mathbb{R}^3$ for moderately soft
 potentials $-2< \gamma< 0$ and Maxwell model $\gamma=0$, and established
 global existence and exponential decay of classical solutions in periodic domain $\mathbb{T}^3$ for hard sphere model $\gamma=1$ in \cite{guo2002cpam,Guo2003}.
 Duan, Yang and Zhao \cite{ DYZ2002, DYZ2003} established global existence and optimal time decay estimate of
 classical solutions in the whole space  $\mathbb{R}^3$ for hard potentials $0\leq\gamma\leq1$ and moderately soft potentials $-2\leq\gamma<0$, respectively.
 Then the analysis was extended to the full range of cutoff soft potentials $-3<\gamma<0$ by \cite{XXZ2017}.
 Recently, Cao, Kim and Lee \cite{CKL2019} established global strong solution and exponential time decay rate for
 diffusive boundary problem with hard sphere model $\gamma=1$, and Li and Wang studied the in-flow boundary problem with soft potentials
 $-3<\gamma<0$. For more related works on the standard VPB system, we refer to
 \cite{DD23,DS11, DY10, DL-13, LYZ2016, wang2013JDE, Yangyu2011CMP, YYZ-06, YZ-06} and the references cited therein.

 Hydrodynamic limit of the VPB system also received an extensive research.
 In the framework of renormalized solutions, Ars\'{e}nio and Saint-Raymond \cite{AS2019} showed that renormalized solutions
 of the VPB system (\ref{GG1}) converged to dissipative solutions of the incompressible
 Navier--Stokes--Fourier--Poisson (NSFP for short) system. In the perturbative framework,
 Guo, Jiang and Luo \cite{GJL2020} justified the incompressible NSFP limit of classical solutions for hard sphere
 model $\gamma=1$, Li, Yang and Zhong \cite{LYZ2020} established decay estimate and incompressible NSFP limit for
 classical solutions by spectral analysis for hard sphere model $\gamma=1$.
 In addition, Guo and Jang \cite{guo2010cmp} studied the compressible Euler--Poisson limit for the
 VPB system in hyperbolic regime by truncated Hilbert expansion \cite{Ca} for hard sphere model $\gamma=1$.
 For more related works on hydrodynamic limit of kinetic equations, we refer
 the readers to \cite{BGL91, BGL93, BU91, CIP, DEL, Esposito2018, GS,   GJJ, GHW21, GX2020, J2009, JL2019, LM, S, WZL, WZL-2}
 and  the references cited therein.

 Unfortunately, hydrodynamic limit of the VPB system (\ref{GG1}) with soft potentials $-3<\gamma<0$ in the perturbative framework
 is still left open. Inspired by the $L^2$--$L^\infty$ approach for weak solutions
 \cite{guo2010arma,guo2010cmp} and the higher order energy approach for classical solutions \cite{guo2002cpam,Guo2003,guo2006},
 we introduce a new weighted $H_{x,v}^2$--$W_{x,v}^{2, \infty}$ approach with time decay for strong solutions to solve this problem
 for both hard potentials $0\leq \gamma\leq 1$ and soft potentials $-3<\gamma<0$.
 The spirit of this approach lies in the interplay between the velocity weighted $H_{x,v}^2$ energy estimate with time decay and the
 time-velocity weighted $W_{x,v}^{2,\infty}$ estimate with time decay for the VPB system (\ref{GG1}), which
 eventually leads to the uniform estimate with respect to the Knudsen number $\e\in (0,1]$ globally in time. Then global existence of
 strong solutions is established and convergence to the incompressible NSFP system is rigorously justified
 for the full range of collision potentials $-3<\gamma\leq 1$. Meanwhile, such uniform estimate with respect to $\e\in (0,1]$ also
 yields the optimal $L^2$ time decay rate as well as the optimal $L^\infty$ time decay rate for the VPB system (\ref{GG1}) and its
 incompressible NSFP limit for  the full range of collision potentials, that is, $(1+t)^{-\frac{3}{4}}$ time decay rate
 in $L^2$ space and $(1+t)^{-\frac{5}{4}}$ time decay rate in $L^\infty$ space.


\subsection{Notations}
\hspace*{\fill}

Before stating the main results of this paper, we need some notations.

 For notation simplicity, we use $X \lesssim Y$ to denote $X \leq CY$, where $C$ is
a constant independent of $X$, $Y$ and $\e$. We also use the notation $X\sim Y$ to represent $X\lesssim Y$
and $Y\lesssim X$. The notation $X \ll 1$ means that $X$ is a constant small enough.
Throughout the paper, $C$ denotes a generic positive constant independent of $\e$.
The multi-indexs $\alpha = [\alpha_1, \alpha_2, \alpha_3]\in\mathbb{N}^3$ and
$\b = [\b_1, \b_2, \b_3]\in\mathbb{N}^3$
 will be used to record spatial and velocity derivatives, respectively.
 We  denote the $\alpha$-th order
 space partial derivatives by $\pt_x^\alpha=\partial_{x_1}^{\alpha_1}\partial_{x_2}^{\alpha_2}\partial_{x_3}^{\alpha_3},$
 and the $\b$-th order
 velocity partial derivatives by $\pt_v^\b=\partial_{v_1}^{\b_1}\partial_{v_2}^{\b_2}\partial_{v_3}^{\b_3}.$
 And  $\partial^\a_\b=\partial^\a_x\partial^\b_v=\partial^{\a_1}_{x_1}\partial^{\a_2}_{x_2}\partial^{\a_3}_{x_3}
 \partial^{\b_1}_{v_1}\partial^{\b_2}_{v_2}\partial^{\b_3}_{v_3}$ stands for the mixed space-velocity derivative.
In addition, the notation $\nabla_x^k=\partial^\a_x$ will be used when $\b=0, |\a|=k$. The length of $\a$
 is denoted by $|\a|=\a_1+\a_2+\a_3.$

We use $\|\cdot\|_{L^q_{x,v}}$  to denote the $L_{x,v}^q$-norm on the variables
$(x, v) \in \mathbb{R}^3\times \mathbb{R}^3$ for $1\leq q \leq \infty$.  For an integer $m \geq 1$, we also
employ $\|\cdot\|_{W^{m,q}_{x,v}}$ to represent  the norm of the Sobolev space
$W^{m,q}(\mathbb{R}^3\times \mathbb{R}^3)$. In particular, if $q=2$ then $\|\cdot\|_{H^{m}_{x,v}}$ stands for
the norm of $H^{m}(\mathbb{R}^3\times \mathbb{R}^3)=W^{m,2}(\mathbb{R}^3\times \mathbb{R}^3)$. Similarly,
$\|\cdot\|_{L^q_{x}}$,  $\|\cdot\|_{H^{m}_{x}}$,  $\|\cdot\|_{W^{m,q}_{x}}$ and
 $\|\cdot\|_{H^{m}_{x}L^2_v}$  stand for
 the norms of the function spaces $L^q(\mathbb{R}^3_x)$, $H^{m}(\mathbb{R}^3_x)$,
 $W^{m,q}(\mathbb{R}^3_x)$ and $L^{2}(\mathbb{R}^3_v,H^m(\mathbb{R}^3_x))$, respectively.
Let $\langle \cdot\;,\;\cdot\rangle_{L^2_{x}}$  and $\langle \cdot\;,\;\cdot\rangle_{L^2_{x, v}}$ be
 the standard  inner product in $L^2_x$ and $L^2_{x, v}$, respectively.
We also define the weighted  $L^2_{x,v}$, $H^1_{x}L^2_{v}$   and $H^2_{x,v}$ space endowed with the norms
\bals
\|g\|^2_{L^2_{x,v}(\nu)}&= \iint_{\mathbb{R}^3\times \mathbb{R}^{3}}\nu(v)|g(x,v)|^2\dd x \dd v,\quad
\|g\|^2_{H^1_{x}L^2_{ v}(\nu)}= \sum_{|\a|\leq1}\|\partial^\a_x g\|^2_{L^2_{x,v}(\nu)},\\
\|g\|^2_{H^2_{x,v}(\nu)}&=  \sum_{|\a|+|\b|\leq2}\|\partial^\a_\b g\|^2_{L^2_{x,v}(\nu)},
\eals
respectively, where $\nu=\nu(v)$ will be given in (\ref{nu-def}).
Moreover,
define the Fourier transform of $f$ with respect to the space variable  $x$ as
\[\widehat{f}(k):=  {(2\pi)^{-\frac{3}{2}}} \int_{\bbR^3}f(x)e^{-ix\cdot k}\d x.\]

\subsection{Main Results}
\hspace*{\fill}

To investigate hydrodynamic limit and optimal time decay rate of global strong solutions $(F^{\e}, \phi^{\e})$ to the VPB system
(\ref{GG1}), we write
\beq\label{eq:zhankaif}
\begin{split}
&F^{\e}(t,x,v)=\mu+\e\sqrt{\mu}f^{\e}(t,x,v),
\\  &F^{\e}_0(x,v)=\mu+\e\sqrt{\mu}f^{\e}_0(x,v)
\end{split}
\eeq
with the fluctuations $f^{\e}$ and $f^{\e}_0$, where
\begin{equation*}
 \mu\equiv \mu(v):=(2\pi)^{-\frac{3}{2}}{e}^{-\frac{|v|^2}{2}}
\end{equation*}
is the global Maxwellian.
 By substituting (\ref{eq:zhankaif}) into (\ref{GG1}),
we get the following equivalent perturbation VPB system

\beq\label{eq:f}
 \left\{
\begin{array}{ll}
\displaystyle \pt_t f^\e +\frac{1}{\e}v\cdot \nabla_x f^\e+\frac{1}{\e}v\cdot\nabla_x\phi^\e\sqrt{\mu}   +\frac{1}{\e^2}Lf^\e
=\frac{\nabla_x \phi^\e\cdot \nabla_v( \sqrt{\mu}f^\e )}{\sqrt{\mu}}
 + \frac{1}{\e}\Gamma(f^\e,f^\e),\,     \\[2mm]
\displaystyle -\Delta_x\phi^\e= \int_{\mathbb{R}^3}\sqrt{\mu}f^\e \d v,\quad\lim_{|x|\rightarrow \infty}\phi^\e(t,x)=0, \\[2mm]
\displaystyle f^\e(0,x,v)=f^\e_0(x,v).
\end{array}
\right.
\eeq
 Here
$$
\Gamma (f,g):=\frac{1}{\sqrt{\mu}}Q(\sqrt{\mu}f,\sqrt{\mu}g)
$$
is the nonlinear Boltzmann collision operator,
\bals
&L f:=-\frac{1}{\sqrt{\mu}}\Big[Q(\mu,\sqrt{\mu}f)+Q(\sqrt{\mu}f,\mu)\Big]:=\nu f -Kf,
\eals
 is the linearized Boltzmann operator
 and
\begin{equation}\label{nu-def}
\begin{split}
\nu\equiv \nu(v)&:=\frac{1}{\sqrt{\mu}}Q_-(\sqrt{\mu},\mu)=
\iint_{\mathbb{R}^3\times\mathbb{S}^2}|v-v_*|^{\gamma} q_0(\theta)\mu(v_*)\mathrm{d}\omega \mathrm{d}v_*,\\
Kf&:=\frac{1}{\sqrt{\mu}}\Big[Q_+(\mu,\sqrt{\mu}f)
+Q_+(\sqrt{\mu}f,\mu)-Q_-(\mu,\sqrt{\mu}f)\Big],
\end{split} \end{equation}
where $\nu=\nu(v)$ is called the collision frequency and $K$ is a compact operator on $L^2(\mathbb{R}_v^3)$.
 The null space of $L$, which we denote by $\mathrm{N}(L)$, is a five-dimensional subspace of $L^2(\mathbb{R}^3_v)$
$$
\mathrm{N}(L) = \text{span} \Big \{ \sqrt{\mu},  \ v\sqrt{\mu} ,  \  \frac{|v|^{2}-3}{2}\sqrt{\mu} \Big \}.
$$
 The orthogonal projection of $f$ onto the null space $\mathrm{N}(L)$ is denoted by
\begin{equation}\label{Pabc}
\mathbf{P}f \  = \  \Big(\rho_f   + u_f\cdot v   + \th_f \frac{|v|^{2}-3}{2} \Big)\sqrt{\mu},
 \end{equation}
 where $\rho_f, u_f, \th_f$ represent the hydrodynamic density, velocity and temperature
 of $f$, respectively. Let $(\mathbf{I}-\mathbf{P})f=f-\mathbf{P} f$ stand for the projection on the
 orthogonal complement of $\mathrm{N}(L)$. As usual, $\mathbf{P} f$ is called the macroscopic
part of $f$, and $(\mathbf{I}-\mathbf{P})f$ is called the microscopic part of $f$.
It is well-known that there exists a positive constant
 $\sigma_0>0$ such that
\begin{equation}\label{spectL}
\langle Lf, f\rangle_{L_{x,v}^2}\ge\sigma_0\|(\mathbf{I}-\mathbf{P})
f\|_{L_{x,v}^2(\nu)}^2.
\end{equation}

We introduce a velocity weight function for both hard and soft potentials
 \bal\label{weight-w}
 w\equiv w(v):=
 \begin{cases}
(1+|v|^2)^{\frac{1}{2}}& \text{for hard potentials}\; 0\leq \gamma\leq 1,\\[2mm]
(1+|v|^2)^{\frac{\gamma}{2}}& \text{for soft potentials}\; -3<\gamma <0.
\end{cases}
 \eal
Moreover, to overcome the singularity brought by the electric field force term
in the forthcoming $W^{2,\infty}_{x,v}$-estimate, we introduce the following time-velocity weight
for the full range of collision potentials
\beq\label{wegiht:express:2}
w_{\widetilde{\vartheta}}\equiv w_{\widetilde{\vartheta}}(v):=e^{\widetilde{\vartheta}|v|^2},
\eeq
where
\beq\label{wegiht:express:1}
\widetilde{\vartheta} \equiv\widetilde{\vartheta}(t):=\vartheta\Big(1+\frac{1}{(1+t)^{\sigma}}\Big)
\eeq
 for some constants $\vartheta> 0$ and $\sigma > 0$ to be determined later.

To state our first main result, we introduce the following temporal energy functional for the equivalent VPB system
(\ref{eq:f}) with hard potentials $0\leq \gamma\leq 1$:
\beq\label{def-energy-R-hard-sum}
\bsp
 \interleave f^\e(t) \interleave^\mathbf{h}:=\;&\| f^\e(t)\|_{H^2_{x,v}}+ \| \nabla_x \phi^\e(t)\|_{H^2_x}+\|w f^\e (t)\|_{H^1_{x}L^2_{v}},
 \\
\esp
\eeq
where $w$ is the weight function defined in (\ref{weight-w}).
 Then we present  the first main result on uniform estimate and time decay rate of global strong solutions to
the perturbation VPB system (\ref{eq:f}) for hard potentials $0\leq \gamma\leq 1$, as well as its hydrodynamic limit to the  incompressible NSFP system.

\begin{theorem}[Hard potentials]   \label{mainth2} \
Let $0\leq \gamma\leq 1$, $0<\e \leq 1$, $w_{\widetilde{\vartheta}}$ be defined in \eqref{wegiht:express:2},
and $\widetilde{\vartheta}$ be defined in \eqref{wegiht:express:1} with $0<\vartheta\ll 1$ and $\displaystyle0<\sigma\leq\frac{1}{4}$.
Assume that
$\iint_{\mathbb{R}^3\times\mathbb{R}^3}\sqrt{\mu}f^\e_0\d x\d v=0$
and there is  a sufficiently small constant $\widetilde{\delta}_0>0$ independent of $\e$ such that
\beq
\bsp
\;&\left\| f^\e_0\right\|_{H^2_{x,v}}+\left\| wf^\e_0\right\|_{H^1_{x}L^2_{v}}
 +\left\|\left(1+|x|\right)f^\e_0\right\|_{L^2_vL^1_x}
 +\e^\frac{1}{2}\left\| w_{\widetilde{\vartheta}} f^\e_0\right\|_{W^{2,\infty}_{x,v}}\leq \widetilde{\delta}_0.
\esp
\eeq
Then the perturbation VPB system \eqref{eq:f} admits a unique global strong solution
$\left(f^{\e},\nabla_{x}\phi^\e\right)$  satisfying
\bal
\bsp\label{eq:theorem2}
&\sup_{0\leq t\leq\infty}\Big\{\left(1+t\right)^{\frac{3}{4}}\interleave f^\e (t) \interleave^\mathbf{h}
+\left(1+t\right)^{\frac{5}{4}}\left\|\nabla_x f^\e(t)\right\|_{   H^1_{x}L^2_{v}}
+\left(1+t\right)^{\frac{5}{4}}\left\|\nabla_x^2\phi^\e(t)\right\|_{L^2_{x}}\Big\}\\
&+\sup_{0\leq t\leq\infty}\Big\{\left(1+t\right)^{\frac{5}{4}}
\e^{\frac{1}{2}}\left\| \left( w_{\widetilde{\vartheta} } f^\e\right)(t)\right\|_{W_{x,v}^{1,\infty}} +\left(1+t\right)^{\frac{5}{4}}\e^{\frac{3}{2}}\left\| \left( w_{\widetilde{\vartheta} } f^\e\right)(t)\right\|_{W_{x,v}^{2,\infty}} \Big\}
\leq C \widetilde{\delta}_0,
\esp
\eal
where $C>0$ is a positive constant independent of $\e$.

Moreover, assume that
there exist functions $\rho_0(x), u_0(x), \th_0(x)\in H^2_x$ such that
\bal\label{limit:initial:hard}
f_0^\e\rightarrow f_0 \;\;\text{strongly in}\;  H^{2}_{x,v} \;\text{ as }\;  \e\rightarrow 0,
\eal
 where $f_0(x, v)$ is of the form
\bals
f_0(x,v)= \Big(\rho_0(x)+u_0(x)\cdot v+\theta_0(x)\frac{|v|^2-3}{2}\Big) \sqrt{\mu}.
\eals
Then  as $\e\rightarrow 0$, there hold
\begin{equation}\label{converg-1}
\begin{split}
&f^\e\rightarrow \Big(\rho+u\cdot v+\theta\frac{|v|^2-3}{2}\Big)\sqrt{\mu}\;
\text{ weakly-}* \text{in } t\ge 0, \text{weakly in } H^2_{x,v},
\\
&\nabla_x\phi^\e\rightarrow \nabla_x\phi\quad
\text{~~~ weakly-}* \text{in } t\ge 0, \text{weakly in } H^2_{x},
\text{ strongly in } C(\mathbb{R}^+; H^{1}_{loc}(\bbR^3_x)),
\end{split}
\end{equation}
 where $\left(\rho, u, \th, \nabla_x\phi\right)\in
  L^\infty\left(\mathbb{R}^+; H^{2}(\bbR^3_x)\right)
  \cap C\left(\mathbb{R}^+; H^{1}_{loc}(\bbR^3_x)\right)
$
 and satisfies the incompressible NSFP system
\beq\label{INSFP}
 \left\{
\begin{array}{ll}
\displaystyle  \pt_t u+u\cdot\nabla_x u-\lambda\Delta_x u+\nabla_x p = \rho\nabla_x\th,\qquad\nabla_x\cdot u=0, \,\quad\qquad\qquad     \\[2mm]
\displaystyle  \pt_t\Big(\dfrac{3}{2}\th-\rho\Big)+u\cdot\nabla_x\Big(\dfrac{3}{2}\th-\rho\Big)
-\dfrac{5}{2}\kappa\Delta_x\th=0,\\[2mm]
\displaystyle   \Delta_x(\rho+\th)=\rho, \quad -\nabla_x\phi=\nabla_x(\rho+\th), \quad\qquad\qquad \\[2mm]
\displaystyle  u(0,x)=\mathcal{P}u_0(x),\quad \th(0,x)=\th_0(x),\quad \rho(0,x)=\rho_0(x)\qquad\qquad
\end{array}
\right.
\eeq
in the sense of distributions, where $\mathcal{P}$ stands for the Leray projection, and $\kappa,$ $\lambda$ and
$p$ represent the heat conductivity coefficient, the viscosity coefficient and the pressure, respectively.
\end{theorem}
\medskip

 Then we turn to consider the more challenging soft potential case $-3<\gamma<0$.
  For this, we introduce a temporal energy functional for the VPB system (\ref{eq:f})  for the soft potential case $-3<\gamma<0$:
\beq\label{def-energy-R-soft-sum}
\bsp
\interleave f^\e (t)\interleave_{\ell}^\mathbf{s}:=\;&\sum_{|\a|+|\b|\leq2}\left\| w^{|\b|}\partial^{\a}_{\b} f^\e(t)\right\|_{L^2_{x,v}}
+ \left\| \nabla_x \phi^\e(t)\right\|_{H^2_x}
+\sum_{\substack{|\a|+|\b|\leq2\\0\leq|\a|\leq 1}}\left\| w^{|\b|-\ell}\partial^{\a}_{\b} f^\e(t)\right\|_{L^2_{x,v}}\\
\;&+\e^{\frac{1}{2}}\left\| {w}^{-\ell}\nabla_x^2f^\e(t)\right\|_{L^2_{x,v}},
\esp
\eeq
 where the weight  function $w$ is defined in \eqref{weight-w} and $\displaystyle\ell\geq\frac{1}{2}$ is a constant to be determined later.
 Note that the weight $w^{|\b|-\ell}$ depends on the order of velocity derivative $\partial^\beta_v$ and
 ${w}^{-\ell}=(1+|v|^2)^{-\frac{\ell\gamma}{2}}\geq 1$ for the soft potential case $-3<\gamma<0$.
 Moreover, the last term in (\ref{def-energy-R-soft-sum}) on the second order space derivative is treated particularly with an extra
 $\displaystyle \e^{\frac{1}{2}}$, due to the singularity brought by $\displaystyle \e^{-2}L(\mathbf{I-P})f^\e$ in weighted
 $H^2_{x,v}$ energy estimate, cf. the proof of Lemma \ref{prop:weighted:f:H^2} in Section \ref{energy-estimate}.
 Then we  state our second main result on the uniform estimate and time decay rate of global strong solutions to
 the perturbation VPB system (\ref{eq:f}) for soft potentials $-3<\gamma<0$.

\begin{theorem}[Soft potentials]    \label{mainth1} \
Let $-3<\gamma<0$, $0<\e \leq 1,$  $\ell_{0}>\frac{3}{4}$, $\ell\geq \max\big\{1,-\frac{1}{\gamma}\big\}+\frac{1}{2}+\ell_{0},$
$w_{\widetilde{\vartheta}}$ be defined in \eqref{wegiht:express:2}, and $\widetilde{\vartheta}$ be defined in \eqref{wegiht:express:1} with
$0<\vartheta\ll 1$  and $0<\sigma\leq\frac{1}{24}$.
Assume that
$\iint_{\mathbb{R}^3\times\mathbb{R}^3}\sqrt{\mu}f^\e_0\d x\d v=0$
and  there is  a sufficiently small constant $\delta_0>0$ independent of $\e$ such that
\beq
\bsp\label{eq:theorem1:initial:2}
\;&\sum_{|\a|+|\b|\leq2}\left\| w^{|\b|}\partial^{\a}_{\b} f^\e_0\right\|_{L^2_{x,v}}
+\sum_{\substack{|\a|+|\b|\leq2\\0\leq|\a|\leq1}}
\left\| w^{|\b|-(\ell+1)}\partial^{\a}_{\b} f^\e_0\right\|_{L^2_{x,v}}
+\e^\frac{1}{2}\left\| {w}^{-(\ell+1)}\nabla_x^2f^\e_0\right\|_{L^2_{x,v}}\\
\;&
 +\left\|\big(1+|x|\big)f^\e_0\right\|_{L^2_vL^1_x}
 +\e^{\frac{1}{2}}\big\|\langle v\rangle^{-\gamma(\frac{1}{2}+\ell_{0})}f^\e_0\big\|_{L^2_vL^1_x}
 +\e^\frac{1}{2}\left\| w_{\widetilde{\vartheta}} f^\e_0\right\|_{W^{2,\infty}_{x,v}}\leq \delta_0.
\esp
\eeq
Then the perturbation VPB system \eqref{eq:f} admits a unique global strong solution
$(f^{\e},\nabla_{x}\phi^\e)$  satisfying
\bal
\bsp\label{eq:theorem1}
&\sup_{0\leq t\leq\infty}\Big\{\left(1+t\right)^{\frac{3}{4}}\interleave f^\e (t) \interleave_{\ell}^\mathbf{s}
+\left(1+t\right)^{\frac{5}{4}}\left\|\nabla_x f^\e(t)\right\|_{   H^1_{x}L^2_{v}}
+\left(1+t\right)^{\frac{5}{4}}\left\|\nabla_x^2\phi^\e(t)\right\|_{L^2_{x}}\Big\}\\
&+\sup_{0\leq t\leq\infty}\Big\{\sum_{|\a|+|\b|\leq 1}\left(1+t\right)^{\frac{5}{4}-\frac{5}{8}|\b|}
\e^{\frac{1}{2}+\frac{1}{5}|\b|}\left\| \partial_{\b}^{\a}(  w_{\widetilde{\vartheta} } f^\e)(t)\right\|_{L_{x,v}^{\infty}}\Big\}\\
&+\sup_{0\leq t\leq\infty}\Big\{\sum_{|\a|+|\b|= 2}\left(1+t\right)^{\frac{5}{4}-\frac{5}{8}|\b|}
\e^{\frac{3}{2}+\frac{1}{5}|\b|}\left\| \partial_{\b}^{\a}( w_{\widetilde{\vartheta} } f^\e)(t)\right\|_{L_{x,v}^{\infty}}\Big\}
\leq C\delta_0,
\esp
\eal
where $C>0$ is a positive constant independent of $\e$.

Furthermore,  if the initial  condition \eqref{limit:initial:hard} of $f^\e_0$  holds, then
 the global solution $\left(f^\e,\nabla_x\phi^\e \right)$ of the perturbation VPB system \eqref{eq:f}  converges to
the solution of the incompressible NSFP system \eqref{INSFP}
in the sense of \eqref{converg-1}.
\end{theorem}

\begin{remark}\label{remark-1.6}
This newly introduced $H_{x,v}^2$--$W_{x,v}^{2, \infty}$ approach with time decay actually provides the time decay estimate
for the solution $(f^\e, \nabla_x\phi^\e)$ uniformly in $\e\in (0,1]$ for the full range of collision potentials $-3<\gamma\leq 1$, namely,
\bals
&\|f^\e(t)\|_{L^2_{x,v}}+\|\nabla_x \phi^\e(t)\|_{L^2_x}\lesssim (1+t)^{-\frac{3}{4}},\\
&\|\nabla_xf^\e(t)\|_{H^1_xL^2_v}+\|\nabla_x^2\phi^\e(t)\|_{H^1_x}+\e^{\frac{1}{2}}\| w_{\widetilde{\vartheta}} f^\e(t)\|_{L_{x,v}^{\infty}} \lesssim(1+t)^{-\frac{5}{4}},
\eals
which are optimal in the sense of Duhamel's principle. This time decay rates covers the results of \cite{DYZ2002, DYZ2003, XXZ2017},
which constructed the same optimal $L^2$ time decay rate of the standard VPB system separately for hard potentials $0\leq \gamma\leq 1$, moderately soft potentials $-2\leq \gamma<0$ and
full soft potentials $-3<\gamma<0$ in higher order energy space $H^N_{x,v}$ with $N\geq 4$.

In addition, as the hydrodynamic limit of the solution $\left(f^\e,\nabla_x\phi^\e \right)$
 to the VPB system \eqref{eq:f} in the sense of \eqref{converg-1} and \eqref{limit:10}, the solution
$(\rho,u,\th,\nabla_x \phi)$ to the incompressible NSFP system \eqref{INSFP} actually inherits the time decay rate, that is,
\bals
\bsp
\big\|(\rho,u,\th,\nabla_x \phi)(t)\big\|_{L^2_x}\lesssim (1+t)^{-\frac{3}{4}},\quad \big\|\nabla_x(\rho,u,\th,\nabla_x\phi)(t)\|_{H^1_x}\lesssim (1+t)^{-\frac{5}{4}},
\esp
\eals
which is consistent with our latest  study  for the incompressible NSFP system \eqref{INSFP} by spectral analysis and
high-low frequency decomposition \cite{GZW-2021}.
\end{remark}

\begin{remark}\label{remark-1.3}
 The weighted $H_{x,v}^2$--$W_{x,v}^{2, \infty}$ approach with time decay integrates
 the advantages of both the $L^2$--$L^\infty$ approach for weak solutions \cite{guo2010arma,guo2010cmp} and the higher order
 energy approach for classical solutions \cite{guo2002cpam,Guo2003}. In fact, this $H_{x,v}^2$--$W_{x,v}^{2, \infty}$ workspace
 is actually the space with the lowest regularity to solve the hydrodynamic limit problem of  the VPB system \eqref{eq:f} in the perturbative framework, and
 spaces with lower regularity like $L^2$--$L^\infty$ or $H^1_{x,v}$--$W^{1,\infty}_{x,v}$ fail due to the singularity brought by the nonlinear term
 $\e^{-1}\Gamma(f^\e, f^\e)$ and $(v\cdot\nabla_x\phi^\e) f^\e$. This is different from the truncated Hilbert expansion approach, cf. \cite{GZW-2021, guo2010cmp}.
 On the other hand, compared with the higher order energy approach, our weighted $H_{x,v}^2$--$W_{x,v}^{2, \infty}$ approach
 needs less regularity on the initial data.

 This weighted $H_{x,v}^2$--$W_{x,v}^{2, \infty}$ approach with time decay is flexible and robust, as it can deal with optimal time decay rate
 problems uniformly in $\e\in (0,1]$ and hydrodynamic limit problems when $\e\rightarrow 0$, in a single unified theory for both
 the Boltzmann equation and the VPB system. It is also expected to shed some light on the more challenging hydrodynamic limit
 of the Landau equation and the Vlasov--Poisson--Landau system.

\end{remark}

\begin{remark}\label{remark-initial}
 The initial conditions $\iint_{\mathbb{R}^3\times\mathbb{R}^3}\sqrt{\mu}f^\e_0\d x\d v=0$ and
 $\left\|\big(1+|x|\big)f^\e_0\right\|_{L^2_vL^1_x}<\infty$ in Theorem \ref{mainth2} and Theorem
 \ref{mainth1} are used to remove the singularity of the Poisson kernel and recover the optimal rate index
 $(1+t)^{-\frac{3}{4}}$, similarly as the standard VPB system \cite{DS11}.

 For the soft potential case, since the dissipation rate \eqref{def-disspation-R} is much weaker than the instant energy \eqref{def-energy-R},
 we have to put the extra velocity weight $\langle v\rangle^{-\gamma(\frac{1}{2}+\ell_{0})}$
 in the initial data $\e^{\frac{1}{2}}\|\langle v\rangle^{-\gamma(\frac{1}{2}+\ell_{0})}f^\e_0\|_{L^2_vL^1_x}$
 in \eqref{eq:theorem1:initial:2} to establish the desired time decay estimate
 for the nonhomogeneous linearized VPB system \eqref{linear:esti:f:decay:1}, see the proof of Lemma \ref{result:linear:estimate:decay}.
 Moreover, when considering the time decay rate for the nonlinear VPB system \eqref{eq:f},
 we need the the extra velocity weight $\langle v\rangle^{-\gamma} =w^{-1}$ in the initial data $\sum_{\substack{|\a|+|\b|\leq2\\0\leq|\a|\leq1}}
\left\| w^{|\b|-(\ell+1)}\partial^{\a}_{\b} f^\e_0\right\|_{L^2_{x,v}}$ and $\e\left\| {w}^{-(\ell+1)}\nabla_x^2f^\e_0\right\|_{L^2_{x,v}}$ in
\eqref{eq:theorem1:initial:2} to apply the time-weighted estimate on \eqref{main-result:1} and  the iterative
technique \eqref{main-result:2}, see \eqref{decay:soft:result:4} for more details.
\end{remark}

\subsection{Main Ideas and Innovations}
\hspace*{\fill}

In this subsection, we outline the main ideas and innovations proposed in this work.
Our analysis is based on the weighted  $H_{x,v}^2$--$W_{x,v}^{2, \infty}$ approach with time decay,
which consists of the velocity weighted $H_{x,v}^2$ energy estimate with time decay and the
 time-velocity weighted $W_{x,v}^{2,\infty}$-estimate with time decay for the perturbation VPB system (\ref{eq:f}), which
 eventually leads to the uniform estimate with respect to the Knudsen number $\e\in (0,1]$ globally in time.

\subsubsection{Velocity weighted $H_{x,v}^2$ energy estimate with time decay}
\hspace*{\fill}

Our velocity weighted $H_{x,v}^2$ energy estimate contains some new ingredients.
We start with the standard $L^2_{x,v}$ energy estimate for the perturbation VPB system  (\ref{eq:f})
\beqs
\begin{split}
&\frac{1}{2}\frac{\d }{\d t}\Big(\|f^\e\|_{L_{x,v}^2}^2
+\|\nabla_x \phi^\e \|_{L^2_x}^2\Big) +\frac{1}{\e^2}\|(\mathbf{I}-\mathbf{P})f^\e(t)\|^2_{L^2_{x,v}(\nu)}\\
 \lesssim\; &\underbrace{-\frac{1}{2} \iint_{\mathbb{R}^3\times\mathbb{R}^3}  v \cdot\nabla_x \phi^\e|(\mathbf{I}-\mathbf{P})f^\e|^2 \d x\d v }_{:={D}_{1}}+\underbrace{\Big\langle\frac{1}{\e}\Gamma(f^\e,f^\e), (\mathbf{I}-\mathbf{P})f^\e \Big\rangle_{L_{x,v}^2}}_{:={D}_{2}}+\cdots,
\end{split}
\eeqs
where  the semi-positivity  of $L$ in (\ref{spectL}) has been used.
Unfortunately, the term $D_1$ cannot be controlled by the weak dissipation due to the collision frequency
$$\nu\sim \langle v\rangle^\gamma \sim (1+|v|^2)^{\frac{\gamma}{2}} \;\text{ for } -3< \gamma<  1.$$
Rather than using the time-velocity weighted higher order energy approach to provide extra dissipation \cite{DYZ2002,DYZ2003},
we control $D_1$ in a creative manner by weighted $L^\infty_{x,v}$-estimate, that is,
\bal
\bsp\label{D:qkj:esti:f:L^2:8}
{D}_{1}
\leq &\;\int_{\mathbb{R}^3} \big\||v|\langle v\rangle^{{\frac{3}{4}}}\langle v\rangle^{-\frac{\gamma}{2}}
(\mathbf{I}-\mathbf{P})f^\e\big\|_{L^3_v}\big\|\langle v\rangle^{-{\frac{3}{4}}}\big\|_{L^6_v}
\big\|\langle v\rangle^{\frac{\gamma}{2}}
(\mathbf{I}-\mathbf{P})f^\e \big\|_{L^2_v}|\nabla_x \phi^\e |\d x\\
\leq&\;
 \big\||v|\langle v\rangle^{{\frac{3}{4}}}\langle v\rangle^{-\frac{\gamma}{2}}
(\mathbf{I}-\mathbf{P})f^\e\big\|_{L^3_{x,v}}
\left\|
(\mathbf{I}-\mathbf{P})f^\e\right\|_{L^2_{x,v}(\nu)}\|\nabla_x \phi^\e \|_{L^6_x}\\
\leq&\;\Big(\frac{1}{\e}\|(\mathbf{I}-\mathbf{P})f^\e\|_{L_{x,v}^2}+
\e^2\|w_{\vartheta}f^{\e}\|_{L^\infty_{x,v}}\Big)
\|(\mathbf{I}-\mathbf{P})f^\e\|_{L^2_{x,v}(\nu)}\|\nabla_x^2 \phi^\e \|_{L^2_x}.
\esp
\eal
Here we have utilized  the interpolation inequality to  estimate
\bals
\bsp
\left\||v|\langle v\rangle^{{\frac{3}{4}}}\langle v\rangle^{-\frac{\gamma}{2}}
(\mathbf{I}-\mathbf{P})f^\e\right\|_{L^3_{x,v}}
\leq\;&\Big(\frac{1}{\e}\left\|
(\mathbf{I}-\mathbf{P})f^\e\right\|_{L^2_{x,v}}\Big)^\frac{2}{3}
\Big(\e^2\big\||v|^3\langle v\rangle^{{\frac{9}{4}}}\langle v\rangle^{-\frac{3\gamma}{2}}
(\mathbf{I}-\mathbf{P})f^\e\big\|_{L^\infty_{x,v}}\Big)^\frac{1}{3}\\
\lesssim\;&
\frac{1}{\e}\|(\mathbf{I}-\mathbf{P})f^\e\|_{L_{x,v}^2}+
\e^2\|w_{\vartheta}f^{\e}\|_{L^\infty_{x,v}}
\esp
\eals
for the exponential weight $w_{\vartheta}=e^{\vartheta|v|^2}\gtrsim |v|^3\langle v\rangle^{{\frac{9}{4}}}\langle v\rangle^{-\frac{3\gamma}{2}}$ with $0<\vartheta \ll 1$,
which holds for the full range of cutoff potentials $ -3< \gamma\leq 1$.
Furthermore,  decomposing   $D_2$ into
\bals
\bsp
D_2  =\;&
 \Big\langle{\frac{1}{\e}\Gamma(w_{\vartheta}^{-1} w_{\vartheta}\mathbf{P}f^\e,f^\e)}, (\mathbf{I}-\mathbf{P})f^\e \Big\rangle_{L_{x,v}^2}+\Big\langle{\frac{1}{\e}\Gamma(\mathbf{(I-P)}f^\e,
 w_{\vartheta}^{-1} w_{\vartheta}f^\e}), (\mathbf{I}-\mathbf{P})f^\e \Big\rangle_{L_{x,v}^2}
 \esp
 \eals
and taking $L^\infty_{x,v}$ norm on $w_{\vartheta}\mathbf{P}f^\e$ and $w_{\vartheta}f^\e$,
we are able to treat $D_2$ as
\beq
\bsp\label{D:Nonlinear:term:proof:1}
D_2
\lesssim
&\;\frac{1}{\e}\| (\mathbf{I}-\mathbf{P})f^\e\|_{L_{v}^2(\nu)}\Big(
\| \mathbf{P}f^\e\|_{L_{x}^\infty L_{v}^2}\|  f^\e\|_{L_{x,v}^2(\nu)} +
\e\| w_\vartheta f^\e\|_{L_{x,v}^\infty}\frac{1}{\e}\| (\mathbf{I}-\mathbf{P})f^\e\|_{L_{x,v}^2(\nu)}
 \Big),
\esp
\eeq
see the proof of Lemma  \ref{es-energy-Nonlinear}.
Here we have used  the well-known result $\|w_{\vartheta}\mathbf{P}f^\e\|_{L^\infty_v}\lesssim \|\mathbf{P}f^\e\|_{L^2_v}$ from the exponential decay
factor in  $\mathbf{P}f^\varepsilon$.  Consequently, motivated by the above  uniform estimates \eqref{D:qkj:esti:f:L^2:8}
 and \eqref{D:Nonlinear:term:proof:1}, we believe that it is hopeful to close the standard energy estimate with the help of  the
 weighted $L^\infty_{x,v}$-estimate, rather than the weighted higher order  energy approach \cite{DYZ2002,DYZ2003}.

 It should be mentioned that lower regularity like $L^2_{x,v}$--$L^\infty_{x,v}$ framework or $H^1_{x,v}$--$W^{1,\infty}_{x,v}$ framework cannot
 solve the hydrodynamic limit of the VPB system (\ref{eq:f}) in the perturbative framework, due to the singularity brought by the nonlinear
 collision term $\e^{-1}\Gamma(f^\e, f^\e)$. This is different from the nonlinear term $\e^{k-2}\Gamma(f^\e,f^\e)$ ($k\geq 2$) in
 truncated Hilbert expansion approach, cf. \cite{GZW-2021, guo2010cmp}. Precisely speaking, for the energy estimate of
 $\|\nabla_xf^\e\|_{L^2_{x,v}}$, we deal with the difficult term related to   nonlinear collision term $\e^{-1}\Gamma(f^\e, f^\e)$ by
\bals
 \Big\langle{\frac{1}{\e}\Gamma(\nabla_x\mathbf{P}f^\e,f^\e)}, (\mathbf{I}-\mathbf{P})\nabla_xf^\e \Big\rangle_{L_{x,v}^2}
 \lesssim
\|\nabla_x\mathbf{P} f^\e\|_{L_{x}^6L_{v}^\infty } \| f^\e\|_{L_{x}^3L_{v}^2(\nu)}\frac{1}{\e}\| (\mathbf{I}-\mathbf{P})\nabla_xf^\e\|_{L_{x,v}^2(\nu)},
\eals
which indicates that this term will be controlled once the dissipation estimate of $\|\nabla_x^2\mathbf{P} f^\e\|_{L^2_{x,v}}$ is added into our analysis,
see Lemma \ref{es-energy-Nonlinear} for more details. Correspondingly, to   pair with the  dissipation   $\|\nabla_x^2\mathbf{P} f^\e\|_{L^2_{x,v}}$,  the energy estimate of $\|\nabla_x^2 f^\e\|_{L^2_{x,v}}$   should be taken  into consideration.
To conclude, we adopt the $H^2_{x,v}$ energy space to treat the perturbation VPB system (\ref{eq:f}) as below
\beqs
 \left\{
\begin{array}{ll}
\displaystyle  \sum_{|\a|+|\b|\leq2}\left\|  \partial^{\a}_{\b} f^\e(t)\right\|_{L^2_{x,v}}+ \| \nabla_x \phi^\e(t)\|_{H^2_x}\quad \quad\quad \text{for hard potentials}\; 0\leq \gamma\leq 1,   \\[4mm]
\displaystyle \sum_{|\a|+|\b|\leq2}\left\| w^{|\b|}\partial^{\a}_{\b} f^\e(t)\right\|_{L^2_{x,v}}
+ \left\| \nabla_x \phi^\e(t)\right\|_{H^2_x}\quad \text{for soft potentials}\; -3< \gamma< 0,
\end{array}
\right.
\eeqs
where the $H^2_{x,v}$ energy space for soft potential case involves the weight function $w^{|\beta|}$
as done in \cite{guo2006}, in order to control the velocity derivatives of the streaming term $v\cdot\nabla_x$ by weak dissipation.

Another feature of our analysis is that we design different frameworks for the weighted energy estimate between the soft potential case and hard potential case.
In fact, to deduce the time decay rate of solutions for the soft potential case $-3< \gamma< 0$, the starting point is
the $H^2_{x,v}$ energy inequality
\bals
\bsp
\frac{\d}{\d t} \mathcal{E}^\mathbf{s}(t)
+\frac{1}{\e^2}\sum_{|\a|+|\b|\leq2}\left\| w^{|\b|}\partial^{\a}_{\b} (\mathbf{I-P})f^\e\right\|_{L^2_{x,v}(\nu)}^2
+ \left\| \nabla_x^2 \phi^\e\right\|_{H^2_x}^2+\left\| \nabla_x \mathbf{P}f^\e\right\|_{H^1_{x,v}}^2
\lesssim\cdots,
\esp
\eals
where $\mathcal{E}^\mathbf{s}(t)$ is equivalent   to  $\sum\limits_{|\a|+|\b|\leq2}\left\| w^{|\b|}\partial^{\a}_{\b} f^\e(t)\right\|_{L^2_{x,v}}^2
+ \left\| \nabla_x \phi^\e(t)\right\|_{H^2_x}^2$, cf. Lemma \ref{prop:estimate:dissipation} and  Lemma \ref{prop:f:H^2:1}.
The microscopic dissipation above obviously contains the extra factor $\nu(v)\sim\langle v\rangle^\gamma$, which is degenerate for large
velocity and causes a difficulty in the time decay estimate of $f^\e$.  Luckily, motivated by \cite{DYZ2003},
we  should consider the velocity weighted $H^2_{x,v}$ energy estimate $\left\| w^{|\b|-\ell}\partial^{\a}_{\b} (\mathbf{I-P})f^\e(t)\right\|_{L^2_{x,v}}$ to
deduce the time decay rate for the VPB system (\ref{eq:f}).
In fact, by applying the weighted  energy estimate on the momentum system of $(\mathbf{I}-\mathbf{P})f^\e$, we have
\bal
\bsp\label{I-P:f:eq}
&\;\pt_t(\mathbf{I}-\mathbf{P})f^\e+\frac{1}{\e}v\cdot \nabla_x (\mathbf{I}-\mathbf{P})f^\e+\frac{1}{\e^2}L((\mathbf{I}-\mathbf{P})f^\e)
=\frac{1}{\e}v\cdot \nabla_x  \mathbf{P}f^\e
+\frac{1}{\e}\mathbf{P}(v\cdot \nabla_x  f^\e)+\cdots.
\esp
\eal
We are able to establish the weighted estimate on $\sum\limits_{\substack{|\a|+|\b|\leq2\\0\leq|\a|\leq 1}}\left\| w^{|\b|-\ell}\partial^{\a}_{\b} (\mathbf{I-P})f^\e(t)\right\|_{L^2_{x,v}}$ with lower order space derivatives, rather than $\sum\limits_{\substack{|\a|+|\b|\leq2\\0\leq|\a|\leq 2}}\left\| w^{|\b|-\ell}\partial^{\a}_{\b} (\mathbf{I-P})f^\e(t)\right\|_{L^2_{x,v}}$ with the full order space derivatives.
The reason is that  when applying $\nabla^2_x$ to \eqref{I-P:f:eq} and taking the  $L^2_{x,v}$ estimate with $w^{-2}\nabla^2_x(\mathbf{I-P})f^\e$, a
new trouble  arises because the third order derivative in the term
\bal\label{D:I-P:f:eq:1}
\Big\langle \frac{1}{\e}v\cdot \nabla_x \nabla^2_x  \mathbf{P}f^\e,
w^{-2l}\nabla^2_x(\mathbf{I-P})f^\e \Big\rangle_{L^2_{x,v}}
\eal
exceeds our $H^2_{x,v}$ framework. We solve this difficulty by applying $\nabla^2_x$ directly on the
VPB system (\ref{eq:f}) rather than on the microscopic system (\ref{I-P:f:eq}), and then taking the $L^2_{x,v}$ inner product with
$\e w^{-2}\nabla^2_x f^\e$
\bals
\frac{\e}{2}\frac{\d}{\d t}\Big\| {w}^{-\ell}\nabla_x^2f^\e(t)\Big\|_{L^2_{x,v}}^2+\Big\langle\frac{1}{\e^2}L((\mathbf{I}-\mathbf{P})f^\e),
\e w^{-2}\nabla^2_x f^\e \Big\rangle_{L^2_{x,v}}=\cdots.
\eals
This treatment not only avoids the appearance of the third order derivative \eqref{D:I-P:f:eq:1}, but also absorbs the singularity of $\e^{-2}$ in front of the
linearized Boltzmann operator $L$. In conclusion, we eventually design the following velocity weighted $H^2_{x,v}$ energy estimate for the soft potential
case $-3< \gamma< 0$
\bals
\sum_{\substack{|\a|+|\b|\leq2\\0\leq|\a|\leq 1}}\left\| w^{|\b|-\ell}\partial^{\a}_{\b} (\mathbf{I-P})f^\e(t)\right\|_{L^2_{x,v}}+\e^{\frac{1}{2}}\left\| {w}^{-\ell}\nabla_x^2f^\e(t)\right\|_{L^2_{x,v}}.
\eals

On the other hand, to establish the time decay estimate for the hard potential case $0\leq \gamma \leq 1$, the starting points are the full
$H^2_{x,v}$ energy estimate and the higher order $H^2_{x, v}$ energy estimate
\bals
\bsp
&\frac{\d}{\d t} \mathcal{E}^\mathbf{h}(t)
+\mathcal{E}^\mathbf{h}(t)
\lesssim\|\nabla_x\phi^\e\|_{L^2_x}^2+\|\mathbf{P}f^\e\|_{L^2_{x,v}}^2\cdots,\\
&\frac{\d}{\d t} \widetilde{\mathcal{E}}^\mathbf{h}(t)+
\widetilde{\mathcal{E}}^\mathbf{h}(t)
\lesssim\left\| \nabla_x \mathbf{P}f^\e\right\|_{L^2_{x,v}}^2+\cdots,
\esp
\eals
  see \eqref{qkj:esti:diss:0-hard}, \eqref{11111-hard} and \eqref{11111-hard-high}  for details.
 This means that we should analyze the time decay estimates of $\|\nabla_x\phi^\e\|_{L^2_x}$, $\|\mathbf{P}f^\e\|_{L^2_{x,v}}$ and $\|\nabla_x\mathbf{P}f^\e\|_{L^2_{x,v}}$
 by linear decay estimate and Duhamel's formula
\beqs
\begin{split}
&\|\partial^{\alpha}_{x}\mathbf{P}f^\e\|_{L^2_{x,v}}+\|\partial^{\alpha}_{x}\nabla_x\phi^\e\|_{L^2_{x}} \\
\lesssim &\int_0^t(1+t-s)^{-(\frac{3}{4}+\frac{|\a|}{2})}\|\partial^{\alpha}_{x}\left[{\color{blue} v}\cdot \nabla_x\phi^\e (\mathbf{I}-\mathbf{P})f^\e\right]\|_{L^2_vL^1_x}+\cdots \\
\lesssim &\int_0^t(1+t-s)^{-(\frac{3}{4}+\frac{|\a|}{2})}\|\nabla_x\phi^\e\|_{L^2_x} \|{\color{blue} \langle v \rangle }\partial^{\alpha}_{x}(\mathbf{I}-\mathbf{P})f^\e\|_{L^2_{x,v}}+\cdots \quad \text{for} ~|\a|\leq1,
\end{split}
\eeqs
 see  Lemma \ref{decay:es:nonlinear:hard} for details. Therefore, the energy functional in (\ref{def-energy-R-hard-sum})
 for the hard potential case  $0\leq \gamma \leq 1$ has weight only for the lower order
 derivatives $\| w (\mathbf{I}-\mathbf{P})f^\e(t)\|_{H^1_{x}L^2_{v}}$.


\subsubsection{Time-velocity weighted $W^{2, \infty}_{x,v}$-estimate with time decay}
\hspace*{\fill}

To close the weighted energy estimates in Proposition \ref{main-weighted-energy-estimate-1} and Proposition  \ref{main-weighted-energy-estimate-2},
the time-velocity weighted $W_{x,v}^{2,\infty}$-estimate for $f^\e$ is carried out and some innovations are proposed.

\vspace{1mm}

Firstly, to tackle the $L^\infty_{x,v}$-estimate of $w_{{\vartheta}}f^\e$, we multiply $w_\vartheta$ by the first equation in (\ref{eq:f})
\bals
\bsp
\;&\left\{\pt_t +\frac{1}{\e}v\cdot\nabla_x -\frac{1}{\e}\nabla_x\phi^\e\cdot\nabla_v\right\} (w_\vartheta f^\e)
+\underbrace{\left(\frac{1}{\e^{2}}{\nu(v)}
+\frac{1}{2}{v}\cdot\nabla_x\phi^\e
+2{\vartheta}{v}\cdot\nabla_x\phi^\e\right)}(w_\vartheta f^\e)
\\
=\;&w_\vartheta \left( \frac{1}{\e^2}Kf^\e+\frac{1}{\e}\Gamma(f^\e, f^\e)
-\frac{1}{\e}v\cdot\nabla_x \phi^\e\sqrt{\mu}\right).
\esp
\eals
 We observe that the underbraced term increases linearly in $\langle v\rangle$ and
 cannot conserve positivity, due to the degenerate collision frequency $\langle v\rangle^\gamma$ for the soft potential case $-3<\gamma< 0$.
 To overcome this difficulty, we adopt a time-velocity weighted function $w_{\widetilde{\vartheta}}(v)$ in \eqref{wegiht:express:2}
 as in \cite{LW2021} and introduce a new unknown $h^\e=w_{\widetilde{\vartheta}} f^\e$ satisfying
\beqs
\begin{split}\label{eq:R}
&\Big\{\pt_t +\frac{1}{\e}v\cdot\nabla_x -\frac{1}{\e}\nabla_x\phi^\e\cdot\nabla_v\Big\}h^\e+ \underbrace{\Big(\frac{1}{\e^2}\nu(v)+  \frac{1}{2}v\cdot \nabla_x\phi^\e +2{\widetilde{\vartheta}}v\cdot\nabla_x\phi^\e + \vartheta\sigma\frac{|v|^2}{(1+t)^{1+\sigma}}\Big)}_{:=\frac{1}{\e^2}{\widetilde{\nu}(v)}}h^\e =\cdots.
\end{split}
\eeqs
Here  the  term $\frac{1}{2}{\color{blue}v}\cdot \nabla_x\phi^\e +2{\widetilde{\vartheta}}{\color{blue}v}\cdot\nabla_x\phi^\e$ in
$\frac{1}{\e^2}{\widetilde{\nu}(v)}$ will be controlled, provided that the electric field $\|\nabla_x\phi^\e\|_{L^\infty_x}$ has the
time decay rate not slower than $(1+t)^{-1}$, cf. Lemma \ref{es:nu} for details.
In addition, straightforward calculation yields $w_{{\vartheta}}\leq w_{\widetilde{{\vartheta}}}\leq w_{2{\vartheta}}$,
which means  that once the $W_{x,v}^{2,\infty}$-estimate of $w_{\widetilde{\vartheta}}f^\e$ is controlled,
then the $W_{x,v}^{2,\infty}$-estimate of $w_{{\vartheta}}f^\e$ will be obtained in the meantime.
Thus, our goal is  to   establish the $W_{x,v}^{2,\infty}$-estimate for the new unknown $h^\e$, which
  relies heavily on the characteristic estimate given in
  Lemma \ref{le:jac} under the \emph{a priori} assumption: there is $\delta>0$ small enough such that
\bal
\bsp\label{eq:assumption:1}
 \sup_{0\leq t \leq T }\left\{(1+t)^\frac{5}{4}\| \nabla_x\phi^\e(t)\|_{W^{1,\infty}_{x }}\right\}
 &\le \delta.
\esp
\eal

Secondly,  when handling the  $L^{{ \infty}}_{x,v}$-estimate of the mixed velocity derivative $\partial_\b^\a h^{\e}$
$(|\a|+|\b|\leq 2, |\b|\geq 1)$ for the soft potential case $-3<\gamma< 1$, we apply  $\partial_\b^\a$   to (\ref{eq:h:estimate:1})  and
exploit  Duhamel's principle to get
\bal
\bsp\label{D:eq:Dh2}
\partial_\b^\a{h}^\varepsilon(t,x,v)
\leq  \;&\underbrace{ \int_0^t \text{exp}\Big\{-\int_s^t\frac{\widetilde{\nu}(\tau)}{\e^2}\text{d}\tau\Big\}
\left|\frac{1}{\e}\nabla_x \partial_{\b-1}^\a {h}^\e (s,X(s),V(s))\right|\text{d}s}_{:=D_3}+\cdots.
\esp
\eal
Here $D_3$ is challenging, in that the collision frequency $\dfrac{1}{\e^2}\widetilde{\nu}$ does not have positive lower bound.
Our treatment is based on the observation that
\bals
{\e}^{-1}| \nabla_x \partial_{\b-1}^\a  h^\e|= {\e}^{-\frac{4}{5}} (1+t)^{-\frac{5}{8}}| {\e}^{-\frac{1}{5}} (1+t)^{\frac{5}{8}} \nabla_x \partial_{\b-1}^\a  h^\e|
\lesssim
\frac{1}{\e^2}\widetilde{\nu}| {\e}^{-\frac{1}{5}} (1+t)^{\frac{5}{8}} \nabla_x \partial_{\b-1}^\a  h^\e|,
\eals
where we have used a significant  result that $\frac{1}{\e^2}\widetilde{\nu}$  has a lower bound $\e^{-\frac{4}{5}}(1+t)^{-\frac{5}{8}}$, see \eqref{es:nu:1} in Lemma \ref{es:nu}.
Then $D_3$ can be controlled by
\bals
D_3\lesssim\;&\int_0^t \text{exp}\Big\{-\int_s^t\frac{\widetilde{\nu}(\tau)}{\e^2}\text{d}\tau\Big\}
\frac{1}{\e^2}\widetilde{\nu}(V(s))\left| {\e}^{-\frac{1}{5}} (1+t)^{\frac{5}{8}} \nabla_x \partial_{\b-1}^\a  h^\e (s,X(s),V(s))\right|\text{d}s\\
\lesssim\;&{\e}^{-\frac{1}{5}}\sup_{0\leq s\leq t}\left \{(1+s)^{\frac{5}{8}} \|\nabla_x \partial_{\b-1}^\a  h^\e(s)\|_{L^\infty_{x,v}}\right\}.
\eals
This means that compared with the $L^\infty_{x,v}$-estimate of $\nabla_x \partial_{\b-1}^\a  h^\e$,  the $L^\infty_{x,v}$-estimate
of $ \partial_{\b}^\a  h^\e$  loses a time decay factor ${\e}^{\frac{1}{5}}(1+t)^{-\frac{5}{8}}$.
Therefore, we have to design  a new time decay rate coupling with the $W^{2,\infty}_{x,v}$-estimate of  $h^\e$  as below
\bal
\bsp\label{weihted:wuqiong:estimate}
\;&\e^{\frac{1}{2}}(1+t)^{\frac{5}{4}}\|  {h^\e}(t)\|_{L_{x,v}^{\infty}}
+\sum_{|\a|+|\b|= 1}\e^{\frac{1}{2}+\frac{1}{5}|\b|}(1+t)^{\frac{5}{4}-\frac{5}{8}|\b|}
\| \partial_{\b}^{\a}h ^\e(t)\|_{L_{x,v}^{\infty}} \\
\;&+\sum_{|\a|+|\b|= 2}\e^{\frac{3}{2}+\frac{1}{5}|\b|}(1+t)^{\frac{5}{4}-\frac{5}{8}|\b|}
\| \partial_{\b}^{\a}h ^\e(t)\|_{L_{x,v}^{\infty}}.
\esp
\eal
This indicates  that we  pair   the $L^\infty_{x,v}$-estimate of the pure space derivative $\partial_x^\a h^\e (|\a|\leq2)$
with the uniform time decay factor $(1+t)^{-\frac{5}{4}}$, whereas  the $L^\infty_{x,v}$-estimate of the mixed space-velocity
derivatives $\partial_\b^\a h^\e$ $(|\a|+|\b|\leq2,1\leq |\b|\leq2)$ with the time decay factor
$\e^{-\frac{1}{5}|\b|}(1+t)^{-\frac{5}{4}+\frac{5}{8}|\b|}$.

Here, it should be mentioned that compared with the lower order $W^{1,\infty}_{x,v}$ part in \eqref{weihted:wuqiong:estimate}, one more $\e$
in $$\sum_{|\a|+|\b|= 2}\e^{\frac{3}{2}+\frac{1}{5}|\b|}\left(1+t\right)^{\frac{5}{4}-\frac{5}{8}|\b|}
\| \partial_{\b}^{\a} h^\e(t)\|_{L_{x,v}^{\infty}}$$ is designed for the second order derivative, in that  we employ different
manners to handle the main contribution term
\bals
\bsp
&D_4:=\frac{C}{\e^4}\int_0^t\int_0^{s-\eta\e^2}\int_{|v'|\le 2N}\int_{|v''|\le 3N}
\text{exp}
\Big\{-\int_s^t\frac{{\nu}(V(\tau))}{2\e^2}\text{d}\tau\Big\}
\nu(V(s))\\
&\qquad\qquad\times
  \text{exp}\Big\{-\int_{s_1}^s\frac{{\nu}(V(\tau'))}{2\e^2}\text{d}\tau'\Big\}
\nu(V(s_1))
 \big|w_{\widetilde{\vartheta}} \partial_{\b}^{\a}{f}^\e (s_1,X(s_1),v'') \big| \dd v''\dd v'\dd s_1\dd s.
\esp
\eals
In fact, employing Lemma \ref{le:jac} and the H\"{o}lder inequality, one easily has
 \bals
\bsp
D_4&\lesssim \frac{1}{\e^4}\int_0^t\int_0^{s-\eta\e^2}\cdots\Big(\int_{\mathbb{R}^3}\Big(\int_{\mathbb{R}^3}\big|\partial_{\b}^{\a}{f}^\e   (s_1,y,v'')\big|^p
  \Big| \mathrm{det}\big( \frac{\pt X(s_1)}{\pt v'}  \big)  \Big|^{-1}
     \dd y\Big)^{\frac{2}{p}}\dd v''\Big)^{\frac{1}{2}}\d s_1\d s\\
&\lesssim\frac{1}{\e^4}\int_0^t\int_0^{s-\eta\e^2}\cdots \Big(\int_{\mathbb{R}^3}\Big(\int_{\mathbb{R}^3}\big|\partial_{\b}^{\a}{f}^\e   (s_1,y,v'')\big|^p
  \frac{2\e^3}{|s-s_1|^3}
     \dd y\Big)^{\frac{2}{p}}\dd v''\Big)^{\frac{1}{2}}\d s_1\d s\\
&\lesssim\e^{-\frac{3}{p}}\frac{1}{\e^4}\int_0^t\int_0^{s}\cdots
\big\| \partial_{\b}^{\a}{f}^\e (s_1) \big\|_{L^p_xL^2_v} \dd s_1\dd s\\
&\lesssim \e^{-\frac{3}{p}} \sup_{0\leq s\leq t}\left\{\big\| \partial_{\b}^{\a}{f}^\e (s) \big\|_{L^p_xL^2_v}\right\},
\esp
\eals
which suggests converting the main term $D_4$  to  $\e^{-\frac{1}{2}}\sup\limits_{0\leq s\leq t}\left\{\big\| \partial_{\b}^{\a}{f}^\e (s) \big\|_{L^6_xL^2_v}\right\}$  only for $|\a|+|\b|\leq 1$, and to $\e^{-\frac{3}{2}}\sup\limits_{0\leq s\leq t}\left\{\big\| \partial_{\b}^{\a}{f}^\e (s) \big\|_{L^2_{x,v}}\right\}$  for $|\a|+|\b|= 2$,
rather than to $\e^{-\frac{1}{2}}\sup\limits_{0\leq s\leq t}\left\{\big\| \partial_{\b}^{\a}{f}^\e (s) \big\|_{L^6_xL^2_v}\right\}$
since the later goes beyond our $H^2_{x,v}$ energy framework.

Thirdly, in the process of deducing the $L_{x,v}^\infty$-estimate of $(1+t)^{\frac{5}{4}} {h^\e}$ for the soft potential  case $-3<\gamma< 1$, we utilize  the Young inequality to get
\bals
\bsp
\left|-\frac{5}{4}(1+t)^{\frac{1}{4}}h^\e\right|
=\;& \frac{5}{4} {(1+t)}^{-\frac{5}{8}}
\left[{(1+t)}^{\frac{5}{4}}|h^\e|\right]^\frac{7}{10}|h^\e|^\frac{3}{10}\\
\leq\;&
\eta {(1+t)}^{-\frac{5}{8}}\|(1+t)^{\frac{5}{4}}h^\e\|_{L^{\infty}_{x,v}}
+ \frac{5}{4}C_{\eta}{(1+t)}^{-\frac{5}{8}} \|h^\e\|_{L^{\infty}_{x,v}}.
\esp
\eals
Thus,   applying Duhamel's principle to the system of $(1+t)^{\frac{5}{4}}h^\e$,   one   establish
\bals
\bsp
&(1+t)^{\frac{5}{4}}|h^\e|\\
\leq\;&
-\int_0^t \text{exp}\Big\{-\int_s^t\frac{\widetilde{\nu}(\tau)}{\e^2}\text{d}\tau\Big\}
\Big[\frac{5}{4}{(1+t)}^{\frac{1}{4}}h^\e\Big] (s,X(s),V(s))\text{d}s
+\cdots\\
\leq\; & \eta  \varepsilon^\frac{4}{5} \sup_{0\le s\le t}\left\{  \| (1+s)^{\frac{5}{4}}{h}^\varepsilon (s)\|_{L^\infty_{x,v}}\right\}
+
 C_{\eta}\varepsilon^\frac{4}{5} \sup_{0\le s\le t}\left\{ \| h^\varepsilon (s)\|_{L^\infty_{x,v}}\right\}+\cdots, \quad\text{for }~\e\in(0,1],
\esp
\eals
implying that when handling the weighted $L^{{ \infty}}_{x,v}$-estimate of $h^{\e}$ with time decay, we also need to consider the
corresponding $L^{{\infty}}_{x,v}$-estimate of $h^{\e}$ without time decay in the meantime.

In addition, to obtain the ${L_{x,v}^\infty}$-estimate of $(1+t)^{\frac{5}{4}} \nabla_x^2 {h^\e}$  for the  soft potential  case $-3<\gamma< 0$, we apply   $\nabla_x^2$   to (\ref{eq:h:estimate:1})  and exploit  Duhamel's principle to deduce
\bal
\bsp\label{D:eq:Dxh2}
(1+t)^{\frac{5}{4}} |\nabla_x^2{h}^\varepsilon|
\leq  \;&\underbrace{ \int_0^t \text{exp}\Big\{-\int_s^t\frac{\widetilde{\nu}(\tau)}{\e^2}\text{d}\tau\Big\}
\left|(1+s)^{\frac{5}{4}} \nabla_x ^3\phi^\e\cdot \nabla_v{h}^\e (s,X(s),V(s))\right|\text{d}s}_{:=D_5}+\cdots,
\esp
\eal
where $D_5$  is  tricky because $\dfrac{1}{\e^2}\widetilde{\nu}$ has no positive lower bound and
$\|\nabla_v h^\e\|_{L^\infty_{x,v}}$ is only coupled with  the time decay factor $(1+t)^{-\frac{5}{8}}$.
Fortunately, thanks to \eqref{es:nu:1}, we observe that
 \bals
 \bsp
|(1+s)^{\frac{5}{4}} \nabla_x ^3\phi^\e\cdot \nabla_vh^\e|
\;&\leq{\e}^{-\frac{4}{5}} (1+s)^{-\frac{5}{8}}\e(1+s)^{\frac{5}{4}} \|\nabla_x ^3\phi^\e\|_{L^\infty_{x,v}}
 {\e}^{-\frac{1}{5}}(1+s)^{\frac{5}{8}} \|\nabla_vh^\e\|_{L^\infty_{x,v}}\\
 \;&\lesssim\frac{1}{\e^2}\widetilde{\nu}(V(s))\left[\e(1+s)^{\frac{5}{4}} \|\nabla_x ^3\phi^\e\|_{L^\infty_{x}}\right]
 \left[{\e}^{-\frac{1}{5}}(1+s)^{\frac{5}{8}} \|\nabla_vh^\e\|_{L^\infty_{x,v}}\right],
 \esp
\eals
indicating that once $\e(1+t)^\frac{5}{4}\| \nabla^3_x\phi^\e(t) \|_{{L^{\infty}_{x}} }$ is bounded, then $D_5$ will be controlled by
\bals
D_5\lesssim {\e}^{-\frac{1}{5}}\sup_{0\leq s\leq t}\left \{(1+s)^{\frac{5}{8}} \| \nabla_v  h^\e(s)\|_{L^\infty_{x,v}}\right\}.
\eals

Consequently, based  on the  analysis above and another \emph{a priori}  assumption: there is $\delta>0$ small enough such that
\bal
\bsp\label{eq:assumption:2}
\;&\e\sup_{0\leq t \leq T}\left\{(1+t)^{\frac{5}{4}}\| \nabla^3_x\phi^\e(t) \|_{{L^{\infty}_{x}} }\right\}+
\e^{\frac{3}{4}}\sup_{0\leq t \leq T}\left\{\| h^\e(t) \|_{{L^{\infty}_{x,v}} }\right\}
+\e\sup_{0\leq t \leq T}\{\| h^\e(t) \|_{{W^{1,\infty}_{x,v}} }\}
\le  \delta,
\esp
\eal
we are able to solve the $W^{2,\infty}_{x,v}$-estimate with time decay \eqref{weihted:wuqiong:estimate}  for the VPB system (\ref{eq:f})
with soft potentials $-3<\gamma<0$, cf. the proof of {Proposition} \ref{result:W2:wuqiong}.

\subsubsection{Time decay estimate of the perturbation}
\hspace*{\fill}

  Now we are in the position to close the velocity weighted $H^2_{x,v}$ energy estimate in Proposition \ref{main-weighted-energy-estimate-1} and
  the time-velocity weighted $W^{2,\infty}_{x,v}$-estimate in Proposition  \ref{result:W2:wuqiong} by
  the time decay estimate under the \emph{a priori} assumptions \eqref{eq:assumption:1} and  \eqref{eq:assumption:2}.
  Note that the time-velocity weighted $W^{2,\infty}_{x,v}$-estimate \eqref{result:W2:wuqiong:1} involves the time decay rate
  $(1+t)^{-\frac{5}{4}}$ for
  $\| \nabla_x^2\phi^\e\|_{H^{2}_{x }}$ and $\|\nabla_xf^\e\|_{H^1_{x}L^2_{v}}$ and the time decay rate
  $(1+t)^{-\frac{5}{8}}$ for $\| \nabla_x\nabla_vf^\e \|_{L^2_{x,v}}$.
  In contrast to the standard VPB system \cite{DYZ2003}, one major difference lies in the singularity of ${\e}^{-1}$ in
  front of  the nonlinear term ${\e}^{-1}\Gamma(f^\e, f^\e)$ in \eqref{eq:f}.
  Therefore, we should first deduce the weighted energy estimate for the linearized system
  \eqref{linear:esti:f:decay:1} with a nonhomogeneous microscopic source in the Fourier space
\bal
\bsp\label{D:linear:estimate:decay:proof:result:2}
\frac{\d}{\d t}  {E}_{\ell}(\hat{f^{\e}})+\sigma_0 {D}_{\ell}(\hat{f^{\e}})
\lesssim
\e^{2}\|{w}^{-\ell-\frac{1}{2}}\hat{h}^\e\|_{L^2_v}^2,
\esp
\eal
where the factor  $\e^2$ on the right-hand side will help us to treat the singularity brought by
${\e}^{-1}\Gamma(f^\e, f^\e)$. Then   employing   the time-frequency weighted   and the time-frequency
splitting technique developed in   \cite{DYZ2003}, we derive  the  time decay estimate for the
linearized system \eqref{linear:esti:f:decay:1}
\bal
\bsp\label{D:result:linear:estimate:decay:3}
 & \|    e^{{t B_{\e}}}f_0^\e\|_{L_{x,v}^2}^2
 + \|  \nabla_x\Delta_x^{-1} \mathbf{P}_0e^{{t B_{\e}}}f_0^\e\|_{L_{x}^2}^2
 \\
\lesssim\;& (1+t)^{-\frac{3 }{2}}
\e^{2}\int_0^t(1+\tau)^{J} \|{w}^{-(\ell+\frac{1}{2})}h^\e(\tau)\|_{{L^1_xL^2_v}}^2\d \tau
+\cdots  \quad \text{for} ~J\geq \frac{3 }{2}.
\esp
\eal
Further, (\ref{D:result:linear:estimate:decay:3}) and Duhamel's formula will lead to the time decay estimates of $\|\mathbf{P}f^\e\|_{L^2_{x,v}}$ and $\|\nabla_x\phi^\e\|_{L^2_{x}}$, cf. Lemma \ref{decay:es:nonlinear} for details. Thus, the weighted $H^2_{x,v}$ energy estimate and weighted
$W^{2,\infty}_{x,v}$-estimate
 are closed via
 the  temporal weighted  method, see Proposition \ref{decay:es:nonlinear:result:1}. In addition, employing the interpolation between H\"{o}lder space and $L^2_x$ space and the Riesz potential theory
\bal
\bsp\label{D:eq:assumption:hard:1-1}
  &\e\|\nabla^3_x\phi^\e\|_{L^\infty_x}
  \lesssim\|\nabla^3_x\phi^\e\|_{L^2_x}^{\frac{1}{4}}
 \left (\e^{\frac{4}{3}}[\nabla^3_x\phi^\e]_{C^{\frac{1}{2}}}\right)^{\frac{3}{4}}
  \lesssim\left(\|\nabla_x f^\e \|_{H^{1}_{x}L^{2}_{v}}\right)^\frac{1}{2}
  \left(\e^2\|\nabla_x^2 h^\e \|_{L^{\infty}_{x,v}}\right)^\frac{1}{2}
  \esp
\eal
 (see \eqref{eq:assumption:soft:1-2}),  we   verify the time weighted \emph{a priori}  assumptions \eqref{eq:assumption:1} and  \eqref{eq:assumption:2}
 in terms of the time decay rate $(1+t)^{-\frac{5}{4}}$ of $\| \nabla_x^2\phi^\e\|_{H^{1}_{x }}$,
 $\| \nabla_xf^\e\|_{H^{1}_{x }L^{2}_{v }}$ and  $\e^2\|\nabla_x^2 h^\e \|_{L^{\infty}_{x,v}}$,
 cf. the proof of   Theorem \ref{mainth1}.
\vspace{2mm}

 The rest of this manuscript is organized as follows. In Section \ref{energy-estimate}, we establish the velocity weighted $H^2_{x,v}$ energy estimate
 for the perturbation VPB system (\ref{eq:f}). Section \ref{wuqiong-estimate} concentrates on the time-velocity weighted $W^{2,\infty}_{x,v}$-estimate
 with time decay. In Section \ref{golbal-decay-soft},  we establish the global strong solution to the VPB system (\ref{eq:f}) and justify
 its limit to the incompressible NSFP system, i.e.  Theorem \ref{mainth1} and  Theorem \ref{mainth2}.
 Particularly, the rest of this manuscript will
 focus on the more challenging soft potential case, while the proof for the hard potential case
 will only be written for the different treatments from the soft potential case.

\section{Weighted \texorpdfstring{$H^2_{x,v}$}{Lg} Energy Estimate}
\label{energy-estimate}

The following two sections concern the uniform-in-time \emph{a priori} estimates on the solutions
$(f^\e, \nabla_x\phi^\e)$ to the perturbation VPB system (\ref{eq:f})  for the full range  of cutoff potentials.
For this purpose, we suppose that  the Cauchy problem (\ref{eq:f})
admits a smooth solution $(f^\e, \nabla_x\phi^\e)$ defined on the time interval $0 \leq t\leq T$ for some $0 < T\leq \infty$.

The main purpose of this section is to derive the weighted $H^2_{x,v}$ energy estimate for the perturbation  VPB system (\ref{eq:f}).
To this end, we begin with some basic properties of the linearized collision operator $L$ and the nonlinear collision operator $\Gamma$.
We first list the velocity weighted estimates on the linearized collision operator $L$ with respect to the velocity weight
$w$ defined in \eqref{weight-w} whose proofs can be found in \cite{DYZ2002,XXZ2013}.

\begin{lemma}\label{es-energy-linear}
Let $\ell\in\mathbb{R}$. For any multi-index  $\b$ with  $|\b|\geq 0$,
 we have the following estimates about the linearized operator $L$.\\
 $(1)$ If $-3< \gamma< 0$, then for any small $\eta \geq 0$, there exists $C_{\eta} > 0$ such that
\begin{equation}
\bsp\label{es-energy-linear:1}
\left\langle \partial_v^\b( Lf) , w^{2|\b|-2\ell}\partial_v^\b f \right\rangle_{L_{x,v}^2} \geq \;& \| w^{|\b|-\ell}
\partial_v^{\b} f\|_{L_{x,v}^2(\nu)}^2
-\eta\sum_{|\b'|\leq |\b|} \| w^{|\b'|-\ell}\partial_v^{\b'}
f\|_{L_{x,v}^2(\nu)}^2\\
\;& -C_{\eta}\| f\|_{L_{x,v}^2(\nu)}^2.
\esp
\end{equation}\\
 $(2)$
If $0\leq \gamma\leq1$,
then for any small $\eta \geq 0$, there exists $C_{\eta} > 0$ such that
\begin{equation}
\bsp\label{es-energy-linear:1-hard}
\left\langle \partial_v^\b( Lf) ,{w}^{2\ell}\partial_v^\b f\right\rangle_{L_{x,v}^2} \geq \;& \|  {w}^{\ell}
\partial_v^{\b} f\|_{L_{x,v}^2(\nu)}^2
-\eta\sum_{|\b'|\leq |\b|} \| {w}^{\ell}\partial_v^{\b'}
f\|_{L_{x,v}^2(\nu)}^2 -C_{\eta}\| f\|_{L_{x,v}^2(\nu)}^2.
\esp
\end{equation}
\end{lemma}

We now turn to deduce the velocity weighted estimates on the nonlinear collision term $\Gamma(f, g)$, which will be frequently
used in the uniform energy estimate. Before stating our result, we recall
\bals
\partial_v^\b\Gamma(f, g)=\sum_{|\b_0|+|\b_1|+|\b_2|=|\b|}
C_{\b}^{\b_0,\b_1, \b_2} \Gamma_v^{\b_0}(\partial_v^{\b_1}f, \partial_v^{\b_2}g),
\eals
where $\Gamma_v^{\b_0}$ is given by
\bals
\Gamma_v^{\b_0}(f, g)
 \equiv\;&\iint_{\bbR^3\times \mathbb{S}^2}
  |v-v_*|^\gamma q_0(\th)
  \partial_{v_*}^{\b_0}\big[\sqrt{\mu}(v_*)\big]
  \Big[f(v')g(v_*')
  -f(v)g(v_*)\Big] \d v_*\d \omega.
\eals
Then  the following main result involves the weighted estimates on $\Gamma$ with the velocity weight $w$ defined in \eqref{weight-w}.
\begin{lemma}\label{es-energy-Nonlinear}
Let  $\ell\in\mathbb{R}$.  For any multi-index $\b$ with $|\b|\geq 0$, we have the following bounds on the nonlinear term $\Gamma_v^{\b}(f, g)$.\\
 $(1)$ For the soft potential case $-3<\gamma< 0$,  there holds
\beq
\bsp\label{es-energy-Nonlinear:2}
&\;\left\langle\Gamma_v^{\b}(f, g)+\Gamma_v^{\b}(g, f),w^{2(|\b|-\ell)} \partial_v^{\b} h\right\rangle_{L^2_{v}}
\\
\lesssim
&\;\left(\| w_\vartheta f\|_{L_{v}^\infty}\|w^{|{\b}|-\ell} (\mathbf{I}-\mathbf{P})g\|_{L_{v}^2(\nu)}+\| \mathbf{P}g\|_{L_{v}^2}
\|w^{|{\b}|-\ell}  f\|_{L_{v}^2(\nu)}\right)
\|w^{|{\b}|-\ell}\partial_v^{\b} h\|_{L_{v}^2(\nu)}.
\esp
\eeq  \\
 $(2)$  For the hard potential case $0\leq \gamma\leq 1$,  there holds
\beq
\bsp\label{es-energy-Nonlinear:2-hard}
&\;\left\langle\Gamma_v^{\b}(f, g)+\Gamma_v^{\b}(g, f),{w}^{2\ell} \partial_v^{\b} h\right\rangle_{L^2_{v}}\\
\lesssim
&\;\| w_\vartheta f\|_{L_{v}^\infty}\|{w}^{\ell}(\mathbf{I}-\mathbf{P})g\|_{L_{v}^2(\nu)} \|{w}^{\ell}\partial_v^\b h\|_{L_{v}^2(\nu)}
+\| \mathbf{P} g\|_{L_{v}^2}\| {w}^{\ell}f\|_{ L_{v}^2(\nu)}
\|{w}^{\ell}\partial_v^{\b} h\|_{L_{v}^2(\nu)}.
\esp
\eeq
Here, the velocity weight function $w_{{\vartheta}}$ is  given by
  \bal\label{weight-w-v}
  w_{{\vartheta}}\equiv w_{{\vartheta}}(v):= e^{{\vartheta}|v|^2}\quad \text{with}\; 0<\vartheta\ll 1
  \eal
  throughout this section.
\end{lemma}
\begin{proof}
We only prove  \eqref{es-energy-Nonlinear:2}, since the argument can be also applied to   \eqref{es-energy-Nonlinear:2-hard} under minor modification.
To this end, we  claim that for any multi-index $|\b|\geq 0,$ there holds
\beq
\bsp\label{Nonlinear:term:proof:1}
&\;\|\nu^{-\frac{1}{2}}w^{|{\b}|-\ell}\Gamma_v^{\b}(f, g)\|_{L^2_v}+\|\nu^{-\frac{1}{2}}w^{|{\b}|-\ell}\Gamma_v^{\b}(g, f)\|_{L^2_v}\\
\lesssim
&\;\| w_\vartheta f\|_{L_{v}^\infty}\|w^{|{\b}|-\ell} (\mathbf{I}-\mathbf{P})g\|_{L_{v}^2(\nu)}
+\| \mathbf{P}g\|_{L_{v}^2}\|w^{|{\b}|-\ell}  f\|_{L_{v}^2(\nu)}
.
\esp
\eeq
Once this estimate is established, we obtain  \eqref{es-energy-Nonlinear:2} by   the H\"{o}lder inequality. Thus,   we focus on verifying \eqref{Nonlinear:term:proof:1} in the rest of the proof.

In fact, using  the decomposition $g=\mathbf{P}g+(\mathbf{I}-\mathbf{P})g$, we have
\beq
\bsp\label{Nonlinear:term:proof:2}
&\;\|\nu^{-\frac{1}{2}}w^{|{\b}|-\ell}\Gamma_v^{\b}(f, g)\|_{L^2_v}+\|\nu^{-\frac{1}{2}}w^{|{\b}|-\ell}\Gamma_v^{\b}(g, f)\|_{L^2_v}\\
\lesssim\;&
\underbrace{\|\nu^{-\frac{1}{2}}w^{|{\b}|-\ell}\Gamma_v^{\b}(f, \mathbf{P}g)\|_{L^2_v}}_{N_1}
+\underbrace{\|\nu^{-\frac{1}{2}}w^{|{\b}|-\ell}\Gamma_v^{\b}(\mathbf{P}g, f)\|_{L^2_v}}
_{N_2}+\underbrace{\|\nu^{-\frac{1}{2}}w^{|{\b}|-\ell}\Gamma_v^{\b}(f, (\mathbf{I}-\mathbf{P})g)\|_{L^2_v}}_{N_3}\\
\;&+\underbrace{\|\nu^{-\frac{1}{2}}w^{|{\b}|-\ell}\Gamma_v^{\b}((\mathbf{I}-\mathbf{P})g, f)\|_{L^2_v}}_{N_4}.
\esp
\eeq
Note that
$$
 \partial_{v_*}^{\b}\big(\mu^\frac{1}{2}(v_*)\big)\leq C\mu^\frac{1}{4}(v_*),\;\;
 \mu^\frac{1}{4}(v_*)w^{2|{\b}|-2\ell}(v)\leq C w^{2|{\b}|-2\ell}(v')w^{2|{\b}|-2\ell}(v'_*).
$$
And using the H\"{o}lder inequality, $N_1$ is bounded by
\bal
\bsp\label{Nonlinear:term:proof:3}
N_1\lesssim\;&\Big[\int_{\bbR^3}\nu(v)^{-1}w^{2|{\b}|-2\ell}(v)\left|\int_{\bbR^3}
  |v-v_*|^\gamma
  \mu^\frac{1}{4}(v_*)
  f(v')\mathbf{P}g(v_*')
   \d v_*\right|^2\d v\Big]^\frac{1}{2}\\
 \;& +
  \Big[\int_{\bbR^3}\nu^{-1}(v)w^{2|{\b}|-2\ell}(v)\left|\int_{\bbR^3}
  |v-v_*|^\gamma
  \mu^\frac{1}{4}(v_*)
  f(v)\mathbf{P}g(v_*) \d v_*\right|^2\d v\Big]^\frac{1}{2}\\
  \lesssim\;&\Big[\iint_{\bbR^3\times\bbR^3}
  |v-v_*|^\gamma
  \mu^\frac{1}{4}(v_*)
  |f(v')|^2|\mathbf{P}g(v_*')|^2w^{2|{\b}|-2\ell}(v)
   \d v_*\d v\Big]^\frac{1}{2}\\
 \;& +
  \Big[\iint_{\bbR^3\times\bbR^3}
  |v-v_*|^\gamma
  \mu^\frac{1}{4}(v_*)
 | f(v)|^2|\mathbf{P}g(v_*)|^2 w^{2|{\b}|-2\ell}(v)\d v_*\d v\Big]^\frac{1}{2}\\
 \lesssim\;&\Big[\iint_{\bbR^3\times\bbR^3}
  |v-v_*|^\gamma
  |f(v')|^2|\mathbf{P}g(v_*')|^2w^{2|{\b}|-2\ell}(v')w^{2|{\b}|-2\ell}(v'_*)
   \d v_*\d v\Big]^\frac{1}{2}\\
 \;& +
  \Big[\iint_{\bbR^3\times\bbR^3}
  |v-v_*|^\gamma
  \mu^\frac{1}{4}(v_*)
 | f(v)|^2|\mathbf{P}g(v_*)|^2 w^{2|{\b}|-2\ell}(v)\d v_*\d v\Big]^\frac{1}{2}.
\esp
\eal
Further, by the change of variables $(v,v_*)\rightarrow(v',v_*')$ with the unit Jacobain for the    right-hand first term of \eqref{Nonlinear:term:proof:3}, we find
\bal
\bsp\label{Nonlinear:term:proof:4}
N_1\lesssim\;&\|w_{\vartheta} \mathbf{P}g\|_{L^\infty_v}\Big[ \int_{\bbR^3}\Big(
\int_{\bbR^3}
  |v-v_*|^\gamma
   w^{2|{\b}|-2\ell}(v_*)e^{-2\vartheta|v_*|^2}\d v_*\Big)
 | f(v)|^2 w^{2|{\b}|-2\ell}(v)\d v\Big]^\frac{1}{2}\\
 \lesssim\;&\|\mathbf{P}g\|_{L^2_v}
 \| w^{|{\b}|-\ell}f \|_{L^2_v(\nu)}.
\esp
\eal
for some constant $0<\vartheta\ll 1$.
Similarly,  we   deduce
\bals
\bsp
N_2\lesssim\;& \Big[
  \iint_{\bbR^3\times\bbR^3}
  |v-v_*|^\gamma
 | \mathbf{P}g(v)|^2|f(v_*)|^2 w^{2|{\b}|-2\ell}(v)w^{2|{\b}|-2\ell}(v_*)\d v_*\d v\Big]^\frac{1}{2}\\
 \lesssim\;& \|w_{\vartheta} \mathbf{P}g\|_{L^\infty_v}\Big[
  \int_{\bbR^3}
  \Big( \int_{\bbR^3}|v-v_*|^\gamma e^{-2\vartheta|v|^2}
   w^{2|{\b}|-2\ell}(v)\d v\Big)
 |f(v_*)|^2w^{2|{\b}|-2\ell}(v_*)\d v_*\Big]^\frac{1}{2}\\
 \lesssim\;&\|\mathbf{P}g\|_{L^2_v}
 \| w^{|{\b}|-\ell}f \|_{L^2_v(\nu)}.
\esp
\eals

To handle $N_3$, applying the same method as the estimate of \eqref{Nonlinear:term:proof:4} on $N_2$, we get
\bals
\bsp
N_3
 \lesssim\;&
  \Big[\iint_{\bbR^3\times\bbR^3}
  |v-v_*|^\gamma
 | f(v)|^2|(\mathbf{I}-\mathbf{P})g(v_*)|^2 w^{2|{\b}|-2\ell}(v)w^{2|{\b}|-2\ell}(v_*)\d v_*\d v\Big]^\frac{1}{2}\\
  \lesssim\;&\|w_{\vartheta}f\|_{L^\infty_v} \Big[\int_{\bbR^3}\Big(\int_{\bbR^3}
  |v-v_*|^\gamma
  e^{-2\vartheta|v|^2}
  w^{2|{\b}|-2\ell}(v)\d v\Big)w^{2|{\b}|-2\ell}(v_*)|(\mathbf{I}-\mathbf{P})g(v_*)|^2\d v_*\Big]^\frac{1}{2}\\
  \lesssim\;&\|w_{\vartheta}f\|_{L^\infty_v}
 \| w^{|{\b}|-\ell}(\mathbf{I}-\mathbf{P})g \|_{L^2_v(\nu)}
\esp
\eals
for some constant $0<\vartheta\ll 1$.
Furthermore, deducing similarly as $N_1$, we have
\bals
\bsp
N_4
 \lesssim\;&
  \Big[\iint_{\bbR^3\times\bbR^3}
  |v-v_*|^\gamma
 | (\mathbf{I}-\mathbf{P})g(v)|^2|f(v_*)|^2 w^{2|{\b}|-2\ell}(v)w^{2|{\b}|-2\ell}(v_*)\d v_*\d v\Big]^\frac{1}{2}\\
  \lesssim\;&\|w_{\vartheta}f\|_{L^\infty_v} \Big[\int_{\bbR^3}\Big(\int_{\bbR^3}
  |v-v_*|^\gamma
  e^{-2\vartheta|v_*|^2}w^{2|{\b}|-2\ell}(v_*)
  \d v_*\Big)w^{2|{\b}|-2\ell}(v)|(\mathbf{I}-\mathbf{P})g(v)|^2\d v\Big]^\frac{1}{2}\\
  \lesssim\;&\|w_{\vartheta}f\|_{L^\infty_v}
 \| w^{|{\b}|-\ell}(\mathbf{I}-\mathbf{P})g \|_{L^2_v(\nu)}.
\esp
\eals

In summary, putting  the above estimates of $N_1\sim N_4$ into \eqref{Nonlinear:term:proof:2}, we obtain the
desired claim \eqref{Nonlinear:term:proof:1}. This completes the proof.
\end{proof}

\subsection{Weighted \texorpdfstring{$H_{x,v}^2$}{Lg} Energy Estimate for Soft Potentials}
\hspace*{\fill}

The main goal of this subsection is to establish the uniform $H_{x,v}^2$  energy estimate for the   VPB system (\ref{eq:f}) with soft potentials.
To be precise, we begin with decompose the perturbation $f^\e$ into
\begin{equation}\label{fdefenjie}
f^\e =\mathbf{P}f^\e +(\mathbf{I}-\mathbf{P})f^\e,
\end{equation}
where the macro part  $\mathbf{P}f^\e$ is defined as
\bal\label{defination:Pf}
\mathbf{P}f^\e=\Big(a^\e(t,x)+b^\e(t,x)\cdot v+c^\e(t,x)\frac{|v|^2-3}{2}\Big)\sqrt{\mu}.
\eal
Then, we introduce the weighted instant energy functional $\mathbf{E}_{\ell}^\mathbf{s}(t)$ satisfying
\bal
\bsp\label{def-energy-R}
\mathbf{E}_{\ell}^\mathbf{s}(t)\sim\;&\sum_{|\a|+|\b|\leq2}
 \| w^{|\b|}\partial_{\b}^{\a} f^\e(t)\|^2_{L^2_{x,v}}
 + \| \nabla_x \phi^\e(t)\|^2_{H^2_x}+\sum_{\substack{|\a|+|\b|\leq2\\0\leq|\a|\leq1}}
 \| w^{|\b|-\ell}\partial_{\b}^{\a} f^\e(t)\|^2_{L^2_{x,v}}\\
 &+\e\| {w}^{-\ell}\nabla_x^2f^\e(t)\|^2_{L^2_{x,v}}
 ,
 \esp
 \eal
 and the corresponding energy  dissipation rate
 \bal
 \bsp\label{def-disspation-R}
  \mathbf{D}_{\ell}^\mathbf{s}(t):=\;&\|\nabla_x \mathbf{P}f^\e(t)\|^2_{H^1_{x,v}}
  + \| \nabla^2_x \phi^\e(t)\|^2_{H^1_x}
  +\frac{1}{\e^2}\sum_{|\a|+|\b|\leq2}\|w^{|\b|}\partial^{\a}_{\b}
  (\mathbf{I}-\mathbf{P})f^\e(t)\|^2_{L^2_{x,v}(\nu)}
  \\
  \;&
 +\frac{1}{\e^2}\sum_{\substack{|\a|+|\b|\leq2\\0\leq|\a|\leq1}}
 \|  w^{|\b|-\ell}\partial_{\b}^{\a} (\mathbf{I}-\mathbf{P})f^\e(t)\|^2_{L^2_{x,v}(\nu)}
 +\frac{1}{\e}\| w^{ -\ell}\nabla_x^2(\mathbf{I}-\mathbf{P})f^\e(t)\|^2_{L^2_{x,v}(\nu)},
 \esp
\eal
where $\ell>0$ and the velocity weight $w$ is given in \eqref{weight-w}.

We are now in a position to state   the main result of this subsection.

\begin{proposition}\label{main-weighted-energy-estimate-1}
Let $-3<\gamma<0,\, \ell>0$. Suppose that $ (f^\e, \nabla_x\phi^\e) $ is a solution to the
VPB system \eqref{eq:f} defined on $ [0, T] \times \mathbb{R}^3 \times \mathbb{R}^3$
and   satisfies the \emph{a priori} assumption \bal\label{energy-assumptition-soft}
  \mathbf{E}_{\ell}^\mathbf{s}(t)\leq \delta\quad \text{for}\;\; 0\leq t\leq T
  \eal
for some sufficiently small $\delta>0$, then there is $\mathbf{E}_{\ell}^\mathbf{s}(t)$ satisfying  \eqref{def-energy-R} such that
  \bal\label{energy-nonlinear-result}
  \frac{\d}{\d t}\mathbf{E}_{\ell}^\mathbf{s}(t)+\mathbf{D}_{\ell}^\mathbf{s}(t)
  \lesssim
\left( \e^3\|w_{\vartheta}f^\e\|_{W^{2,\infty}_{x,v}}+\e
 \|w_{\vartheta}f^\e\|_{W^{1,\infty}_{x,v}}\right)\mathbf{D}_{\ell}^\mathbf{s}(t) \quad \text{for  all}\;\; 0\leq t\leq T.
  \eal
\end{proposition}

The main idea we adopt to prove Proposition \ref{main-weighted-energy-estimate-1} parallels to the
Boltzmann equation   in \cite{guo2006}, that is, by combining direct energy methods
on the perturbation VPB system (\ref{eq:f}) together with
the so-called micro-macro decomposition method to   acquire the missing macro dissipations.
However, the analysis for the VPB system for the whole range of soft potentials  is quite complex
since we meet the challenge of overcoming the high velocity part of the external force term $v \cdot\nabla_x \phi^\e f^\e$ and trying to deduce the energy estimate on the lower regularity space $H^2_{x,v}$, thus we prove it in details as follows.
\medskip

To make the presentation clearly, we  define
  the instant energy $\mathcal{E}^\mathbf{s}(t)$ and the dissipation $\mathcal{D}^\mathbf{s}(t)$
 for the  $H_{x,v}^2$ energy estimate
\beq\label{def-energy-R:1}
\bsp
 \mathcal{E}^\mathbf{s}(t):=\;&\sum_{|\a|+|\b|\leq2}
 \| w^{|\b|}\partial^{\a}_{\b} f^\e(t)\|^2_{L^2_{x,v}}
 + \| \nabla_x \phi^\e(t)\|^2_{H^2_x},\\
  \mathcal{D}^\mathbf{s}(t):=\;&\frac{1}{\e^2}\sum_{|\a|+|\b|\leq2}\|w^{|\b|}\partial^{\a}_{\b}
  (\mathbf{I}-\mathbf{P})f^\e(t)\|^2_{L^2_{x,v}(\nu)}+\|\nabla_x \mathbf{P}f^\e(t)\|^2_{H^1_{x,v}}
  + \| \nabla^2_x \phi^\e(t)\|^2_{H^1_x}
 ,
\esp
\eeq
and  the
  instant energy $\mathcal{E}^\mathbf{s}_\ell(t)$ and the dissipation $\mathcal{D}^\mathbf{s}_\ell(t)$
 for the  weighted $H_{x,v}^2$ energy estimate
\beq\label{def-weighted-energy-R:1}
\bsp
 \mathcal{E}_\ell^\mathbf{s}(t):=\;&\sum_{\substack{|\a|+|\b|\leq2\\ 0\leq |\a| \leq 1}}
 \| w^{|\b|-\ell}\partial^{\a}_{\b}( \mathbf{I}-\mathbf{P})f^\e(t)\|^2_{L^2_{x,v}}
 + \e\| w^{-\ell}\nabla_{x}^2f^\e(t)\|^2_{L^2_{x,v}},\\
  \mathcal{D}_\ell^\mathbf{s}(t):=\;&
  \frac{1}{\e^2}\sum_{\substack{|\a|+|\b|\leq2\\ 0\leq |\a| \leq 1}}\|w^{|\b|-\ell}\partial^{\a}_{\b}
  (\mathbf{I}-\mathbf{P})f^\e(t)\|^2_{L^2_{x,v}(\nu)}+
  \frac{1}{\e}\|w^{-\ell}\nabla_{x}^{2}
  (\mathbf{I}-\mathbf{P})f^\e(t)\|^2_{L^2_{x,v}(\nu)}
  \\
  \;&
 + \|\nabla_x \mathbf{P}f^\e(t)\|^2_{H^1_{x,v}}
  + \| \nabla^2_x \phi^\e(t)\|^2_{H^1_x}
 ,
\esp
\eeq
where $\ell>0$ and the velocity weight  $w$ is given in \eqref{weight-w}.

\medskip

First of all, we have
 the following macro dissipation estimate for the VPB system  (\ref{eq:f}) with soft potentials $-3<\gamma<0$.
\begin{lemma}\label{prop:estimate:dissipation}
Assume that $ (f^\e, \nabla_x\phi^\e) $ is a solution to the
VPB system \eqref{eq:f} defined on $ [0, T] \times \mathbb{R}^3 \times \mathbb{R}^3$.
Then there is a temporal functional $\mathcal{E}_{int}^\mathbf{s}(t)$ such that
\bal
\bsp\label{qkj:esti:diss:0}
&\|\nabla_x \mathbf{P}f^\e(t)\|_{H_{x,v}^1}^2+\|\nabla_x^2\phi^\e(t)\|_{H_{x}^1}^2+\e\frac{\d}{\d t}\mathcal{E}_{int}^\mathbf{s}(t)
\\
\lesssim \;&\big[\mathcal{E}^\mathbf{s}(t)\big]^{\frac{1}{2}} \mathcal{D}^\mathbf{s}(t)
+\e
 \|w_{\vartheta}f^\e\|_{L^{\infty}_{x,v}} \mathcal{D}^\mathbf{s}(t) +\frac{1}{\e^2}\|  (\mathbf{I}-\mathbf{P}) f^\e \|_{H_{x}^1L_{v}^2(\nu)}^2
\esp
\eal
holds for any $0\leq t\leq T$,
where $\mathcal{E}_{int}^\mathbf{s}(t)$  given in \eqref{qkj:esti:diss:0-2} satisfies
\bal\label{qkj:esti:diss:0-1}
|\mathcal{E}_{int}^\mathbf{s}(t)|\lesssim \|a^\e\|_{H^1_x}^2+\|  (\mathbf{I}-\mathbf{P}) f^\e \|_{H_{x}^1L_{v}^2}^2+\|  \nabla_x \mathbf{P}f^\e \|_{H_{x,v}^1}^2.
\eal
\end{lemma}

\begin{proof}
We prove this lemma by an analysis of the macroscopic equations and the local conservation laws. Firstly,
plugging (\ref{fdefenjie})  and  (\ref{defination:Pf}) into  the first equation in (\ref{eq:f}), we have
\newcommand{\W}{W}
\newcommand{\M}{M}
\newcommand{\T}{T}
\newcommand{\Q}{Q}
\bals
\bsp
&\pt_t \Big(a^\e+b^\e\cdot v+c^\e\frac{|v|^2-3}{2} \Big)\sqrt{\mu}+\frac{1}{\e}\sqrt{\mu}v\cdot \nabla_x \Big(a^\e+b^\e\cdot v+c^\e\frac{|v|^2-3}{2} \Big) +\frac{1}{\e}\sqrt{\mu}v\cdot \nabla_x \phi^\e\\
=\;&-\pt_t(\mathbf{I}-\mathbf{P}) f^\e-\frac{1}{\e}v\cdot\nabla_x (\mathbf{I}-\mathbf{P}) f^\e-\frac{1}{\e^2}L f^\e
 + \frac{\nabla_x \phi^\e\cdot\nabla_v(\sqrt{\mu} f^\e)}{\sqrt{\mu}}
 +\frac{1}{\e}\Gamma(f^\e,f^\e).
\esp
\eals
Fixing $t$ and $x$ and comparing the coefficients on both sides in front of
$\big[\sqrt{\mu},\,v_i\sqrt{\mu},\,v_iv_j\sqrt{\mu},\,v_i|v|^2\sqrt{\mu}\big]$
we  thereby obtain the so-called macroscopic equations
\newcommand{\J}{\mathcal{Y}} 
\beq\label{qkj:esti:diss:2}
 \left\{
\begin{array}{ll}
\sqrt{\mu}:  &\qquad         \displaystyle \pt_t\Big(a^\e-\frac{3}{2}c^\e\Big) = \J_1 , \,\quad\qquad\qquad     \\[2mm]
v_i\sqrt{\mu}:&\qquad   \displaystyle \pt_tb_i^\e+\frac{1}{\e}\pt_i\Big(a^\e-\frac{3}{2}c^\e\Big)+\frac{1}{\e}\pt_i\phi^\e=\J _2 ^{i}  , \quad\qquad\qquad\\[2mm]
v_i^2\sqrt{\mu}:&\qquad      \displaystyle   \frac{1}{2}\pt_t c^\e+\frac{1}{\e}\pt_ib_i^\e= \J_3 ^{i}  , \quad\qquad\qquad \\[2mm]
v_iv_j\sqrt{\mu}(i\ne j):&\qquad    \displaystyle   \frac{1}{\e}\left(\pt_ib_j^\e+\pt_jb_i^\e\right)= \J _4 ^{i,j}, \,\,\qquad\qquad \\[2mm]
v_i|v|^2\sqrt{\mu}:&\qquad  \displaystyle   \frac{1}{\e}\pt_ic^\e= \J _5^i. \quad\qquad\qquad
\end{array}
\right.
\eeq
Here the right-hand terms $ \J_1,\J _2 ^{i},\J_3 ^{i},\J _4 ^{i,j}$ and $\J _5^i (i,j=1,2,3) $ are all of the form
\bals
\bsp
-\Big\langle\pt_t(\mathbf{I}-\mathbf{P}) f^\e, \zeta \Big\rangle_{L^2_v}
+\Big\langle-\frac{1}{\e}v\cdot\nabla_x (\mathbf{I}-\mathbf{P}) f^\e-\frac{1}{\e^2}L f^\e
 + \frac{\nabla_x \phi^\e\cdot\nabla_v(\sqrt{\mu} f^\e)}{\sqrt{\mu}}
 +\frac{1}{\e}\Gamma(f^\e,f^\e), \zeta \Big\rangle_{L^2_v},
\esp
\eals
where $\zeta$ is a (different) linear combination of the basis $\left[\sqrt{\mu},\,v_i\sqrt{\mu},\,v_iv_j\sqrt{\mu},\,v_i|v|^2\sqrt{\mu} \right]$.

The second set of equations we consider are the local conservation laws satisfied
by $(a^\e, b^\e, c^\e)$. To derive them, multiplying the first equation in (\ref{eq:f}) by the collision
invariant $(1, v, \frac{|v|^2-3}{2})\sqrt{\mu}$ and integrating  the resulting  identity over $\mathbb{R}^3_v$,  we derive by
using  (\ref{defination:Pf}) that $a^\e$, $b^\e$ and $c^\e$ obey the local macroscopic balance laws of mass, moment and energy
\beq\label{qkj:esti:f:L^2:12}
 \left\{
\begin{array}{ll}
 \displaystyle\pt_t a^\e=-\frac{1}{ \e}\nabla_x\cdot b^\e,\\[2mm]
 \displaystyle \pt_t b^\e=-\frac{1}{\e}\nabla_x \phi^\e+\mathcal{X}_{b^\e},\\[2mm]
 \displaystyle\pt_t c^\e
 =-\frac{2}{3\e}\nabla_x\cdot b^\e+\mathcal{X}_{c^\e}.
\end{array}
\right.
\eeq
Here $\mathcal{X}_{b^\e}$ and $\mathcal{X}_{c^\e}$ are given by
\bals
\mathcal{X}_{b^\e}&:=- \frac{1}{\e}\nabla_x(a^\e+c^\e)-a^\e\nabla_x \phi^\e
 - \frac{1}{\e}\int_{\mathbb{R}^3}v\cdot \nabla_x (\mathbf{I}-\mathbf{P})f^\e v\sqrt{\mu} \d v,\\
\mathcal{X}_{c^\e}&:=- \frac{2}{3\e }\nabla_x\phi^\e\cdot b^\e
-\frac{1}{3\e}\int_{\bbR^3}v\cdot\nabla_x(\mathbf{I}-\mathbf{P})f^\e|v|^2\sqrt{\mu}\d v.
\eals

Based on the analysis of the macro fluid-type system \eqref{qkj:esti:diss:2} and the local conservation
laws \eqref{qkj:esti:f:L^2:12},  the following nonlinear
estimate deduced by \eqref{Nonlinear:term:proof:1}
\bals
\sum_{|\a|\leq 1}\left\langle\partial_x^\a \Gamma(f^\e,f^\e), \zeta \right\rangle_{L^2_v}
\lesssim
\| w_\vartheta f^\e\|_{L_{v}^\infty}\sum_{|\a|\leq 1}\| \partial_x^\a(\mathbf{I}-\mathbf{P})f^\e\|_{L_{v}^2(\nu)}
+\| \mathbf{P}f^\e\|_{L_{v}^2}\sum_{|\a|\leq 1}\|\partial_x^\a f^\e\|_{L_{v}^2(\nu)},
\eals
 and  the arguments employed in \cite{GZW-2021}, we conclude the desired estimate  \eqref{qkj:esti:diss:0}.  The details are omitted for
 simplicity. Here, we only point out the representation of
$\mathcal{E}^\mathbf{s}_{int}(t)$ as
\bal
\bsp\label{qkj:esti:diss:0-2}
\mathcal{E}^\mathbf{s}_{int}(t)=\;& \eta_0\sum_{|\a|\leq 1}\big\langle   \nabla_x\cdot \partial_x^\a b ,\, \partial_x^\a a\big\rangle_{L^2_x}
+\sum_{|\a|\leq 1}\sum_{i,j=1}^3\big\langle \partial_x^\a (\mathbf{I}-\mathbf{P}) f^\e ,\,g_b^{i,j}\pt_j \partial_x^\a b_i \big\rangle_{L^2_{x,v}}\\
\;&
+\sum_{|\a|\leq 1}\sum_{i=1}^3\big\langle  \partial_x^\a(\mathbf{I}-\mathbf{P}) f^\e ,\,g_c^{i}\pt_i \partial_x^\a c \big\rangle_{L^2_{x,v}},
\esp
\eal
where $ g_b^{i,j}$ and $g_c^{i}$ are linear combinations of
$\left[\sqrt{\mu},\,v_i\sqrt{\mu},\,v_iv_j\sqrt{\mu},\,v_i|v|^2\sqrt{\mu} \right]$
and  $\eta_0> 0$ is a small   constant. Further, \eqref{qkj:esti:diss:0-1} directly follows from \eqref{qkj:esti:diss:0-2} and the H\"{o}lder inequality. This   completes the proof.
\end{proof}
Next, we   perform the
 $H^2_{x,v}$  energy estimate for the VPB system  (\ref{eq:f}) with soft potentials $-3<\gamma<0$.
\begin{lemma}\label{prop:f:H^2:1}
Assume that $ (f^\e, \nabla_x\phi^\e) $ is a solution to the VPB system \eqref{eq:f} defined on $ [0, T] \times \mathbb{R}^3 \times \mathbb{R}^3$. Then, there holds
\bal
\bsp\label{11111}
 &\frac{\d}{\d t}\Big[C_0\sum_{|\a|\leq2}\Big(\|\partial_x^\a f^\e\|_{L^2_{x,v}}^2 +\|\partial_x^\a\nabla_x  \phi^\e \|_{L^2_x}^2\Big)
 +\sum_{m=1}^2 C_m \sum_{\substack{|\a|+|\b|\leq 2\\|\b|=m}} \|w^{|\b|}\partial^\a_\b(\mathbf{I}-\mathbf{P})f^\e\|_{L_{x,v}^2}^2\Big] \\
  &
  -C_0\e\frac{\d}{\d t}  \int_{\mathbb{R}^3}|b^\e|^2(a^\e+c^\e)\dd x
  +\frac{\sigma_0}{2\e^2}
  \sum_{\substack{|\a|+|\b|\leq2}}
\|w^{|\b|}\partial_\b^\a(\mathbf{I}-\mathbf{P})f^\e\|_{L^2_{x,v}(\nu) }^2\\
 \lesssim  \;&
 \left(\big[\mathcal{E}^\mathbf{s}(t)\big]^{\frac{1}{2}}
 + \mathcal{E}^\mathbf{s}(t)\right) \mathcal{D}^\mathbf{s}(t)+
\Big( \e^3\|w_{\vartheta}f^\e\|_{W^{2,\infty}_{x,v}}+\e
 \|w_{\vartheta}f^\e\|_{W^{1,\infty}_{x,v}}\Big)\mathcal{D}^\mathbf{s}(t)+
 \1\|\nabla_x\mathbf{P}f^\e\|_{H^1_{x}L^2_{v}}^2
\esp
\eal
 for any $0\leq t \leq  T$, where $C_0\gg C_1\gg C_2>0$ are   fixed large constants
 and  $\eta> 0$ is an enough small   constant throughout this section.
\end{lemma}

\begin{proof}
The proof  of (\ref{11111}) is divided into
three steps as follows.

 \emph{Step 1. $L_{x,v}^2$-estimate of $f^\e$.}\;
Taking the  $L^2_{x,v}(\bbR^3\times\mathbb{R}^3)$ inner product of the first equation in (\ref{eq:f}) with $f^\e$ and then employing (\ref{spectL}), we derive
\beq
\begin{split}\label{qkj:esti:f:L^2:1}
&\frac{1}{2}\frac{\d }{\d t}\Big(\|f^\e\|_{L_{x,v}^2}^2
+\|\nabla_x \phi^\e \|_{L^2_x}^2\Big)
+\frac{\sigma_0}{\e^2}\|(\mathbf{I}-\mathbf{P})f^\e\|_{L_{x,v}^2(\nu)}^2\\
\le\;&\underbrace{\Big\langle  \frac{\nabla_x \phi^{\e} \cdot \nabla_v(\sqrt{\mu}f^{\e}) }{\sqrt{\mu}} , f^\e \Big \rangle_{L_{x,v}^2}}_{\mathcal{X}_1}
+
\underbrace{\Big\langle\frac{1}{\e}\Gamma(f^\e,f^\e), f^\e \Big\rangle_{L_{x,v}^2}}_{\mathcal{X}_2}
.
\end{split}
\eeq
Using the decomposition (\ref{fdefenjie})  of $f^\e$  leads to
\bals
\bsp
\mathcal{X}_1=\;&  -\frac{1}{2} \iint_{\mathbb{R}^3\times\mathbb{R}^3} v \cdot\nabla_x \phi^\e |f^\e|^2\d x\d v \\
 =\;&\underbrace{-\frac{1}{2} \iint_{\mathbb{R}^3\times\mathbb{R}^3}  v \cdot\nabla_x \phi^\e|(\mathbf{I}-\mathbf{P})f^\e|^2 \d x\d v }_{\mathcal{X}_{1,1}}
  \underbrace{-\iint_{\mathbb{R}^3\times\mathbb{R}^3}  v \cdot\nabla_x \phi^\e \mathbf{P}f^\e \cdot(\mathbf{I}-\mathbf{P})f^\e  \d x\d v}_{\mathcal{X}_{1,2}}\\
              \;&\underbrace{ -\frac{1}{2} \iint_{\mathbb{R}^3\times\mathbb{R}^3}  v \cdot\nabla_x \phi^\e |\mathbf{P}f^\e|^2 \d x\d v}_{\mathcal{X}_{1,3}}.
\esp
\eals
Then the H\"{o}lder inequality yields
\bal
\bsp\label{qkj:esti:f:L^2:8}
\mathcal{X}_{1,1}
\leq &\;\int_{\mathbb{R}^3}\left\||v|w_{\vartheta}^{\frac{1}{4}}\langle v\rangle^{-\frac{\gamma}{2}}
(\mathbf{I}-\mathbf{P})f^\e\right\|_{L^3_v}\left\|w_{\vartheta}^{-\frac{1}{4}}\right\|_{L^6_v}
\left\|\langle v\rangle^{\frac{\gamma}{2}}
(\mathbf{I}-\mathbf{P})f^\e \right\|_{L^2_v}|\nabla_x \phi^\e |\d x\\
\leq&\;
 \left\||v|w_{\vartheta}^{\frac{1}{4}}\langle v\rangle^{-\frac{\gamma}{2}}
(\mathbf{I}-\mathbf{P})f^\e\right\|_{L^3_{x,v}}
\left\|
(\mathbf{I}-\mathbf{P})f^\e\right\|_{L^2_{x,v}(\nu)}\|\nabla_x \phi^\e \|_{L^6_x}\\
\leq&\;\Big(\frac{1}{\e}\|(\mathbf{I}-\mathbf{P})f^\e\|_{L_{x,v}^2}+
\e^2\|w_{\vartheta}f^{\e}\|_{L^\infty_{x,v}}\Big)\|
(\mathbf{I}-\mathbf{P})f^\e\|_{L^2_{x,v}(\nu)}\|\nabla_x^2 \phi^\e \|_{L^2_x}\\
\lesssim&\;\big[\mathcal{E}^\mathbf{s}(t)\big]^{\frac{1}{2}}\mathcal{D}^\mathbf{s}(t)+ \e^{3}\|w_{\vartheta}f^{\e}\|_{L^\infty_{x,v}}\mathcal{D}^\mathbf{s}(t),
\esp
\eal
where we have utilized  the Young inequality to  bound
\bals
\bsp
\left\||v|w_{\vartheta}^{\frac{1}{4}}\langle v\rangle^{-\frac{\gamma}{2}}
(\mathbf{I}-\mathbf{P})f^\e\right\|_{L^3_{x,v}} \leq\;
&\Big(\frac{1}{\e}\left\|
(\mathbf{I}-\mathbf{P})f^\e\right\|_{L^2_{x,v}}\Big)^\frac{2}{3}
\Big(\e^2\left\||v|^3w_{\vartheta}^{\frac{3}{4}}\langle v\rangle^{-\frac{3\gamma}{2}}
(\mathbf{I}-\mathbf{P})f^\e\right\|_{L^\infty_{x,v}}\Big)^\frac{1}{3}\\
\lesssim\;&
\frac{1}{\e}\|(\mathbf{I}-\mathbf{P})f^\e\|_{L_{x,v}^2}+
\e^2\|w_{\vartheta}f^{\e}\|_{L^\infty_{x,v}}.
\esp
\eals
For the term $\mathcal{X}_{1,2},$
  the H\"{o}lder inequality and the Sobolev embedding inequality  imply
  \bals
\bsp
&\;\mathcal{X}_{1,2}
\lesssim \|\nabla_x \phi^\e\|_{L_x^3} \|\mathbf{P}f^\e\|_{L^6_xL^2_v}\|(\mathbf{I}-\mathbf{P})f^\e\|_{L_{x,v}^2(\nu)}
\lesssim \e\big[\mathcal{E}^\mathbf{s}(t)\big]^{\frac{1}{2}}\mathcal{D}^\mathbf{s}(t).
\esp
\eals
For the term $\mathcal{X}_{1,3}$, we derive from   (\ref{qkj:esti:f:L^2:12}) and the H\"{o}lder inequality that
\bals
\bsp
\mathcal{X}_{1,3}
=\;&-\frac{1}{2}\iint_{\mathbb{R}^3\times\mathbb{R}^3}    v \cdot\nabla_x \phi^\e \Big[\Big(a^\e+b^\e\cdot v+c^\e\frac{|v|^2-3}{2}\Big)\sqrt{\mu}\Big]^2  \d x\d v \\
=\;&-\int_{\bbR^3}    (a^\e+c^\e)b^\e\cdot\nabla_x \phi^\e    \d x \\
\le\;&\int_{\bbR^3}  \e  (a^\e+c^\e)b^\e\cdot\pt_t b^\e\d x+\int_{\bbR^3} | \e (a^\e+c^\e)b^\e\cdot\mathcal{X}_{b^\e}|\d x\\
 =\;&\frac{\e}{2}\frac{\d}{\d t}\int_{\bbR^3}     (a^\e+c^\e)|b^\e|^2\d x-\frac{1}{2} \int_{\bbR^3}  \e \pt_t (a^\e+c^\e)|b^\e|^2\d x+\int_{\bbR^3} | \e (a^\e+c^\e)b^\e\cdot\mathcal{X}_{b^\e}|\d x \\
 \le\;&\frac{\e}{2}\frac{\d}{\d t}\int_{\bbR^3}   (a^\e+c^\e)|b^\e|^2\d x+C\int_{\bbR^3}  \e \Big(|\mathcal{X}_{c^\e}|+\frac{1}{ \e}|\nabla_x\cdot b^\e|\Big)  |{b^\e}|^2\d x\\
\;& +\int_{\bbR^3} | \e (a^\e+c^\e)b^\e\cdot\mathcal{X}_{b^\e}|\d x\\
 \leq\;& \frac{\e}{2}\frac{\d}{\d t}\int_{\bbR^3}    (a^\e+c^\e)|b^\e|^2\d x
 +C\big[\mathcal{E}^\mathbf{s}(t)\big]^{\frac{1}{2}} \mathcal{D}^\mathbf{s}(t)+C \mathcal{E}^\mathbf{s}(t) \mathcal{D}^\mathbf{s}(t),
\esp
\eals
where the fourth equality   above is derived from the integrating by parts with respect to $t$.
Consequently,   the above estimates on $\mathcal{X}_{1,i} (i=1,2,3)$  give us
\bal
\bsp\label{qkj:esti:f:L^2:14}
\mathcal{X}_1&\leq \frac{1}{2}\frac{\d}{\d t}\int_{\bbR^3}  \e (a^\e+c^\e)|b^\e|^2\d x+ C\Big(\big[\mathcal{E}^\mathbf{s}(t)\big]^{\frac{1}{2}}+ \mathcal{E}^\mathbf{s}(t) + \e^{3}\|w_{\vartheta}f^{\e}\|_{L^\infty_{x,v}} \Big)\mathcal{D}^\mathbf{s}(t).
\esp
\eal

We now deal with the term $\mathcal{X}_2$ via the collision symmetry of $\Gamma $, \eqref{es-energy-Nonlinear:2}, the H\"{o}lder inequality and  the Sobolev embedding inequality. More precisely,
\bal
\bsp\label{qkj:esti:f:L^2:3}
\mathcal{X}_2=&\;\Big\langle\frac{1}{\e}\Gamma(f^\e,f^\e), (\mathbf{I}-\mathbf{P})f^\e \Big\rangle_{L_{x,v}^2}\\
\lesssim &\;
\frac{1}{\e}\| w_\vartheta f^\e\|_{L_{x,v}^\infty} \|(\mathbf{I}-\mathbf{P})f^\e\|_{L_{x,v}^2(\nu)}^2
+\frac{1}{\e}\|\mathbf{P} f^\e\|_{L_{x}^\infty L_{v}^2}
\|f^\e\|_{L_{x,v}^2(\nu)}\|(\mathbf{I}-\mathbf{P})f^\e\|_{L_{x,v}^2(\nu)}\\
\lesssim&\;
\e\|w_{\vartheta}f^{\e}\|_{L^{\infty}_{x,v}} \mathcal{D}^\mathbf{s}(t)
+\big[\mathcal{E}^\mathbf{s}(t)\big]^{\frac{1}{2}}\mathcal{D}^\mathbf{s}(t).
\esp
\eal
Therefore, inserting the bounds  (\ref{qkj:esti:f:L^2:14}) and (\ref{qkj:esti:f:L^2:3})
into (\ref{qkj:esti:f:L^2:1}), we derive  \begin{align}
\begin{split}\label{qkj:esti:f:L^2}
 &\;\frac{1}{2}\frac{\d }{\d t}\Big(\|f^\e\|_{L_{x,v}^2}^2 +\|\nabla_x \phi^\e\|_{L_{x}^2}^2 -\e \int_{\mathbb{R}^3}|b^\e|^2(a^\e+c^\e)\dd x\Big)+\frac{\sigma_0}{\e^2}\|(\mathbf{I}-\mathbf{P})f^\e\|_{L_{x,v}^2(\nu)}^2\\
 \lesssim &\;\big[\mathcal{E}^\mathbf{s}(t)\big]^{\frac{1}{2}}\mathcal{D}^\mathbf{s}(t)+
 \mathcal{E}^\mathbf{s}(t)\mathcal{D}^\mathbf{s}(t)+\e\|w_{\vartheta}f^{\e}\|_{L^{\infty}_{x,v}} \mathcal{D}^\mathbf{s}(t).
\end{split}
\end{align}

\emph{Step 2. $L_{x,v}^2$-estimate of the pure space derivative  $\partial_x^\a f^\e$ $(1\leq|\a|\leq 2).$\;}\;
Applying $\partial_x^\a $ to (\ref{eq:f}) and then taking the $L^2_{x,v}(\bbR^3\times\bbR^3)$ inner product with $\partial_x^\a f^\e$, we  obtain   by employing (\ref{spectL}) that
\bal
\bsp\label{qkj:esti:Dxf:L^2:1}
& \frac{1}{2}\frac{\d }{\d t} \big(\| \partial_x^\a f^\e\|_{L_{x,v}^2}^2+ \| \partial_x^\a\nabla_x\phi^\e\|_{L_{x}^2}^2\big)+\frac{\sigma_0}{\e^2}\| \partial_x^\a (\mathbf{I}-\mathbf{P})f^\e\|_{L_{x,v}^2(\nu)}^2  \\
\le \;& \underbrace{ \Big\langle  \partial_x^\a\Big[ \frac{\nabla_x \phi^{\e} \cdot \nabla_v(\sqrt{\mu}f^{\e}) }{\sqrt{\mu}}\Big] ,  \partial_x^\a f^\e \Big \rangle_{L_{x,v}^2}}_{\mathcal{X}_3}
+
\underbrace{\Big\langle\frac{1}{\e}\partial_x^\a\Gamma(f^\e,f^\e), \partial_x^\a f^\e \Big\rangle_{L_{x,v}^2}}_{\mathcal{X}_4}.
\esp
\eal
Recalling the decomposition \eqref{fdefenjie}  and applying a straightforward computation lead  to
\bal
\mathcal{X}_3=\;&\sum_{\substack{|\a'|<|\a|}}\Big\langle  \partial_x^{\a-\a'}\nabla_x \phi^{\e} \cdot \nabla_v\partial_x^{\a'}f^{\e},
 \partial_x^\a f^\e \Big \rangle_{L_{x,v}^2}
-\sum_{\substack{|\a'|\leq|\a|}}
\Big\langle \frac{v}{2}\cdot  \partial_x^{\a-\a'}\nabla_x \phi^{\e} \partial_x^{\a'} f^{\e},
 \partial_x^\a f^\e \Big \rangle_{L_{x,v}^2}\notag\\
 = \;&\underbrace{\sum_{\substack{|\a'|<|\a|}}\Big\langle  \partial_x^{\a-\a'}\nabla_x \phi^{\e} \cdot \nabla_v  \partial_x^{\a'}f^{\e},
 \partial_x^\a (\mathbf{I}-\mathbf{P}) f^\e \Big \rangle_{L_{x,v}^2}}_{\mathcal{X}_{3,1}}\notag\\
 \;&+\underbrace{\sum_{\substack{|\a'|<|\a|}}\Big\langle  \partial_x^{\a-\a'}\nabla_x \phi^{\e} \cdot \nabla_v \partial_x^{\a'}f^{\e},
 \partial_x^\a \mathbf{P}f^\e \Big \rangle_{L_{x,v}^2}}_{\mathcal{X}_{3,2}}\label{qkj:esti:Dxf:L^2:3-1}\\
 \;&\underbrace{-\sum_{\substack{|\a'|\leq|\a|}}
\Big\langle \frac{v}{2}\cdot  \partial_x^{\a-\a'}\nabla_x \phi^{\e}\partial_x^{\a'} f^{\e},
 \partial_x^\a (\mathbf{I}-\mathbf{P})f^\e \Big \rangle_{L_{x,v}^2}}_{\mathcal{X}_{3,3}}\notag\\
 \;&\underbrace{-\sum_{\substack{|\a'|\leq|\a|}}
\Big\langle \frac{v}{2}\cdot  \partial_x^{\a-\a'}\nabla_x \phi^{\e}\partial_x^{\a'} f^{\e},
 \partial_x^\a \mathbf{P}f^\e \Big \rangle_{L_{x,v}^2}}_{\mathcal{X}_{3,4}}\notag
\eal
Then taking the similar way as the estimate of $\mathcal{X}_{1,1}$, we find   that
 \bals
 \bsp
 {\mathcal{X}_{3,1}}\lesssim\;&\sum_{\substack{|\a'|<|\a|}}\| \partial_x^{\a-\a'}\nabla_x \phi^{\e}\|_{L^6_x}
\|\langle v\rangle^{-\frac{\gamma}{2}}w_{\vartheta}^{\frac{1}{4}}\nabla_v \partial_x^{\a'} f^\e\|_{L_{x,v}^3}
\|\partial_x^{\a}(\mathbf{I}-\mathbf{P})f^\e\|_{L_{x,v}^2(\nu)}\\
\lesssim\;&\e\sum_{\substack{|\a'|<|\a|}}\| \nabla_x a^\e\|_{H^1_x}
\Big(\frac{1}{\e}\|w\nabla_v \partial_x^{\a'}f^\e\|_{L_{x,v}^2}+
\e^2\|w_{\vartheta}\nabla_v \partial_x^{\a'} f^{\e}\|_{L^\infty_{x,v}}\Big)
\big[\mathcal{D}^\mathbf{s}(t)\big]^{\frac{1}{2}}\\
\lesssim\;&
\big[\mathcal{E}^\mathbf{s}(t)\big]^{\frac{1}{2}}\mathcal{D}^\mathbf{s}(t)
+\e^3\|w_{\vartheta}\nabla_v f^{\e}\|_{W^{1,\infty}_{x,v}}\mathcal{D}^\mathbf{s}(t)
 \esp
\eals
and
\bals
 \bsp
 {\mathcal{X}_{3,3}}\lesssim\;&\sum_{\substack{|\a'|\leq|\a|}}\| \partial_x^{\a-\a'}\nabla_x \phi^{\e}\|_{L^6_x}
\||v|\langle v\rangle^{-\frac{\gamma}{2}}w_{\vartheta}^{\frac{1}{4}} \partial_x^{\a'} f^\e\|_{L_{x,v}^3}
\|\partial_x^{\a}(\mathbf{I}-\mathbf{P})f^\e\|_{L_{x,v}^2(\nu)}\\
 \lesssim \;&\e
\sum_{\substack{|\a'|\leq|\a|}}\|a^\e\|_{H^2_x}
\Big(\frac{1}{\e}\|\partial_x^{\a'} f^\e\|_{L_{x,v}^2}+\e^2\| w_\vartheta \partial_x^{\a'} f^\e\|_{L_{x,v}^\infty} \Big)
\big[\mathcal{D}^\mathbf{s}(t)\big]^{\frac{1}{2}}
\\
 \lesssim\;&
\big[\mathcal{E}^\mathbf{s}(t)\big]^{\frac{1}{2}}\mathcal{D}^\mathbf{s}(t)
+\e^3\|w_{\vartheta}  f^{\e}\|_{W^{2,\infty}_{x,v}} \mathcal{D}^\mathbf{s}(t).
 \esp
\eals
Here we have used a basic inequality
\bal\label{esti:Dxxa}
\|a^\e\|_{H^2_x}=\|a^\e\|_{L^2_x}+\|\nabla_xa^\e\|_{H^1_x}\leq
\|\nabla_x^2\phi^\e\|_{L^2_x}+\|\nabla_x\mathbf{P}f^\e\|_{H^1_xL^2_v}
\lesssim\big[\mathcal{D}^\mathbf{s}(t)\big]^{\frac{1}{2}},
\eal
which is derived from \eqref{defination:Pf} and the Riesz potential theory to the Poisson equation  in (\ref{eq:f}).
For the term  $\mathcal{X}_{3,2}$, we  deduce  from integrating by parts with respect to $v$,
 the H\"{o}lder inequality and  \eqref{esti:Dxxa} that
\bals
 \bsp
 \mathcal{X}_{3,2}\;&=-\sum_{\substack{|\a'|<|\a|}}\Big\langle  \partial_x^{\a-\a'}\nabla_x \phi^{\e} \cdot  \partial_x^{\a'}f^{\e},
\nabla_v  \partial_x^\a \mathbf{P} f^\e \Big \rangle_{L_{x,v}^2}\\
\;&\lesssim\sum_{\substack{|\a'|<|\a|}}\| \partial_x^{\a-\a'}\nabla_x \phi^{\e}\|_{L^3_x}
\|\partial_x^{\a'}f^{\e}\|_{L_{{x}}^6L_{v}^2}
\|\nabla_v \partial_x^{\a}\mathbf{P} f^\e\|_{L_{x,v}^2}\\
\;&\lesssim\| a^{\e}\|_{H^2_x}
\|\nabla_xf^{\e}\|_{H_{{x}}^1L_{v}^2}
\|\partial_x^{\a}\mathbf{P} f^\e\|_{L_{x,v}^2}\\
\;&\lesssim\big[\mathcal{E}^\mathbf{s}(t)\big]^{\frac{1}{2}}\mathcal{D}^\mathbf{s}(t).
  \esp
\eals
Further, by the H\"{o}lder inequality, the Sobolev embedding and  \eqref{esti:Dxxa}, we have
 \bals
 \bsp
 {\mathcal{X}_{3,4}}\lesssim \;&
\Big(\sum_{\substack{|\a'|<|\a|}}\|\partial_x^{\a-\a'}\nabla_x \phi^{\e}\|_{L^6_x}
\|\partial_x^{\a'}f^\e\|_{L_{x}^3L_{v}^2}
+\|\nabla_x \phi^{\e}\|_{L^\infty_x}
\|\partial_x^\a f^\e\|_{L_{x,v}^2}\Big)
\|\partial_x^\a\mathbf{P} f^\e\|_{L_{x,v}^2}\\
\lesssim\;&
\Big(\|a^{\e}\|_{H^2_x}
\|f^\e\|_{H_{x}^2L_{v}^2}
+\|\nabla_x^2 \phi^{\e}\|_{H^1_x}
\|\partial_x^\a f^\e\|_{L_{x,v}^2}\Big)
\|\partial_x^\a\mathbf{P} f^\e\|_{L_{x,v}^2}\\
\lesssim\;&\big[\mathcal{E}^\mathbf{s}(t)\big]^{\frac{1}{2}}\mathcal{D}^\mathbf{s}(t).
 \esp
\eals
Hence, substituting the above bounds on ${\mathcal{X}_{3,i}} (i=1,2,3,4)$ into \eqref{qkj:esti:Dxf:L^2:3-1} yields
 \bal
 \bsp\label{qkj:esti:Dxxf:L^2:32}
 {\mathcal{X}_{3}}\lesssim
\big[\mathcal{E}^\mathbf{s}(t)\big]^{\frac{1}{2}}\mathcal{D}^\mathbf{s}(t)
+\e^3\|w_{\vartheta} f^{\e}\|_{W^{2,\infty}_{x,v}} \mathcal{D}^\mathbf{s}(t).
 \esp
\eal

For the term $\mathcal{X}_4$,
the collision symmetry of $\Gamma $ implies
\bal
\bsp\label{qkj:esti:Dxxf:L^2:8}
\mathcal{X}_4=&\;\underbrace{\frac{1}{\e}\Big\langle
\Gamma(f^\e,\partial_x^\a f^\e)+\Gamma(\partial_x^\a f^\e,f^\e), \partial_x^\a(\mathbf{I}-\mathbf{P}) f^\e \Big\rangle_{L_{x,v}^2}}_{\mathcal{X}_{4,1}}\\
&\;+\underbrace{\frac{1}{\e}\Big\langle
\Gamma(\nabla_xf^\e, \nabla_xf^\e), \partial_x^{\a}(\mathbf{I}-\mathbf{P}) f^\e \Big\rangle_{L_{x,v}^2}}_{\mathcal{X}_{4,2}},
\esp
\eal
where the term ${\mathcal{X}_{4,2}}$  only exists when $|\a|=2$.
Similar to the estimate of $\mathcal{X}_{2}$ in (\ref{qkj:esti:f:L^2:3}),  we derive
\bals
\bsp
\mathcal{X}_{4,1}\lesssim\e\|w_{\vartheta}f^{\e}\|_{L^{\infty}_{x,v}} \mathcal{D}^\mathbf{s}(t)
+\big[\mathcal{E}^\mathbf{s}(t)\big]^{\frac{1}{2}}\mathcal{D}^\mathbf{s}(t).
\esp
\eals
Besides, thanks to   \eqref{es-energy-Nonlinear:2},   the H\"{o}lder inequality, the Sobolev embedding $H^{1}_x(\mathbb{R}^3)\hookrightarrow L^3_x(\mathbb{R}^3)$ and  $H^{1}_x(\mathbb{R}^3)\hookrightarrow L^6_x(\mathbb{R}^3)$,  we obtain
\bals
\bsp
{\mathcal{X}_{4,2}} \lesssim&\;\frac{1}{\e}\| w_\vartheta \nabla_xf^\e\|_{L_{x,v}^\infty} \|\nabla_x(\mathbf{I}-\mathbf{P}) f^\e\|_{L_{x,v}^2(\nu)}
\|\nabla_x^2(\mathbf{I}-\mathbf{P}) f^\e\|_{L_{x,v}^2(\nu)}\\
&\;
+\frac{1}{\e}\|\nabla_x\mathbf{P} f^\e\|_{L_{x}^6 L_{v}^2}
 \|\nabla_xf^\e\|_{L_{x}^3 L_{v}^2(\nu)}
\|\nabla_x^2(\mathbf{I}-\mathbf{P})f^\e\|_{L_{x,v}^2(\nu)}\\
\lesssim&\;
\e\|w_{\vartheta}\nabla_xf^{\e}\|_{L^{\infty}_{x,v}} \mathcal{D}^\mathbf{s}(t)
+\big[\mathcal{E}^\mathbf{s}(t)\big]^{\frac{1}{2}}\mathcal{D}^\mathbf{s}(t).
\esp
\eals
Thus, putting  the above estimates on $\mathcal{X}_{4,1}$ and $\mathcal{X}_{4,2}$
into \eqref{qkj:esti:Dxxf:L^2:8} leads to
\bal
\bsp\label{qkj:esti:Dxxf:L^2:10}
\mathcal{X}_{4}
\lesssim\e\|w_{\vartheta}\nabla_xf^{\e}\|_{L^{\infty}_{x,v}} \mathcal{D}^\mathbf{s}(t)
+\e\|w_{\vartheta} f^{\e}\|_{L^{\infty}_{x,v}} \mathcal{D}^\mathbf{s}(t)
+\big[\mathcal{E}^\mathbf{s}(t)\big]^{\frac{1}{2}}\mathcal{D}^\mathbf{s}(t).
\esp
\eal
In summary, gathering the bounds  \eqref{qkj:esti:Dxxf:L^2:32} and (\ref{qkj:esti:Dxxf:L^2:10}) into \eqref{qkj:esti:Dxf:L^2:1},   we conclude
\bal
\bsp\label{qkj:esti:Dxxf:L^2}
\;&\frac{1}{2}\frac{\d }{\d t}\left( \|\partial_x^\a f^\e\|_{L_{x,v}^2}^2 + \|\partial_x^\a\nabla_x\phi^\e\|_{L_{x}^2}^2 \right)+\frac{\sigma_0}{\e^2}\|\partial_x^\a (\mathbf{I}-\mathbf{P})f^\e\|_{L_{x,v}^2(\nu)}^2\\
 \lesssim\;&\big[\mathcal{E}^\mathbf{s}(t)\big]^{\frac{1}{2}}\mathcal{D}^\mathbf{s}(t)+
\left(\e\|w_{\vartheta}f^{\e}\|_{W^{1,\infty}_{x,v}}
+\e^3\|w_{\vartheta}f^{\e}\|_{W^{2,\infty}_{x,v}}\right)\mathcal{D}^\mathbf{s}(t).
\esp
\eal

\emph{Step 3. $L_{x,v}^2$-estimate of the   mixed derivative  $\partial_{{\b}}^\a f^\varepsilon$ $(|\a|+|\b|\leq2,1\leq|\b|\leq 2).$}\;
Indeed, thanks to the Maxwellian
structure in the macro part $\mathbf{P}f^\varepsilon$, there holds
\begin{equation*}
 \|w^{|\b|}\partial_{{\b}}^\a\mathbf{P}f^\varepsilon\|_{L^2_{x, v}} \lesssim\|\partial^\a_x\mathbf{P}f^\varepsilon\|_{L^2_{x, v}},
\end{equation*}
which helps us conclude the  ${L^2_{x, v}}$-estimate of $ w^{|\b|}\partial_{{\b}}^\a\mathbf{P}f^\varepsilon$.
Thus it is adequate to deduce the momentum derivative estimate of
$ w^{|\b|}\partial_{{\b}}^\a(\mathbf{I}-\mathbf{P})f^\varepsilon$.
Applying microscopic projection to   the first equation in (\ref{eq:f}) and using  $(\mathbf{I}-\mathbf{P})( v\cdot \nabla_x \phi^\e\sqrt{\mu})=0$, we have
\bal
\bsp\label{qkj:esti:Dv{I-P}f:L^2:1}
&\;\pt_t(\mathbf{I}-\mathbf{P})f^\e+\frac{1}{\e}v\cdot \nabla_x (\mathbf{I}-\mathbf{P})f^\e+\frac{1}{\e^2}L((\mathbf{I}-\mathbf{P})f^\e)\\
=&\;\nabla_x \phi^\e\cdot \nabla_v (\mathbf{I}-\mathbf{P}) f^\e
-\frac{1}{2} v \cdot\nabla_x \phi^\e(\mathbf{I}-\mathbf{P}) f^\e +\frac{1}{\e}\Gamma(f^\e,f^\e)
+[[\mathbf{P},\mathcal{A}_{\phi^\e}]]f^\e.
\esp
\eal
Here $ [[\mathbf{P},\mathcal{A}_{\phi^\e}]]
:=\mathbf{P}\mathcal{A}_{\phi^\e}-\mathcal{A}_{\phi^\e}\mathbf{P}$ denotes the commutator of two operators  and $\mathcal{A}_{\phi^\e}$ is given by
$$\mathcal{A}_{\phi^\e}:=\frac{1}{\e}v\cdot \nabla_x - \nabla_x \phi^\e\cdot \nabla_v+\frac{1}{2} v \cdot\nabla_x \phi^\e .$$

Applying $\partial_{\b}^\a$ to (\ref{qkj:esti:Dv{I-P}f:L^2:1}) and then taking the $L^2_{x,v}(\bbR^3\times\bbR^3)$ inner product with $w^{2|\b|}\partial_{\b}^\a(\mathbf{I}-\mathbf{P})f^\e$, we  obtain by (\ref{es-energy-linear:1})   that
\bal
\bsp\label{qkj:esti:Dv{I-P}f:L^2:2}
& \frac{1}{2}\frac{\d }{\d t} \|w^{|\b|}\partial_{\b}^\a(\mathbf{I}-\mathbf{P})f^\e\|_{L_{x,v}^2}^2
+\frac{\sigma_0}{\e^2}\| w^{|\b|}\partial_{\b}^\a(\mathbf{I}-\mathbf{P})f^\e\|_{L_{x,v}^2(\nu)}^2
\\
\le
&\;\frac{\eta}{\e^2}\sum_{|\b'|\leq|\b|}\| w^{|\b'|}\partial_{\b'}^\a(\mathbf{I}-\mathbf{P})f^\e\|_{L_{x,v}^2(\nu)}^2
+\frac{C_{\sigma_0}}{\e^2}\| \partial^\a_x(\mathbf{I}-\mathbf{P})f^\e\|_{L_{x,v}^2(\nu)}^2\\
&\;+\underbrace{\sum_{|\b_1|=1}\Big\langle- \frac{1}{\e}C_{\b}^{\b_1}\partial_v^{\b_1}v\cdot\partial_{\b-\b_1}^\a\nabla_x (\mathbf{I}-\mathbf{P}) f^\e , w^{2|\b|}\partial_{\b}^\a(\mathbf{I}-\mathbf{P})f^\e
\Big\rangle_{L_{x,v}^2}}_{\mathcal{X}_5}\\
&\;
+\underbrace{\Big\langle
\partial_{\b}^\a\Big(\nabla_x \phi^\e \cdot\nabla_v(\mathbf{I}-\mathbf{P}) f^\e
\Big) , w^{2|\b|}\partial_{\b}^\a(\mathbf{I}-\mathbf{P})f^\e\Big
\rangle_{L_{x,v}^2}}_{\mathcal{X}_6}\\
&\;+\underbrace{\Big\langle
\partial_{\b}^\a\Big(
-\frac{1}{2}v\cdot\nabla_x \phi^\e (\mathbf{I}-\mathbf{P}) f^\e\Big) , w^{2|\b|}\partial_{\b}^\a(\mathbf{I}-\mathbf{P})f^\e\Big
\rangle_{L_{x,v}^2}}_{\mathcal{X}_7}
\\
&+\underbrace{\Big\langle\frac{1}{\e} \partial_{\b}^\a\Gamma(f^\e,f^\e), w^{2|\b|}\partial_{\b}^\a(\mathbf{I}-\mathbf{P})f^\e\Big
\rangle_{L_{x,v}^2}}_{\mathcal{X}_8}
+\underbrace{\Big\langle \partial_{\b}^\a\big([[\mathbf{P},\mathcal{A}_{\phi^\e}]]f^\e\big),  w^{2|\b|}\partial_{\b}^\a(\mathbf{I}-\mathbf{P})f^\e\Big \rangle_{L_{x,v}^2}.}
_{\mathcal{X}_9}
\esp
\eal

For the term ${\mathcal{X}_5}$, the H\"{o}lder inequality and the Young inequality lead to
\bals
{\mathcal{X}_5}\leq \frac{\eta}{\e^2} \|w^{|\b|}\partial_{\b}^\a(\mathbf{I}-\mathbf{P})f^\e\|_{L^2_{x,v}(\nu)}^2
+C_{\eta  }\|w^{|\b-\b_1|}\partial^\a_{\b-\b_1}\nabla_x (\mathbf{I}-\mathbf{P}) f^\e \|_{L^2_{x,v}(\nu)}^2,
\eals
where we have used
\bals
w^{2|\b|}(v)=w^{|\b|+\frac{1}{2}}(v)
w^{|\b|-1+\frac{1}{2}}(v)\leq \nu^{\frac{1}{2}}(v) w^{|\b|}(v)\cdot
\nu^{\frac{1}{2}}(v)w^{|\b-\b_1|}(v).
\eals

To estimate the term $\mathcal{X}_6$, 
we   integrate by part  over $v\in\mathbb{R}^3$ and  employ  the similar procedure  as  \eqref{qkj:esti:f:L^2:8} that
 \bals
\bsp
{\mathcal{X}_6}
=\;&-\sum_{|\a'|\leq|\a|}\Big\langle
\partial^{\a-\a'}_x\nabla_x \phi^\e \cdot\frac{\nabla_v(w^{2|\b|})}{w^{2|\b|}}\partial^{\a'}_{\b}(\mathbf{I}-\mathbf{P}) f^\e , w^{2|\b|}\partial_{\b}^\a(\mathbf{I}-\mathbf{P})f^\e\Big
\rangle_{L_{x,v}^2}\\
\lesssim\;&\sum_{|\a'|\leq|\a|}\|\partial^{\a-\a'}_x\nabla_x \phi^\e\|_{L^6_x}
\|w^{|\b|}\langle v\rangle^{-\frac{\gamma}{2}}w_{\vartheta}^{\frac{1}{4}}
\partial_{\b}^\a(\mathbf{I}-\mathbf{P})f^\e\|_{L_{x,v}^3}
\|w^{|\b|}\partial_{\b}^{\a'}(\mathbf{I}-\mathbf{P})f^\e\|_{L_{x,v}^2(\nu)}\\
\lesssim\;&\e
\|\nabla_x^2\phi^\e\|_{H^1_x}
\Big(\frac{1}{\e}\|\partial_{\b}^\a(\mathbf{I}-\mathbf{P}) f^\e\|_{L_{x,v}^2}+\e^2\| w_\vartheta \partial_{\b}^\a f^\e\|_{L_{x,v}^\infty} \Big)
\big[\mathcal{D}^\mathbf{s}(t)\big]^{\frac{1}{2}}\\
\lesssim\;&\big[\mathcal{E}^\mathbf{s}(t)\big]^{\frac{1}{2}}\mathcal{D}^\mathbf{s}(t)
+\e^3\|w_{\vartheta} \partial_{\b}^\a f^{\e}\|_{L^{\infty}_{x,v}} \mathcal{D}^\mathbf{s}(t).
\esp
\eals
Similarly, noting that $|\a|\leq1$, we have
\bals
\bsp
\mathcal{X}_7=\;&-\sum_{|\a'|\leq|\a|,|\b'|\leq 1}\Big\langle \frac{1}{2}\partial^{\b'}_v v\cdot \partial^{\a-\a'}_x \nabla_x \phi^\e \partial_{\b-\b'}^{\a'}(\mathbf{I}-\mathbf{P}) f^\e, w^{2|\b|}\partial_{\b}^\a(\mathbf{I}-\mathbf{P})f^\e\Big
\rangle_{L_{x,v}^2}\\
\lesssim\;&\sum_{|\a'|\leq|\a|,|\b'|\leq1}\|\partial^{\a-\a'}_x\nabla_x \phi^\e\|_{L^6_x}
\|w^{|\b|}\langle v\rangle ^{-\frac{\gamma}{2}+1} w_{\vartheta}^{\frac{1}{4}}\partial_{\b}^\a
(\mathbf{I}-\mathbf{P})f^\e\|_{L_{x,v}^3}\\
\;&
\quad\quad\quad\quad\quad\cdot\|w^{|\b|}\partial_{\b-\b '}^{\a'}(\mathbf{I}-\mathbf{P})f^\e\|_{L_{x,v}^2(\nu)}\\
\lesssim\;&\e
\|\nabla_x^2\phi^\e\|_{H^1_x}
\Big(\frac{1}{\e}\|\partial_{\b}^\a(\mathbf{I}-\mathbf{P}) f^\e\|_{L_{x,v}^2}+\e^2\| w_\vartheta \partial_{\b}^\a f^\e\|_{L_{x,v}^\infty} \Big)
\big[\mathcal{D}^\mathbf{s}(t)\big]^{\frac{1}{2}}\\
\lesssim\;&\big[\mathcal{E}^\mathbf{s}(t)\big]^{\frac{1}{2}}\mathcal{D}^\mathbf{s}(t)
+\e^3\|w_{\vartheta} \partial_{\b}^\a f^{\e}\|_{L^{\infty}_{x,v}} \mathcal{D}^\mathbf{s}(t).
\esp
\eals

Further, with the aid of Lemma \ref{es-energy-Nonlinear} and following  the estimate of $\mathcal{X}_4$ in \eqref{qkj:esti:Dxxf:L^2:10}, we infer that
\bals
\bsp
\mathcal{X}_{8}\lesssim
\big[\mathcal{E}^\mathbf{s}(t)\big]^{\frac{1}{2}}\mathcal{D}^\mathbf{s}(t)
+\e\|w_{\vartheta}f^{\e}\|_{L^{\infty}_{x,v}} \mathcal{D}^\mathbf{s}(t)
+\e\|w_{\vartheta}\nabla_vf^{\e}\|_{L^{\infty}_{x,v}} \mathcal{D}^\mathbf{s}(t).
\esp
\eals

Thanks to $|\a|\leq1$,
using the integration by parts with respect to $v$ when $\partial_v^{\b}$
acts on $(\mathbf{I}-\mathbf{P})f^\e$, the H\"{o}lder inequality, the Sobolev embedding inequality and  the Young inequality,  we are able to deduce that
\bals
\bsp
\mathcal{X}_9=\;&\Big\langle \partial_{\b}^\a\big([[\mathbf{P},\mathcal{A}_{\phi^\e}]]f^\e\big),  w^{2|\b|}\partial_{\b}^\a(\mathbf{I}-\mathbf{P})f^\e\Big \rangle_{L_{x,v}^2}\\
\lesssim\;&
\frac{1}{ \e}\|\partial_x^\a (\mathbf{I}-\mathbf{P})f^\e\|_{L_{x,v}^2(\nu)}
\Big(\|\partial_x^\a \nabla_x f^\e\|_{L_{x,v}^2(\nu)}
+\|\partial_x^\a \nabla_x \mathbf{P} f^\e\|_{L_{x,v}^2}\Big)\\
\;&+\|\partial_x^\a (\mathbf{I}-\mathbf{P})f^\e\|_{L_{x,v}^2(\nu)}
\|\nabla_x \phi^\e\|_{W_{x}^{1,\infty}}
\Big(\|w\nabla_v f^\e\|_{{H_{x}^1L_{v}^2}}+\|\nabla_v \mathbf{P}f^\e\|_{{H_{x}^1L_{v}^2}}\Big)\\
\;&+\|\partial_x^\a (\mathbf{I}-\mathbf{P})f^\e\|_{L_{x,v}^2(\nu)}
\|\nabla_x \phi^\e\|_{W_{x}^{1,\infty}}
\Big(\| f^\e\|_{H_{x}^1L_{v}^2}+\| \mathbf{P}f^\e\|_{{H_{x}^1L_{v}^2}}\Big)\\
\lesssim\;&\eta\|\nabla_x\mathbf{P}f^\e\|_{H_{x}^1L_{v}^2}^2
 +\frac{C_\eta}{ \e^2}\|(\mathbf{I}-\mathbf{P})f^\e\|_{H_{x}^2L_{v}^2(\nu)}^2
  +\e\big[\mathcal{E}^\mathbf{s}(t)\big]^{\frac{1}{2}}\mathcal{D}^\mathbf{s}(t),
\esp
\eals
where we have used a fact
\bal
\bsp\label{eq:assumption:hard:1-1}
\| \nabla_x\phi^\e\|_{W^{1,\infty}_{x }}
\lesssim \| \nabla_x^2\phi^\e\|_{H^{1}_{x }}+ \| \nabla_x^3\phi^\e\|_{H^{1}_{x }}
\lesssim \;&\| \nabla_x^2\phi^\e\|_{H^{1}_{x }}+\| \nabla_x\mathbf{P}f^\e\|_{H^{1}_{x }L^{2}_{v }},
\esp
\eal
which is deduced from  the Sobolev embedding and the Riesz potential theory to  the Poisson equation in (\ref{eq:f}).

Thus,    substituting the above estimates on $\mathcal{X}_5\sim\mathcal{X}_9$  into \eqref{qkj:esti:Dv{I-P}f:L^2:2}, and
then
taking the proper linear combination of the $L_{x,v}^2$-estimate
 on $w^{|\b|}\partial^\a_\b(\mathbf{I}-\mathbf{P})f^\e$ $(|\b|=m, |\a| + |\b| \leq 2)$ with properly chosen constants $C_m>0$ for each given $m$,
  we   get  by choosing $\1$ sufficiently small that
   \bal
\bsp\label{qkj:esti:Dv{I-P}f:L^2}
&\;\frac{1}{2}\frac{\d }{\d t}  \sum_{m=1}^2 C_m \sum_{\substack{|\a|+|\b|\leq 2\\|\b|=m}} \|w^{|\b|}\partial^\a_\b(\mathbf{I}-\mathbf{P})f^\e\|_{L_{x,v}^2}^2
+\sum_{\substack{|\a|+|\b|\leq 2\\1\leq|\b|\leq2}}\frac{\sigma_0}{2\e^2}\|w^{|\b|}\partial^\a_\b(\mathbf{I}-\mathbf{P})f^\e\|_{L_{x,v}^2(\nu)}^2\\
\lesssim  &\;\big[\mathcal{E}^\mathbf{s}(t)\big]^{\frac{1}{2}}\mathcal{D}^\mathbf{s}(t)
+\e\|w_{\vartheta}  f^{\e}\|_{W^{1,\infty}_{x,v}}\mathcal{D}^\mathbf{s}(t)+\e^3\|w_{\vartheta} f^{\e}\|_{W^{2,\infty}_{x,v}}\mathcal{D}^\mathbf{s}(t)+\1\|\nabla_x \mathbf{P}f^\e\|_{H_{x}^1 L_{v}^2}^2 \\
&\;
+\| \nabla_x(\mathbf{I}-\mathbf{P})f^\e\|_{H_{x}^1 L^2_{v}(\nu)}^2+\frac{1}{ \e^2}\|(\mathbf{I}-\mathbf{P})f^\e\|_{H_{x}^2L_{v}^2(\nu)}^2.
\esp
\eal

In summary, multiplying (\ref{qkj:esti:f:L^2}) and (\ref{qkj:esti:Dxxf:L^2}) by a sufficiently large constant $C_0$ and
combining the result with   \eqref{qkj:esti:Dv{I-P}f:L^2},   we  conclude (\ref{11111}). This completes the proof.
\end{proof}

Then we are devoted to the weighted $H_{x,v}^2$  energy estimate for the VPB system \eqref{eq:f}  with soft potentials $-3<\gamma<0$.
\begin{lemma}\label{prop:weighted:f:H^2}
Let $\ell>0.$
Assume that $ (f^\e, \nabla_x\phi^\e) $ is a solution to the
VPB system \eqref{eq:f} defined on $  [0, T] \times \mathbb{R}^3 \times \mathbb{R}^3$. Then there holds
\bal
\bsp\label{weighted:11111}
 &\frac{\d}{\d t}\Big(\sum_{m=0}^2 \widetilde{C}_m \sum_{\substack{|\a|+|\b|\leq 2\\0\leq|\a|\leq1,|\b|=m}} \|w^{|\b|-\ell}\partial^\a_\b(\mathbf{I}-\mathbf{P})f^\e\|_{L_{x,v}^2}^2
 +\widetilde{C}_0 \e
\|{w}^{-\ell}\nabla_x^2f^\e\|_{L_{x,v}^2}^2\Big)\\
\;&
  \\
  &+\frac{\sigma_0}{2\e^2}
  \sum_{\substack{|\a|+|\b|\leq2\\0\leq|\a|\leq1}}
\|w^{|\b|-\ell}\partial^\a_\b(\mathbf{I}-\mathbf{P})f^\e\|_{L^2_{x,v}(\nu) }^2
  +\frac{\sigma_0}{4\e}\|{w}^{-\ell}\nabla_x^2(\mathbf{I}-\mathbf{P})f^\e\|_{L_{x,v}^2(\nu)}^2\\
 \lesssim  \;&
\big[\mathcal{E}^\mathbf{s}(t)\big]^{\frac{1}{2}}\mathcal{D}_{\ell}^\mathbf{s}(t)
+\e
 \|w_{\vartheta}f^\e\|_{L^{\infty}_{x,v}}\mathcal{D}_\ell^\mathbf{s}(t)
+\e^3\|w_{\vartheta} f^{\e}\|_{W^{1,\infty}_{x,v}}
\mathcal{D}^\mathbf{s}_\ell(t)+\mathcal{D}^\mathbf{s}(t)
\esp
\eal
for any $0\leq t\leq T$, where $\widetilde{C}_0\gg \widetilde{C}_1\gg \widetilde{C}_2>0$ are   fixed large constants.
\end{lemma}
\begin{proof}
The proof  of (\ref{weighted:11111}) is divided into
three steps.

 \emph{Step 1. The  $L_{x,v}^2$-estimate of ${w}^{-\ell}\partial^\a_x (\mathbf{I}-\mathbf{P}) f^\e$ $(0\leq |\a|\leq 1)$.}\;
Applying $\partial_{x}^\a$ to \eqref{qkj:esti:Dv{I-P}f:L^2:1},
 and  then taking the $L^2_{x,v}$ inner product on the resulting equation   with $w^{-2\ell}\partial_{x}^\a(\mathbf{I}-\mathbf{P})f^\e$, we  obtain by employing \eqref{es-energy-linear:1} that
 \bal
\bsp\label{weighted-qkj:esti:Dx{I-P}f:L^2:2}
& \frac{1}{2}\frac{\d }{\d t} \|w^{-\ell}\partial_{x}^\a(\mathbf{I}-\mathbf{P})f^\e\|_{L_{x,v}^2}^2
+\frac{\sigma_0}{\e^2}\| w^{-\ell}\partial_{x}^\a(\mathbf{I}-\mathbf{P})f^\e\|_{L_{x,v}^2(\nu)}^2
-\frac{C_{\sigma_0}}{\e^2}\| \partial_{x}^\a(\mathbf{I}-\mathbf{P})f^\e\|_{L_{x,v}^2(\nu)}^2\\
\le&\;
\underbrace{\Big\langle\partial_{x}^\a\Big(\nabla_x \phi^\e \cdot\nabla_v(\mathbf{I}-\mathbf{P}) f^\e-
\frac{1}{2}v\cdot\nabla_x \phi^\e (\mathbf{I}-\mathbf{P}) f^\e
\Big) , w^{-2\ell}\partial_{x}^\a(\mathbf{I}-\mathbf{P})f^\e\Big
\rangle_{L_{x,v}^2}}_{\mathcal{X}_{10}}\\
&\;
+\underbrace{\frac{1}{\e}\left\langle \partial_{x}^\a\Gamma(f^\e, f^\e), w^{-2\ell}\partial_{x}^\a(\mathbf{I}-\mathbf{P})f^\e\right
\rangle_{L_{x,v}^2}}_{\mathcal{X}_{11}}
+\underbrace{\left\langle\partial_{x}^\a\big([[\mathbf{P},\mathcal{A}_{\phi^\e}]]f^\e\big), w^{-2\ell}\partial_{x}^\a(\mathbf{I}-\mathbf{P})f^\e\right
\rangle_{L_{x,v}^2}}_{\mathcal{X}_{12}}.
\esp
\eal
Noting that $|\a|\leq1$, by the H\"{o}lder
inequality,  the  Sobolev
embedding inequality and the Young inequality,  we deduce
\bals
\bsp
\mathcal{X}_{10}
 =\;&
 -\Big\langle \Big( \frac{\nabla_v( w^{-2\ell} )}{w^{-2\ell}}+\frac{v}{2}\Big)\cdot\partial_x^\a\left(\nabla_x \phi^{\e} (\mathbf{I}-\mathbf{P})f^{\e}\right), w^{-2\ell}\partial_x^\a  (\mathbf{I}-\mathbf{P})f^\e \Big \rangle_{L_{x,v}^2}\\
 =\;&-\sum_{|\a'|\leq |\a|}\Big\langle \left( \frac{\nabla_v( w^{-2\ell} )}{w^{-2\ell}}+\frac{v}{2}\right)\cdot \partial_x^{\a-\a'}\nabla_x \phi^{\e} \partial_x^{\a'}(\mathbf{I}-\mathbf{P}) f^{\e}, w^{-2\ell}\partial_x^\a (\mathbf{I}-\mathbf{P})f^\e \Big \rangle_{L_{x,v}^2}\\
 \lesssim \;&\sum_{\substack{|\a'|\leq|\a|}}\|  \partial_x^{\a-\a'}\nabla_x \phi^{\e}\|_{L^6_x}\|\langle v \rangle w^{-\ell}w_{\vartheta}^{\frac{1}{4}}\partial_x^{\a'} (\mathbf{I}-\mathbf{P})f^\e \|_{L^3_{x,v}}\|w^{-\ell}\partial_x^\a  (\mathbf{I}-\mathbf{P})f^\e \|_{L^2_{x,v}(\nu)}\\
\lesssim\;& \e\sum_{\substack{|\a'|\leq|\a|}}\| \nabla_x^2 \phi^{\e}\|_{H^1_x}\Big(\frac{1}{\e}\| \partial_x^{\a'}f^\e\|_{L_{x,v}^2}+
\e^2\|w_{\vartheta} \partial_x^{\a'} f^{\e}\|_{L^\infty_{x,v}}\Big)
\big[\mathcal{D}^\mathbf{s}_\ell(t)\big]^{\frac{1}{2}}\\
\lesssim\;& \big[\mathcal{E}^\mathbf{s}(t)\big]^{\frac{1}{2}} \mathcal{D}^\mathbf{s}_\ell(t)+
\e^3\|w_{\vartheta} f^{\e}\|_{W^{1,\infty}_{x,v}}
\mathcal{D}^\mathbf{s}_\ell(t).
\esp
\eals
By exploiting (\ref{es-energy-Nonlinear:2}) in Lemma \ref{es-energy-Nonlinear}, the term $\mathcal{X}_{11}$  is bounded by
\bals
\bsp
\mathcal{X}_{11}=&\;\frac{1}{\e}\Big\langle \Gamma(f^\e,\partial_x^{\a}f^\e)
+\Gamma(\partial_x^{\a}f^\e,f^\e) , w^{-2\ell} \partial_x^{\a}(\mathbf{I}-\mathbf{P})f^\e \Big\rangle_{L_{x,v}^2}\\
\lesssim &\;
\frac{1}{\e}\| w_\vartheta f^\e\|_{L_{x,v}^\infty}
 \|w^{-\ell} \partial_x^{\a}(\mathbf{I}-\mathbf{P})f^\e\|_{L_{x,v}^2(\nu)}^2\\
&\;+\frac{1}{\e}\|\mathbf{P} \partial_x^{\a}f^\e\|_{L_{x}^6 L_{v}^2}
\|w^{-\ell} \partial_x^{\a}f^\e\|_{L_{x}^3 L_{v}^2(\nu)}\|w^{-\ell} \partial_x^{\a}(\mathbf{I}-\mathbf{P})f^\e\|_{L_{x,v}^2(\nu)}\\
\lesssim&\;
\e\|w_{\vartheta}f^{\e}\|_{L^{\infty}_{x,v}} \mathcal{D}^\mathbf{s}_{\ell}(t)
+\big[\mathcal{E}^\mathbf{s}(t)\big]^{\frac{1}{2}}\mathcal{D}^\mathbf{s}_{\ell}(t).
\esp
\eals
Observing $|\a|\leq1$ and arguing similarly as the estimate of $\mathcal{X}_9$,
we deduce
\bals
\bsp
\mathcal{X}_{12}
\lesssim
 { \frac{1}{\e}}\|(\mathbf{I}-\mathbf{P})f^\e\|_{H_{x}^2L_{v}^2(\nu)}
 \|\nabla_x\mathbf{P}f^\e\|_{H_{x}^1L_{v}^2}
  +\e\big[\mathcal{E}^\mathbf{s}(t)\big]^{\frac{1}{2}}\mathcal{D}^\mathbf{s}_\ell(t).
\esp
\eals
Therefore,  plugging the above estimates on $\mathcal{X}_{10}$, $\mathcal{X}_{11}$ and $\mathcal{X}_{12}$ into \eqref{weighted-qkj:esti:Dx{I-P}f:L^2:2}, we obtain
\bal
\bsp\label{qkj:weighted:esti:{I-P}f:L^2}
&\; \frac{\d }{\d t}
\|{w}^{-\ell}\partial^\a_x (\mathbf{I}-\mathbf{P})f^\e\|_{L_{x,v}^2}^2
+\frac{\sigma_0}{ \e^2}\|{w}^{-\ell}\partial^\a_x (\mathbf{I}-\mathbf{P})f^\e\|_{L_{x,v}^2(\nu)}^2\\
\lesssim  &\;
\big[\mathcal{E}^\mathbf{s}(t)\big]^{\frac{1}{2}}\mathcal{D}_{\ell}^\mathbf{s}(t)
+\e
 \|w_{\vartheta}f^\e\|_{L^{\infty}_{x,v}}\mathcal{D}_\ell^\mathbf{s}(t)
+\e^3\|w_{\vartheta} f^{\e}\|_{W^{1,\infty}_{x,v}}
\mathcal{D}^\mathbf{s}_\ell(t)+\mathcal{D}^\mathbf{s}(t).
\esp
\eal

\emph{Step 2. The $L_{x,v}^2$-estimate of  $\e^{\frac{1}{2}} {w}^{-\ell} \nabla^2_x f^\e$.}\;
Acting $\nabla_x^2 $ on the first equation in (\ref{eq:f}) and then taking the $L^2_{x,v}(\bbR^3\times\bbR^3)$ inner product with $\e w^{-2\ell}\nabla_x^2 f^\e$, we get
\bal
\bsp\label{weighted:qkj:esti:Dxxf:L^2:1}
& \frac{\e}{2}\frac{\d }{\d t} \|w^{-\ell} \nabla_x^2 f^\e\|_{L_{x,v}^2}^2+\underbrace{ \frac{1}{\e}\Big\langle L (\nabla_x^2(\mathbf{I}-\mathbf{P})f^\e),  w^{-2\ell} \nabla_x^2 (\mathbf{I}-\mathbf{P})f^\e \Big \rangle_{L_{x,v}^2}}_{\mathcal{X}_{13}} \\
\le \;&
\underbrace{- \frac{1}{\e}\Big\langle L (\nabla_x^2(\mathbf{I}-\mathbf{P})f^\e),  w^{-2\ell} \nabla_x^2 \mathbf{P}f^\e \Big \rangle_{L_{x,v}^2}}_{\mathcal{X}_{14}}
+\underbrace{ \Big\langle -v\cdot \nabla_x^2\nabla_x \phi^{\e}\sqrt{\mu},   w^{-2\ell} \nabla_x^2 f^\e \Big \rangle_{L_{x,v}^2}}_{\mathcal{X}_{15}}\\
\;&
+\underbrace{ \Big\langle \nabla_x^2\Big[ \frac{\nabla_x \phi^{\e} \cdot \nabla_v(\sqrt{\mu}f^{\e}) }{\sqrt{\mu}}\Big] , \e w^{-2\ell} \nabla_x^2 f^\e \Big \rangle_{L_{x,v}^2}}_{\mathcal{X}_{16}}
+
\underbrace{\Big\langle \nabla_x^2\Gamma(f^\e,f^\e),  w^{-2\ell} \nabla_x^2 f^\e \Big\rangle_{L_{x,v}^2}}_{\mathcal{X}_{17}}
.
\esp
\eal
By the spectral inequality  \eqref{es-energy-linear:1} in Lemma \ref{es-energy-linear}, we have
\bals
\mathcal{X}_{13}\gtrsim\frac{\sigma_0}{2\e}\|w^{-\ell} \nabla_x^2 (\mathbf{I}-\mathbf{P})f^\e\|_{L_{x,v}^2(\nu)}^2 -\frac{C_{\sigma_0}}{\e}\| \nabla_x^2 (\mathbf{I}-\mathbf{P})f^\e\|_{L_{x,v}^2(\nu)}^2.
\eals
Observing the Maxwellian
structure \eqref{defination:Pf}  of $\mathbf{P}f^\varepsilon$, we also  derive
\bals
\mathcal{X}_{14}
\lesssim \frac{1}{\e} \|\nabla_x^2(\mathbf{I}-\mathbf{P})f^{\e}\|_{L_{x,v}^2(\nu)}
\|\nabla_x^2\mathbf{P}f^\e\|_{L_{x,v}^2}.
\eals
For the term $\mathcal{X}_{15}$,  the  H\"{o}lder inequality and  \eqref{esti:Dxxa} induce
\bals
\mathcal{X}_{15}
\lesssim  \|\nabla_x^2\nabla_x \phi^{\e}\|_{L^2_x}\|\nabla_x^2f^\e\|_{L_{x,v}^2(\nu)}
 \lesssim  \|\nabla_x a^\e\|_{L^2_x}\|\nabla_x^2f^\e\|_{L_{x,v}^2(\nu)}.
\eals
Direct calculations yield
\bals
\bsp
\mathcal{X}_{16}=\;&
 \Big\langle \nabla_x^2\Big( \nabla_x \phi^{\e}\cdot\nabla_v f^{\e}\Big), \e^2 w^{-2\ell} \nabla_x^2 f^\e \Big \rangle_{L_{x,v}^2}-
 \Big\langle \nabla_x^2\Big( \frac{v}{2}\cdot\nabla_x \phi^{\e} f^{\e} \Big) , \e w^{-2\ell} \nabla_x^2 f^\e \Big \rangle_{L_{x,v}^2}\\
 =\;&
 -\Big\langle \nabla_x^2\Big(\Big( \frac{\nabla_v( w^{-2\ell} )}{w^{-2\ell}}+\frac{v}{2}\Big)\cdot\nabla_x \phi^{\e} f^{\e}\Big), \e w^{-2\ell}\nabla_x^2 f^\e \Big \rangle_{L_{x,v}^2}\\
 =\;&-\sum_{|\a'|\leq 2}\Big\langle \Big( \frac{\nabla_v( w^{-2\ell} )}{w^{-2\ell}}+\frac{v}{2}\Big)\cdot \nabla_x^{2-|\a'|}\nabla_x \phi^{\e} \partial_x^{\a'} f^{\e}, \e w^{-2\ell}\nabla_x^2 f^\e \Big \rangle_{L_{x,v}^2},\\
 \esp
 \eals
 which means, by the similar argument as the estimate on the term $\mathcal{X}_{10}$, that
 \bals
\bsp
\mathcal{X}_{16}\lesssim \;&\e\sum_{\substack{|\a'|\leq|\a|}}\|  \nabla_x^{2-|\a'|}\nabla_x \phi^{\e}\|_{L^6_x}\|\langle v \rangle w^{-\ell}w_{\vartheta}^{\frac{1}{4}}\partial_x^{\a'} f^{\e}\|_{L^3_{x,v}}\|w^{-\ell}\nabla_x^{2} f^{\e}\|_{L^2_{x,v}(\nu)}\\
\lesssim\;& \e\sum_{\substack{|\a'|\leq|\a|}}\| a^{\e}\|_{H^2_x}\Big(\frac{1}{\e}\| \partial_x^{\a'}f^\e\|_{L_{x,v}^2}+
\e^2\|w_{\vartheta} \partial_x^{\a'} f^{\e}\|_{L^\infty_{x,v}}\Big)
\big[\mathcal{D}^\mathbf{s}_\ell(t)\big]^{\frac{1}{2}}\\
\lesssim\;& \big[\mathcal{E}^\mathbf{s}(t)\big]^{\frac{1}{2}} \mathcal{D}^\mathbf{s}_\ell(t)+
\e^3\|w_{\vartheta} f^{\e}\|_{W^{2,\infty}_{x,v}}
\mathcal{D}^\mathbf{s}_\ell(t).
\esp
\eals
Based on   \eqref{fdefenjie}   and Lemma \ref{es-energy-Nonlinear}, as in the derivation of
\eqref{qkj:esti:Dxxf:L^2:10}, we deduce
\bals
\bsp
\mathcal{X}_{17}=\;&
\Big\langle \nabla_x^2\Gamma(f^\e,f^\e),  w^{-2\ell} \nabla_x^2 (\mathbf{I}-\mathbf{P})f^\e \Big\rangle_{L_{x,v}^2}
+\Big\langle \big|\nabla_x^2\Gamma(f^\e,f^\e)\big|,  |\nabla_x^2 \mathbf{P}f^\e| \Big\rangle_{L_{x,v}^2}\\
\lesssim\;&\big[\mathcal{E}^\mathbf{s}(t)\big]^{\frac{1}{2}}\mathcal{D}^\mathbf{s}_{\ell}(t)
+\e\|w_{\vartheta}f^{\e}\|_{L^{\infty}_{x,v}} \mathcal{D}^\mathbf{s}_{\ell}(t)
+\e\|w_{\vartheta}\nabla_xf^{\e}\|_{L^{\infty}_{x,v}} \mathcal{D}^\mathbf{s}_{\ell}(t).
\esp
\eals
Hence,  plugging the above bounds on  $\mathcal{X}_{13}\sim \mathcal{X}_{17}$   into \eqref{weighted:qkj:esti:Dxxf:L^2:1}  implies
\bal
\bsp\label{qkj:weighted:esti:Dx{I-P}f:L^2}
&\;\e\frac{\d }{\d t}
\|{w}^{-\ell}\nabla_x^2f^\e\|_{L_{x,v}^2}^2
+\frac{\sigma_0}{\e}\|{w}^{-\ell}\nabla_x^2(\mathbf{I}-\mathbf{P})f^\e\|_{L_{x,v}^2(\nu)}^2\\
\lesssim  &\; \left[\mathcal{E}^\mathbf{s}(t)\right]^{\frac{1}{2}} \mathcal{D}^\mathbf{s}_{\ell}(t)
+\e\|w_{\vartheta}  f^{\e}\|_{W^{1,\infty}_{x,v}}\mathcal{D}^\mathbf{s}_{\ell}(t)+
\e^3\|w_{\vartheta} f^{\e}\|_{W^{2,\infty}_{x,v}}
\mathcal{D}^\mathbf{s}_\ell(t)
 +
  \mathcal{D}^\mathbf{s}(t).
\esp
\eal

\emph{Step 3. The
 $L_{x,v}^2$-estimate of the  mixed derivative  $w^{|\b|-\ell}\partial_{{\b}}^\a(\mathbf{I}-\mathbf{P})f^\varepsilon$ $ (|\a|+|\b|\leq2,1\leq|\b|\leq 2).$}\;
 Indeed, the process of the $L_{x,v}^2$-estimate on  $w^{|\b|-\ell}\partial_{{\b}}^\a(\mathbf{I}-\mathbf{P})f^\varepsilon$  is analogous to the $L_{x,v}^2$-estimate on  $ w^{|\b|}\partial^\a_\b(\mathbf{I}-\mathbf{P})f^\e$  under some necessary modifications to match the new velocity weighted $w^{|\b|-\ell}$. Thus, the method of verifying \eqref{qkj:esti:Dv{I-P}f:L^2}  applied   here  yields
\bal
\bsp\label{qkj:weighted:esti:DvDx{I-P}f:L^2}
&\;\frac{\d }{\d t}  \sum_{m=1}^2 \widetilde{C}_m \sum_{\substack{|\a|+|\b|\leq 2\\|\b|=m}} \|w^{|\b|-\ell}\partial^\a_\b(\mathbf{I}-\mathbf{P})f^\e\|_{L_{x,v}^2}^2
+\sum_{\substack{|\a|+|\b|\leq 2\\1\leq |\b|\leq2}}\frac{\sigma_0}{\e^2}\|w^{|\b|-\ell}\partial^\a_\b(\mathbf{I}-\mathbf{P})f^\e\|_{L_{x,v}^2(\nu)}^2\\
\lesssim  &\; \big[\mathcal{E}^\mathbf{s}(t)\big]^{\frac{1}{2}}
\mathcal{D}^\mathbf{s}_\ell(t)+\e\|w_{\vartheta}  f^{\e}\|_{W^{1,\infty}_{x,v}}\mathcal{D}^\mathbf{s}_\ell(t)+\e^3\|w_{\vartheta} f^{\e}\|_{W^{2,\infty}_{x,v}}\mathcal{D}^\mathbf{s}_\ell(t)+\eta\| \nabla_x\mathbf{P} f^\e\|_{H_{x}^1 L^2_{v}}^2\\
&\;+\|w^{-\ell} \nabla_x(\mathbf{I}-\mathbf{P})f^\e\|_{H_{x}^1 L^2_{v}(\nu)}^2
+\frac{1}{ \e^2}\|(\mathbf{I}-\mathbf{P})f^\e\|_{H_{x}^2 L^2_{v}(\nu)}^2,
\esp
\eal
where $\widetilde{C}_1\gg \widetilde{C}_2>0$ are   fixed large constants.

\medskip

Consequently,
multiplying \eqref{qkj:weighted:esti:{I-P}f:L^2} and \eqref{qkj:weighted:esti:Dx{I-P}f:L^2} by a sufficiently large constant $\widetilde{C}_0$ and then
combining the resulting inequalities with  \eqref{qkj:weighted:esti:DvDx{I-P}f:L^2}, we get (\ref{weighted:11111}). This completes the proof of Lemma \ref{prop:weighted:f:H^2}.
\end{proof}

With  Lemma \ref{prop:estimate:dissipation}, Lemma \ref{prop:f:H^2:1} and Lemma \ref{prop:weighted:f:H^2} in hand,   we are now in the position to give the proof of
Proposition \ref{main-weighted-energy-estimate-1}.
\begin{proof}[\textbf{Proof of Proposition \ref{main-weighted-energy-estimate-1}}] \ \
Choosing $0<\eta\ll\eta_2\ll \eta_1\ll 1$   small enough and taking the linear combination   (\ref{qkj:esti:diss:0})$\times \eta_1+$(\ref{11111})$+$(\ref{weighted:11111})$\times \eta_2$ induce  that
\bal
\bsp\label{eq:energy:estimate:result-soft}
\frac{\d}{\d t}\mathbf{E}^\mathbf{s}_{\ell}(t)
 +\mathbf{D}^\mathbf{s}_{\ell}(t)
 \lesssim
 \Big([\mathcal{E}^\mathbf{s}(t)]^{\frac{1}{2}}
 + \mathcal{E}^\mathbf{s}(t) \Big) \mathbf{D}^\mathbf{s}_{\ell}(t)+
\Big( \e^3\|w_{\vartheta}f^\e\|_{W^{2,\infty}_{x,v}}+\e
 \|w_{\vartheta}f^\e\|_{W^{1,\infty}_{x,v}}\Big)\mathbf{D}^\mathbf{s}_{\ell}(t).
\esp
\eal
where ${\mathbf{E}}^\mathbf{s}_{\ell}(t)$ comes from the energy given in  \eqref{qkj:esti:diss:0-2}, (\ref{11111}) and   (\ref{weighted:11111}), as follows
\bal
\bsp\label{eq:energy:estimate:result-soft-1}
{\mathbf{E}}^\mathbf{s}_{\ell}(t):=\;&C_0\sum_{|\a|\leq2}\Big(\|\partial_x^\a f^\e\|_{L^2_{x,v}}^2 +\|\partial_x^\a\nabla_x  \phi^\e \|_{L^2_x}^2\Big)
 -C_0\e \int_{\mathbb{R}^3}|b^\e|^2(a^\e+c^\e)\dd x\\
 \;&+\sum_{m=1}^2 C_m \sum_{\substack{|\a|+|\b|\leq 2\\|\b|=m}} \|w^{|\b|}\partial^\a_\b(\mathbf{I}-\mathbf{P})f^\e\|_{L_{x,v}^2}^2
 +\eta_1\e \mathcal{E}_{int}^\mathbf{s}(t)\\
 \;&
 +\eta_2\sum_{m=0}^2 \widetilde{C}_m \sum_{\substack{|\a|+|\b|\leq 2\\0\leq|\a|\leq1,|\b|=m}} \|w^{|\b|-\ell}\partial^\a_\b(\mathbf{I}-\mathbf{P})f^\e\|_{L_{x,v}^2}^2
 +\eta_2\widetilde{C}_0 \e
\|{w}^{-\ell}\nabla_x^2f^\e\|_{L_{x,v}^2}^2.
\esp
\eal
Thus, \eqref{eq:energy:estimate:result-soft-1}  and \eqref{qkj:esti:diss:0-1}  mean that $\mathbf{E}_{\ell}^\mathbf{s}(t)$ satisfies \eqref{def-energy-R}. Further,  we
utilize \eqref{eq:energy:estimate:result-soft} and the \emph{a priori} assumption \eqref{energy-assumptition-soft} to conclude \eqref{energy-nonlinear-result}.
This completes the proof of Proposition \ref{main-weighted-energy-estimate-1}.
\end{proof}

\subsection{Weighted \texorpdfstring{$H^2_{x,v}$}{Lg} Energy Estimate for Hard Potentials}
\hspace*{\fill}

The main purpose in this subsection is to perform the full $H_{x,v}^2$  energy estimate coupled with the  weighted $H_{x}^1L_{v}^2$ energy estimate,  and the high-order  $H_{x,v}^2$ energy estimate   with the   weighted $H_{x}^1L_{v}^2$ energy estimate for the   VPB system (\ref{eq:f}) with    hard potentials $0\leq \gamma\leq 1$.
To be exact, we introduce the weighted instant full energy $\mathbf{E}^\mathbf{h}(t)$ and the weighted instant high-order energy  $\mathbf{\widetilde{E}}^\mathbf{h}(t)$
\bal
\mathbf{E}^\mathbf{h}(t)\sim\;&
 \|  f^\e(t)\|^2_{H^2_{x,v}}
 + \| \nabla_x \phi^\e(t)\|^2_{H^2_x}+\| {w}(\mathbf{I}-\mathbf{P}) f^\e (t)\|_{H^1_{x}L^2_{v}}^2\label{def-energy-R-hard}
 ,\\
 \mathbf{\widetilde{E}}^\mathbf{h}(t)\sim\;&
 \sum_{1\leq|\a|\leq2}
 \Big(\|\partial_{x}^{\a}f^\e(t)\|^2_{L^2_{x,v}}+\|\partial_{x}^{\a} \nabla_x \phi^\e(t)\|^2_{L^2_x}\Big)
 +\sum_{|\a|+|\b|\leq2}
 \|\partial_{\b}^{\a}(\mathbf{I}-\mathbf{P}) f^\e(t)\|^2_{L^2_{x,v}}\notag
\\&
+\|{w} (\mathbf{I}-\mathbf{P})f^\e (t)\|_{H^1_{x}L^2_{v}}^2.  \label{def-energy-R-hard-2}
\eal
 The corresponding energy dissipation rate is defined as
\bal
\bsp\label{def-energy-R-hard-3}
  \mathbf{D}^\mathbf{h}(t):=\;&\frac{1}{\e^2}\|
  (\mathbf{I}-\mathbf{P})f^\e(t)\|^2_{H^2_{x,v}(\nu)}+\frac{1}{\e^2} \|{w} (\mathbf{I}-\mathbf{P})f^\e (t)\|_{H^1_{x}L^2_{v}(\nu)}^2+\|\nabla_x \mathbf{P}f^\e(t)\|^2_{H^1_{x,v}}\\
  \;&+ \| \nabla^2_x \phi^\e(t)\|^2_{H^1_x},
  \esp
\eal
where the velocity weight $w$   is given in \eqref{weight-w}.
Then we   state the   main result of this section.
\begin{proposition}\label{main-weighted-energy-estimate-2}
Let $0\leq\gamma\leq1$.
Suppose that $ (f^\e, \nabla_x\phi^\e) $ is a solution to the VPB system \eqref{eq:f} defined on $ [0, T] \times \mathbb{R}^3 \times \mathbb{R}^3$ and
 satisfies the \emph{a priori} assumption
  \bal\label{energy-assumptition-hard}
  \mathbf{E}^\mathbf{h}(t)\leq \delta\quad  \text{for} \quad 0\leq t\leq T
  \eal
  for some sufficiently small $\delta>0$, then there exist $\mathbf{E}^\mathbf{h}(t)$
  and $\mathbf{\widetilde{E}}^\mathbf{h}(t)$ satisfying
\eqref{def-energy-R-hard}--\eqref{def-energy-R-hard-2} such that
  \bal
  \frac{\d}{\d t}\mathbf{E}^\mathbf{h}(t)+\mathbf{D}^\mathbf{h}(t)
  \;&\lesssim
\left( \e^3\|w_{\vartheta}f^\e\|_{W^{2,\infty}_{x,v}}+\e
 \|w_{\vartheta}f^\e\|_{W^{1,\infty}_{x,v}}\right)
 \mathbf{D}^\mathbf{h}(t)\label{energy-nonlinear-result-3},\\
 \frac{\d}{\d t}\mathbf{\widetilde{E}}^\mathbf{h}(t)+\mathbf{D}^\mathbf{h}(t)
  \;&\lesssim
\left( \e^3\|w_{\vartheta}f^\e\|_{W^{2,\infty}_{x,v}}+\e
 \|w_{\vartheta}f^\e\|_{W^{1,\infty}_{x,v}}\right)
 \mathbf{D}^\mathbf{h}(t)+\|\nabla_x \mathbf{P}f^\e(t)\|_{L^2_{x,v}}^2\label{energy-nonlinear-result-4}
  \eal
   for  all $0\leq t\leq T$.
\end{proposition}

\begin{proof}  \
Observing that the technique \eqref{qkj:esti:f:L^2:8},  which we adopt to  overcome the trouble force term $v \cdot\nabla_x \phi^\e f^\e$ for the soft potential case, can be applied to the hard potential case.
 Therefore, thanks to the properties \eqref{es-energy-linear:1-hard} and \eqref{es-energy-Nonlinear:2-hard},
 we adopt the methods of justifying  Lemma \ref{prop:estimate:dissipation}, Lemma \ref{prop:f:H^2:1} and Lemma \ref{prop:weighted:f:H^2} to formulate
 the corresponding results for the hard potential case under minor modification.
That is, for any $0\leq t\leq T$,
we are able to establish the macroscopic dissipation estimate
\bal
\bsp\label{qkj:esti:diss:0-hard}
\;&\|\nabla_x \mathbf{P}f^\e(t)\|_{H_{x,v}^1}^2+\|\nabla_x^2\phi^\e(t)\|_{H_{x}^1}^2
+\e\frac{\d}{\d t}\mathbf{E}_{int}^\mathbf{h}(t)\\
\lesssim\;& \left[\mathbf{E}^\mathbf{h}(t)\right]^{\frac{1}{2}} \mathbf{D}^\mathbf{h}(t)+\e
 \|w_{\vartheta}f^\e\|_{L^{\infty}_{x,v}} \mathbf{D}^\mathbf{h}(t)  +\frac{1}{\e^2}\|  (\mathbf{I}-\mathbf{P}) f^\e \|_{H_{x}^2L_{v}^2(\nu)}^2
\esp
\eal
with the  temporal interactive functional $\mathbf{E}_{int}^\mathbf{h}(t)$ satisfying
\bal\label{qkj:esti:diss:0-1-hard}
|\mathbf{E}_{int}^\mathbf{h}(t)|\lesssim \|a^\e\|_{H^1_x}^2+\|  (\mathbf{I}-\mathbf{P}) f^\e \|_{H_{x}^1L_{v}^2}^2+\|  \nabla_x \mathbf{P}f^\e \|_{H_{x}^1L_{v}^2}^2,
\eal
the  full standard  $H^2_{x,v}$  energy estimate
\bal
\bsp\label{11111-hard}
 &\frac{\d}{\d t}\Big[\widetilde{C}_0\sum_{|\a|\leq2}\left(\|\partial_x^\a f^\e\|_{L^2_{x,v}}^2 +\|\partial_x^\a\nabla_x  \phi^\e \|_{L^2_x}^2\right)
 +\sum_{m=1}^2 \widetilde{C}_m \sum_{\substack{|\a|+|\b|\leq 2\\|\b|=m}} \|\partial^\a_\b(\mathbf{I}-\mathbf{P})f^\e\|_{L_{x,v}^2}^2\Big] \\
  &
  -\widetilde{C}_0\e\frac{\d}{\d t}  \int_{\mathbb{R}^3}|b^\e|^2(a^\e+c^\e)\dd x
  +\frac{\sigma_0}{2\e^2}
  \sum_{\substack{|\a|+|\b|\leq2}}
\|\partial_\b^\a (\mathbf{I}-\mathbf{P})f^\e\|_{L^2_{x,v}(\nu) }^2\\
 \lesssim &
 \Big(\big[\mathbf{E}^\mathbf{h}(t)\big]^{\frac{1}{2}}
 + \mathbf{E}^\mathbf{h}(t)\Big) \mathbf{D}^\mathbf{h}(t)+
 \left( \e^3\|w_{\vartheta}f^\e\|_{W^{2,\infty}_{x,v}}+\e
 \|w_{\vartheta}f^\e\|_{W^{1,\infty}_{x,v}}\right)\mathbf{D}^\mathbf{h}(t)+
 \1\|\nabla_x\mathbf{P}f^\e\|_{H^1_{x,v}}^2
\esp
\eal
with fixed large constants $\widetilde{C}_0\gg \widetilde{C}_1\gg \widetilde{C}_2>0$,
and the velocity weighted $H_{x}^1L_{v}^2$  energy estimate
\bal
\bsp\label{result:weighted:f:H^1-1}
&\frac{\d}{\d t}\|{w}(\mathbf{I}-\mathbf{P})f^\e\|_{H^1_{x}L^2_{v}}^2
+\frac{\sigma_0}{\e^2}\|{w}(\mathbf{I}-\mathbf{P})f^\e\|_{H^1_{x}L^2_{v}(\nu)}^2\\
 \lesssim\;&
\left[\mathbf{E}^\mathbf{h}(t)\right]^{\frac{1}{2}}
 \mathbf{D}^\mathbf{h}(t)
+\left( \e^3\|w_{\vartheta}f^\e\|_{W^{2,\infty}_{x,v}}+\e
 \|w_{\vartheta}f^\e\|_{W^{1,\infty}_{x,v}}\right)\mathbf{D}^\mathbf{h}(t)
+ { \frac{1}{\e^2}}\|(\mathbf{I}-\mathbf{P})f^\e\|_{H_{x}^2L_{v}^2(\nu)}^2\\
 \;&
 +\|\nabla_x\mathbf{P}f^\e\|_{H_{x}^1L_{v}^2}^2
.
\esp
\eal
Consequently, employing   the linear combination   \eqref{qkj:esti:diss:0-hard}$\times \eta_3+$(\ref{11111-hard})$+$\eqref{result:weighted:f:H^1-1}$\times\eta_4$ for $0<\eta\ll\eta_4\ll \eta_3\ll 1$   small enough and  the \emph{a priori} assumption \eqref{energy-assumptition-hard}, we conclude \eqref{energy-nonlinear-result-3}.
Here ${\mathbf{E}}^\mathbf{h}(t)$ is given by
\bals
\bsp
{\mathbf{E}}^\mathbf{h}(t):=\;&\widetilde{C}_0\sum_{|\a|\leq2}\Big(\|\partial_x^\a f^\e\|_{L^2_{x,v}}^2 +\|\partial_x^\a\nabla_x  \phi^\e \|_{L^2_x}^2\Big)
 -\widetilde{C}_0\e \int_{\mathbb{R}^3}|b^\e|^2(a^\e+c^\e)\dd x\\
 \;&+\sum_{m=1}^2 \widetilde{C}_m \sum_{\substack{|\a|+|\b|\leq 2\\|\b|=m}} \|w^{|\b|}\partial^\a_\b(\mathbf{I}-\mathbf{P})f^\e\|_{L_{x,v}^2}^2
 +\eta_3\e \mathbf{E}_{int}^\mathbf{h}(t)
 +\eta_4\|{w}(\mathbf{I}-\mathbf{P})f^\e\|_{H^1_{x}L^2_{v}}^2,
\esp
\eals
which combines with \eqref{qkj:esti:diss:0-1-hard}  meaning that $\mathbf{E}^\mathbf{h}(t)$ satisfies \eqref{def-energy-R-hard}.

\medskip

To prove \eqref{energy-nonlinear-result-4}, we proceed along the same scheme as  the proof of \eqref{energy-nonlinear-result-3}.
Comparing   the relations of $\mathbf{{E}}^\mathbf{h}(t)$
with
$\mathbf{\widetilde{E}}^\mathbf{h}(t)$ given in \eqref{def-energy-R-hard} and \eqref{def-energy-R-hard-2},   we find that the main difference lies in the  zero-order energy estimate.
 In fact, it suffices to replace the zero-order estimate in  (\ref{11111-hard}) by the following estimate
 \begin{align}
\begin{split}\label{qkj:esti:f:L^2:high:hard}
 &\;\frac{1}{2}\frac{\d }{\d t}\|(\mathbf{I}-\mathbf{P})f^\e\|_{L_{x,v}^2}^2 +\frac{\sigma_0}{2\e^2}\|(\mathbf{I}-\mathbf{P})f^\e\|_{L_{x,v}^2(\nu)}^2\\
 \lesssim &\;
 \big[\mathbf{E}^{\mathbf{h}}(t)\big]^{\frac{1}{2}}\mathbf{D}^{\mathbf{h}}(t)
+\left( \e^3\|w_{\vartheta}f^\e\|_{L^{\infty}_{x,v}}+\e
 \|w_{\vartheta}f^\e\|_{L^{\infty}_{x,v}}\right) \mathbf{D}^{\mathbf{h}}(t)+\|\nabla_x \mathbf{P}f^\e\|_{L_{x,v}^2}^2,
\end{split}
\end{align}
which further implies that the $H^2_{x, v}$ energy estimate  (\ref{11111-hard}) is replaced by  the following high-order $H^2_{x, v}$ energy estimate
\bal
\bsp\label{11111-hard-high}
 &\frac{\d}{\d t}\Big[\widetilde{C}_0\sum_{1\leq|\a|\leq2}\left(\|\partial_x^\a f^\e\|_{L^2_{x,v}}^2 +\|\partial_x^\a\nabla_x  \phi^\e \|_{L^2_x}^2\right)
 +\sum_{m=0}^2 \widetilde{C}_m \sum_{\substack{|\a|+|\b|\leq 2\\|\b|=m}} \|\partial^\a_\b(\mathbf{I}-\mathbf{P})f^\e\|_{L_{x,v}^2}^2\Big] \\
  &
  +\frac{\sigma_0}{2\e^2}
  \sum_{\substack{|\a|+|\b|\leq2}}
\|\partial_\b^\a (\mathbf{I}-\mathbf{P})f^\e\|_{L^2_{x,v}(\nu) }^2\\
 \lesssim &
 \left(\left[\mathbf{E}^\mathbf{h}(t)\right]^{\frac{1}{2}}+ \e^3\|w_{\vartheta}f^\e\|_{W^{2,\infty}_{x,v}}+\e
 \|w_{\vartheta}f^\e\|_{W^{1,\infty}_{x,v}}\right)\mathbf{D}^\mathbf{h}(t)+
 \|\nabla_x\mathbf{P}f^\e\|_{L^2_{x,v}}^2+
 \eta\|\nabla_x^2\mathbf{P}f^\e\|_{L^2_{x,v}}^2.
\esp
\eal
To prove (\ref{qkj:esti:f:L^2:high:hard}),  multiplying \eqref{qkj:esti:Dv{I-P}f:L^2:1} by $(\mathbf{I}-\mathbf{P})f^\e$ and integrating over $\mathbb{R}^3 \times \mathbb{R}^3$, we have
\bals
\bsp
& \frac{1}{2}\frac{\d }{\d t} \|(\mathbf{I}-\mathbf{P})f^\e\|_{L_{x,v}^2}^2
+\frac{\sigma_0}{\e^2}\| (\mathbf{I}-\mathbf{P})f^\e\|_{L_{x,v}^2(\nu)}^2
\\
\le
&\;
\left\langle
\frac{1}{\e}v\cdot \nabla_x\mathbf{P}f^\e - \nabla_x \phi^\e\cdot \nabla_v\mathbf{P}f^\e+\frac{1}{2} v \cdot\nabla_x \phi^\e\mathbf{P}f^\e
,(\mathbf{I}-\mathbf{P})f^\e\right
\rangle_{L_{x,v}^2}\\
&\;-\left\langle
\frac{v}{2}\cdot\nabla_x \phi^\e (\mathbf{I}-\mathbf{P}) f^\e, (\mathbf{I}-\mathbf{P})f^\e\right
\rangle_{L_{x,v}^2}
+\frac{1}{\e}\big\langle \Gamma(f^\e,f^\e), (\mathbf{I}-\mathbf{P})f^\e\big
\rangle_{L_{x,v}^2}.
\esp
\eals
With the aid of the H\"{o}lder inequality and the Young inequality, the right-hand first term of the inequality above  is bounded by
\bals
 \frac{\eta}{\e^2}
\|(\mathbf{I}-\mathbf{P})f^\e\|_{L^2_{x,v}(\nu) }^2+C_{\eta}\|\nabla_x\mathbf{P}f^\e\|_{L^2_{x,v}}^2+
\|\mathbf{P}f^\e\|_{L^2_{x,v}}\|\nabla_x\phi^\e\|_{L^\infty_x}
\|(\mathbf{I}-\mathbf{P})f^\e\|_{L^2_{x,v}(\nu)}.
\eals
 Other terms  have  been handled in  \eqref{qkj:esti:f:L^2:8} and \eqref{qkj:esti:f:L^2:3}.
 This hence proves  (\ref{qkj:esti:f:L^2:high:hard}).

Thus,  using the linear combination  \eqref{qkj:esti:diss:0-hard}$\times \eta_5$+\eqref{11111-hard-high}+\eqref{result:weighted:f:H^1-1}$\times \eta_6$  for $0<\eta\ll\eta_6\ll \eta_5\ll 1$  small enough  and   the \emph{a priori} assumption \eqref{energy-assumptition-hard}, we derive \eqref{energy-nonlinear-result-4}
with ${\widetilde{\mathbf{E}}}^\mathbf{h}(t)$  given by
\bals
\bsp
{\widetilde{\mathbf{E}}}^\mathbf{h}(t):=\;&\widetilde{C}_0
\sum_{1\leq|\a|\leq2}\left(\|\partial_x^\a f^\e\|_{L^2_{x,v}}^2 +\|\partial_x^\a\nabla_x  \phi^\e \|_{L^2_x}^2\right)
 +\sum_{m=0}^2 \widetilde{C}_m \sum_{\substack{|\a|+|\b|\leq 2\\|\b|=m}} \|\partial^\a_\b(\mathbf{I}-\mathbf{P})f^\e\|_{L_{x,v}^2}^2\\
 \;&
 +\eta_5\e \mathbf{E}_{int}^\mathbf{h}(t)
 +\eta_6\|{w}(\mathbf{I}-\mathbf{P})f^\e\|_{H^1_{x}L^2_{v}}^2.
\esp
\eals
This gathers with \eqref{qkj:esti:diss:0-1}  signifying that $\mathbf{E}^\mathbf{h}(t)$ satisfies \eqref{def-energy-R-hard-2}.
This completes the proof.
\end{proof}

\section{Weighted \texorpdfstring{$W^{{2,\infty}}_{x,v}$}{Lg}-estimate with Time Decay}\label{wuqiong-estimate}

To  close the weighted $H^2_{x,v}$ energy   estimates presented in Proposition \ref{main-weighted-energy-estimate-1} and Proposition  \ref{main-weighted-energy-estimate-2}, we need to handle the
term involving the $W_{x,v}^{2,\infty}$-estimate of  $w_{{\widetilde{\vartheta}}}f^\e$,
where $w_{{\widetilde{\vartheta}}}$ is given in \eqref{wegiht:express:2}. To this end, we    introduce a new unknown function
\bals
h^\e:= w_{\widetilde{\vartheta}}f^\e.
\eals
Multiplying  (\ref{eq:f}) by $w_{\widetilde{\vartheta}}$, we get the following system satisfied by $h^\e $
\bal
\bsp\label{eq:h}
&\pt_t h^\e+\frac{1}{\e}v\cdot\nabla_x h^\e-\nabla_x\phi^\e\cdot\nabla_v h^\e +\frac{1}{\e^2}\widetilde{\nu }(v) h^\e-\frac{1}{\e^2}K_{w_{\widetilde{\vartheta}}}(h^\e)\\
=\;& \frac{1}{\e}\frac{w_\vartheta}{\sqrt{\mu}}Q\Big(\frac{\sqrt{\mu}}
{w_{\widetilde{\vartheta}}}h^\e,
\frac{\sqrt{\mu}}{w_{\widetilde{\vartheta}}}h^\e\Big)
-\frac{1}{\e}v\cdot\nabla_x \phi^\e\sqrt{\mu}w_{\widetilde{\vartheta}},
\esp
\eal
where
\bal
\frac{1}{\e^{2}}\widetilde{\nu}(v):=\;&\frac{1}{\e^{2}}{\nu}(v)
+\frac{v}{2}\cdot\nabla_x\phi^\e
+2\widetilde{\vartheta}v\cdot\nabla_x\phi^\e+
\frac{\vartheta\sigma}{(1+t)^{1+\sigma}}|v|^2,\label{new:nu:def:1}\\
K_{w_{\widetilde{\vartheta}}}(h^\e):=\;&w_{\widetilde{\vartheta}} K\Big(\frac{h^\e} {w_{\widetilde{\vartheta}}}\Big).\label{new:K:def:1}
\eal
%

In the following, we gather some basic results for the later analysis. Firstly,  we present the  positive
lower bound estimate for $\frac{1}{\e^2}\widetilde{\nu }(v)$, which
is crucial for dealing with the $W^{2,\infty}_{x,v}$-estimate for $h^\e$.

\begin{lemma}\label{es:nu}
Let $-3<\gamma\leq1$, $0<\vartheta\ll 1$, $0<\sigma\leq \frac{1}{4}$  and the {\emph{a priori}} assumption \eqref{eq:assumption:1} hold. Then  there holds
\beq
\bsp\label{eq:nu:result:1}
\frac{1}{\e^{2}}\widetilde{\nu}(v)
\gtrsim\;&\frac{1}{2\e^{2}}{\nu}(v)+\frac{\langle v\rangle^2}{(1+t)^{1+\sigma}}.
\esp
\eeq
In particular, if $-3<\gamma<0$ and $0<\sigma\leq \frac{1}{24}$,  then there also holds
\beq
\bsp\label{es:nu:1}
\qquad\qquad \qquad\qquad\frac{1}{\e^{2}}\widetilde{\nu}(v)
\gtrsim\frac{1}{\e^{\frac{4}{5}}}(1+t)^{\varrho-1}\qquad \text{with}~~\varrho:=\frac{\sigma \gamma+2}{2-\gamma}\in\big(\frac{3}{8},1\big).
\esp
\eeq
\end{lemma}
\begin{proof}
Via \eqref{new:nu:def:1} and  the  Young  inequality, we have, for any small $\eta>0$,
\beq
\bsp\label{eq:nu:1-1}
\frac{1}{\e^{2}}\widetilde{\nu}
\gtrsim\;&\frac{1}{\e^{2}}\Big({\nu}-\eta\nu(v)\Big)
+\Big[\vartheta\sigma-\e^{\frac{2}{1-\gamma}}(1+t)^{(1+\sigma)}\Big(\frac{1}{2}+2\vartheta\Big)
\|\nabla_x\phi^\e\|_{L^\infty_x}^{\frac{2-\gamma}{1-\gamma}}\Big]
\frac{\langle v\rangle^2}{(1+t)^{1+\sigma}}\\
\gtrsim\;&\frac{1}{\e^{2}}\Big({\nu}-\eta\nu(v)\Big)
+\Big[\vartheta\sigma-\e^{\frac{1}{2}}(1+t)^{(1+\sigma)}\Big(\frac{1}{2}+2\vartheta\Big)
\|\nabla_x\phi^\e\|_{L^\infty_x}^{\frac{5}{4}}\Big]
\frac{\langle v\rangle^2}{(1+t)^{1+\sigma}}.
\esp
\eeq
Here, to deduce the second inequality in \eqref{eq:nu:1-1}, we have used the fact
\bals
|v\cdot \nabla_x\phi^\e|\leq \langle v \rangle \|\nabla_x\phi^\e\|_{L^\infty_x}
\leq  \Big[\frac{1}{\e^{2}}\nu(v)\Big]^{\frac{1}{2-\gamma}}
\Big[\e^\frac{2}{1-\gamma}\langle v \rangle ^2 \|\nabla_x\phi^\e\|_{L^\infty_x}^{\frac{2-\gamma}{1-\gamma}}\Big]^{\frac{1-\gamma}{2-\gamma}}.
\eals
Thus, combining (\ref{eq:nu:1-1}) with the {\emph{a priori}} assumption (\ref{eq:assumption:1}) and  choosing $\eta$ and $\delta$  small enough, we conclude (\ref{eq:nu:result:1}).

To prove \eqref{es:nu:1}, noticing $-3<\gamma<0$,  using \eqref{eq:nu:result:1} and the Young  inequality again,  we have
\beq
\bsp\label{eq:nu:2}
\frac{1}{\e^{2}}\widetilde{\nu}
\gtrsim\Big( \frac{1}{\e^{2}}\Big)^{\frac{2}{2-\gamma}}
\langle v\rangle^{\frac{2\gamma}{2-\gamma}}
\Big[\frac{1}{(1+t)^{1+\sigma}}\langle v\rangle^2\Big]^{\frac{-\gamma}{2-\gamma}}
=\;&\Big(\frac{1}{\e}\Big)^{\frac{4}{2-\gamma}}
(1+t)^{\frac{(1+\sigma)\gamma}{2-\gamma}}\\
\gtrsim\;&\e^{-\frac{4}{5}}(1+t)^{\frac{\sigma\gamma+\gamma}{2-\gamma}}.
\esp
\eeq
Observing $0<\sigma\leq \frac{1}{24}$ and denoting
 $\varrho:=\frac{\sigma \gamma+2}{2-\gamma}$, we
conclude  \eqref{es:nu:1}. This completes the proof.
\end{proof}
 Now we  turn to the analysis of $K$. From (\ref{nu-def}) and   changing the variables $v-v_* \rightarrow v_*$, we obtain
\bals
\bsp
Kf=& \;\iint_{\mathbb{R}^3\times\mathbb{S}^2 }|v-v_*|^{\gamma}
q_0(\th)\sqrt{\mu(v_*)}\left[f(v')\sqrt{\mu(v_*')}+f(v_*')\sqrt{\mu(v')}\right]\d v_* \d \o\\
& \;-\sqrt{\mu(v)}\iint_{\mathbb{R}^3\times\mathbb{S}^2 }|v-v_*|^{\gamma}
q_0(\th)f( v_*)\sqrt{\mu(v_*)} \d v_* \d \o\\
=& \;\iint_{\mathbb{R}^3\times\mathbb{S}^2 }|v_*|^{\gamma}
q_0(\th)\sqrt{\mu(v-v_*)}\left[f(v-v_*^{\parallel})
\sqrt{\mu(v-v_*^{\perp})}+f(v-v_*^{\perp})\sqrt{\mu(v-v_*^{\parallel})}\right]\d v_* \d \o\\
& \;-\sqrt{\mu(v)}\iint_{\mathbb{R}^3\times\mathbb{S}^2 }|v_*|^{\gamma}
q_0(\th)f(v- v_*)\sqrt{\mu(v-v_*)} \d v_* \d \o
,
\esp
\eals
where $v_*^{\parallel} = (v_* \cdot\omega)\omega$ and
$v_*^\perp=v_*- v_*^{\parallel}.$
Then  by means of \eqref{new:K:def:1} and direct computations, we have
\bals
\bsp
\partial_v^{\b}\big(K_{w_{\widetilde{\vartheta}}}f\big)
=
\sum_{|\b_0|+|\b_1|=|\b|}
\big(\partial_v^{\b_0}K_{w_{\widetilde{\vartheta}}}\big)(\partial_v^{\b_1}f)
,
\esp
\eals
where $\big(\partial_v^{\b_0}K_{w_{\widetilde{\vartheta}}}\big)f$
 is defined as
\bals
 \bsp
 \big(\partial_v^{\b_0}K_{w_{\widetilde{\vartheta}}}\big)f
:=& \;\iint_{\mathbb{R}^3\times\mathbb{S}^2 }q_0(\th)|v_*|^{\gamma}
{\partial_v^{\b_0}\Big[\frac{w_{\widetilde{\vartheta}}(v )\sqrt{\mu(v-v_* )}\sqrt{\mu(v-v_*^{\perp})}}{w_{\widetilde{\vartheta}}(v-v_*^{\|}  )}\Big]}{f(v-v_*^{\|})}\d v_* \d \o
\\
&\;
+\iint_{\mathbb{R}^3\times\mathbb{S}^2 }q_0(\th)|v_*|^{\gamma}{\partial_v^{\b_0}
\Big[\frac{w_{\widetilde{\vartheta}}(v )\sqrt{\mu(v-v_* )}\sqrt{\mu(v -v_*^{\parallel})}}{w_{\widetilde{\vartheta}}(v-v_*^{\perp}  )}\Big]}
{f(v-v_*^{\perp})}
\d v_* \d \o\\
& \;-\iint_{\mathbb{R}^3\times\mathbb{S}^2 }q_0(\th)|v_*|^{\gamma}
{\partial_v^{\b_0}\Big[\frac{w_{\widetilde{\vartheta}}(v )\sqrt{\mu(v-v_* )}\sqrt{\mu(v)}}{w_{\widetilde{\vartheta}}(v-v_*  )}\Big]}
{f(v-v_*)}
\d v_* \d \o\\
\equiv&\;\int_{\mathbb{R}^3}l_{w_{\widetilde{\vartheta}}}(v,v_*)f(v_*)\d v_*.
 \esp
 \eals

In the case of soft potentials, to treat the singularity in $\partial_v^{\b_0}K_{w_{\widetilde{\vartheta}}}$, we also introduce a smooth cutoff function $0 \leq \chi_m\leq 1$ such that
\beqs
\chi_m(s)=\left\{
\begin{array}{ll}
\displaystyle  1,  \quad\quad          \text{for} \;|s|\leq m,&\\[2mm]
\displaystyle 0, \quad\quad       \text{for} \;|s|\geq 2m.&\\
\end{array}
\right.
\eeqs
Then we  use $\chi_m$ to split $\partial_v^{\b_0}K_{w_{\widetilde{\vartheta}}} = \partial_v^{\b_0}K_{w_{\widetilde{\vartheta}}}^m + \partial_v^{\b_0}K_{w_{\widetilde{\vartheta}}}^c$
, where
\bals
 \bsp
 \big(\partial_v^{\b_0}K_{w_{\widetilde{\vartheta}}}^m\big)f
:=& \;\iint_{\mathbb{R}^3\times\mathbb{S}^2 }q_0(\th)|v_*|^{\gamma}\chi_m(v_*)
{\partial_v^{\b_0}\Big[\frac{w_{\widetilde{\vartheta}}(v )\sqrt{\mu(v-v_* )}\sqrt{\mu(v-v_*^{\perp})}}{w_{\widetilde{\vartheta}}(v-v_*^{\|}  )}\Big]}{f(v-v_*^{\|})}\d v_* \d \o
\\
&\;
+\iint_{\mathbb{R}^3\times\mathbb{S}^2 }q_0(\th)|v_*|^{\gamma}\chi_m(v_*){\partial_v^{\b_0}
\Big[\frac{w_{\widetilde{\vartheta}}(v )\sqrt{\mu(v-v_* )}\sqrt{\mu(v -v_*^{\parallel})}}{w_{\widetilde{\vartheta}}(v-v_*^{\perp}  )}\Big]}
{f(v-v_*^{\perp})}
\d v_* \d \o\\
& \;-\iint_{\mathbb{R}^3\times\mathbb{S}^2 }q_0(\th)|v_*|^{\gamma}\chi_m(v_*)
{\partial_v^{\b_0}\Big[\frac{w_{\widetilde{\vartheta}}(v )\sqrt{\mu(v-v_* )}\sqrt{\mu(v)}}{w_{\widetilde{\vartheta}}(v-v_*  )}\Big]}
{f(v-v_*)}
\d v_* \d \o.\\
\big(\partial_v^{\b_0}K_{w_{\widetilde{\vartheta}}}^c\big)f=\;&
\big(\partial_v^{\b_0}K_{w_{\widetilde{\vartheta}}}\big)f
-\big(\partial_v^{\b_0}K_{w_{\widetilde{\vartheta}}}^m\big)f
\equiv\int_{\mathbb{R}^3}l^c_{w_{\widetilde{\vartheta}}}(v,v_*)f(v_*)\d v_*.
 \esp
 \eals

Then the following  result  is concerned with  the  properties of the operator $\partial_v^{\b}K_{w_{\widetilde{\vartheta}}}$.

\begin{lemma}[\cite{LW2021}]  \label{eq:es:K}
$(1)$ Let $0\leq\gamma\leq1$. For  the kernel $l_{w_{\widetilde{\vartheta}}}(v,v_*)$ of $\partial_v^{\b}K_{w_{\widetilde{\vartheta}}}$,  there holds
\bal\label{eq:es:Kc:1:hard}
&\int_{\mathbb{R}^3}|l_{w_{\widetilde{\vartheta}}}(v,v_*)e^{\frac{\eta}{8}|v-v_*|^2}|\dd v_* \lesssim\nu(v)\langle v\rangle^{-2},
\eal
where $\eta>0$ is a  sufficiently small constant.

$(2)$  Let $-3<\gamma<0$.
For the kernel $l_{w_{\widetilde{\vartheta}}}^c(v,v_*)$ of $\big(\partial_v^{\b}K_{w_{\widetilde{\vartheta}}}^c\big)f$, it satisfies
\bal\label{eq:es:Kc:1}
&\int_{\mathbb{R}^3}|l_{w_{\widetilde{\vartheta}}}^c(v,v_*)e^{\frac{\eta}{8}|v-v_*|^2}|\dd v_* \lesssim\nu(v)\langle v\rangle^{-2},
\eal
where $\eta>0$ is a  sufficiently small constant. For the operator $\partial_v^{\b}K_{w_{\widetilde{\vartheta}}}^m$,  there holds
\bal\label{eq:es:Km:1}
\big(\partial_v^{\b}K_{w_{\widetilde{\vartheta}}}^m\big)f
&\;\lesssim  m^{3+\gamma}\mu^{\frac{1}{8}}(v)\|f\|_{L^\infty_{x,v}}
\eal
for any multi-index  $\b$ with $|\beta|\geq 0$.

%
\end{lemma}

Next, we turn to deducing the ${L^{\infty}_{x,v}}$-estimate  on the terms related to the nonlinear collision operator $Q$.
Before stating our results, we recall
\bals
\partial_v^\beta\Big[
\frac{w_{\widetilde{\vartheta}}}{\sqrt{\mu}}Q\Big(
\frac{\sqrt{\mu}}{w_{\widetilde{\vartheta}}}f,
\frac{\sqrt{\mu}}{w_{\widetilde{\vartheta}}}g\Big)\Big]
=\sum_{|\beta_0|+ |\beta_1 |+ |\beta_2|= |\beta|} C^{\beta_0,\beta_1,\beta_2}_\beta
\frac{w_{\widetilde{\vartheta}}}{\sqrt{\mu}}Q{_v^{\beta_0}}\Big(
\frac{\sqrt{\mu}}{w_{\widetilde{\vartheta}}}\partial_v^{\beta_1}f,
\frac{\sqrt{\mu}}{w_{\widetilde{\vartheta}}}\partial_v^{\beta_2}g\Big),
\eals
where $Q{_v^{\beta_0}}$ is given by
\bals
Q{_v^{\beta_0}}\Big(
\frac{\sqrt{\mu}}{w_{\widetilde{\vartheta}}}f,
\frac{\sqrt{\mu}}{w_{\widetilde{\vartheta}}}g\Big)
 :=\frac{\sqrt{\mu}}{w_{\widetilde{\vartheta}}}\iint_{\bbR^3\times \mathbb{S}^2}
  |v-v_*|^\gamma q_0(\th)
  \partial_{v_*}^{\beta_0}\Big(\frac{\sqrt{\mu(v_*)}}{w_{\widetilde{\vartheta}}(v_*)}\Big)
  \Big[f(v')g(v_*')
  -f(v)g(v_*)\Big] \d v_*\d \omega.
\eals

Then we have the following result  on the nonlinear collision
operator $Q$.
\begin{lemma}\label{eq:es:Q}
Let $-3<\gamma\leq1$. For any multi-index $\b$ with $|\beta|\geq 0$,
there holds
\beqs
\bsp
\Big\|\frac{w_{\widetilde{\vartheta}}}{\sqrt{\mu}}Q_v^{\beta}\Big(
\frac{\sqrt{\mu}}{w_{\widetilde{\vartheta}}}f,
\frac{\sqrt{\mu}}{w_{\widetilde{\vartheta}}}g\Big)\Big\|_{L^{\infty}_{x,v}}
\lesssim \nu(v)\|  f\|_{L^{\infty}_{x,v}}
\| g\|_{L^{\infty}_{x,v}}.
\esp
\eeqs
\end{lemma}
\begin{proof}Notice that
$\partial_{v_*}^{\beta}\Big(\frac{\sqrt{\mu(v_*)}}{w_{\widetilde{\vartheta}}(v_*)}\Big)
\leq\frac{{\mu(v_*)}^{\frac{1}{4}}}{w_{\widetilde{\vartheta}}(v_*)}.
$
Then the proof proceeds  in an analogous manner of Lemma 5 in \cite{guo2010arma} for the hard potential case
and Lemma 2.3 in \cite{LW2021} for the soft potential case. The details are omitted here for simplicity.
\end{proof}

Finally, we define the characteristics $(X(\tau;t,x,v),V(\tau;t,x,v))$ passing through $(t,x,v)$ as
\beq
\left\{
\begin{array}{ll}\label{tezhengxian}
\displaystyle \frac{\d X(\tau;t,x,v)}{\d \tau}=\frac{1}{\e}V(\tau;t,x,v),  \quad\quad\quad\quad\quad\quad           X(t;t,x,v)=x,&\\[2mm]
\displaystyle \frac{\d V(\tau;t,x,v)}{\d \tau}=-\nabla_x\phi^\e (\tau,X(\tau;t,x,v)), \quad\quad  V(t;t,x,v)=v.&\\
\end{array}
\right.
\eeq
Then under the \emph{a priori} assumption (\ref{eq:assumption:1}), we  study the characteristics estimate \cite{guo2010arma, guo2010cmp}, which will be
used to control the $W_{x,v}^{2,\infty}$-estimate.

\begin{lemma}\label{le:jac}
Let $-3<\gamma\leq1$ and $T_0$ be a small fixed number.
Then for $0\le \tau \le t\leq  \e^{\frac{1}{2}}T_0$,
there holds
\bals
&\frac{1}{2\e^3}|t-\tau|^3\le \Big| \mathrm{det}\Big( \frac{\pt X(\tau;t,x,v)}{\pt v}  \Big)    \Big|\le\frac{2}{\e^3}|t-\tau|^3.
\eals
\end{lemma}
\begin{proof}
From (\ref{tezhengxian}), we find
\bal
& X(\tau;t,x,v)=x-\frac{v}{\e}(t-\tau)
-\frac{1}{\e}\int_\tau^t\int_{\tau'}^t\nabla_{x}\phi^\e (\tau'',X(\tau'';t,x,v) )\d\tau''\d\tau'.\label{tezhengxian1}
\eal
Applying $\partial_v$ to (\ref{tezhengxian1}) leads to
\bal
\label{tezhengxian3}
\pt_v X(\tau;t,x,v)
\lesssim \frac{1}{\e}|t-\tau|+ \frac{1}{\e}|t-\tau|^2  \sup_{0\le \tau\le t}\big\{\|\nabla_{x}^2\phi^\e(\tau) \|_{L_{x }^\infty}\big\} \sup_{0\le \tau\le t}\big\{|\pt_v X(\tau;t,x,v)|\big\}.
\eal
 Thus, we derive from (\ref{eq:assumption:1}) and (\ref{tezhengxian3})  that
 \bal
 \label{tezhengxianjielun1}&\sup_{0\le \tau\le t}\left\{ \|\pt_v X(\tau;t,x,v)\|_{L_{x,v}^\infty}   \right\}\lesssim \frac{1}{\e}t\lesssim {\e^{-\frac{1}{2}}}T_0.
 \eal

Then consider the Taylor expansion of $\frac{\pt X(\tau;t,x,v)}{\pt v}$ in $\tau$ around $t$
\bal
\bsp\label{tezhengxian5}
\frac{\pt X(\tau;t,x,v)}{\pt v}=\;&\frac{\pt X(t;t,x,v)}{\pt v}+\Big.(\tau-t)\frac{\d }{\d \tau}\frac{\pt X(\tau;t,x,v)}{\pt v}\Big|_{\tau=t}
+\frac{(\tau-t)^2}{2}\frac{\d ^2}{\d \tau^2}\frac{\pt X(\tilde{\tau};t,x,v)}{\pt v}\\
=\;&\frac{1}{\e} (\tau-t)\mathbb{I}_{3\times 3}
+\frac{1}{\e} (\tau-t)\int_\tau^t\nabla_{x}^2\phi^\e (\tau,X(\tau;t,x,v) )
\cdot \partial_v X(\tau;t,x,v)\d\tau\\
\;&+(\tau-t)\Big[\frac{(\tau-t)}{2}\frac{\d ^2}{\d \tau^2}\frac{\pt X(\tilde{\tau};t,x,v)}{\pt v}\Big],
\esp
\eal
where $\mathbb{I}_{3\times 3}$ is the identity matrix. From (\ref{tezhengxian1}), we  find
\bal
\bsp\label{tezhengxian6}
&\frac{1}{\e}\int_\tau^t\nabla_{x}^2\phi^\e (\tau,X(\tau;t,x,v) )
\cdot \partial_v X(\tau;t,x,v)\d\tau+\frac{\tau-t}{2}\frac{\d ^2}{\d \tau^2}\frac{\pt X(\tilde{\tau};t,x,v)}{\pt v}\\
\le\;& \frac{2}{\e}|\tau-t| \sup_{0\le \tau\le t}\left\{\|\nabla_{x}^2\phi^\e(\tau)
\|_{L_{x }^\infty}\right\} \sup_{0\le \tau\le t}\left\{\|\pt_v X(\tau;t,x,v)\|_{L_{x,v}^\infty}\right\}\\
\le\;& \frac{1}{8\e},
\esp
\eal
where (\ref{eq:assumption:1}) and (\ref{tezhengxianjielun1}) are utilized. Therefore, we combine (\ref{tezhengxian5}) and (\ref{tezhengxian6}) to conclude  Lemma \ref{le:jac}. This completes the proof.
\end{proof}

\subsection{Weighted \texorpdfstring{$W^{{2,\infty}}_{x,v}$}{Lg}-estimate with Time Decay for Soft Potentials}\label{W-infty-soft}
\hspace*{\fill}

In this subsection,  we   aim to  set up the $W^{{2,\infty}}_{x,v}$-estimate of $h^{\e}$ for soft potentials $-3 <\gamma<0$.
The main result of this subsection is stated as  follows.

\begin{proposition}\label{result:W2:wuqiong}
Let $-3 <\gamma<0$.
Suppose that $ (f^\e, \nabla_x\phi^\e) $ is a solution to the
VPB system \eqref{eq:f}   defined on $ [0, T] \times \mathbb{R}^3 \times \mathbb{R}^3$. And, assume that the crucial {\em a priori} assumptions \eqref{eq:assumption:1}
and \eqref{eq:assumption:2} hold.
Then for any $0\leq t \leq T$, there holds
\beq
\bsp
\label{result:W2:wuqiong:1}
 \;&\e^{\frac{1}{2}}(1+t)^{\frac{5}{4}} \| {h^\e}(t)\|_{L_{x,v}^{\infty}}
+\sum_{|\a|+|\b|= 1}\e^{\frac{1}{2}+\frac{1}{5}|\b|}(1+t)^{\frac{5}{4}-\frac{5}{8}|\b|}
\| \partial_{\b}^{\a}h ^\e(t)\|_{L_{x,v}^{\infty}} \\
\;&+\sum_{|\a|+|\b|= 2}\e^{\frac{3}{2}+\frac{1}{5}|\b|}(1+t)^{\frac{5}{4}-\frac{5}{8}|\b|}
\| \partial_{\b}^{\a}h ^\e(t)\|_{L_{x,v}^{\infty}}\\
\lesssim \;&\e^\frac{1}{2}\| h^\e _0\|_{W^{2,\infty}_{x,v}}
+\e^{\frac{3}{2}}\sup_{0\le \tau\le t}\left\{(1+\tau)^{\frac{5}{4}} \| \nabla_x^2\phi^\e(\tau)\|_{H^2_{x}}\right\}
+
\e^{\frac{5}{2}}\sup_{0\le \tau\le t}\left\{ (1+ \tau)^{\frac{5}{4}} \|\nabla_x^3{\phi}^\varepsilon (\tau)\|_{L^\infty_{x}}\right\}\\
\;&
+\sup_{0\le \tau\le t}\left\{(1+\tau)^{\frac{5}{4}}\| \nabla_xf^\e(\tau)\|_{H^1_{x}L^2_{v}} \right\}
 +\e^{\frac{1}{5}}\sup_{0\le \tau\le t}\left\{(1+\tau)^{\frac{5}{8}} \| \nabla_x\nabla_vf^\e(\tau)\|_{L^2_{x,v}}\right\}\\
\;&+\e^{\frac{2}{5}}\sup_{0\le \tau\le t}\left\{\| \nabla_v^2f^\e(\tau)\|_{L^2_{x,v}} \right\}
.
\esp
\eeq
\end{proposition}

\begin{proof}  \
To prove Proposition \ref{result:W2:wuqiong}, the key point is to obtain the $L^{\infty}_{x,v}$-estimate  with time decay
for $h^\e$ from the zero order to the second order  via Duhamel's principle.

For convenience, we introduce the temporal weighted  function
\bal\label{express:weighted:h}
\widetilde{h}^\e:=(1+t)^q h^\e \qquad\text{for } \quad 0\leq q\leq \frac{5}{4}.
\eal
In terms of (\ref{eq:h}) and (\ref{express:weighted:h}), we have the weighted system
\bal
\bsp\label{eq:h:estimate:1}
&\pt_t \widetilde{h}^\e+\frac{1}{\e}v\cdot\nabla_x \widetilde{h}^\e-\nabla_x\phi^\e\cdot\nabla_v \widetilde{h}^\e +\frac{1}{\e^2}\widetilde{\nu } \widetilde{h}^\e-\frac{1}{\e^2}K_{w_{\widetilde{\vartheta}}}(\widetilde{h}^\e)\\
=\;& \frac{1}{\e}\frac{w_\vartheta}{\sqrt{\mu}}Q\Big(\frac{\sqrt{\mu}}
{w_{\widetilde{\vartheta}}}\widetilde{h}^\e,
\frac{\sqrt{\mu}}{w_{\widetilde{\vartheta}}}h^\e\Big)
-\frac{1}{\e}v\cdot\nabla_x \phi^\e\sqrt{\mu}w_{\widetilde{\vartheta}}(1+t)^q
-q{(1+t)}^{q-1}h^\e,
\esp
\eal
where the last term on the right-hand side of \eqref{eq:h:estimate:1}  only exists when $q>0$.
Then the proof  of (\ref{result:W2:wuqiong:1}) is divided into three steps.

\medskip
 \emph{Step 1. \; $L_{x,v}^\infty$-estimate of $\widetilde{h}^\e$.}\;
 In fact,
let $T_0$ be the small fixed number as in Lemma \ref{le:jac}. We first consider $\| \widetilde{h}^\e\|_{L^\infty_{x,v}}$  on the time interval
$ 0\le   t \le \e^{\frac{1}{2}} T_0 \ll1.$
Applying Duhamel's principle to the system \eqref{eq:h:estimate:1}, $\widetilde{h}^\e$ can be bounded by
\bal\label{eq:h1}
&\widetilde{h}^\varepsilon(t,x,v) \leq \text{exp}\Big\{-\frac{1}{\varepsilon^2}\int_0^t\widetilde{\nu}(\tau)\text{d}\tau\Big\}
|\widetilde{h}^\varepsilon(0,X(0),V(0))|+\sum_{i=1}^5|  I_i|,
\eal
where $X(s):=X(s;t,x,v)$, $V(s):=V(s;t,x,v)$ and
\begin{align*}
I_1&=\frac{1}{\e^2}\int_0^t \text{exp}\Big\{-\frac{1}{\e^2}\int_s^t\widetilde{\nu}(\tau)\text{d}\tau\Big\}
\left[K_{w_{\widetilde{\vartheta}}}^c\widetilde{h}^\varepsilon\right](s,X(s),V(s))\text{d}s,\\
I_2&=\frac{1}{\e^2}\int_0^t \text{exp}\Big\{-\frac{1}{\e^2}\int_s^t\widetilde{\nu}(\tau)\text{d}\tau\Big\}
\left[K_{w_{\widetilde{\vartheta}}}^m\widetilde{h}^\varepsilon\right](s,X(s),V(s))\text{d}s,\\
I_3&=\frac{1}{\e}\int_0^t \text{exp}\Big\{-\frac{1}{\e^2}\int_s^t\widetilde{\nu}(\tau)\text{d}\tau\Big\}
\Big[\frac{w_\vartheta}{\sqrt{\mu}}Q\Big(\frac{\sqrt{\mu}}
{w_{\widetilde{\vartheta}}}\widetilde{h}^\e,
\frac{\sqrt{\mu}}{w_{\widetilde{\vartheta}}}h^\e\Big)\Big](s,X(s),V(s))\text{d}s,\\
I_4&=-\frac{1}{\e}\int_0^t \text{exp}\Big\{-\frac{1}{\e^2}\int_s^t\widetilde{\nu}(\tau)\text{d}\tau\Big\}
\Big[v\cdot\nabla_x \phi^\e\sqrt{\mu}w_{\widetilde{\vartheta}}(1+t)^q\Big] (s,X(s),V(s))\text{d}s,\\
I_5&=-\int_0^t \text{exp}\Big\{-\frac{1}{\e^2}\int_s^t\widetilde{\nu}(\tau)\text{d}\tau\Big\}
\Big[q{(1+t)}^{q-1}h^\e\Big] (s,X(s),V(s))\text{d}s
.
\end{align*}
Based on Lemma \ref{es:nu}, we have
\bal
\;&\int_{0}^{t}
\text{exp}\Big\{-\frac{1}{\e^2}\int_s^t\widetilde{\nu}(\tau)\text{d}\tau\Big\}
\nu(V(s))
e^{-\frac{s^\varrho  }{4\e^a}}\dd s
\lesssim
\e^2e^{-\frac{ t^\varrho  }{4\e^a}}, \label{eq:Dhnu} \\
\;&\int_{0}^{t}
\text{exp}\Big\{-\frac{1}{\e^2}\int_s^t\widetilde{\nu}(\tau)\text{d}\tau\Big\}
\frac{\langle V(s)\rangle^2}{(1+s)^{1+\sigma}}e^{-\frac{ s^\varrho  }{4\e^a}}
\dd s
\lesssim
e^{-\frac{ t^\varrho  }{4\e^a}} \label{eq:Dhnu-2}, \\
\;&\int_{0}^{t}
\text{exp}\Big\{-\frac{1}{\e^2}\int_s^t\widetilde{\nu}(\tau)\text{d}\tau\Big\}
(1+s)^{-\frac{5}{8}}e^{-\frac{s^\varrho  }{4\e^a}}\dd s
\lesssim
\e^\frac{4}{5}e^{-\frac{t^\varrho  }{4\e^a}},\label{eq:Dhnu-1}
\eal
where  $a\equiv\frac{4}{5}$ and  we have used a basic inequality
derived from \eqref{es:nu:1}
\bals
\text{exp}\Big\{-\frac{1}{2\e^2}\int_s^t\widetilde{\nu}(\tau)\text{d}\tau\Big\}
\lesssim \text{exp}\Big\{-\frac{1}{2\e^{\frac{4}{5}}}\int_s^t(1+\tau)^{\varrho-1}\text{d}\tau\Big\}
\lesssim\text{exp}\Big\{-\frac{t^\varrho}{2\e^{\frac{4}{5}}} +\frac{s^\varrho}{2\e^{\frac{4}{5}}}\Big\}.
\eals

In the following, we  estimate $I_i$ $(1 \leq i \leq 5)$ in (\ref{eq:h1}) term by term.
Using \eqref{eq:es:Km:1}   and (\ref{eq:Dhnu}),  we have
\bals
\bsp
|I_2|&\lesssim m^{3+\gamma}
   e^{-\frac{ t^\varrho}{4\e^{a}}} \sup_{0\le s\le t}\left\{\| e^{\frac{ s^\varrho}{4\e^{a}}} \widetilde{h}^\varepsilon (s)\|_{L^\infty_{x,v}}\right\}.
\esp
\eals
Recalling the property of $Q$  in Lemma \ref{eq:es:Q}
and the \emph{a priori} assumption \eqref{eq:assumption:2}, we derive by  employing (\ref{eq:Dhnu}) that
\bals
\bsp
|I_3|\lesssim\;& \e
   e^{-\frac{ t^\varrho}{4\e^{a}}} \sup_{0\le s\le t}\left\{\| e^{\frac{ s^\varrho}{4\e^{a}}} \widetilde{h}^\varepsilon (s)\|_{L^\infty_{x,v}}\right\}
   \sup_{0\le s\le t}\left\{\| {h}^\varepsilon (s)\|_{L^\infty_{x,v}}\right\}
   \\
    \lesssim\;&\delta\e^{\frac{1}{4}}
   e^{-\frac{ t^\varrho}{4\e^{a}}} \sup_{0\le s\le t}\left\{\| e^{\frac{ s^\varrho}{4\e^{a}}} \widetilde{h}^\varepsilon (s)\|_{L^\infty_{x,v}}\right\}
   .
\esp
\eals
By (\ref{eq:Dhnu}) and the Sobolev embedding inequality, we get
\bals
\bsp
|I_4|\lesssim &\; \varepsilon \sup_{0\le s\le t}\left\{(1+s)^q \| \nabla_x\phi^\varepsilon (s)\|_{L^\infty_{x}}\right\}\lesssim  \varepsilon \sup_{0\le s\le t}\left\{\| \nabla_x^2\widetilde{\phi}^\varepsilon (s)\|_{H^1_{x}}\right\} \quad \text{with}~ \widetilde{\phi}^\varepsilon\equiv(1+t)^q{\phi}^\varepsilon.
\esp
\eals
The Young inequality yields
\bals
q{(1+t)}^{q-\frac{3}{8}}h^\e
\leq q\left[{(1+t)}^{q}|h^\e|\right]^\frac{8q-3}{8q}|h^\e|^\frac{3}{8q}
\leq \eta q|\widetilde{h}^\e|+qC_\eta|h^\e|,
\eals
which combines with (\ref{eq:Dhnu-1}) inducing
\bals
\bsp
|I_5|\lesssim &\; \eta\varepsilon^\frac{4}{5} e^{-\frac{ t^\varrho}{4\e^{a}}}  \sup_{0\le s\le t}\left\{e^{\frac{ s^\varrho}{4\e^{a}}} \| \widetilde{h}^\varepsilon (s)\|_{L^\infty_{x,v}}\right\}
+
 q\varepsilon^\frac{4}{5} \sup_{0\le s\le t}\left\{ \| h^\varepsilon (s)\|_{L^\infty_{x,v}}\right\}.
\esp
\eals

To handle the term $I_1$, we  iterate the bound (\ref{eq:h1}) of $\widetilde{h}^\e$  into  $I_1$ again to get
\bals
\bsp
|I_1|
  \lesssim  \;&e^{-\frac{t^\varrho}{2\e^a}}\|  h^\varepsilon_0\|_{L^\infty_{x,v}}+\left(m^{3+\gamma}+\delta\e^{\frac{1}{4}}+\eta \e^{\frac{4}{5}}\right)
   e^{-\frac{ t^\varrho}{4\e^{a}}} \sup_{0\le s\le t}\left\{\| e^{\frac{ s^\varrho}{4\e^{a}}} \widetilde{h}^\varepsilon (s)\|_{L^\infty_{x,v}}\right\}\\
   \;&
+\e\sup_{0\le s\le t}\left\{\|  \nabla_x^2\widetilde{\phi}^\varepsilon (s)\|_{H^1_{x}}\right\}
+
 q\varepsilon^\frac{4}{5} \sup_{0\le s\le t}\left\{ \| h^\varepsilon (s)\|_{L^\infty_{x,v}}\right\}+I_{1}^r,
\esp
\eals
where $I_{1}^r$ is given by
\bals
\bsp
I_{1}^r:=\frac{C}{\varepsilon^4}\int_0^t&\text{exp}
\Big\{-\frac{1}{\e^2}\int_s^t\widetilde{\nu}(V(\tau))\text{d}\tau\Big\}\int_0^s  \text{exp}\Big\{-\frac{1}{\e^2}\int_{s_1}^s\widetilde{\nu}(V(\tau'))\text{d}\tau'\Big\}\\
& \times \iint_{\bbR^3\times\bbR^3}|l_{w_{\widetilde{\vartheta}}}^c (V(s),v')l_{w_{\widetilde{\vartheta}}}^c (V(s_1),v'')  \widetilde{h}^\varepsilon (s_1,X(s_1),V(s_1))|\text{d}v''\text{d}v'\text{d}s_1\dd s
\esp
\eals
with $X(s_1):=X(s_1;s,X(s),v')$ and $V(s_1):=V(s_1;s,X(s),v')$.
Observing that  $l_{w_{\widetilde{\vartheta}}}^c(v,v')$ has a possible integrable
singularity of $\frac{1}{|v-v'|},$ we  choose a smooth function $l_N^c(v,v')$ with compact support satisfying
\begin{equation*}
\sup\limits_{|p|\leq 3N}\int_{|v'|\leq 2N}|l_N^c(p,v')-l_{w_{\widetilde{\vartheta}}}^c(p,v')|\dd v'\leq\frac{1}{N}.
\end{equation*}
Using the above   approximation and the  characteristics estimate in Lemma \ref{le:jac}, then taking the similar proof as that  of the $L^\infty_{x,v}$-estimate   in \cite{WZL-2},    we  establish
\bals
\bsp
I_{1}^r
\le\;&\Big(\frac{C}{N}+C\eta+o(1)\Big) e^{-\frac{ t^\varrho}{4\e^a}} \sup_{0\le s\le t}
 \Big\{ e^{ \frac{ s^\varrho}{4\e^a}}\|  \widetilde{h}^\e (s)\|_{L^\infty_{x,v}}\Big\} \\
 \;&+\frac{C}{\e^4}\int_0^t\int_0^{s-\eta\e^2}\int_{|v'|\le 2N}\int_{|v''|\le 3N}
\text{exp}
\Big\{-\frac{1}{2\e^2}\int_s^t{\nu}(V(\tau))\text{d}\tau\Big\}
\nu(V(s))\\
&\qquad\qquad\times
  \text{exp}\Big\{-\frac{1}{2\e^2}\int_{s_1}^s{\nu}(V(\tau'))\text{d}\tau'\Big\}
\nu(V(s_1))
 \big| \widetilde{h}^\e (s_1,X(s_1),v'') \big| \dd v''\dd v'\dd s_1\dd s\\
 \le\;&\left(\frac{C}{N}+C\eta+o(1)\right) e^{-\frac{ t^\varrho}{4\e^a}} \sup_{0\le s\le t}
 \Big\{ e^{ \frac{ s^\varrho}{4\e^a}}\|  \widetilde{h}^\e (s)\|_{L^\infty_{x,v}}\Big\} +\frac{C_{N,\1}}{\e^{{1}/{2}}}  \sup_{0\le s\le t}\left\{\| \nabla_x \widetilde{f}^\e (s)\|_{L^2_{x,v}}\right\}.
\esp
\eals

Thus, substituting the above estimates on $I_1\sim I_5$ into (\ref{eq:h1}), we derive
that, for any $0\le t\le \e^\frac{1}{2} T_0$, there holds
\bals
\bsp
\;&e^{\frac{t^\varrho}{4\e^a}}\|\widetilde{h}^\varepsilon (t)\|_{L^\infty_{x,v}} \\
\lesssim \;& e^{-\frac{t^\varrho}{4\e^a}}\|  h^\varepsilon_0\|_{L^\infty_{x,v}}
+e^{\frac{t^\varrho}{4\e^a}}\e\sup_{0\le s\le t}\left\{ \|  \nabla_x^2\widetilde{\phi}^\varepsilon (s)\|_{H^1_{x}}\right\}
+e^{\frac{t^\varrho}{4\e^a}}\e^{-\frac{1}{2}}\sup_{0\le s\le t}\left\{\| \nabla_x \widetilde{f}^\e (s)\|_{L^2_{x,v}}\right\}\\
\;&
+\Big(m^{3+\gamma}+\delta\e^{\frac{1}{4}}+\eta \e^{\frac{4}{5}}+\frac{1}{N}+\eta+o(1)\Big)
    \sup_{0\le s\le t}\left\{\| e^{\frac{ s^\varrho}{4\e^{a}}} \widetilde{h}^\varepsilon (s)\|_{L^\infty_{x,v}}\right\}\\
   \;&+
e^{\frac{t^\varrho}{4\e^a}} q\varepsilon^\frac{4}{5} \sup_{0\le s\le t}\left\{ \| h^\varepsilon (s)\|_{L^\infty_{x,v}}\right\}.
\esp
\eals
Choosing $\frac{1}{N}$ , $\eta$,  $m$ and $\delta$ small enough, we obtain
\bal
\bsp\label{eq:h7}
 &\sup_{0\le s \le \e^{\frac{1}{2}} T_0 }\{e^{\frac{t^\varrho}{4\e^a}}\|\widetilde{h}^\varepsilon (s) \|_{L^\infty_{x,v}} \} \\ \lesssim \;&\| h^\varepsilon_0\|_{L^\infty_{x,v}}+
 e^{\frac{({\e}^{1/2}T_0)^\varrho}{4\e^a}}\Big[\e\sup_{0\le s\le \e^{\frac{1}{2}} T_0 }\left\{ \|  \nabla_x^2\widetilde{\phi}^\varepsilon (s)\|_{H^1_{x}}\right\}+\e^{-\frac{1}{2}}\sup_{0\le s\le \e^{\frac{1}{2}} T_0 }\left\{\| \nabla_x \widetilde{f}^\e (s)\|_{L^2_{x,v}}\right\}\Big]
\\
\;&+
 \e^{\frac{({\e}^{1/2}T_0)^\varrho}{4\e^a}} q\varepsilon^\frac{4}{5} \sup_{0\le s\le \e^{\frac{1}{2}} T_0 }\left\{ \| h^\varepsilon (s)\|_{L^\infty_{x,v}}\right\}.
\esp
\eal
Further,  we have
\bal
\bsp\label{eq:h8}
\varepsilon^{\frac{1}{2} }\| \widetilde{h}^\varepsilon (\varepsilon^{\frac{1}{2}} T_0) \|_{L^\infty_{x,v}}\le \;& \frac{1}{2}\varepsilon^{\frac{1}{2} }\| h^\varepsilon_0\|_{L^\infty_{x,v}}+ \varepsilon^{\frac{3}{2} }\sup_{0\le s\le \e^{\frac{1}{2}} T_0 }\left\{ \|  \nabla_x^2\widetilde{\phi}^\varepsilon (s)\|_{H^1_{x}}\right\}\\
\;&+\sup_{0\le s\le \e^{\frac{1}{2}} T_0 }\left\{\| \nabla_x \widetilde{f}^\e (s)\|_{L^2_{x,v}}\right\}
+
 q\varepsilon^\frac{13}{10} \sup_{0\le s\le \e^{\frac{1}{2}} T_0 }\left\{ \| h^\varepsilon (s)\|_{L^\infty_{x,v}}\right\}.
\esp
\eal
For $\varepsilon^\frac{1}{2} T_0< t \leq T$, there exist $n\in \mathbb{Z}^+$ and $~\tilde\tau\in[0,\e^\frac{1}{2} T_0)$  such that  $t=\varepsilon^\frac{1}{2} T_0\times n+\tilde{\tau}$.  Consequently, using (\ref{eq:h7}) and (\ref{eq:h8}) for  $q=0$ leads to the estimate of  $\e^{\frac{1}{2}}\|h^\e(t)\|_{L^\infty_{x,v}}$.
Then taking the resulting estimate  into (\ref{eq:h7}) and (\ref{eq:h8}), we conclude that for each $0\leq q \leq\frac{5}{4}$, there holds
\beq
\bsp
\label{eqle:result:1}
\e^{\frac{1}{2}}\|  {\widetilde{h}^\e}(t)\|_{L_{x,v}^{\infty}}
\lesssim \;&\e^{\frac{1}{2}}\| h^\e_0\|_{L_{x,v}^{\infty}}
+\e^{\frac{3}{2}}\sup_{0\le \tau\le t}\left\{ \|  \nabla_x^2\widetilde{\phi}^\varepsilon (\tau)\|_{H^1_{x}}\right\}
+\sup_{0\le \tau\le t}\left\{\| \nabla_x\widetilde{f}^\e(\tau)\|_{L^2_{x,v}} \right\}.
\esp
\eeq

\emph{Step 2. $L_{x,v}^\infty$-estimates of $\nabla_x\widetilde{h}^\varepsilon$  and $\nabla_v\widetilde{h}^\varepsilon$.}\;
Applying $\nabla_x$ and  $\nabla_v$ to (\ref{eq:h:estimate:1}), respectively and exploiting Duhamel's principle, we get
\bal
\bsp\label{eq:Dh1}
\nabla_x\widetilde{h}^\varepsilon(t,x,v)
 \leq\; & \text{exp}\Big\{-\frac{1}{\varepsilon^2}\int_0^t\widetilde{\nu}(\tau)\text{d}\tau\Big\}
|\nabla_x\widetilde{h}^\varepsilon(0,X(0),V(0))|+\sum_{i=6}^{11}|  {I}_i|\\
\; &+\int_0^t \text{exp}\Big\{-\frac{1}{\e^2}\int_s^t\widetilde{\nu}(\tau)\text{d}\tau\Big\}
\left|\big(\nabla_x^2\phi^\e\cdot\nabla_v \widetilde{h}^\e\big)(s,X(s),V(s))\right|\text{d}s
\esp
\eal
and
\bal
\bsp\label{eq:Dh2}
\nabla_v\widetilde{h}^\varepsilon(t,x,v)
\leq  \;& \text{exp}\Big\{-\frac{1}{\varepsilon^2}\int_0^t\widetilde{\nu}(\tau)\text{d}\tau\Big\}
|\nabla_v\widetilde{h}^\varepsilon(0,X(0),V(0))|+\sum_{i=6}^{11}|   I_i|\\
\; &+\frac{1}{\e}\int_0^t \text{exp}\Big\{-\frac{1}{\e^2}\int_s^t\widetilde{\nu}(\tau)\text{d}\tau\Big\}
\left|\big(\nabla_x \widetilde{h}^\e\big)(s,X(s),V(s))\right|\text{d}s.
\esp
\eal
Here ${I}_6\sim{I}_{11}$ are given by
\bals
\bsp
{I}_6&=\frac{1}{\e^2}\int_0^t \text{exp}\Big\{-\frac{1}{\e^2}\int_s^t\widetilde{\nu}(\tau)\text{d}\tau\Big\}
\nabla( K_{w_{\widetilde{\vartheta}}}^c\widetilde{h}^\varepsilon)(s,X(s),V(s))\text{d}s,\\
{I}_7&=\frac{1}{\e^2}\int_0^t \text{exp}\Big\{-\frac{1}{\e^2}\int_s^t\widetilde{\nu}(\tau)\text{d}\tau\Big\}
\nabla(K_{w_{\widetilde{\vartheta}}}^m\widetilde{h}^\varepsilon)(s,X(s),V(s))\text{d}s,\\
{I}_8&=\frac{1}{\e}\int_0^t \text{exp}\Big\{-\frac{1}{\e^2}\int_s^t\widetilde{\nu}(\tau)\text{d}\tau\Big\}
\nabla\Big[\frac{w_\vartheta}{\sqrt{\mu}}Q\Big(\frac{\sqrt{\mu}}
{w_{\widetilde{\vartheta}}}\widetilde{h}^\e,
\frac{\sqrt{\mu}}{w_{\widetilde{\vartheta}}}h^\e\Big)\Big](s,X(s),V(s))\text{d}s,\\
{I}_9&=-\frac{1}{\e}\int_0^t \text{exp}\Big\{-\frac{1}{\e^2}\int_s^t\widetilde{\nu}(\tau)\text{d}\tau\Big\}
\nabla\Big[v\cdot\nabla_x \widetilde{\phi}^\e\sqrt{\mu}w_{\widetilde{\vartheta}}\Big] (s,X(s),V(s))\text{d}s,\\
{I}_{10}&=-\int_0^t \text{exp}\Big\{-\frac{1}{\e^2}\int_s^t\widetilde{\nu}(\tau)\text{d}\tau\Big\}
\left[q{(1+t)}^{q-1}\nabla h^\e\right] (s,X(s),V(s))\text{d}s,\\
{I}_{11}&=-\int_0^t \text{exp}\Big\{-\frac{1}{\e^2}\int_s^t\widetilde{\nu}(\tau)\text{d}\tau\Big\}
\Big[\nabla\Big(\frac{1}{\e^2}\widetilde{\nu}(v)\Big) \widetilde{h}^\e\Big] (s,X(s),V(s))\text{d}s
,
\esp
\eals
where  $\nabla$ represents $\nabla_x$ if
${I}_6\sim{I}_{11}$ is in  \eqref{eq:Dh1}   and   $\nabla_v$ if
${I}_6\sim{I}_{11}$   in  \eqref{eq:Dh2}.

Thanks to Lemma \ref{eq:es:K} and Lemma  \ref{eq:es:Q},   and taking the  similar approach adopted in the  estimate  of  $I_1\sim I_5$,   we are able to derive
\bal
\bsp\label{eq:DI2}
\sum_{i=6}^{10}|{I}_i|
  \lesssim  \;&\left(m^{3+\gamma}+\delta\e^{\frac{1}{4}}+\eta \e^{\frac{4}{5}}+\frac{1}{N}+\eta+o(1)\right)
   e^{-\frac{ t^\varrho}{4\e^{a}}} \sup_{0\le s\le t}\left\{\| e^{\frac{ s^\varrho}{4\e^{a}}} \nabla\widetilde{h}^\varepsilon (s)\|_{L^\infty_{x,v}}\right\}\\
   \;&+e^{-\frac{t^\varrho}{2\e^a}}\|  \nabla h^\varepsilon_0\|_{L^\infty_{x,v}}+
\e\sup_{0\le s\le t}\left\{ \|  \nabla_x^2\widetilde{\phi}^\varepsilon (s)\|_{H^2_{x}}\right\}
+
 \sup_{0\le s\le t}\left\{ \|  \widetilde{h}^\varepsilon (s)\|_{L^\infty_{x,v}}\right\}\\
 \;& +\frac{C_{N,\1}}{\e^{{1}/{2}}}  \sup_{0\le s\le t}\left\{\| \nabla_x\nabla \widetilde{f}^\e (s)\|_{L^2_{x,v}}\right\}+
 q\varepsilon^\frac{4}{5} \sup_{0\le s\le t}\left\{ \| \nabla  h^\varepsilon (s)\|_{L^\infty_{x,v}}\right\}.
\esp
\eal
Using the definition  \eqref{new:nu:def:1} of $\frac{1}{\e^2}\widetilde{\nu}(v)$ and applying direct computations, we obtain
\bals
\nabla\Big(\frac{1}{\e^2}\widetilde{\nu}(v)\Big)
\lesssim\;&
\langle v\rangle|\nabla_x^2\phi^\e|+|\nabla_x\phi^\e|
+\frac{1}{\e^{2}}\nabla_v{\nu}(v)
+\frac{\vartheta\sigma}{(1+t)^{1+\sigma}}\langle v\rangle\\
\lesssim\;&
\langle v\rangle\|\nabla_x^2\phi^\e\|_{L^\infty_{x}}
+\|\nabla_x\phi^\e\|_{L^\infty_{x}}
+\frac{1}{\e^{2}}{\nu}(v)
+\frac{\langle v\rangle^2}{(1+t)^{1+\sigma}},
\eals
which  means, by  \eqref{eq:assumption:1}, \eqref{eq:nu:result:1},
\eqref{eq:Dhnu} and \eqref{eq:Dhnu-2}, that
\bal
\bsp\label{DI6}
|{I}_{11}|\lesssim &\;
 \sup_{0\le s\le t}\left\{ \|\widetilde{h }^\varepsilon (s)\|_{L^\infty_{x,v}}\right\}.
\esp
\eal

From the  \emph{a priori}  assumption \eqref{eq:assumption:1} and the lower bound estimate
\eqref{es:nu:1} of $\frac{1}{\e^2}\widetilde{\nu}$, the last term on the right-hand side of \eqref{eq:Dh1} is bounded by
\bal
\bsp\label{eq:Dh1-1}
\;&\int_0^t \text{exp}\Big\{-\frac{1}{\e^2}\int_s^t\widetilde{\nu}(\tau)\text{d}\tau\Big\}
\|\nabla_x^2\phi^\e(s)\|_{L^\infty_x}(1+s)^{\frac{5}{8}}  \d s\sup_{0\le s\le t}\left\{\| (1+s)^{-\frac{5}{8}}  \nabla_v{\widetilde{h}^\e}(s)\|_{L_{x,v}^{\infty}}\right\}
\\
\lesssim\;&\delta\int_0^t \text{exp}\Big\{-\frac{1}{\e^{\frac{4}{5}}}\int_s^t(1+\tau)^{-\frac{5}{8}} \text{d}\tau\Big\}
(1+s)^{-\frac{5}{8}}  \d s\sup_{0\le s\le t}\left\{(1+s)^{-\frac{5}{8}}  \| \nabla_v{\widetilde{h}^\e}(s)\|_{L_{x,v}^{\infty}}\right\}
\\
\lesssim\;&\delta\e^{\frac{4}{5}}\sup_{0\le s\le t}\left\{(1+s)^{-\frac{5}{8}}  \| \nabla_v{\widetilde{h}^\e}(s)\|_{L_{x,v}^{\infty}}\right\}.
\esp
\eal
Moreover,
applying
\eqref{es:nu:1} and observing  the trivial fact that $ |\nabla_x{\widetilde{h}^\e}|\leq (1+s)^{-\frac{5}{8}}|(1+s)^{\frac{5}{8}}  \nabla_x{\widetilde{h}^\e}|$,   the right-hand last term  of \eqref{eq:Dh2} is controlled by
\bal
\bsp\label{eq:Dh2-1}
\;&\frac{1}{\e}\int_0^t \text{exp}\Big\{-\frac{1}{\e^2}\int_s^t\widetilde{\nu}(\tau)\text{d}\tau\Big\}
(1+s)^{-\frac{5}{8}}  \d s\sup_{0\le s\le t}\left\{(1+s)^{\frac{5}{8}}  \| \nabla_x{\widetilde{h}^\e}(s)\|_{L_{x,v}^{\infty}}\right\}
\\
\lesssim\;&\frac{1}{\e}\int_0^t \text{exp}\Big\{-\frac{1}{\e^{\frac{4}{5}}}\int_s^t(1+\tau)^{-\frac{5}{8}} \text{d}\tau\Big\}
(1+s)^{-\frac{5}{8}}  \d s\sup_{0\le s\le t}\left\{(1+s)^{\frac{5}{8}}  \| \nabla_x{\widetilde{h}^\e}(s)\|_{L_{x,v}^{\infty}}\right\}
\\
\lesssim\;&\e^{-\frac{1}{5}}\sup_{0\le s\le t}\left\{(1+s)^{\frac{5}{8}} \|  \nabla_x{\widetilde{h}^\e}(s)\|_{L_{x,v}^{\infty}}\right\}.
\esp
\eal

Thus, putting the bounds (\ref{eq:DI2}), \eqref{DI6} and \eqref{eq:Dh1-1} into \eqref{eq:Dh1},
and further arguing  in the same procedure as   the proof of \eqref{eqle:result:1} via  \eqref{eq:h7}--\eqref{eq:h8}, we obtain that,
for each $0\leq q\leq\frac{5}{4}$ and for any $0\le t\le T$, there holds
\beq
\bsp
\label{eqle:result:2}
\|  \nabla_x{\widetilde{h}^\e}(t)\|_{L_{x,v}^{\infty}}
\lesssim
   \;&\|  \nabla_x h^\varepsilon_0\|_{L^\infty_{x,v}}+
\e\sup_{0\le s\le t}\left\{ \|  \nabla_x^2\widetilde{\phi}^\varepsilon (s)\|_{H^2_{x}}\right\}
+
 \sup_{0\le s\le t}\left\{ \|  \widetilde{h}^\varepsilon (s)\|_{L^\infty_{x,v}}\right\}\\
 \;& +{\e^{-\frac{1}{2}}}  \sup_{0\le s\le t}\left\{\| \nabla_x^2 \widetilde{f}^\e (s)\|_{L^2_{x,v}}\right\}
+\delta\e^{\frac{4}{5}}\sup_{0\le s\le t}\left\{ (1+s)^{-\frac{5}{8}}  \| \nabla_v{\widetilde{h}^\e}(s)\|_{L_{x,v}^{\infty}}\right\}.
\esp
\eeq
Furthermore, based on \eqref{eq:Dh2}--\eqref{DI6} and \eqref{eq:Dh2-1} and following the same procedure as the proof of  \eqref{eqle:result:2}, we easily verify that for each $0\leq q\leq\frac{5}{8}$ and  for any $0\le t\le T$, there holds
\beq
\bsp
\label{eqle:result:3}
\e^{\frac{1}{5}}\|\nabla_v{\widetilde{h}^\e}(t)\|_{L_{x,v}^{\infty}}
\lesssim\;&\e^{\frac{1}{5}}\| \nabla_v{h^\e_0}\|_{L_{x,v}^{\infty}}
+\e^{\frac{6}{5}}\sup_{0\le s\le t}\left\{\|  \nabla_x^2\widetilde{\phi}^\varepsilon (s)\|_{H^2_{x}}\right\}
+
\e^{\frac{1}{5}} \sup_{0\le s\le t}\left\{ \|  \widetilde{h}^\varepsilon (s)\|_{L^\infty_{x,v}}\right\}\\
\;&
+\e^{-\frac{3}{10}}\sup_{0\le s\le t}\left\{\| \nabla_x\nabla_v\widetilde{f}^\e(s)\|_{L^2_{x,v}} \right\}
+\sup_{0\le s\le t}\left\{(1+s)^{\frac{5}{8}}  \| \nabla_x{{\widetilde{h}}^\e}(s)\|_{L_{x,v}^{\infty}}\right\}.
\esp
\eeq
Consequently, taking the  linear combination  of \eqref{eqle:result:2} with $q=\frac{5}{4}$  and
  $\eta\times$\eqref{eqle:result:3}
with $q=\frac{5}{8}$ for  $\eta$ sufficiently small,  we conclude by applying  \eqref{eqle:result:1}  that
\begin{align}\label{eqle:result:4}
\begin{split}
&\sum_{|\a|+|\b|= 1}\e^{\frac{1}{2}+\frac{1}{5}|\b|}(1+t)^{\frac{5}{4}-\frac{5}{8}|\b|}
\| \partial_{\b}^{\a}h ^\e(t)\|_{L_{x,v}^{\infty}} \\
\lesssim \;&
\e^\frac{1}{2}\| h^\e_0\|_{W^{1,\infty}_{x,v}}
+\e^{\frac{3}{2}}\sup_{0\le \tau\le t}\left\{(1+\tau)^{\frac{5}{4}}  \|  \nabla_x^2{\phi}^\varepsilon (\tau)\|_{H^2_{x}}\right\}
+\sup_{0\le \tau\le t}\left\{(1+\tau)^{\frac{5}{4}} \| \nabla_xf^\e(\tau)\|_{H^1_{x}L^2_{v}} \right\}\\
\;& +\e^{\frac{1}{5}}\sup_{0\le \tau\le t}\left\{(1+\tau)^{\frac{5}{8}} \| \nabla_x\nabla_vf^\e(\tau)\|_{L^2_{x,v}} \right\}
\end{split}
\end{align}
holds for any $0\le t\le T$.
\medskip

\emph{Step 3. $L_{x,v}^\infty$-estimates of $ \nabla^2_x h^\e, \nabla_x\nabla_v h^\e$
and $ \nabla^2_v h^\e$.}\;
Applying $\nabla_x^2$, $\nabla_x\nabla_v$ and $\nabla_v^2$ to (\ref{eq:h:estimate:1}) respectively and employing Duhamel's principle, we have
\bal
\bsp\label{eq:DDh1}
\nabla_x^2\widetilde{h}^\varepsilon(t,x,v)
 \leq\; & \text{exp}\Big\{-\frac{1}{\varepsilon^2}\int_0^t\widetilde{\nu}(\tau)\text{d}\tau\Big\}
|\nabla_x^2\widetilde{h}^\varepsilon(0,X(0),V(0))|+\sum_{i={12}}^{17}|   {I}_i|\\
\; &+\int_0^t \text{exp}\Big\{-\frac{1}{\e^2}\int_s^t\widetilde{\nu}(\tau)\text{d}\tau\Big\}
\left|\nabla_x\left(\nabla_x^2\phi^\e\cdot\nabla_v \widetilde{h}^\e\right)(s,X(s),V(s))\right|\text{d}s\\
\; &+\int_0^t \text{exp}\Big\{-\frac{1}{\e^2}\int_s^t\widetilde{\nu}(\tau)\text{d}\tau\Big\}
\left|\left[\nabla_x^2\phi^\e\cdot\nabla_v\nabla_x \widetilde{h}^\e\right](s,X(s),V(s))\right|\text{d}s
,
\esp
\eal
\bal
\bsp\label{eq:DDh2}
\nabla_x\nabla_v\widetilde{h}^\varepsilon(t,x,v)
\leq  \;& \text{exp}\Big\{-\frac{1}{\varepsilon^2}\int_0^t\widetilde{\nu}(\tau)\text{d}\tau\Big\}
|\nabla_x\nabla_v\widetilde{h}^\varepsilon(0,X(0),V(0))|+\sum_{i={12}}^{17}|   {I}_i|\\
\; &
+\int_0^t \text{exp}\Big\{-\frac{1}{\e^2}\int_s^t\widetilde{\nu}(\tau)\text{d}\tau\Big\}
\left|[\nabla_x^2\phi^\e\cdot\nabla_v \nabla_v\widetilde{h}^\e](s,X(s),V(s))\right|\text{d}s\\
\; &+\frac{1}{\e}\int_0^t \text{exp}\Big\{-\frac{1}{\e^2}\int_s^t\widetilde{\nu}(\tau)\text{d}\tau\Big\}
\left|(\nabla_x^2 \widetilde{h}^\e)(s,X(s),V(s))\right|\text{d}s,
\esp
\eal
and
\bal
\bsp\label{eq:DDh3}
\nabla_v^2\widetilde{h}^\varepsilon(t,x,v)
\leq  \;& \text{exp}\Big\{-\frac{1}{\varepsilon^2}\int_0^t\widetilde{\nu}(\tau)\text{d}\tau\Big\}
|\nabla_v^2\widetilde{h}^\varepsilon(0,X(0),V(0))|+\sum_{i={12}}^{17}|   {I}_i|\\
\; &+\frac{1}{\e}\int_0^t \text{exp}\Big\{-\frac{1}{\e^2}\int_s^t\widetilde{\nu}(\tau)\text{d}\tau\Big\}
\left|(\nabla_x\nabla_v \widetilde{h}^\e)(s,X(s),V(s))\right|\text{d}s.
\esp
\eal
Here ${I}_{12}\sim {I}_{17}$ are given by
\bals
\bsp
{I}_{12}&=\frac{1}{\e^2}\int_0^t \text{exp}\Big\{-\frac{1}{\e^2}\int_s^t\widetilde{\nu}(\tau)\text{d}\tau\Big\}
\partial^\a_\b( K_{w_{\widetilde{\vartheta}}}^c\widetilde{h}^\varepsilon)(s,X(s),V(s))\text{d}s,\\
{I}_{13}&=\frac{1}{\e^2}\int_0^t \text{exp}\Big\{-\frac{1}{\e^2}\int_s^t\widetilde{\nu}(\tau)\text{d}\tau\Big\}
\partial^\a_\b(K_{w_{\widetilde{\vartheta}}}^m\widetilde{h}^\varepsilon)(s,X(s),V(s))\text{d}s,\\
{I}_{14}&=\frac{1}{\e}\int_0^t \text{exp}\Big\{-\frac{1}{\e^2}\int_s^t\widetilde{\nu}(\tau)\text{d}\tau\Big\}
\partial^\a_\b\Big[\frac{w_\vartheta}{\sqrt{\mu}}Q\Big(\frac{\sqrt{\mu}}
{w_{\widetilde{\vartheta}}}\widetilde{h}^\e,
\frac{\sqrt{\mu}}{w_{\widetilde{\vartheta}}}h^\e\Big)\Big](s,X(s),V(s))\text{d}s,\\
{I}_{15}&=-\frac{1}{\e}\int_0^t \text{exp}\Big\{-\frac{1}{\e^2}\int_s^t\widetilde{\nu}(\tau)\text{d}\tau\Big\}
\partial^\a_\b\left[v\cdot\nabla_x \widetilde{\phi}^\e\sqrt{\mu}w_{\widetilde{\vartheta}}\right] (s,X(s),V(s))\text{d}s,\\
{I}_{16}&=-\int_0^t \text{exp}\Big\{-\frac{1}{\e^2}\int_s^t\widetilde{\nu}(\tau)\text{d}\tau\Big\}
\left[q{(1+t)}^{q-1}\partial^\a_\b h^\e\right] (s,X(s),V(s))\text{d}s,\\
{I}_{17}&=-\int_0^t \text{exp}\Big\{-\frac{1}{\e^2}\int_s^t\widetilde{\nu}(\tau)\text{d}\tau\Big\}
\Big[\partial^\a_\b\Big(\frac{1}{\e^2}\widetilde{\nu}(v) \widetilde{h}^\e\Big)
-\frac{1}{\e^2}\widetilde{\nu}(v) \partial^\a_\b\widetilde{h}^\e\Big]
(s,X(s),V(s))\text{d}s
,
\esp
\eals
where $\partial^\a_\b$ represents $\nabla_x^2$, $\nabla_x\nabla_v$ or $\nabla_v^2$
according to the position of ${I}_{12}\sim {I}_{17}$.

Exploiting Lemma \ref{eq:es:K}, Lemma  \ref{eq:es:Q}    and the  similar method applied in the  estimate  of $I_1\sim I_5$,   we  establish
\bal
\bsp\label{eq:DD12}
\sum_{i={12}}^{16}|{I}_i|
  \lesssim  \;&\left(m^{3+\gamma}+\delta\e^{\frac{1}{4}}+\eta \e^{\frac{4}{5}}+\frac{1}{N}+\eta+o(1)\right)
   e^{-\frac{ t^\varrho}{4\e^{a}}} \sup_{0\le s\le t}\left\{\| e^{\frac{ s^\varrho}{4\e^{a}}}\partial_\b^\a\widetilde{h}^\varepsilon (s)\|_{L^\infty_{x,v}}\right\}\\
   \;&+e^{-\frac{t^\varrho}{2\e^a}}\|  \partial_\b^\a h^\varepsilon_0\|_{L^\infty_{x,v}}+
\e\sup_{0\le s\le t}\left\{ \|\partial_x^\a  \nabla_x\widetilde{\phi}^\varepsilon (s)\|_{L^\infty_{x}}\right\}
\\
 \;& +\frac{C_{N,\1}}{\e^{{3}/{2}}}  \sup_{0\le s\le t}\left\{\|  \partial_\b^\a \widetilde{f}^\e (s)\|_{L^2_{x,v}}\right\}+
 q\varepsilon^\frac{4}{5} \sup_{0\le s\le t}\left\{ \| \partial_\b^\a  h^\varepsilon (s)\|_{L^\infty_{x,v}}\right\}\\
 \;&+
 \e\sup_{0\le s\le t}\left\{ \| {h}^\varepsilon (s)\|_{W^{1,\infty}_{x,v}}\right\}\sup_{0\le s\le t}\left\{ \| \nabla_x \widetilde{h}^\varepsilon (s)\|_{L^{\infty}_{x,v}}\right\}+
 \sup_{0\le s\le t}\left\{ \|  \widetilde{h}^\varepsilon (s)\|_{W^{1,\infty}_{x,v}}\right\}\\
  \;&+
 \e\sup_{0\le s\le t}\left\{ \| {h}^\varepsilon (s)\|_{W^{1,\infty}_{x,v}}\right\}\left[\sup_{0\le s\le t}\left\{ \| \nabla_v \widetilde{h}^\varepsilon (s)\|_{L^{\infty}_{x,v}}\right\}
 +\sup_{0\le s\le t}\left\{ \|  \widetilde{h}^\varepsilon (s)\|_{L^{\infty}_{x,v}}\right\}\right],
\esp
\eal
where the last two terms on the right-hand side of \eqref{eq:DD12}  only exist  when $\partial^\a_\b$ represents $\nabla_x\nabla_v$ or $\nabla_v^2$.
The \emph{a priori} assumptions \eqref{eq:assumption:1} and \eqref{eq:assumption:2} imply
\bal
\bsp\label{DDI17-1}
\nabla^2_x\Big(\frac{1}{\e^2}\widetilde{\nu}(v) \widetilde{h}^\e\Big)
-\frac{1}{\e^2}\widetilde{\nu}(v) \nabla^2_x\widetilde{h}^\e
\;&\lesssim
\langle v\rangle|\nabla_x^2\phi^\e||\nabla_x\widetilde{h}^\e|
+\langle v\rangle|\nabla_x^3\phi^\e||\widetilde{h}^\e|\\
\;&
\lesssim\frac{\langle v\rangle^2}{(1+t)^{1+\sigma}}
\left(|\nabla_x\widetilde{h}^\e|+|\widetilde{h}^\e|\right),
\esp
\eal
which combines with
\eqref{eq:Dhnu-2} inducing
\bal
\bsp\label{DDI17-2}
|{I}_{17}|\lesssim &\;
 \sup_{0\le s\le t}\left\{ \|\widetilde{h }^\varepsilon (s)\|_{L^\infty_{x,v}}\right\}+\sup_{0\le s\le t}\left\{ \|\nabla_x\widetilde{h }^\varepsilon (s)\|_{L^\infty_{x,v}}\right\} \quad \text{for } \partial^\a_\b=\nabla_x^2.
\esp
\eal
Furthermore, direct  computations   tell us
\bals
\;&\nabla_x\nabla_v\Big(\frac{1}{\e^2}\widetilde{\nu}(v) \widetilde{h}^\e\Big)
-\frac{1}{\e^2}\widetilde{\nu}(v) \nabla_x\nabla_v\widetilde{h}^\e
+\nabla_v^2\Big(\frac{1}{\e^2}\widetilde{\nu}(v) \widetilde{h}^\e\Big)
-\frac{1}{\e^2}\widetilde{\nu}(v) \nabla_v^2\widetilde{h}^\e\\
\lesssim\;&
\left(|\nabla_x\phi^\e|+\langle v\rangle|\nabla_x^2\phi^\e|+\frac{1}{\e^{2}}{\nu}(v)
+\frac{\langle v\rangle^2}{(1+t)^{1+\sigma}}\right)
\left(|\widetilde{h}^\e|+|\nabla_x\widetilde{h}^\e|+|\nabla_v\widetilde{h}^\e|\right),
\eals
which  combines with \eqref{eq:assumption:1}, \eqref{eq:nu:result:1}, \eqref{eq:Dhnu} and
\eqref{eq:Dhnu-2} inferring
\bal
\bsp\label{DDI17-3}
|{I}_{17}|\lesssim &\;
 \sup_{0\le s\le t}\left\{ \|\widetilde{h }^\varepsilon (s)\|_{W^{1,\infty}_{x,v}}\right\} \quad \text{for } \partial^\a_\b=\nabla_x\nabla_v~\text{or}~\partial^\a_\b=\nabla_v^2.
\esp
\eal

According to  \eqref{eq:assumption:1}, \eqref{eq:assumption:2}
 and \eqref{eq:Dhnu-1},  the last two terms on the right-hand side of \eqref{eq:DDh1} are estimated by
\bal
\bsp\label{eq:DDh1-1}
\;&\delta\e^{-1}\int_0^t \text{exp}\Big\{-\frac{1}{\e^2}\int_s^t\widetilde{\nu}(\tau)\text{d}\tau\Big\}
(1+s)^{-\frac{5}{8}}  \d s\sup_{0\le s\le t}\left\{(1+s)^{-\frac{5}{8}}  \| \nabla_v{\widetilde{h}^\e}(s)\|_{L_{x,v}^{\infty}}\right\}\\
\;&+\delta\int_0^t \text{exp}\Big\{-\frac{1}{\e^2}\int_s^t\widetilde{\nu}(\tau)\text{d}\tau\Big\}
(1+s)^{-\frac{5}{8}}  \d s\sup_{0\le s\le t}\left\{(1+s)^{-\frac{5}{8}}  \| \nabla_x\nabla_v{\widetilde{h}^\e}(s)\|_{L_{x,v}^{\infty}}\right\}\\
\lesssim\;&\delta\e^{-\frac{1}{5}}\sup_{0\le s\le t}\left\{(1+s)^{-\frac{5}{8}}  \| \nabla_v{\widetilde{h}^\e}(s)\|_{L_{x,v}^{\infty}}\right\}
+\delta\e^{\frac{4}{5}}\sup_{0\le s\le t}\left\{(1+s)^{-\frac{5}{8}}  \| \nabla_x\nabla_v{\widetilde{h}^\e}(s)\|_{L_{x,v}^{\infty}}\right\}.
\esp
\eal
Further,
adopting the similar technique
as in the proof of  \eqref{eq:Dh1-1} and \eqref{eq:Dh2-1},   the last two terms on the right-hand side of \eqref{eq:DDh2} can be controlled by
\bal
\bsp\label{eq:DDh2-1}
\delta\e^{\frac{4}{5}}\sup_{0\le s\le t}\left\{(1+s)^{-\frac{5}{8}}  \| \nabla_v^2{\widetilde{h}^\e}(s)\|_{L_{x,v}^{\infty}}\right\}+\e^{-\frac{1}{5}}\sup_{0\le s\le t}\left\{(1+s)^{\frac{5}{8}} \|  \nabla_x^2{\widetilde{h}^\e}(s)\|_{L_{x,v}^{\infty}}\right\}.
\esp
\eal
Similarly,  the right-hand last term   of \eqref{eq:DDh3} can be  bounded by
\bal
\bsp\label{eq:DDh3-1}
\e^{-\frac{1}{5}}\sup_{0\le s\le t}\left\{(1+s)^{\frac{5}{8}} \|  \nabla_x\nabla_v{\widetilde{h}^\e}(s)\|_{L_{x,v}^{\infty}}\right\}.
\esp
\eal

Thus, plugging the bounds (\ref{eq:DD12}), \eqref{DDI17-2} and \eqref{eq:DDh1-1} into \eqref{eq:DDh1},
and arguing  similarly as the proof of \eqref{eqle:result:1}, we find that,
for each $0\leq q\leq\frac{5}{4}$ and for any $0\le t\le T$, there holds
\beq
\bsp
\label{eqle:result:6}
\;&\|{\nabla_x^2{\widetilde{h}^\e}}(t)\|_{L_{x,v}^{\infty}}\\
\lesssim\;&\| \nabla_x^2{h^\e_0}\|_{L_{x,v}^{\infty}}
+
\e\sup_{0\le s\le t}\left\{ \|\nabla_x^3\widetilde{\phi}^\varepsilon (s)\|_{L^\infty_{x}}\right\}
 +\frac{1}{\e^{{3}/{2}}}  \sup_{0\le s\le t}\left\{\|  \nabla_x^2 \widetilde{f}^\e (s)\|_{L^2_{x,v}}\right\}\\
 \;&+
 \delta\sup_{0\le s\le t}\left\{ \| \nabla_x \widetilde{h}^\varepsilon (s)\|_{L^{\infty}_{x,v}}\right\}+\sup_{0\le s\le t}\left\{ \|\widetilde{h }^\varepsilon (s)\|_{L^\infty_{x,v}}\right\}+\sup_{0\le s\le t}\left\{ \| \nabla_x \widetilde{h}^\varepsilon (s)\|_{L^{\infty}_{x,v}}\right\}\\
 \;&
 +\delta\e^{-\frac{1}{5}}\sup_{0\le s\le t}\left\{(1+s)^{-\frac{5}{8}}  \| \nabla_v{\widetilde{h}^\e}(s)\|_{L_{x,v}^{\infty}}\right\}
+\delta\e^{\frac{4}{5}}\sup_{0\le s\le t}\left\{(1+s)^{-\frac{5}{8}}  \| \nabla_x\nabla_v{\widetilde{h}^\e}(s)\|_{L_{x,v}^{\infty}}\right\}.\\
\esp
\eeq
Similarly,  collecting the estimates \eqref{eq:DDh2}--(\ref{eq:DD12}), \eqref{DDI17-3}, \eqref{eq:DDh2-1}  and \eqref{eq:DDh3-1},
we directly establish that
 for any $0\le t\le  T$, there hold
\beq
\bsp
\label{eqle:result:7}
\;&\e^{\frac{1}{5}}\|{\nabla_x\nabla_v {\widetilde{h}^\e}}(t)\|_{L_{x,v}^{\infty}}\\
\lesssim\;&\e^{\frac{1}{5}}\| \nabla_x \nabla_v{h^\e_0}\|_{L_{x,v}^{\infty}}
+
\e^{\frac{6}{5}}\sup_{0\le s\le t}\left\{ \|  \nabla_x^2\widetilde{\phi}^\varepsilon (s)\|_{L^\infty_{x}}\right\}
 +\frac{\e^{{1}/{5}}}{\e^{{3}/{2}}}  \sup_{0\le s\le t}\left\{\| \nabla_x \nabla_v \widetilde{f}^\e (s)\|_{L^2_{x,v}}\right\}\\
 \;&+
 \e^{\frac{1}{5}}(1+\delta)\sup_{0\le s\le t}\left\{ \|  \widetilde{h}^\varepsilon (s)\|_{W^{1,\infty}_{x,v}}\right\}
 +\delta\e\sup_{0\le s\le t}\left\{(1+s)^{-\frac{5}{8}}  \| \nabla_v^2{\widetilde{h}^\e}(s)\|_{L_{x,v}^{\infty}}\right\}\\
 \;&+\sup_{0\le s\le t}\left\{(1+s)^{\frac{5}{8}} \|  \nabla_x^2{\widetilde{h}^\e}(s)\|_{L_{x,v}^{\infty}}\right\}
 ~\qquad\qquad\qquad\qquad\qquad~~~~\text{for each}~~~ 0\leq q\leq\frac{5}{8}
\esp
\eeq
and
\beq
\bsp
\label{eqle:result:8}
\;&\e^{\frac{2}{5}}\|{\nabla_v^2 {\widetilde{{h}}^\e}}(t)\|_{L_{x,v}^{\infty}}\\
\lesssim\;&\e^{\frac{2}{5}}\| \nabla_v^2{h^\e_0}\|_{L_{x,v}^{\infty}}
+
\e^{\frac{7}{5}}\sup_{0\le s\le t}\left\{ \|  \nabla_x\widetilde{{\phi}}^\varepsilon (s)\|_{L^\infty_{x}}\right\}
 +\frac{\e^{{2}/{5}}}{\e^{{3}/{2}}}  \sup_{0\le s\le t}\left\{\| \nabla_v^2 \widetilde{{f}}^\e (s)\|_{L^2_{x,v}}\right\}\\
 \;&+
 \e^{\frac{2}{5}}(1+\delta)\sup_{0\le s\le t}\left\{ \| \widetilde{{h}}^\varepsilon (s)\|_{W^{1,\infty}_{x,v}}\right\}
+\e^{\frac{1}{5}}\sup_{0\le s\le t}\left\{(1+s)^{\frac{5}{8}} \|  \nabla_x\nabla_v{\widetilde{{h}}^\e}(s)\|_{L_{x,v}^{\infty}}\right\}
 ~\text{for}~q=0.
\esp
\eeq

Thus, for any $0\le t\le T$, taking a proper linear combination of \eqref{eqle:result:6}, \eqref{eqle:result:7} and \eqref{eqle:result:8}, and further gathering \eqref{eqle:result:1} and  \eqref{eqle:result:4},
we  conclude  
\begin{align}\label{eqle:result:5}
\begin{split}
&\sum_{|\a|+|\b|= 2}\e^{\frac{3}{2}+\frac{1}{5}|\b|}(1+t)^{\frac{5}{4}-\frac{5}{8}|\b|}
\| \partial_{\b}^{\a}h ^\e(t)\|_{L_{x,v}^{\infty}} \\
\lesssim \;&
\e^\frac{1}{2}\| h^\e _0\|_{W^{2,\infty}_{x,v}}
+\e^{\frac{21}{10}}\sup_{0\le \tau\le t}\left\{(1+\tau)^{\frac{5}{4}}  \|  \nabla_x^2{\phi}^\varepsilon (\tau)\|_{H^2_{x}}\right\}+
\e^{\frac{5}{2}}\sup_{0\le \tau\le t}\left\{ (1+ \tau)^{\frac{5}{4}} \|\nabla_x^3{\phi}^\varepsilon (\tau)\|_{L^\infty_{x}}\right\}
\\
 \;&+\sup_{0\le \tau\le t}\left\{(1+ \tau)^{\frac{5}{4}}\| \nabla_xf^\e( \tau)\|_{H^1_{x}L^2_{v}} \right\}
 +\e^{\frac{1}{5}}\sup_{0\le \tau\le t}\left\{(1+ \tau)^{\frac{5}{8}} \| \nabla_x\nabla_vf^\e( \tau)\|_{L^2_{x,v}} \right\}
  \\
\;&
 +\e^{\frac{2}{5}}\sup_{0\le  \tau\le t}\left\{\|\nabla_v^2f^\e( \tau)\|_{L^2_{x,v}} \right\}
.
\end{split}
\end{align}

In summary, combining  \eqref{eqle:result:1}, \eqref{eqle:result:4} and \eqref{eqle:result:5}, we complete the proof of  Proposition \ref{result:W2:wuqiong}.
\end{proof}

\subsection{Weighted \texorpdfstring{$W^{{2,\infty}}_{x,v}$}{Lg}-estimate with Time Decay for Hard Potentials}\label{W-infty-hard}
\hspace*{\fill}

In this subsection,   we  turn to establish the $W^{{2,\infty}}_{x,v}$-estimate for $h^{\e}$ with    hard potentials $0\leq\gamma\leq 1$.
\begin{proposition}\label{result:W2:wuqiong-hard}
Let $0\leq\gamma\leq 1$. Suppose that $ (f^\e, \nabla_x\phi^\e) $ is a solution to the
VPB system \eqref{eq:f}   defined on $ [0, T] \times \mathbb{R}^3 \times \mathbb{R}^3$.
And, assume that the crucial bootstrap assumptions \eqref{eq:assumption:1}
and \eqref{eq:assumption:2} hold.
Then for each $0\leq q\leq \frac{5}{4}$, there holds
\beqs
\bsp
\;&\sum_{|\a|+|\b|\leq 1}\e^{\frac{1}{2}}(1+t)^{q}
\| \partial_{\b}^{\a}h ^\e(t)\|_{L_{x,v}^{\infty}}
+\sum_{|\a|+|\b|= 2}\e^{\frac{3}{2}}(1+t)^{q}
\| \partial_{\b}^{\a}h ^\e(t)\|_{L_{x,v}^{\infty}}\\
\lesssim \;&\e^\frac{1}{2}\| h^\e_0\|_{W^{2,\infty}_{x,v}}
+\e^{\frac{3}{2}}\sup_{0\le \tau\le t}\left\{(1+\tau)^{q}\| \nabla_x^2\phi^\e(\tau)\|_{H^2_{x}}\right\}+
\e^{\frac{5}{2}}\sup_{0\le\tau\le t}\left\{ (1+ \tau)^{q} \|\nabla_x^3{\phi}^\varepsilon (\tau)\|_{L^\infty_{x}}\right\}\\
\;&
+\sup_{0\le \tau\le t}\left\{(1+\tau)^{q}\| \nabla_xf^\e(\tau)\|_{H^1_{x,v}} \right\}
+\sup_{0\le \tau\le t}\left\{(1+\tau)^{q}\| \nabla_v^2(\mathbf{I}-\mathbf{P})f^\e(\tau)\|_{L^2_{x,v}} \right\}
\esp
\eeqs
  for any $0\leq t \leq T$.
\end{proposition}
\begin{proof}
The main idea to deduce   Proposition \ref{result:W2:wuqiong-hard} is along the same line as    Proposition \ref{result:W2:wuqiong} for the   soft potentials case, but   $\frac{1}{\e^2}\widetilde{\nu}$  for   hard potentials has positive bound   here, leading to some different treatments in our analysis.
Indeed, thanks to Lemma \ref{es:nu} and $\nu(v)\geq \nu_0$, we have
\bal
\;&\int_{0}^{t}
\text{exp}\Big\{-\frac{1}{\e^2}\int_s^t\widetilde{\nu}(\tau)\text{d}\tau\Big\}
\nu(V(s))
e^{-\frac{ \nu_0 s }{4\e^2}}\dd s
\lesssim
\e^2e^{-\frac{  \nu_0 t }{4\e^2}}, \label{eq:Dhnu-hard} \\
\;&\int_{0}^{t}
\text{exp}\Big\{-\frac{1}{\e^2}\int_s^t\widetilde{\nu}(\tau)\text{d}\tau\Big\}
\frac{\langle V(s)\rangle^2}{(1+s)^{1+\sigma}}e^{-\frac{  \nu_0 s  }{4\e^a}}
\dd s
\lesssim
e^{-\frac{  \nu_0 t  }{4\e^2}} \label{eq:Dhnu-2-hard}.
\eal
From   \eqref{eq:assumption:1} and
\eqref{eq:Dhnu-hard}, we find that the last term on the right-hand side of \eqref{eq:Dh1} and \eqref{eq:Dh2}  are bounded by
\bals
\bsp
\delta\e^{2}\sup_{0\le s\le t}\left\{ \| \nabla_v{\widetilde{h}^\e}(s)\|_{L_{x,v}^{\infty}}\right\},\qquad \e\sup_{0\le s\le t}\left\{ \|  \nabla_x{\widetilde{h}^\e}(s)\|_{L_{x,v}^{\infty}}\right\},
\esp
\eals
respectively.
Then  exploiting the above two estimates, \eqref{eq:Dhnu-hard},  \eqref{eq:Dhnu-2-hard},  and adopting the similar analysis as   the proof of \eqref{eqle:result:1} and \eqref{eqle:result:4}, we   conclude that, for each $0\leq q \leq \frac{5}{4}$ and for any $0\le t\le T$, there holds
\begin{align}\label{eqle:result:4-hard}
\begin{split}
\e^{\frac{1}{2}}
\| \widetilde{h} ^\e(t)\|_{W_{x,v}^{1,\infty}}
\lesssim \;&
\e^\frac{1}{2}\| h^\e_0\|_{W^{1,\infty}_{x,v}}
+\e^{\frac{3}{2}}\sup_{0\le \tau\le t}\left\{\|  \nabla_x^2{\widetilde{\phi}}^\varepsilon (\tau)\|_{H^2_{x}}\right\}
+\sup_{0\le \tau\le t}\left\{ \| \nabla_x\widetilde{f}^\e(\tau)\|_{H^1_{x,v}} \right\}.
\end{split}
\end{align}

On the other hand, from   \eqref{eq:assumption:1} and
\eqref{eq:Dhnu-hard}, we deduce  that
   the last two terms on the right-hand side of \eqref{eq:DDh1}, \eqref{eq:DDh2},  and  the last  term  on the right-hand side of \eqref{eq:DDh3} can be controlled by
\bals
\bsp
\;&\delta\e^2\sup_{0\le s\le t}\left\{   \| \nabla_v{\widetilde{h}^\e}(s)\|_{L_{x,v}^{\infty}}\right\}
+\delta\e^2\sup_{0\le s\le t}\left\{  \| \nabla_x\nabla_v{\widetilde{h}^\e}(s)\|_{L_{x,v}^{\infty}}\right\},\\
\;&
\delta\e^2\sup_{0\le s\le t}\left\{  \| \nabla_v^2{\widetilde{h}^\e}(s)\|_{L_{x,v}^{\infty}}\right\}+\e\sup_{0\le s\le t}\left\{\|  \nabla_x^2{\widetilde{h}^\e}(s)\|_{L_{x,v}^{\infty}}\right\},\\
\;& \e\sup_{0\le s\le t}\left\{\|  \nabla_x\nabla_v{\widetilde{h}^\e}(s)\|_{L_{x,v}^{\infty}}\right\},
\esp
\eals
respectively. Then employing  the estimates above, \eqref{eq:Dhnu-2-hard}  and taking the similar process as   the proof of \eqref{eqle:result:5},
we  obtain  
\begin{align}\label{eqle:result:5-hard}
\begin{split}
\;&\sum_{|\a|+|\b|= 2}\e^{\frac{3}{2}}
\| \partial_{\b}^{\a}\widetilde{h} ^\e(t)\|_{L_{x,v}^{\infty}}   \\
\lesssim \;&
\e^\frac{1}{2}\| h^\e _0\|_{W^{2,\infty}_{x,v}}
+\e^{\frac{3}{2}}\sup_{0\le \tau\le t}\left\{ \|  \nabla_x^2{\widetilde{\phi}}^\varepsilon (\tau)\|_{H^2_{x}}\right\}+
\e^{\frac{5}{2}}\sup_{0\le \tau\le t}\left\{  \|\nabla_x^3{\widetilde{\phi}}^\varepsilon (\tau)\|_{L^\infty_{x}}\right\}
\\
 \;&+\sup_{0\le \tau\le t}\left\{\| \nabla_x\widetilde{f}^\e( \tau)\|_{H^1_{x,v}} \right\}
+\sup_{0\le \tau\le t}\left\{\| \nabla_v^2(\mathbf{I}-\mathbf{P})\widetilde{f}^\e(\tau)\|_{L^2_{x,v}} \right\}
\end{split}
\end{align}
  for each $0\leq q \leq \frac{5}{4}$ and for any $0\le t\le T$.
Consequently, gathering \eqref{eqle:result:4-hard} and  \eqref{eqle:result:5-hard}, we complete the proof of    Proposition \ref{result:W2:wuqiong-hard}.
\end{proof}
\begin{remark}For the hard potential case,
the \emph{a priori} assumption    \eqref{eq:assumption:2}
can be replaced by a weaker      priori  assumption: there is $\delta>0$ small enough such that
\bal
\bsp\label{eq:assumption:3}
\;&\e\sup_{0\leq t \leq T}\left\{(1+t)^{\frac{5}{8}}\| \nabla^3_x\phi^\e(t) \|_{{L^{\infty}_{x,v}} }\right\}+
\e^{\frac{3}{4}}\sup_{0\leq t \leq T}\left\{\| h^\e(t) \|_{{L^{\infty}_{x,v}} }\right\}
+\e\sup_{0\leq t \leq T}\{\| h^\e(t) \|_{{W^{1,\infty}_{x,v}} }\}
\le  \delta.
\esp
\eal
Indeed,    the  term $D_5$ defined in \eqref{D:eq:Dxh2} for  hard potentials  can be controlled directly without the time decay rate $(1+t)^{-\frac{5}{4}}$ of $\|\nabla_x\phi^\e\|_{L^\infty_x}$. On the other hand, the \emph{a priori}  assumption  \eqref{eq:assumption:3} is also enough to  control the
velocity growth of $v\cdot\nabla_x^3\phi^\e h^\e$
with the aid of the Young inequality.

In addition, according to \eqref{D:eq:assumption:hard:1-1},   the \emph{a priori}  assumption \eqref{eq:assumption:3}
 can be verified only with the help of the time decay rate $(1+t)^{-\frac{5}{4}}$ of
 $\| \nabla_x^2\phi^\e\|_{H^{1}_{x }}$ and $\| \nabla_xf^\e\|_{H^{1}_{x }L^{2}_{v }}$, rather than the time decay rate of $h^\e$.
\end{remark}

\section{Global Solutions  and  Limit to the Incompressible NSFP System}\label{golbal-decay-soft}
In this section, our main goal is to establish   the global existence and  temporal decay estimate for the perturbation VPB system
(\ref{eq:f}), and
justify the convergence  limit to the incompressible NSFP system  \eqref{INSFP}  as $\e\rightarrow 0$.

\subsection{Time Decay Estimate of Energy  for Soft Potentials}
\hspace*{\fill}

To close the energy estimate and $W^{2,\infty}_{x,v}$-estimate for  soft potentials
$-3<\gamma<0$ under the {\em{a priori}} assumptions \eqref{eq:assumption:1}, \eqref{eq:assumption:2} and  \eqref{energy-assumptition-soft},  we have to investigate the time decay of $\mathbf{E}_\ell^\mathbf{s}(t)$. For this purpose,
denote
\bal
\bsp\label{1:estimate:decay:10}
\mathbf{X}_\ell(t):=\;&\sup_{0\le\tau\le t}\left\{(1+\tau)^{\frac{3}{2}} \mathbf{E}_\ell^\mathbf{s}(\tau)
 +(1+\tau)^{\frac{5}{2}} \|\nabla_x f^\e(\tau)\|_{H^1_xL^2_v}^2
 +(1+\tau)^{\frac{5}{2}} \|\nabla_x \phi^\e(\tau)\|_{H^1_x}^2\right\}\\
 \;&+\sup_{0\le\tau\le t}\Big\{\sum_{|\a|+|\b|\leq 1}\e^{1+\frac{2}{5}|\b|}(1+\tau)^{\frac{5}{2}-\frac{5}{4}|\b|}
\| \partial_{\b}^{\a}h ^\e(\tau)\|_{L_{x,v}^{\infty}}^2
 \Big\}
 \\
 \;&+\sup_{0\le\tau\le t}\Big\{\sum_{|\a|+|\b|= 2}\e^{3+\frac{2}{5}|\b|}(1+\tau)^{\frac{5}{2}-\frac{5}{4}|\b|}
\| \partial_{\b}^{\a}h ^\e(\tau)\|_{L_{x,v}^{\infty}}^2\Big\}
 .
\esp
\eal

The key point is to prove the following main result of this subsection.

\begin{proposition}\label{decay:es:nonlinear:result:1}
Let $-3<\gamma<0,$
$\ell_{0}, \ell,
\widetilde{\vartheta}$ be fixed parameters  stated in Theorem \ref{mainth1}.
Assume that
$
\iint_{\mathbb{R}^3\times\mathbb{R}^3}\sqrt{\mu} f^\e_0 \d x\d v=0,
$
and   the \emph{a priori} assumptions   \eqref{eq:assumption:1}, \eqref{eq:assumption:2} and \eqref{energy-assumptition-soft}  hold true for $\delta>0$ small enough.
Then,  any strong solution $(f^\e, \nabla_x\phi^\e)$ to
  the VPB system \eqref{eq:f} over $0\leq t\leq T$ with $0 <T \leq \infty$ satisfies
\bal\label{decay:result:1}
\mathbf{X}_\ell(t)\lesssim \mathbf{\kappa}_0^2+\delta\mathbf{X}_\ell(t)+\mathbf{X}_\ell^2(t)
\eal
for any $0\leq t\leq T$, where ${\kappa}_0$ is defined by
\beq
\bsp\label{initial-data-soft}
\mathbf{\kappa}_0:=\;&\e^\frac{1}{2}\| h^\e_0\|_{W^{2,\infty}_{x,v}}+\sum_{|\a|+|\b|\leq2}\| w^{|\b|}\partial^{\a}_{\b} f^\e_0\|_{L^2_{x,v}}+\sum_{\substack{|\a|+|\b|\leq2\\0\leq|\a|\leq1}}\| w^{|\b|-(\ell+1)}\partial^{\a}_{\b} f^\e_0\|_{L^2_{x,v}}\\
\;&
+\e^{\frac{1}{2}}\| {w}^{-(\ell+1)}\nabla_x^2f^\e_0\|_{L^2_{x,v}}
 +\big\|\big(1+|x|\big)f^\e_0\big\|_{L^2_vL^1_x}
 +\big\|w^{-(\frac{1}{2}+\ell_{0})}f^\e_0\big\|_{L^2_vL^1_x}.
\esp
\eeq
\end{proposition}

Indeed,  our main idea to deduce Proposition \ref{decay:es:nonlinear:result:1}  is based on  the  temporal decay estimate  on the solution of  the corresponding linearized system of VPB system (\ref{eq:f})  and Duhamel's formula. We shall first consider the time decay properties of solutions to  the linearized VPB system with a nonhomogeneous microscopic source
\beq
 \left\{
\begin{array}{ll}\label{linear:esti:f:decay:1}
 \pt_t f^\e +\frac{1}{\e} v\cdot \nabla_x f^\e+\frac{1}{\e} v\cdot\nabla_x\phi^\e\sqrt{\mu}   +\frac{1}{\e^2}Lf^\e
=h^\e, \quad\quad \mathbf{P}h^\e=0,\\[3mm]
\phi^\e= \frac{1}{4\pi|x|}\ast_{x}\int_{\mathbb{R}^3}\sqrt{\mu}f^\e \d v\rightarrow 0\quad\quad  \text{as}\; |x|\rightarrow\infty.
\end{array}
\right.
\eeq
Given the initial data  $f^\e_0=f^\e_0(x,v)$, we formally define
$e^{tB_{\e}}f^\e_0$ to be  the solution to the linearized homogeneous system \eqref{linear:esti:f:decay:1} when $h^\e= 0$.
Then
we state the following result about the  linearized VPB system \eqref{linear:esti:f:decay:1}  for  soft potentials $-3<\gamma<0$.
\begin{lemma}     \label{result:linear:estimate:decay}
Let $-3<\gamma<0, \ell\geq \frac{1}{2}$,  $\frac{3}{2}<J<2$ and $\ell_0>\frac{3}{4}$.
Assume
$\iint_{\mathbb{R}^3\times\mathbb{R}^3}\sqrt{\mu}f^\e_0\d x\d v=0$
and
\beqs
\e^{\frac{1}{2}}\|{w}^{-({\ell+\ell_0})}f^\e_0\|_{L^2_vL^1_x}
+\e^{\frac{1}{2}}\|{w}^{-(\ell+\ell_0)} f^\e_0\|_{L^2_{x,v}}
+\||x|a^\e_0\|_{L^1_x}+\|f^\e_0\|_{L^2_vL^1_x}
+\|  f^\e_0\|_{L^2_{x,v}}<\infty,
\eeqs
Then for any $t\geq 0$, the solution to the   system \eqref{linear:esti:f:decay:1}  satisfies
\bal
\bsp\label{result:linear:estimate:decay:3}
 & \|  e^{{t B_{\e}}}f_0^\e\|_{L_{x,v}^2}^2
 + \| \nabla_x\Delta_x^{-1} \mathbf{P}_0e^{{t B_{\e}}}f_0^\e\|_{L_{x}^2}^2
 +\e \|{w}^{-\ell}
 e^{{t B_{\e}}}f_0^\e\|_{L_{x,v}^2}^2
 \\
\lesssim\;& (1+t)^{-\frac{3}{2}}
\Big(\e\|{w}^{-(\ell+\ell_0)}f^\e_0\|_{L^2_vL^1_x}^2
+\||x|a_0^\e\|_{L^1_x}^2+\|f^\e_0\|_{L^2_vL^1_x}^2
+\e\|{w}^{-(\ell+\ell_0)}  f^\e_0\|_{L^2_{x,v}}^2
+\|  f^\e_0\|_{L^2_{x,v}}^2
\Big)\\
\;&+(1+t)^{-\frac{3}{2}} \Big(
\e^{2}\int_0^t(1+\tau)^{J} \|{w}^{-(\ell+\frac{1}{2})}h^\e(\tau)\|_{{L^2_vL^1_x}}^2\d \tau
+ \e^{2}\int_0^t \| {w}^{-(\ell+\ell_0+\frac{1}{2}) }h^\e(\tau)\|_{{L^2_vL^1_x}}^2\d \tau
 \Big)\\
\;&+(1+t)^{-\frac{3}{2}}\Big(\e^{2}\int_0^t (1+\tau)^{J} \| {w}^{-(\ell+\frac{1}{2})}  h^\e(\tau)\|_{L^2_{x,v}}^2\d s
+\e^{2}\int_0^t  \| {w}^{-(\ell+\ell_0+\frac{1}{2})}  h^\e(\tau)\|_{L^2_{x,v} }^2\d \tau \Big),
\esp
\eal
where $\mathbf{P}_0$ is a projector corresponding to the hyperbolic part of the macro  component defined by  $\mathbf{P}_0f:=\langle f,\,\sqrt{\mu}\rangle_{L^2_v}\sqrt{\mu}$.
\end{lemma}

\begin{proof} To make the presentation clearly, the proof of \eqref{result:linear:estimate:decay:3} is divided into three steps.

 \emph{Step 1. Weighted $L_{v}^2$-estimate of $\hat{f^{\e}}$.}\; We claim that whenever $\ell>0$,
\bal
\bsp\label{linear:estimate:decay:proof:result:2}
\frac{\d}{\d t}  {E}_{\ell}(\hat{f^{\e}})+\sigma_0 {D}_{\ell}(\hat{f^{\e}})
\lesssim
\e^{2}\|{w}^{-\ell-\frac{1}{2}}\hat{h}^\e\|_{L^2_v}^2
\esp
\eal
holds for any $t>0$ and $k\in \mathbb{R}^3$. Here ${E}_{\ell}(\hat{f^{\e}})$
 and  ${D}_{\ell}(\hat{f^{\e}})$  are given by
\bals
\bsp
{E}_{\ell}(\hat{f^{\e}})=\;&\|\hat{f^{\e}}\|_{L^2_v}^2+\frac{|\hat{a}^\e|^2}{|k|^2}
+
\kappa_1\mathfrak{R}{E}_{lin}(\hat{f}^\e(t))
+\kappa_2\e\|{w}^{-\ell}(\mathbf{I}-\mathbf{P})\hat{f^{\e}}\|_{L^2_v}^2\chi_{|k|\leq 1}\
\\
\;&
+\kappa_2\e\|{w}^{-\ell}\hat{f^{\e}}\|_{L^2_v}^2\chi_{|k|\geq 1},
\\
{D}_{\ell}(\hat{f^{\e}})=\;&
\frac{1}{\e^2}\|(\mathbf{I}-\mathbf{P})\hat{f^{\e}}\|_{L^2_v(\nu)}^2
+\kappa_1\frac{|k|^{2}}{1+|k|^{2}}\|\mathbf{P}\hat f\|_{L^2_v}^2+\kappa_1|\hat{a}^\e|^2
+\frac{\kappa_2}{\e}\|{w}^{-\ell}(\mathbf{I}-\mathbf{P})\hat{f^{\e}}\|_{L^2_v(\nu)}^2,
\esp
\eals
 where $0< \kappa_2\ll \kappa_1\ll 1 $ are chosen later and  ${E}_{lin}(\hat{f}^\e(t))$ is defined in  \eqref{linear:estimate:decay:proof:result:1-3}.

In fact, the proof of \eqref{linear:estimate:decay:proof:result:2} is   similar   to the proof of Proposition \ref{main-weighted-energy-estimate-1}. The difference is that all calculations are made in the Fourier space. Thus, some details in the following proof will be omitted for simplicity. Firstly,
 the Fourier transform in $x$ from \eqref{linear:esti:f:decay:1} gives rise to
\beq
 \left\{
\begin{array}{ll}\label{linear:esti:f:1}
 \pt_t \hat{f}^\e +\frac{1}{\e}i v\cdot k \hat{f}^\e+\frac{1}{\e}i v\cdot k\hat{\phi}^\e\sqrt{\mu}   +\frac{1}{\e^2}L\hat{f}^\e
=\hat{h}^\e, \\[2mm]
|k|^2\hat{\phi}^\e=\hat{a}^\e.
\end{array}
\right.
\eeq
Then taking the $L^2_v(\mathbb{R}^3)$   complex inner product of the first equation in \eqref{linear:esti:f:1}  with $\hat{f^{\e}}$ and employing (\ref{spectL}), we easily deduce
 \bal
\bsp\label{linear:estimate:decay:proof:result:1-1}
\frac{\d}{\d t}\Big(\|\hat{f^{\e}}\|_{L^2_v}^2+\frac{|\hat{a^\e}|^2}{|k|^2}\Big)
 +\frac{\sigma_0}{\e^2}\|(\mathbf{I}-\mathbf{P})\hat{f^{\e}}\|^2_{L^2_v(\nu)}
\lesssim \e^{2} \|\nu^{-\frac{1}{2}}\hat{h}^\e\|_{L^2_v}^2.
\esp
\eal

Secondly,  similar to deriving  \eqref{qkj:esti:diss:2} and
\eqref{qkj:esti:f:L^2:12} from the   VPB system (\ref{eq:f}), we also derive  a macro fluid-type system and the local conservation laws  for the linearized system \eqref{linear:esti:f:1}
\newcommand{\J}{\mathfrak{J}} 
\beqs
 \left\{
\begin{array}{ll}
\displaystyle\sqrt{\mu}:      & \displaystyle \pt_t(\hat{a}^\e-\frac{3}{2}\hat{c}^\e) = \J_1 , \,\quad\qquad\qquad     \\[2mm]
\displaystyle v_l\sqrt{\mu}:  &\displaystyle \pt_t\hat{b}_l+\frac{1}{\e}i k_l(\hat{a}^\e-\frac{3}{2}\hat{c}^\e)+\frac{1}{\e}i k_l\hat{\phi}^\e=\J _2 ^{l}  , \quad\qquad\qquad\\[2mm]
\displaystyle v_l^2\sqrt{\mu}:     & \displaystyle \frac{1}{2}\pt_t \hat{c}^\e+\frac{1}{\e}ik_l\hat{b}_l^\e= \J_3 ^{l}  , \quad\qquad\qquad \\[2mm]
\displaystyle v_lv_j\sqrt{\mu}(l\ne j):  & \displaystyle  \frac{1}{\e}i(k_l\hat{b}_j^\e+k_j\hat{b}_l^\e)= \J _4 ^{l,j}, \,\,\qquad\qquad \\[2mm]
\displaystyle v_l|v|^2\sqrt{\mu}:   & \displaystyle \frac{1}{\e}i k_l\hat{c}^\e= \J _5^l, \quad\qquad\qquad
\end{array}
\right.
\eeqs
and
\beqs
 \left\{
\begin{array}{ll}
\displaystyle \pt_t \hat{a}^\e+\frac{1}{ \e}ik\cdot \hat{b}^\e=0,\\[2mm]
  \displaystyle\pt_t \hat{b}^\e+\frac{1}{\e}i k \hat{\phi}^\e+
   \frac{1}{\e}i k(\hat{a}^\e+\hat{c}^\e)=- \frac{1}{\e}\int_{\mathbb{R}^3}v\cdot \nabla_x (\mathbf{I}-\mathbf{P})f^\e v\sqrt{\mu} \d v
   +\int_{\mathbb{R}^3}\hat{h}^\e v\sqrt{\mu} \d v
   ,\\[2mm]
 \displaystyle\pt_t \hat{c}^\e
 +\frac{2}{3\e}ik\cdot \hat{b}^\e=-
 \frac{1}{3\e}\int_{\bbR^3}v\cdot
 \nabla_x(\mathbf{I}-\mathbf{P})\hat{f}^\e|v|^2\sqrt{\mu}\d v
 + \frac{1}{3}\int_{\mathbb{R}^3}\hat{h}^\e |v|^2\sqrt{\mu} \d v.
\end{array}
\right.
\eeqs
Here the right-hand terms $ \J_1,\J _2 ^{l},\J_3 ^{l},\J _4 ^{l,j}$ and $\J _5^l (l,j=1,2,3) $ are all of the form
\bals
\bsp
-\left\langle\pt_t(\mathbf{I}-\mathbf{P}) \hat{f}^\e, \zeta \right\rangle_{L^2_v}
+\left\langle-\frac{1}{\e}v\cdot\nabla_x (\mathbf{I}-\mathbf{P}) \hat{f}^\e-\frac{1}{\e^2}L \hat{f}^\e+ \hat{h}^\e, \zeta \right\rangle_{L^2_v},
\esp
\eals
where $\zeta$ represents a (different) linear combination of the basis $[\sqrt{\mu},\,v_l\sqrt{\mu},\,v_lv_j\sqrt{\mu},\,v_l|v|^2\sqrt{\mu} ]$.
Thus, thanks to the above fluid system, as in Lemma \ref{prop:estimate:dissipation}, we establish  the macro  dissipation of ${|k|^2}\|\mathbf{P}\hat{f^{\e}}\|_{L^2_v}^2$ in the complex space as follows
\bal
\bsp\label{linear:estimate:decay:proof:result:1-2}
\e\frac{\d}{\d t}\mathfrak{R}{E}_{lin}(\hat{f}^\e(t))
 +
 \frac{|k|^2}{1+|k|^2}\|\mathbf{P}\hat{f^{\e}}\|_{L^2_v}^2+|\hat{a}^{\e}|^2
\lesssim  \frac{1}{\e^2}\|(\mathbf{I}-\mathbf{P})\hat{f^{\e}}\|_{L^2_v(\nu)}^2+ \e^{2} \|\nu^{-\frac{1}{2}}\hat{h}^\e\|_{L^2_v}^2,
\esp
\eal
where ${E}_{lin}(\hat{f}^\e(t))$ is given by
\bal
\bsp\label{linear:estimate:decay:proof:result:1-3}
{E}_{lin}(\hat{f}^\e(t))
=\;&\frac{1}{1+|k|^2}\eta_0\mathfrak{R}(ik\hat{a}^\e|\hat{b}^\e)
+\frac{1}{1+|k|^2}\sum_{l=1}^3\mathfrak{R}\left(ik_l\hat{c}^\e\Big|
\int_{\mathbb{R}^3}(\mathbf{I}-\mathbf{P})\hat{f^{\e}}g_c^{l}\d v\right)\\
\;&+\frac{1}{1+|k|^2}\sum_{l,j=1}^3
\mathfrak{R}\left(ik_j\hat{b}^\e_l\Big|\int_{\mathbb{R}^3}(\mathbf{I}-\mathbf{P})\hat{f^{\e}}g_b^{l,j}\d v\right)
\esp
\eal
with $g_c^{l}$ and $ g_b^{l,j}$  representing some linear combinations of
$\left[\sqrt{\mu},\,v_l\sqrt{\mu},\,v_lv_j\sqrt{\mu},\,v_l|v|^2\sqrt{\mu} \right]$.

Thirdly,  applying $(\mathbf{I}-\mathbf{P})$ to the first equation in \eqref{linear:esti:f:1} leads to
\bals
 \pt_t(\mathbf{I}-\mathbf{P}) \hat{f}^\e +\frac{1}{\e}i v\cdot k (\mathbf{I}-\mathbf{P})\hat{f}^\e +\frac{1}{\e^2}L(\mathbf{I}-\mathbf{P})\hat{f}^\e
\;&=\frac{1}{\e}\left(\mathbf{P}iv\cdot k\hat{f}^\e- iv\cdot k\mathbf{P}\hat{f}^\e\right)+\hat{h}^\e. 
\eals
Then
 taking the $L^2_v(\mathbb{R}^3)$  complex inner product of the above equation with $\e w^{-2\ell}(\mathbf{I}-\mathbf{P})\hat{f^{\e}}$, we are able to deduce  the velocity weighted estimate for the pointwise time-frequency variables $(t, k)$ over $|k|\leq1$
\bal
\bsp\label{linear:estimate:decay:proof:result:3}
\;&\e \frac{\d}{\d t} \|w^{-\ell}(\mathbf{I}-\mathbf{P})\hat{f^{\e}}\|_{L^2_v}^2\chi_{|k|\leq1}
+\frac{\sigma_0}{\e}
\|w^{-\ell}(\mathbf{I}-\mathbf{P})\hat{f^{\e}}\|_{L^2_v(\nu)}^2\chi_{|k|\leq1}\\
\lesssim\;&
\frac{1}{\e}
\|(\mathbf{I}-\mathbf{P})\hat{f^{\e}}\|_{L^2_v(\nu)}^2
+
\frac{|k|^2}{1+|k|^2}\|\mathbf{P}\hat f\|_{L^2_v}
+
\|(\mathbf{I}-\mathbf{P})\hat{f^{\e}}\|_{L^2_v(\nu)}^2
+\e^{3}\|{w}^{-\ell-\frac{1}{2}}\hat{h}^\e\|_{L^2_v}^2.
\esp
\eal
On the other hand,  applying the  $L^2_v(\mathbb{R}^3)$ complex inner product of \eqref{linear:esti:f:1} with $\e w^{-2\ell} \hat{f}^\e$,
we  obtain
\bal
\bsp\label{linear:estimate:decay:proof:result:4}
\;&\e\frac{\d}{\d t}  \|w^{-\ell}\hat{f^{\e}}\|_{L^2_v}^2\chi_{|k|\geq1}
+\frac{\sigma_0}{\e}
\|w^{-\ell}(\mathbf{I}-\mathbf{P})\hat{f^{\e}}\|_{L^2_v(\nu)}^2\chi_{|k|\geq1}\\
\lesssim\;&\frac{|k|^2}{1+|k|^2}\|\mathbf{P}\hat f\|_{L^2_v}+\|(\mathbf{I}-\mathbf{P})\hat{f^{\e}}\|_{L^2_v(\nu)}^2
+\frac{1}{\e^2}
\|(\mathbf{I}-\mathbf{P})\hat{f^{\e}}\|_{L^2_v(\nu)}^2
+\e^{2}\|{w}^{-\ell-\frac{1}{2}}\hat{h}^\e\|_{L^2_v}^2.
\esp
\eal

Finally,   choosing $0< \kappa_2\ll \kappa_1\ll 1 $ small enough and  taking the linear combination
 (\ref{linear:estimate:decay:proof:result:1-1})$+$ (\ref{linear:estimate:decay:proof:result:1-2})$\times\kappa_1+$ (\ref{linear:estimate:decay:proof:result:3})$\times\kappa_2+$
\eqref{linear:estimate:decay:proof:result:4}$\times\kappa_2$, we  conclude our desired claim \eqref{linear:estimate:decay:proof:result:2}.

\emph{Step 2.  Time frequency Lyapunov inequality in integral form.}\;
Let $0<\widetilde{\eta}\ll 1$ and $J>0$ be a constant chosen later. Multiplying \eqref{linear:estimate:decay:proof:result:2} by
$\left( 1+\frac{\widetilde{\eta}|k|^2}{1+|k|^2}t \right)^{J}$, we have
\bal
\bsp\label{linear:decay:proof:1}
\;&\frac{\d}{\d t}\left[ \left( 1+\frac{\widetilde{\eta}|k|^2}{1+|k|^2}t \right)^{J} {E}_{\ell}(\hat{f^{\e}})\right]+\sigma_0  \left( 1+\frac{\widetilde{\eta}|k|^2}{1+|k|^2}t \right)^{J} {D}_{\ell}(\hat{f^{\e}})\\
\lesssim\;& \widetilde{\eta}  J\left( 1+\frac{\widetilde{\eta}|k|^2}{1+|k|^2}t \right)^{J-1} \frac{|k|^2}{1+|k|^2} {E}_{\ell}(\hat{f^{\e}})
+ \e^{2}\left( 1+\frac{\widetilde{\eta}|k|^2}{1+|k|^2}t \right)^{J} \|{w}^{-\ell-\frac{1}{2}}\hat{h}^\e\|_{L^2_v}^2.
\esp
\eal
Choosing $0< \kappa_2\ll \kappa_1\ll 1 $ small enough, we get
\bals
\bsp
{E}_{\ell}(\hat{f^{\e}})&\thicksim\|\hat{f^{\e}}\|_{L^2_v}^2+
\frac{|\hat{a}^\e|^2}{|k|^2}
+\e\|{w}^{-\ell}\hat{f^{\e}}\|_{L^2_v}^2,\\
&\lesssim \|(\mathbf{I}-\mathbf{P})\hat{f^{\e}}\|_{L^2_v}^2
+\e\|{w}^{-\ell}(\mathbf{I}-\mathbf{P})\hat{f^{\e}}\|_{L^2_v}^2+
\frac{|\hat{a}^\e|^2}{|k|^2}+\|\mathbf{P}\hat{f^{\e}}\|_{L^2_v}^2.
\esp
\eals

Then  similar as
\cite{DYZ2003}, by splitting the velocity integration into  $\left\{w^{-1}(v)\leq  1+\frac{\widetilde{\eta}|k|^2}{1+|k|^2}t\right \}$
and $\left\{w^{-1}(v)\geq 1+\frac{\widetilde{\eta}|k|^2}{1+|k|^2}t\right \}$, we get that for  $p>1$ and $2\ell_0:=J+p-1$,  the first term on the right-hand side of  \eqref{linear:decay:proof:1} is bounded by
\bals
\frac{\sigma_0}{2}  \left( 1+\frac{\widetilde{\eta}|k|^2}{1+|k|^2}t \right)^{J} {D}_{\ell}(\hat{f^{\e}})
+ J\left( 1+\frac{\widetilde{\eta}|k|^2}{1+|k|^2}t \right)^{-p}\frac{\widetilde{\eta}|k|^2}{1+|k|^2} {E}_{\ell+\ell_0}(\hat{f^{\e}})
\quad\text{for} \quad \ell\geq \frac{1}{2}.
\eals
Furthermore, utilizing  \eqref{linear:estimate:decay:proof:result:2} implies
\bals
 {E}_{\ell+\ell_0}(\hat{f^{\e}})\lesssim {E}_{\ell+\ell_0}(\hat{f^{\e}_0})
 +\e^{2}\int_{0}^t \|{w}^{-\ell-\ell_0-\frac{1}{2}}\hat{h}^\e(\tau)\|_{L^2_v}^2\d \tau.
\eals
Consequently, noting $p>1$, by  the above two estimates and  \eqref{linear:decay:proof:1}, we have
\bal
\bsp\label{linear:decay:proof:5}
{E}_{\ell}(\hat{f^{\e}})
\lesssim\;& \left( 1+\frac{\widetilde{\eta}|k|^2}{1+|k|^2}t \right)^{-J}
\Big({E}_{\ell+\ell_0}(\hat{f_0^{\e}})
+\e^2\int_0^t\|w^{-\ell-\ell_0-\frac{1}{2}}\hat{h}^\e(\tau)\|_{L^2_v}^2\d \tau\Big)\\
\;&+\left( 1+\frac{\widetilde{\eta}|k|^2}{1+|k|^2}t \Big)^{-J}\e^2\int_0^t\Big( 1+\frac{\widetilde{\eta}|k|^2}{1+|k|^2}\tau \right)^{J}\|w^{-\ell-\frac{1}{2}}\hat{h}^\e(\tau)\|_{L^2_v}^2\d \tau
\esp
\eal
for any $t>0$ and $k\in\mathbb{R}^3.$

\medskip

\emph{Step 3.  Temporal decay  of  solution to linearized VPB system.}\;
{Noting that whenever} $J>\frac{3}{2},$ we have
\bals
\bsp
\int_{|k|\leq1}\left(1+\frac{\widetilde{\eta}|k|^2}{1+|k|^2}t \right)^{-J}\d k
\lesssim(1+t)^{-\frac{3}{2}},\qquad
\sup_{|k|\geq1}\left(1+\frac{\widetilde{\eta}|k|^2}{1+|k|^2}t \right)^{-J}
\lesssim(1+t)^{-J}.
\esp
\eals
 Utilizing the above two estimates  and the initial condition in Lemma \ref{result:linear:estimate:decay},       integrating \eqref{linear:decay:proof:5} over $\mathbb{R}^3_k$ and then employing a similar process as in \cite{DYZ2003}  by splitting $\mathbb{R}^3_k=\{|k| \leq 1\} \cup \{|k| \geq 1\}$, we
  derive the desired temporal decay property \eqref{result:linear:estimate:decay:3}.
This completes the proof.
\end{proof}

Secondly,    we  establish   the temporal decay estimate  on $\|\mathbf{P}f^\e\|_{L^2_{x,v}}$ and $\|\nabla_x\phi^\e\|_{L^2_{x}}$ with the help of Lemma \ref{result:linear:estimate:decay}.

\begin{lemma}\label{decay:es:nonlinear}
Under the assumptions of Proposition \ref{decay:es:nonlinear:result:1}, we have
\bal\label{hongguan:decay:result}
\|\mathbf{P}f^\e(t)\|_{L^2_{x,v}}^2+\|\nabla_x\phi^\e(t)\|_{L^2_{x}}^2
\lesssim (1+t)^{-\frac{3}{2}}\left(\mathbf{\kappa}_0^2
+\delta\mathbf{X}_\ell(t)+\mathbf{X}_\ell^2(t)\right)
\eal
for any $0\leq t\leq T.$
\end{lemma}
\begin{proof}By Duhamel's principle, the solution $f^\e$ to
  the   VPB system (\ref{eq:f}) can be expressed as
\bal\label{hongguan:decay:equation:1}
f^\e=e^{tB_{\e}}f^\e_0+\int_{0}^t e^{(t-s)B_{\e}}h_1^\e(s)\d s
+\int_{0}^t e^{(t-s)B_{\e}}h_2^\e(s)\d s,
\eal
where the nonlinear terms $h_1^\e$ and $h_2^\e$ are given by
\bal\label{hongguan:decay:equation:2}
h_1^\e=\frac{1}{\e}\Gamma(f^\e, f^\e),\quad
h_2^\e=\nabla_x \phi^\e\cdot \nabla_v f^\e-\frac{1}{2}v\cdot\nabla_x \phi^\e  f^\e.
\eal
Note that $\mathbf{P}_0(h_1^\e+h_2^\e)=0$ and $\mathbf{P}h_1^\e=0$ for all $t\geq0$.
Then applying   Lemma \ref{result:linear:estimate:decay} to   \eqref{hongguan:decay:equation:1}, we have
\bal
\bsp\label{hongguan:decay:result:1}
 &\| \mathbf{P}f^\e(t)\|_{L_{x,v}^2}^2
 + \| \nabla_x\phi^\e(t)\|_{L_{x}^2}^2+\e\| w^{-\frac{1}{2}}f^\e(t)\|_{L_{x,v}^2}^2
\\
\lesssim\;& (1+t)^{-\frac{3}{2}}
\left(\e\|{w}^{-(\frac{1}{2}+\ell_0)}f^\e_0\|_{{L^2_{v}L^1_{x}}}^2
 +\|(1+|x|)f^\e_0\|_{{L^2_{v}L^1_{x}}}^2
+\e\|{w}^{-(\frac{1}{2}+\ell_0)}  f^\e_0\|_{L^2_{x,v}}^2
+\|  f^\e_0\|_{L^2_{x,v}}^2
\right)\\
\;&+(1+t)^{-\frac{3}{2}} \Big(
\e^{2}\int_0^t(1+\tau)^{J} \|{w}^{-1}h_1^\e(\tau)\|_{L^2_vL^1_x}^2\d \tau
+\e^{2}\int_0^t \| {w}^{-(1+\ell_0) }h_1^\e(\tau)\|_{{L^2_vL^1_x}}^2\d \tau
\Big)\\
\;&+(1+t)^{-\frac{3}{2}}\Big(\e^{2}\int_0^t (1+\tau)^{J} \| {w}^{-1}  h_1^\e(\tau)\|_{L^2_{x,v}}^2\d \tau
+\e^{2}\int_0^t  \| {w}^{-(1+\ell_0)}  h_1^\e(\tau)\|_{L^2_{x,v} }^2\d \tau \Big)\\
\;&+\int_0^t (1+t-\tau)^{-\frac{3}{2}}\Big(\e\| {w}^{-(\frac{1}{2}+\ell_0) }h_2^\e(\tau)\|_{{L^2_vL^1_x}}^2+\| h_2^\e(\tau)\|_{{L^2_vL^1_x}}^2\Big)\d \tau\\
\;&+\int_0^t (1+t-\tau)^{-\frac{3}{2}}\Big(\e\| {w}^{-(\frac{1}{2}+\ell_0) }h_2^\e(\tau)\|_{L^2_{x,v}}^2+\| h_2^\e(\tau)\|_{L^2_{x,v}}^2\Big)\d \tau,
L^2_vL^1_x\leq  L^1_xL^2_v
\esp
\eal
where $\ell_0>\frac{3}{4}$ is  a constant.

On  one hand, applying the similar analysis as that of (6.10)  in \cite{DYZ2003}, the H\"{o}lder inequality,
 the definition \eqref{1:estimate:decay:10} of $\mathbf{X}_{\ell}(t)$ and  the assumption $\ell\geq \max\{1,-\frac{1}{\gamma}\}+\frac{1}{2}+\ell_{0}$,  we   directly infer
\bals
\bsp
\e^2\|w^{-1} h_1^\e(t)\|_{L^2_{v}L^1_{x}}^2
=
\;&\big\|w^{-1} \Gamma(f^\e, f^\e)(t)\big\|_{L^2_{v}L^1_{x}}^2\\
\lesssim \;&\|f^\e\|_{L_{v}^\infty L_{x}^2}^2
\|f^\e\|_{L_{x,v}^2}^2
\\
\lesssim \;&\Big(\|w^{-1}w\nabla_vf^\e\|_{ L_{x,v}^2}^2+
\|w^{-2}w^2\nabla_v^2f\|_{ L_{x,v}^2}^2\Big)
\|f^\e\|_{L_{x,v}^2}^2\\
\lesssim \;&(1+t)^{-3}\left[\mathbf{X}_\ell(t)\right]^2
.
\esp
\eals
Similarly, we are also able to prove
\bals
\bsp
\e^2\|w^{-(1+\ell_0)} h_1^\e(t)\|_{L^2_{v}L^1_{x}}^2
\lesssim \|f^\e\|_{ L_{v}^\infty L_{x}^2}^2
\|w^{-{\ell_0}}f^\e\|_{L_{x,v}^2}^2
\lesssim \;&(1+t)^{-3}\left[\mathbf{X}_\ell(t)\right]^2
.
\esp
\eals
In terms of the bound \eqref{Nonlinear:term:proof:1}, the H\"{o}lder inequality,
  \eqref{1:estimate:decay:10} and  the assumption $\ell\geq \max\{1,-\frac{1}{\gamma}\}+\frac{1}{2}+\ell_{0}$,  we   get
\bals
\bsp
&\e^2\|w^{-1} h_1^\e(t)\|_{L^2_{x,v}}^2
=
\big\|w^{-\frac{1}{2}}\nu^{-\frac{1}{2}} \Gamma(f^\e, f^\e)(t)\big\|_{L^2_{x,v}}^2
\\
\lesssim \;&\| w_\vartheta f^\e(t)\|_{L_{x,v}^\infty}^2
\|w^{-\frac{1}{2}}(\mathbf{I}-\mathbf{P})f^\e(t)\|_{L_{x,v}^2(\nu)}^2
 +\| \mathbf{P}f^\e(t)\|_{L_{x}^\infty L_{v}^2}^2
\|w^{-\frac{1}{2}}f^\e(t)\|_{L_{x,v}^2(\nu)}^2
\\
\lesssim\;& \sup_{0\leq \tau \leq t}\Big\{(1+\tau)^{\frac{5}{2}}\e\| w_\vartheta f^\e(\tau)\|_{L_{x,v}^\infty}^2\Big\}(1+t)^{-\frac{5}{2}}
\left[\mathbf{D}_{\frac{1}{2}}^\mathbf{s}(t)\right]^\frac{1}{2}
\| f^\e(t)\|_{L_{x,v}^2}
+\| f^\e(t)\|_{H_{x}^2 L_{v}^2}^4\\
\lesssim \;&
\mathbf{X}_\ell(t)(1+t)^{-3}
\mathbf{D}_{\frac{1}{2}}^\mathbf{s}(t)
+\mathbf{X}_\ell(t)(1+t)^{-2}
\| f^\e(t)\|_{L_{x,v}^2}^2
+
\| f^\e(t)\|_{H_{x}^2 L_{v}^2}^4
\\
\lesssim \;&
\mathbf{X}_\ell(t)(1+t)^{-3}
\mathbf{D}_{\ell}^\mathbf{s}(t)
+(1+t)^{-3}\left[\mathbf{X}_\ell(t)\right]^2
.
\esp
\eals
Similarly, we also have
\bals
\bsp
\;&\e^2\|w^{-(1+\ell_0)} h_1^\e(t)\|_{L^2_{x,v}}^2\\
\lesssim \;&\| w_\vartheta f^\e\|_{L_{x,v}^\infty}^2
\|w^{-(\frac{1}{2}+\ell_0)}
(\mathbf{I}-\mathbf{P})f^\e\|_{L_{x,v}^2(\nu)}^2
 +\| \mathbf{P}f^\e\|_{L_{x}^\infty L_{v}^2}^2
\|w^{-(\frac{1}{2}+\ell_0)}f^\e\|_{L_{x,v}^2(\nu)}^2\\
\lesssim \;&
\mathbf{X}_\ell(t)(1+t)^{-3}
\mathbf{D}_{\ell}^\mathbf{s}(t)
+(1+t)^{-3}\left[\mathbf{X}_\ell(t)\right]^2
.
\esp
\eals


On the other hand,
exploiting the expression \eqref{hongguan:decay:equation:2} of $h_2^\e$, the H\"{o}lder inequality and  the assumption $\ell\geq \max\{1,-\frac{1}{\gamma}\}+\frac{1}{2}+\ell_{0}$, we find
\bals
\bsp
&
\e^\frac{1}{2}\| {w}^{-(\frac{1}{2}+\ell_0) }h_2^\e(t)\|_{{L^2_vL^1_x}}
+\| h_2^\e(t)\|_{L^2_vL^1_x}\\
\lesssim \;&
\|\nabla_x \phi^\e(t)\|_{L^2_x}\left(\e^\frac{1}{2}\|{w}^{-(1+\frac{1}{2}+\ell_0) }w\nabla_v f^\e(t)\|_{L^2_{x,v}}+\e^\frac{1}{2}\|{w}^{-(-\frac{1}{\gamma}+\frac{1}{2}+\ell_0) } f^\e(t)\|_{L^2_{x,v}}\right)\\
 \;&+\|\nabla_x \phi^\e(t)\|_{L^2_x}
\left(\left\|w^{-1}w\nabla_v f^\e(t)\right\|_{L^2_{x,v}}+\|w^{-
(-\frac{1}{\gamma})} f^\e(t)\|_{L^2_{x,v}}\right)\\
\lesssim \;&
\mathbf{E}_{\ell}^\mathbf{s}(t)\lesssim(1+t)^{-\frac{3}{2}}\mathbf{X}_\ell(t).
\esp
\eals
Taking the similar argument  as above and    the Sobolev embedding inequality,  we also deduce that
\bals
\bsp
&
\e^\frac{1}{2}\| {w}^{-(\frac{1}{2}+\ell_0) }h_2^\e(t)\|_{{L^2_{x,v}}}
+\| h_2^\e(t)\|_{L^2_{x,v}}\\
\lesssim \;&
\|\nabla_x \phi^\e(t)\|_{L^\infty_x}\left(\e^\frac{1}{2}\|{w}^{-(1+\frac{1}{2}+\ell_0) }w\nabla_v f^\e(t)\|_{L^2_{x,v}}+\e^\frac{1}{2}\|{w}^{-(-\frac{1}{\gamma}+\frac{1}{2}+\ell_0) } f^\e(t)\|_{L^2_{x,v}}\right)
\\
\;&+\|\nabla_x \phi^\e(t)\|_{L^\infty_x}
\left(\left\|w^{-1}w\nabla_v f^\e(t)\right\|_{L^2_{x,v}}+\|w^{-
(-\frac{1}{\gamma})} f^\e(t)\|_{L^2_{x,v}}\right)\\
\lesssim\;&(1+t)^{-\frac{3}{2}}\mathbf{X}_\ell(t).
\esp
\eals

Therefore, substituting the above estimates into \eqref{hongguan:decay:result:1} and using $J<2$, we conclude
\bal
\bsp\label{hongguan:decay:result:1-6}
 & \| \mathbf{P}f^\e(t)\|_{L_{x,v}^2}^2
 + \| \nabla_x\phi^\e(t)\|_{L_{x}^2}^2+\e\| w^{-\frac{1}{2}}f^\e(t)\|_{L_{x,v}^2}^2
\\
\lesssim\;& (1+t)^{-\frac{3}{2}}
\kappa_0^2+(1+t)^{-\frac{3}{2}}
\int_0^t(1+\tau)^{J} (1+\tau)^{-3}
\mathbf{D}_{\ell}^\mathbf{s}(\tau)\d \tau
\mathbf{X}_\ell(t)\\
\;&+(1+t)^{-\frac{3}{2}}\int_0^t (1+\tau)^{J} (1+\tau)^{-3}\d \tau\left[\mathbf{X}_\ell(t)\right]^2
+\int_0^t (1+t-\tau)^{-\frac{3}{2}}(1+\tau)^{-3}\d \tau\left[\mathbf{X}_\ell(t)\right]^2
\\
\lesssim\;& (1+t)^{-\frac{3}{2}}\kappa_0^2
+(1+t)^{-\frac{3}{2}}\delta \mathbf{X}_\ell(t)
+(1+t)^{-\frac{3}{2}} \left[\mathbf{X}_\ell(t)\right]^2,
\esp
\eal
which leads to \eqref{hongguan:decay:result} immediately. This completes the proof.
\end{proof}

Finally, with Lemma \ref{decay:es:nonlinear} in hand, we    give the proof of Proposition \ref{decay:es:nonlinear:result:1}.
\begin{proof}[\textbf{Proof of Proposition \ref{decay:es:nonlinear:result:1}}]
To make the presentation clear, we divide the proof of Proposition \ref{decay:es:nonlinear:result:1} into three steps.

\emph{Step 1.  Optimal time decay estimate on $\mathbf{E}_{\ell}^\mathbf{s}(t)$.}\;
From Proposition \ref{main-weighted-energy-estimate-1} and the \emph{a priori} assumption \eqref{eq:assumption:2}, we have
  \bal\label{main-result:1}
  \frac{\d}{\d t}\mathbf{E}_{\ell}^\mathbf{s}(t)+\mathbf{D}_{\ell}^\mathbf{s}(t)
  \lesssim0.
  \eal
Take $0 <p<\frac{1}{2}$ small enough. Multiplying the uniform energy  inequality \eqref{main-result:1} by $(1 + t)^{\frac{3}{2}+p}$, and further  taking the time integration over $[0, t]$, we deduce
  \bal
  \bsp\label{decay:soft:result:2}
\;&(1 + t)^{\frac{3}{2}+p}\mathbf{E}_{\ell}^\mathbf{s}(t)
+\int_0^t(1 + \tau)^{\frac{3}{2}+p}\mathbf{D}_{\ell}^\mathbf{s}(\tau)\d \tau\\
\lesssim\;& \mathbf{E}_{\ell}^\mathbf{s}(0)+\big(\frac{3}{2}+p\big)\int_0^t(1 + \tau)^{\frac{1}{2}+p}\mathbf{E}_{\ell}^\mathbf{s}(\tau)\d \tau\\
\lesssim\;&\mathbf{E}_{\ell}^\mathbf{s}(0)+\int_0^t(1 + \tau)^{\frac{1}{2}+p}\left(\|\mathbf{P}f^\e(\tau)\|_{L^2_{x,v}}^2+ \|\nabla_x\phi^\e(\tau)\|_{L^2_{x}}^2\right)\d \tau+
\int_0^t(1 + \tau)^{\frac{1}{2}+p}\mathbf{D}_{\ell+\frac{1}{2}}^\mathbf{s}(\tau)\d \tau.
  \esp
  \eal
Here we have used the inequality
\bal\label{main-result:2}
\mathbf{E}_{\ell}^\mathbf{s}(t)\lesssim
\|\mathbf{P}f^\e(t)\|_{L^2_{x,v}}^2+ \|\nabla_x\phi^\e(t)\|_{L^2_{x}}^2+\mathbf{D}_{\ell+\frac{1}{2}}^\mathbf{s}(t)\quad \;\text{for any}  \; \ell>0,
\eal
which is derived from  \eqref{def-energy-R} and
\eqref{def-disspation-R}.

Similarly, we derive from \eqref{main-result:1} with $\ell$ replaced by $\ell+\frac{1}{2}$ and further multiply it by $(1 + t)^{\frac{1}{2}+p}$
  \bal
  \bsp\label{decay:soft:result:3}
\;&(1 + t)^{\frac{1}{2}+p}\mathbf{E}_{\ell+\frac{1}{2}}^\mathbf{s}(t)+
\int_0^t(1 + \tau)^{\frac{1}{2}+p}\mathbf{D}_{\ell+\frac{1}{2}}^\mathbf{s}(\tau)\d \tau\\
\lesssim\;& \mathbf{E}_{\ell+\frac{1}{2}}^\mathbf{s}(0)+\big(\frac{1}{2}+p\big)\int_0^t(1 + \tau)^{p-\frac{1}{2}}\mathbf{E}_{\ell+\frac{1}{2}}^\mathbf{s}(\tau)\d \tau\\
\lesssim\;&\mathbf{E}_{\ell+\frac{1}{2}}^\mathbf{s}(0)
+\int_0^t\left(\|\mathbf{P}f^\e(\tau)\|_{L^2_{x,v}}^2+ \|\nabla_x\phi^\e(\tau)\|_{L^2_{x}}^2\right)\d \tau+
\int_0^t\mathbf{D}_{\ell+1}^\mathbf{s}(\tau)\d \tau\\
\lesssim\;&\mathbf{E}_{\ell+1}^\mathbf{s}(0)+\int_0^t
\left(\|\mathbf{P}f^\e(\tau)\|_{L^2_{x,v}}^2+ \|\nabla_x\phi^\e(\tau)\|_{L^2_{x}}^2\right)\d \tau,
  \esp
  \eal
  where the last inequality   above  used a fact
\bals
\mathbf{E}_{\ell+1}^\mathbf{s}(t)+\int^t_0\mathbf{D}_{\ell+1}^\mathbf{s}(\tau) \d \tau\lesssim\mathbf{E}_{\ell+1}^\mathbf{s}(0),
\eals
which is directly derived from the time integration of \eqref{main-result:1}   with $\ell$ replaced by $\ell+1$.
Further, substituting the  estimate  \eqref{decay:soft:result:3} into \eqref{decay:soft:result:2} and applying  Lemma \ref{decay:es:nonlinear}, we  arrive at
  \bal
  \bsp\label{decay:soft:result:4}
\;&(1 + t)^{\frac{3}{2}+p}\mathbf{E}_{\ell}^\mathbf{s}(t)
+\int_0^t(1 + \tau)^{\frac{3}{2}+p}\mathbf{D}_{\ell}^\mathbf{s}(\tau)\d \tau\\
\lesssim\;&\mathbf{E}_{\ell+1}^\mathbf{s}(0)+
\int_0^t(1 + \tau)^{\frac{1}{2}+p}\left(\|\mathbf{P}f^\e(\tau)\|_{L^2_{x,v}}^2+ \|\nabla_x\phi^\e(\tau)\|_{L^2_{x}}^2\right)\d \tau\\
\lesssim\;&\mathbf{E}_{\ell+1}^\mathbf{s}(0)+
\int_0^t(1 + \tau)^{\frac{1}{2}+p}(1+\tau)^{-\frac{3}{2}}\d \tau\left(\mathbf{\kappa}_0^2
+\delta\mathbf{X}_\ell(t)+\mathbf{X}_\ell^2(t)\right)\\
\lesssim\;&\mathbf{E}_{\ell+1}^\mathbf{s}(0)+
(1+t)^{p}\left(\mathbf{\kappa}_0^2
+ \delta\mathbf{X}_\ell(t)+\mathbf{X}_\ell^2(t)\right),
  \esp
  \eal
where we have used  a basic inequality
\bals
\int_0^t(1 + \tau)^{\frac{1}{2}+p}(1+\tau)^{-\frac{3}{2}}\d \tau\lesssim(1+t)^{p}.
\eals
On the other hand, in terms of  $ \Delta_x\phi^\e_0( x) = a^\e_0(x)$, the interpolation  inequality and the Young inequality, we have
\bals
\|\nabla_x\phi^\e_0\|_{L^2_x}\;&\lesssim  \|a_0^\e\|_{L^{6/5}_x}\lesssim  \|a_0^\e\|_{L^1_x}^{\frac{2}{3}} \|a_0^\e\|_{L^2_x}^{\frac{1}{3}}
\lesssim  \|f_0^\e\|_{L^1_xL^2_v}+ \|f_0^\e\|_{L^2_xL^2_v},\\[2mm]
\|\nabla_x^2\phi^\e_0\|_{H^1_x}\;&\lesssim  \|\nabla_x^2\Delta^{-1}a_0^\e\|_{H^1_x}
\lesssim  \|a_0^\e\|_{H^1_x}\lesssim  \|f_0^\e\|_{H^1_xL^2_v}.
\eals
These mean that  $\|\nabla_x\phi^\e_0\|_{H^2_x}$ can  be removed from $\mathbf{E}_{\ell+1}^\mathbf{s}(0)$ by the definition of \eqref{def-energy-R} and
 \bal\label{decay:soft:result:4-1}
 \mathbf{E}_{\ell+1}^\mathbf{s}(0)\lesssim \kappa_0^2
  \eal
  holds true with $\kappa_0$ given by \eqref{initial-data-soft}.
Therefore, combining \eqref{decay:soft:result:4} with \eqref{decay:soft:result:4-1}, we get
  \bal
  \bsp\label{decay:soft:result:5}
\sup_{0\le\tau\le t}\left\{(1+\tau)^{\frac{3}{2}} \mathbf{E}_\ell^\mathbf{s}(\tau)
\right\}+(1+t)^{-p}\int_0^t(1 + \tau)^{\frac{3}{2}+p}\mathbf{D}_{\ell}^\mathbf{s}(\tau)\d \tau\lesssim
\mathbf{\kappa}_0^2
+\delta\mathbf{X}_\ell(t)+\mathbf{X}_\ell^2(t).
 \esp
  \eal

\emph{Step 2. Optimal time decay estimate on $\|\nabla_x f^\e(t)\|_{H^1_xL^2_v}$.\;}
To this end, we
first perform the energy  estimate without time decay factor.
In fact,
employing similar derivation in \eqref{qkj:esti:Dxf:L^2:1} leads to
\bal
\bsp\label{qkj:decay:esti:Dxf:L^2:1}
&\frac{1}{2}\frac{\d }{\d t} \left(\| \partial_x^\a f^\e\|_{L_{x,v}^2}^2+ \| \partial_x^\a\nabla_x\phi^\e\|_{L_{x}^2}^2\right)+\frac{\sigma_0}{\e^2}\|\partial_x^\a (\mathbf{I}-\mathbf{P})f^\e\|_{L_{x,v}^2(\nu)}^2  \\
\le&
 \underbrace{- \sum_{\substack{|\a'|<|\a|}}\left\langle \partial_x^{\a-\a'}\nabla_x \phi^{\e} \cdot \nabla_v \partial_x^{\a'}f^{\e},
 \nabla_x^\a f^\e \right \rangle_{L_{x,v}^2}}_{J_1} +\underbrace{\sum_{\substack{|\a'|\leq|\a|}}
 \left\langle \frac{v}{2}\cdot  \partial_x^{\a-\a'}\nabla_x \phi^{\e} \partial_x^{\a'} f^{\e},
 \partial_x^\a f^\e \right\rangle_{L_{x,v}^2}}_{J_2}
\\
& +
\underbrace{\frac{1}{\e}\left\langle
\Gamma(f^\e,\partial_x^\a f^\e)+\Gamma(\partial_x^\a f^\e,f^\e), \partial_x^\a f^\e \right\rangle_{L_{x,v}^2}}_{J_3}+\underbrace{\frac{1}{\e}\left\langle
\Gamma(\nabla_xf^\e, \nabla_xf^\e), \partial_x^{\a} f^\e \right\rangle_{L_{x,v}^2}}_{J_4}
,
\esp
\eal
where the last term  $J_4$  only exists   when $|\a|=2$.

Applying the same   estimate of
 $\mathcal{X}_{3}$ in  (\ref{qkj:esti:Dxxf:L^2:32}),
    the Young inequality and the definition \eqref{1:estimate:decay:10} of $\mathbf{X}_\ell(t)$,  we  derive
 \bals
 \bsp
 J_1
\lesssim\;&\sum_{\substack{|\a'|<|\a|}}\|a^\e\|_{H^2_x}
\Big(\frac{1}{\e}\|w\nabla_v \partial_x^{\a'}f^\e\|_{L_{x,v}^2}+
\e^2\|w_{\vartheta}\nabla_v \partial_x^{\a'} f^{\e}\|_{L^\infty_{x,v}}\Big)
\|\partial_x^{\a}(\mathbf{I}-\mathbf{P})f^\e\|_{L_{x,v}^2(\nu)}\\
\;&+
\sum_{\substack{|\a'|<|\a|}}\|a^\e\|_{H^2_x}\|\partial_x^{\a'}f^{\e}\|_{L_{{x}}^6L_{v}^2}
\|\nabla_v \partial_x^{\a}\mathbf{P} f^\e\|_{L_{x,v}^2}\\
\lesssim\;&\frac{\eta}{\e^2}\|\partial_x^\a (\mathbf{I}-\mathbf{P})f^\e\|_{L_{x,v}^2(\nu)}^2+
\|a^\e\|_{H^2_x}^2\left(\|w\nabla_v f^\e\|_{H_{x}^1L_{v}^2}^2
+\e^6\|w_{\vartheta}\nabla_v f^{\e}\|_{W^{1,\infty}_{x,v}}^2\right)\\
\;&+\|a^\e\|_{H^2_x}\|\nabla_xf^{\e}\|_{H_{{x}}^1L_{v}^2}
\| \partial_x^{\a}\mathbf{P} f^\e\|_{L_{x,v}^2}\\
\lesssim\;&\frac{\eta}{\e^2}\|\partial_x^\a (\mathbf{I}-\mathbf{P})f^\e\|_{L_{x,v}^2(\nu)}^2
+(1+t)^{-\frac{15}{4}}\mathbf{X}_\ell^\frac{3}{2}(t)
+(1+t)^{-\frac{15}{4}}\mathbf{X}_\ell^2(t)
 \esp
\eals
 and
\bals
 \bsp
 J_2
 \lesssim \;&
\sum_{\substack{|\a'|\leq|\a|}}\|a^\e\|_{H^2_x}
\Big(\frac{1}{\e}\|\partial_x^{\a'} f^\e\|_{L_{x,v}^2(\nu)}+\e^2\| w_\vartheta \partial_x^{\a'} f^\e\|_{L_{x,v}^\infty} \Big)
\|\partial_x^{\a}(\mathbf{I}-\mathbf{P})f^\e\|_{L_{x,v}^2(\nu)}\\
\;&+\Big(\|a^{\e}\|_{H^2_x}
\|\nabla_xf^\e\|_{H_{x}^1L_{v}^2}
+\|\nabla_x^2 \phi^{\e}\|_{H^1_x}
\|\partial_x^\a f^\e\|_{L_{x,v}^2}\Big)
\|\partial_x^\a\mathbf{P} f^\e\|_{L_{x,v}^2}\\
\lesssim\;&\frac{\eta}{\e^2}\|\partial_x^\a (\mathbf{I}-\mathbf{P})f^\e\|_{L_{x,v}^2(\nu)}^2
+(1+t)^{-\frac{15}{4}}\mathbf{X}_\ell^\frac{3}{2}(t)
+(1+t)^{-4}\mathbf{X}_\ell^2(t).
 \esp
\eals
For the terms $J_3$ and $J_4$,
applying the same procedure as the estimate on
 $\mathcal{X}_{4}$ in (\ref{qkj:esti:Dxxf:L^2:10}), and further using the \emph{a priori} assumption \eqref{eq:assumption:2},   the Young inequality and  \eqref{1:estimate:decay:10},  we   deduce
\bals
\bsp
J_3
\lesssim &\;
\frac{1}{\e}\| w_\vartheta f^\e\|_{L_{x,v}^\infty} \|\partial^\a_x(\mathbf{I}-\mathbf{P})f^\e\|_{L_{x,v}^2(\nu)}^2
+\frac{1}{\e}\|\mathbf{P}\partial^\a_x f^\e\|_{L_{x,v}^2}
\|f^\e\|_{L_x^\infty L_{v}^2(\nu)}\|\partial^\a_x(\mathbf{I}-\mathbf{P})f^\e\|_{L_{x,v}^2(\nu)}\\
\lesssim&\;
\frac{\delta\e^\frac{1}{4}}{\e^2}\|\partial_x^\a (\mathbf{I}-\mathbf{P})f^\e\|_{L_{x,v}^2(\nu)}^2
+\frac{\eta}{\e^2}\|\partial_x^\a (\mathbf{I}-\mathbf{P})f^\e\|_{L_{x,v}^2(\nu)}^2
+\|\partial^\a_xf^\e\|_{L_{x,v}^2}^2\|\nabla_xf^\e\|_{H_{x}^1 L_{v}^2}^2 \\
\lesssim&\;\frac{\delta\e^\frac{1}{4}+\eta}{\e^2}\|\partial_x^\a (\mathbf{I}-\mathbf{P})f^\e\|_{L_{x,v}^2(\nu)}^2
+(1+t)^{-5}\mathbf{X}_\ell^2(t)
\esp
\eals
and
\bals
\bsp
J_4 \lesssim&\;\frac{1}{\e}\| w_\vartheta \nabla_xf^\e\|_{L_{x,v}^\infty} \|\nabla_x(\mathbf{I}-\mathbf{P}) f^\e\|_{L_{x,v}^2(\nu)}\|\nabla_x^2(\mathbf{I}-\mathbf{P}) f^\e\|_{L_{x,v}^2(\nu)}
\\
&\;
+\frac{1}{\e}\|\nabla_x\mathbf{P} f^\e\|_{L_{x}^6 L_{v}^2}
 \|\nabla_xf^\e\|_{L_{x}^3 L_{v}^2(\nu)}
\|\nabla_x^2(\mathbf{I}-\mathbf{P}) f^\e\|_{L_{x,v}^2(\nu)}\\
\lesssim&\;
\frac{\delta}{\e^2}\|\nabla_x (\mathbf{I}-\mathbf{P})f^\e\|_{H_{x}^1L_{v}^2(\nu)}^2
+\frac{\eta}{\e^2}\|\nabla_x^2 (\mathbf{I}-\mathbf{P})f^\e\|_{L_{x,v}^2(\nu)}^2
+(1+t)^{-5}\mathbf{X}_\ell^2(t).
\esp
\eals
Therefore, we obtain from
substituting the estimates on $J_1\sim J_4$ into \eqref{qkj:decay:esti:Dxf:L^2:1},  choosing $\eta$ and $\delta$ small enough, and taking summation over $1\leq|\a|\leq 2$  that
\bals
\bsp
\;&\frac{1}{2}\sum_{1\leq|\a|\leq2}\frac{\d }{\d t}\left( \|\partial_x^\a f^\e\|_{L_{x,v}^2}^2 + \|\partial_x^\a\nabla_x\phi^\e\|_{L_{x}^2}^2 \right)+\sum_{1\leq|\a|\leq2}\frac{\sigma_0}{2\e^2}\|\partial_x^\a (\mathbf{I}-\mathbf{P})f^\e\|_{L_{x,v}^2(\nu)}^2\\
 \lesssim\;&(1+t)^{-\frac{15}{4}}\mathbf{X}_\ell^\frac{3}{2}(t)
+(1+t)^{-\frac{15}{4}}\mathbf{X}_\ell^2(t).
\esp
\eals

On the other hand, we multiply the  above inequality by $(1+t)^{\frac{5}{2}+p}$ and integrate
the resulting inequality with respect to $t$ over $[0,t]$ to yield
\bals
\bsp
\;&(1+t)^{\frac{5}{2}+p} \|\nabla_x  f^\e(t)\|_{H_{x}^1L_{v}^2}^2 + (1+t)^{\frac{5}{2}+p}\|\nabla_x^2\phi^\e(t)\|_{H_{x}^1}^2  \\
 \lesssim\;&\|\nabla_x  f^\e_0\|_{H_{x}^1L_{v}^2}^2 + \|\nabla_x^2\phi^\e_0\|_{H_{x}^1}^2+
\int_{0}^t (1+\tau)^{\frac{3}{2}+p}\left( \|\nabla_x  f^\e(\tau)\|_{H_{x}^1L_{v}^2}^2 + \|\nabla_x^2\phi^\e(\tau)\|_{H_{x}^1}^2\right)\d \tau\\
\;&
+\int_{0}^t (1+\tau)^{-\frac{5}{4}+p}\d  \tau \mathbf{X}_\ell^\frac{3}{2}(t)
+\int_{0}^t (1+\tau)^{-\frac{5}{4}+p}\d  \tau\mathbf{X}_\ell^2(t)\\
 \lesssim\;&\| f^\e_0\|_{H_{x}^1L_{v}^2}^2 +
\int_{0}^t (1+\tau)^{\frac{3}{2}+p}\mathbf{D}^\mathbf{s}_{\ell}(\tau)\d \tau
+ \mathbf{X}_\ell^\frac{3}{2}(t)
+\mathbf{X}_\ell^2(t),
\esp
\eals
which further implies
\bal
\bsp\label{qkj:esti:decay:Dxxf:L^2:11}
\;&\sup_{0\leq\tau\leq t}\left\{(1+\tau)^{\frac{5}{2}} \|\nabla_x f^\e(\tau)\|_{H_{x}^1L_{v}^2}^2 + (1+\tau)^{\frac{5}{2}} \|\nabla_x^2\phi^\e(\tau)\|_{H_{x}^1}^2  \right\} \\
 \lesssim\;&\kappa_0^2 +
 (1+t)^{-p}\int_{0}^t (1+\tau)^{\frac{3}{2}+p}\mathbf{D}^\mathbf{s}_{\ell}(\tau)\d \tau
+ \mathbf{X}_\ell^\frac{3}{2}(t)
+\mathbf{X}_\ell^2(t).
\esp
\eal

\emph{Step 3. Optimal time decay estimate on $\|h^\e(t)\|_{W_{x,v}^{2,\infty}}$.\;}
In fact, we  deduce from the interpolation inequality between H\"{o}lder space and $L^2_x$ space, the Riesz potential theory and the Young inequality   that
\bal
\bsp\label{eq:assumption:soft:1-2}
  \e(1+t)^{\frac{5}{4}}\|\nabla^3_x\phi^\e\|_{L^\infty_x}
  &\lesssim(1+t)^{\frac{5}{4}}\|\nabla^3_x\phi^\e\|_{L^2_x}^{\frac{1}{4}}
 \left (\e^{\frac{4}{3}}[\nabla^3_x\phi^\e]_{{C^{\frac{1}{2}}}}\right)^{\frac{3}{4}}\\
 & \lesssim(1+t)^{\frac{5}{4}}\left(\|\nabla_x f^\e \|_{H^{1}_{x}L^{2}_{v}}\right)^\frac{1}{2}
  \left(\e^2\|\nabla_x^2 f^\e \|_{L^{\infty}_{x,v}}\right)^\frac{1}{2}\\
  & \lesssim C_\eta(1+t)^{\frac{5}{4}}\|\nabla_x f^\e \|_{H^{1}_{x}L^{2}_{v}}+
  \eta \e^2(1+t)^{\frac{5}{4}}\|\nabla_x^2 h^\e \|_{L^{\infty}_{x,v}}.
\esp
\eal
Here we have used a fact derived from 
the Schauder  estimate and the  Sobolev embedding inequality
\bals
\bsp
  \e^{\frac{4}{3}}[\nabla^3_x\phi^\e]_{{C^{\frac{1}{2}}}}&\lesssim\e^{\frac{4}{3}}\|\nabla^4_x\phi^\e \|_{L^{6}_{x}}\lesssim \e^{\frac{4}{3}}\|\nabla^2_xa^\e \|_{L^{6}_{x}}\lesssim\|\nabla_x^2 f^\e \|_{L^{2}_{x,v}}^\frac{1}{3}
  \left(\e^2\|\nabla_x^2 f^\e \|_{L^{\infty}_{x,v}}\right)^\frac{2}{3}.
\esp
\eals
Then we   deduce   from \eqref{result:W2:wuqiong:1}   and \eqref{eq:assumption:soft:1-2} that
\beq
\bsp
\label{result:decay:W2:wuqiong:1}
 \;&\sup_{0\le\tau\le t}\Big\{\sum_{|\a|+|\b|\leq 1}\e^{1+\frac{2}{5}|\b|}(1+\tau)^{\frac{5}{2}-\frac{5}{4}|\b|}
\| \partial_{\b}^{\a}h ^\e(\tau)\|_{L_{x,v}^{\infty}}^2
 \Big\}
 \\
 \;&+\sup_{0\le\tau\le t}\Big\{\sum_{|\a|+|\b|= 2}\e^{3+\frac{2}{5}|\b|}(1+\tau)^{\frac{5}{2}-\frac{5}{4}|\b|}
\| \partial_{\b}^{\a}h ^\e(\tau)\|_{L_{x,v}^{\infty}}^2\Big\}\\
\lesssim \;&\e\| h^\e_0\|_{W^{2,\infty}_{x,v}}^2+
\sup_{0\leq\tau\leq t}\left\{(1+\tau)^{\frac{5}{2}} \|\nabla_x f^\e(\tau)\|_{H_{x}^1L_{v}^2}^2 \right\} + \sup_{0\leq\tau\leq t}\left\{ (1+\tau)^{\frac{5}{2}}\|\nabla_x^2\phi^\e(\tau)\|_{H_{x}^1}^2 \right\} \\
\;&+\e^{\frac{2}{5}}\sup_{0\le\tau\le t}\left\{(1+\tau)^{\frac{3}{2}} \mathbf{E}_\ell^\mathbf{s}(\tau)
\right\}
.
\esp
\eeq
In summary, by taking a suitable linear combination   \eqref{decay:soft:result:5}$\times M_1 +$\eqref{qkj:esti:decay:Dxxf:L^2:11}$\times M_2+$\eqref{result:decay:W2:wuqiong:1}
for some suitably large positive constants $M_1 \gg M_2>0$, we  conclude \eqref{decay:result:1}. This  completes the proof Proposition \ref{decay:es:nonlinear:result:1}.
\end{proof}

\subsection{Proof of the   Main Theorem \ref{mainth1}}
\hspace*{\fill}

With Proposition \ref{decay:es:nonlinear:result:1} in hand, we are in the position to complete the proof of Theorem
\ref{mainth1} and verify the {\emph{a priori}} assumptions  \eqref{eq:assumption:1}, \eqref{eq:assumption:2} and \eqref{energy-assumptition-soft}.
\begin{proof}[\textbf{Proof of Theorem \ref{mainth1}}] \
The proof of Theorem \ref{mainth1} consists of two steps. The global existence and time decay estimate of the VPB system  (\ref{eq:f})
uniformly in $\e \in (0,1]$ for soft potentials is presented in \emph{Step 1}, and  the justification
of the incompressible NSFP limit from the VPB system  (\ref{eq:f}) is verified in \emph{Step  2}.

\emph{Step 1. Global existence and time decay estimate.}

First of all, the local solvability of the VPB system (\ref{eq:f}) for soft potentials can be obtained in terms of the energy functional
$\interleave f^\e (t)\interleave_{\ell}^\mathbf{s}$ given by \eqref{def-energy-R-soft-sum} and the argument used in \cite{Guo2003-SOFT}.
The details are omitted for simplicity.

Next,  assume that such a local solution
$(f^\e, \nabla_x\phi^\e)$  has been extended to the time step $t = T$ for some $T > 0$, that is, $(f^\e, \nabla_x\phi^\e)$  is a
solution to  the VPB system (\ref{eq:f}) defined on
$[0,T]\times \mathbb{R}^3\times \mathbb{R}^3$. Then we derive from the
 estimate \eqref{decay:result:1}   that
\bal\label{energy:result:1}
\mathbf{X}_\ell(t)\leq C\mathbf{\kappa}_0^2,
\eal
for some constant $C$ independent of $\e$. Therefore,  the global  strong  solution $(f^\e, \nabla_x\phi^\e)$ to the VPB system  (\ref{eq:f})  for   soft potentials can be obtained with the help of  \eqref{energy:result:1} and a standard continuation argument as in \cite{guo2002cpam}. The details are omitted here  for brevity.

Finally, the \emph{a priori}
assumptions \eqref{eq:assumption:1}, \eqref{eq:assumption:2} and \eqref{energy-assumptition-soft} follow by combining \eqref{energy:result:1}, \eqref{eq:assumption:hard:1-1}, \eqref{eq:assumption:soft:1-2} and \eqref{1:estimate:decay:10}.

\emph{Step 2. \  Limits to the incompressible NSFP system.}

The approach is similar to that one in \cite{JL2019}, but here  we  provide a full proof for completeness.

\emph{Step 2.1. \  Limits from the global energy estimate.}

Based on \eqref{eq:theorem2} in Theorem \ref{mainth2},   the VPB system (\ref{eq:f})
admits a global solution $f^\e\in L^\infty(\bbR^+;H^2_{x,v})$ and $\nabla_x\phi^\e \in L^\infty(\bbR^+;H^2_{x})$ and the uniform global energy estimate \eqref{main-result:1},
there exists a positive constant $C$, independent of $\e$, such that
\bal
\bsp\label{limit:1}
&\sup_{t\ge 0} \left\{\| f^{\e}(t)\|^2_{H^2_{x,v}}+ \| \nabla_x\phi^{\e}(t)\|^2_{H^2_{x}}\right\}\lesssim  C,\\
&\int_0^\infty\| (\mathbf{I}-\mathbf{P}){f}^{\e}(\tau)\|^2_{H^2_{x,v}}\d \tau\leq C \e, \\
&   \int_0^\infty\| (\mathbf{I}-\mathbf{P}){f}^{\e}(\tau)\|^2_{H^2_{x}L^2_{v}(\nu)}\d \tau\leq C \e^2.
\esp
\eal
From the energy bound in \eqref{limit:1},   there exist $f\in L^\infty(\bbR^+;H^2_{x,v})$ and $\nabla_x\phi\in L^\infty(\bbR^+;H^2_{x})$  such that
\bal
\bsp\label{limit:3}
f^{\e} &\to  f \quad \qquad\;\text{~~weakly-}* ~{\text{for}~t\geq 0 }, \; \text{weakly in~} H^2_{x,v},\\
\nabla_x\phi^{\e}  &\to \nabla_x\phi\qquad \text{~~weakly-}* ~\text{for}~ t\geq 0 , \; \text{weakly in} ~H^2_{x},
\esp
\eal
as $\e \to 0$. We still employ the original notation of sequence to denote its subsequence for convenience,
although the limit may hold for some subsequence.
From the dissipation bound in \eqref{limit:1}, we deduce that
\bal
\bsp\label{limit:4}
&(\mathbf{I}-\mathbf{P})f^{\e} \to  0 \;\;\,\text{~~strongly} {\text{~in}~L^2(\bbR^+;H^2_{x,v})} \text{~~as~~} \e \to 0.
\esp
\eal
Then, the convergence in \eqref{limit:3} and \eqref{limit:4} lead to
\bals
\bsp
&(\mathbf{I}-\mathbf{P})f=0.
\esp
\eals
This implies that there exist functions $\rho,u,\theta\in L^\infty(\bbR^+;H^2_{x})$ such that
\bal
\bsp\label{limit:6}
f(t,x,v)=\Big(\rho(t,x)+v \cdot u(t,x) +\frac{|v|^2 -3}{2}\th(t,x)\Big)\sqrt{\mu}.
\esp
\eal

Next, we introduce the following fluid variables
\bal
\bsp\label{limit:7}
&\rho^\e:=\langle f^{\e},\sqrt{\mu}\rangle_{L^2_{v}},\;\;\;\; u^\e:=\langle f^{\e},v\sqrt{\mu}\rangle_{L^2_{v}},\;\;\;\; \theta^\e :=\Big\langle f^{\e}, \Big(\frac{|v|^2}{3}-1\Big)\sqrt{\mu}\Big\rangle_{L^2_{v}}.
\esp
\eal
Multiplying (\ref{eq:f}) by the collision
invariant $(1, v, \frac{|v|^2-3}{3})\sqrt{\mu}$ and integrating  over $\mathbb{R}^3_v$,  we derive that $\rho^\e$, $u^\e$ and $\th^\e$ obey the local conservation laws
\beq
 \left\{
\begin{array}{ll}\label{limit:8-1:1}
\displaystyle
\pt_t \rho^\e+\frac{1}{\e}\nabla_x\cdot u^\e=0,~\\[2mm]
\displaystyle\pt_t u^\e+\frac{1}{\e}\nabla_x(\rho^\e+\theta^\e)+\frac{1}{\e}\nabla_x\phi^\e+\frac{1}{\e}
\nabla_x\cdot \langle \hat{A}(v)\sqrt{\mu}, L(\mathbf{I}-\mathbf{P})f^{\e}\rangle_{L^2_{v}}=-\rho^\e \nabla_x\phi^{\e},~\\[2mm]
\displaystyle\pt_t \theta^\e+\frac{2}{3\e}\nabla_x\cdot u^\e+\frac{2}{3\e}\nabla_x\cdot\langle \hat{B}(v)\sqrt{\mu}, L(\mathbf{I}-\mathbf{P})f^{\e}\rangle_{L^2_{v}}=-\frac{2}{3}u^\e\cdot \nabla_{x}\phi^\e,
\end{array}
\right.
\eeq
where
\bals
\bsp
& A (v):=v\otimes v -\frac{|v|^2}{3}\mathbb{I}_{3\times 3},~\quad B (v):=v\Big(\frac{|v|^2}{2}-\frac{5}{2}\Big),\\
&\hat{A} (v):=L^{-1}A (v),~\qquad \qquad\;\hat{B} (v):=L^{-1}B(v).~
\esp
\eals
From (\ref{eq:f}) and the definition of $\rho^\e$ in \eqref{limit:7}, we also find the equation of $\nabla_x\phi^{\e}$
\bal
\label{limit:8-1:4}
&-\Delta _x\phi^{\e}=\rho^\e.
\eal

Furthermore,
via the definitions of $\rho^\e$, $u^\e$ and $\th^\e$ in \eqref{limit:7} and the uniform energy bound in \eqref{limit:1}, we
obtain
\bal\label{limit:9}
\sup_{t\ge 0}\Big\{\| \rho^\e(t)\|_{H^2_{x}}+ \| u^\e(t)\|_{H^2_{x}}+\| \theta^\e(t)\|_{H^2_{x}}\Big\}\le C.
\eal
From   \eqref{limit:3} and the limit function $f(t, x, v)$ given in \eqref{limit:6}, we thereby deduce the following convergences
\bal
\bsp\label{limit:10}
&\rho^{\e}=\langle f^{\e},\sqrt{\mu}\rangle_{L^2_{v}}\to \langle f,\sqrt{\mu}\rangle_{L^2_{v}}
=\rho,~\\
&u^{\e}=\langle f^{\e},v\sqrt{\mu}\rangle_{L^2_{v}}\to \langle f,v\sqrt{\mu}\rangle_{L^2_{v}}=u,~\\
&\theta^{\e} =\Big\langle f^{\e}, \Big(\frac{|v|^2}{3}-1\Big)\sqrt{\mu}\Big\rangle_{L^2_{v}}\to\Big\langle f, \Big(\frac{|v|^2}{3}-1\Big)\sqrt{\mu}\Big\rangle_{L^2_{v}}=\theta,
\esp
\eal
\text{weakly}-$* $ {\text{for} $t\geq 0$},  \text{weakly in} $H^2_{x}$ and strongly in
$H^1_{loc}(\mathbb{R}^3_x)$ as $\e \to 0.$

\medskip

\emph{Step 2.2. \  Convergences to  the limiting system.} \

Next, we deduce the incompressible NSFP system \eqref{INSFP}
from the local conservation laws \eqref{limit:8-1:1} and the convergences \eqref{limit:3}, \eqref{limit:4}
and \eqref{limit:10} obtained in the \emph{Step 2.1}.

\emph{{Step 2.2.1.} \  Incompressibility and Poisson equation.} \
From the  equation \eqref{limit:8-1:1}  and the energy uniform bound \eqref{limit:9}, it is easy to deduce
\bals
\bsp
&\nabla_x \cdot u^{\e}=-\e\pt_t \rho^{\e}\to 0
\esp
\eals
in the sense of distributions as $\e\to 0$. By using the first convergence in \eqref{limit:10}, we have
$\nabla_x \cdot u^\e $ converges in the sense of distributions to $\nabla_x \cdot u $ as $\e$ tends to zero.
By uniqueness of distribution limit, we obtain
\bal\label{limit:13}
\nabla_x \cdot u = 0.
\eal

The  $u^\e$-equation of \eqref{limit:8-1:1}  yields
\bals
&\nabla_x(\rho^\e+\theta^\e)+\nabla_x\phi^\e=-\e\pt_t u^\e
-  \nabla_x\cdot \langle\widehat{ {A}}(v)\sqrt{\mu}, L(\mathbf{I}-\mathbf{P})f^\e \rangle_{L^2_{v}}-\e\rho^\e \nabla_x\phi^{\e}.
\eals
The bound \eqref{limit:9} implies   $\e\pt_t u^\e \to 0$ in the sense of distribution as $\e\to 0$.
With the aid of the self-adjointness of $L$, we derive from the  H{\"{o}}lder inequality and the uniform dissipation bound
in \eqref{limit:1} that
\bals
\bsp
\int_0^\infty
\|\nabla_x\cdot \langle\widehat{ {A}}(v)\sqrt{\mu}, L(\mathbf{I}-\mathbf{P})f^\e\rangle_{L^2_{v}}\|_{H^1_x}^2\d t
=\;&\int_0^\infty
\|\nabla_x\cdot \langle{ {A}}(v)\sqrt{\mu}, (\mathbf{I}-\mathbf{P})f^\e\rangle_{L^2_{v}}\|_{H^1_x}^2\d t\\
\lesssim\;& \int_0^\infty
\| (\mathbf{I}-\mathbf{P})f^\e(t)\|_{H^2_{x}L^2_v}^2\d t\\
\lesssim\;&\e.
\esp
\eals
Thus,  $$
\nabla_x\cdot \langle \hat{A}(v)\sqrt{\mu}, L(\mathbf{I}-\mathbf{P})f^\e\rangle_{L^2_{v}}\to 0 \;\; \text{strongly in} \; L^2(\mathbb{R}^+; H^1(\mathbb{R}^3_x))\; \text{as} \;  \e\to 0.
$$

Moreover, the bounds \eqref{limit:1} and \eqref{limit:9} reveal that
\bals
\bsp
\sup_{t\ge 0}\left\{\|\e\rho^\e \nabla_x\phi^{\e}(t)\|_{H^{2}_x}\right\}
\lesssim\e\sup_{t\ge 0}\left\{\|\rho^\e(t)\|_{H^{2}_x}\| \nabla_x\phi^{\e}(t)\|_{H^{2}_x}\right\}
 \le C\e,
\esp
\eals
which means that
$$\e\rho^\e \nabla_x\phi^{\e}\to 0
\;\;\text{strongly in} \; L^\infty(\mathbb{R}^+, H^2(\mathbb{R}^3_x))\; \text{as} \;  \e\to 0.
$$
In summary, we have
\bals
\nabla_x(\rho^\e + \th^\e)+\nabla_x\phi^\e\rightarrow 0 \;\;\text{in the sense of distributions as } \; \e \rightarrow 0.
\eals
By using the convergences \eqref{limit:3} and \eqref{limit:10}, we have
$$
\nabla_x(\rho^\e + \th^\e)+\nabla_x\phi^\e \rightarrow \nabla_x(\rho + \th)+\nabla_x\phi\;\;\text{in the sense of distributions as } \; \e \rightarrow 0.
$$
By uniqueness of distribution limit, we get
\bal\label{limit:16-1}
\nabla_x(\rho + \th) +\nabla_x\phi=0.
\eal
Then, applying $\nabla_x\cdot$ to  \eqref{limit:16-1} and combining the resulting equation with \eqref{limit:8-1:4}, by the uniform energy bound \eqref{limit:1} and \eqref{limit:9}, we obtain  the Poisson equation
\bal\label{limit:16}
\Delta_x(\rho + \th)=-\Delta_x\phi=\rho.
\eal

\emph{{{Step 2.2.2.}} \  Equations of $u^\e$ and $\frac{3}{2}\theta^{\e}-\rho^{\e}$.} \
 Following the standard formal derivations of fluid dynamic
limits of the Boltzmann equation (see \cite{BGL93} for instance), we obtain
\beq
 \left\{
\begin{array}{ll}\label{A(v):weiguan:zuoyong}
\displaystyle\frac{1}{\e} \left\langle \hat{A}(v)\sqrt{\mu}, L(\mathbf{I}-\mathbf{P})f^{\e}\right\rangle_{L^2_{v}}= u^{\e}\otimes u^{\e}-\frac{|u^{\e}|^2}{3}\mathbb{I}_3-\lambda\Sigma(u^{\e})-R_{\e,A},\\[2mm]
\displaystyle\frac{1}{\e}\left\langle \hat{B}(v)\sqrt{\mu},\;\; L(\mathbf{I}-\mathbf{P})f^{\e}\right\rangle_{L_v^2}
=\frac{5}{2}u^{\e}\theta^{\e}-\frac{5}{2}\kappa\nabla_x\theta^{\e}-R_{\e,B},
\end{array}
\right.
\eeq
where
\bals
\Sigma(u^{\e}):=\;&\nabla_xu^{\e}+({\nabla_xu^{\e}})^{T}-\frac{2}{3}\nabla_x\cdot u^{\e}\mathbb{I}_{3\times3},\\
\lambda:=\;&\frac{1}{10}\left\langle {A}(v)\sqrt{\mu},\;\;\hat{A}(v)\sqrt{\mu} \right\rangle_{L_v^2},\\
 \kappa:=\;&\frac{2}{15}\left\langle {B}(v)\sqrt{\mu},\;\;\hat{B}(v)\sqrt{\mu} \right\rangle_{L_v^2},\\
R_{\e,\Xi}:=\;&\left\langle \hat{\Xi}(v)\sqrt{\mu},\;\;\e\partial_t f^{\e} \right\rangle_{L_v^2}
+\left\langle \hat{\Xi}(v)\sqrt{\mu},\;\;v\cdot\nabla_x(\mathbf{I}-\mathbf{P})f^{\e} \right\rangle_{L_v^2}\\
\;&+\Big\langle \hat{\Xi}(v)\sqrt{\mu},\;\; \e \nabla_x\phi^\e\cdot
\frac{\nabla_v \left({\sqrt{\mu}f^{\e}}\right)}{{\sqrt{\mu}}}\Big\rangle_{L_v^2}
-\left\langle \hat{\Xi}(v)\sqrt{\mu},\;\; \Gamma((\mathbf{I}-\mathbf{P})f^{\e},f^{\e})\right\rangle_{L_v^2}\\
\;&-\left\langle \hat{\Xi}(v)\sqrt{\mu},\;\; \Gamma(\mathbf{P}f^{\e},(\mathbf{I}-\mathbf{P})f^{\e})\right\rangle_{L_v^2},\quad \quad
\text{with}\; \Xi = A~ \text{or}~ B
.
\eals

For the vector field $u^\e$, we decompose
\bal\label{jixian:u:fenjie}
u^{\e}=\mathcal{P}u^{\e}+\mathcal{P}^{\perp} u^{\e},
\eal
where $\mathcal{P}:=I-\nabla_x\Delta_x^{-1}\nabla_x\cdot$ is the Leray projection operator and $\mathcal{P}^{\perp}={\mathrm{I}}-\mathcal{P}$.
Plugging the first equation of (\ref{A(v):weiguan:zuoyong}) into \eqref{limit:8-1:1} and utilizing \eqref{jixian:u:fenjie}, we have
\begin{equation}\label{jixian:u:zongjie}
  \pt_t\mathcal{P}u^{\e}+\mathcal{P}\text{div}_x(\mathcal{P}u^{\e}\otimes\mathcal{P}u^{\e})
  -\lambda\Delta_x\mathcal{P}u^{\e}=-\mathcal{P}(\rho^\e\nabla_x\phi^\e)+R_{\e,u},
\end{equation}
where $R_{\e,u}$ is given by
\begin{equation*}
  R_{\e,u}:=\mathcal{P}\nabla_x\cdot R_{\e,A}-\mathcal{P}\nabla_x\cdot
  (\mathcal{P}u^{\e}\otimes\mathcal{P}^\perp u^{\e}+\mathcal{P}^\perp u^{\e}\otimes\mathcal{P}u^{\e}+\mathcal{P}^\perp u^{\e}\otimes\mathcal{P}^\perp u^{\e}).
\end{equation*}

For the equation of  $\frac{3}{2}\theta^{\e}-\rho^{\e}$, we substitute  the second equation in (\ref{A(v):weiguan:zuoyong})  into
 \eqref{limit:8-1:1} to obtain
\begin{equation}\label{jixian:th:zongjie}
  \pt_t\Big(\frac{3}{2}\theta^{\e}-\rho^{\e}\Big)+\frac{5}{2}\nabla_x\cdot (\mathcal{P}u^{\e}\theta^{\e})-\frac{5}{2}\kappa\Delta_x\theta^{\e}=-u^\e\cdot \nabla_{x}\phi^\e-\frac{5}{2}\nabla_x\cdot (\mathcal{P}^\perp u^{\e}\theta^{\e}) +\nabla_x\cdot (R_{\e,B}).
\end{equation}

\emph{{Step 2.2.3.} \  Convergences of $\mathcal{P}u^{\e}$ and $\frac{3}{2}\theta^{\e}-\rho^{\e}$.} \
On   one hand,  applying $\mathcal{P}$ on the $u^\e$-equation in \eqref{limit:8-1:1} and H\"{o}lder inequality, we have
\begin{equation*}
  \begin{split}
\|\pt_t\mathcal{P}u^{\e}\|_{H_x^{1}}
=&\;\Big\|-\mathcal{P}( \rho^{\e}\nabla_x\phi^{\e})-\frac{1}{\e}\mathcal{P}\Big(\int_{\mathbb{R}^3}
A(v)\nabla_x(\mathbf{I}-\mathbf{P})f^{\e}\sqrt{\mu}\mathrm{d}v\Big)\Big\|_{H_{x}^{1}}\\
\lesssim&\; \|\nabla_x\phi^{\e}\|_{H_x^{1}}\|\nabla_x\rho^{\e}\|_{H_{x}^{1}}
+
\frac{1}{\e}\|(\mathbf{I}-\mathbf{P})f^{\e}\|_{H_x^2L_v^2(\nu)},
  \end{split}
\end{equation*}
which immediately infers from the uniform   bound  \eqref{limit:1}   that
\begin{equation}\label{jixian:u:shijianjie}
  \begin{split}
\|\pt_t\mathcal{P}u^{\e}\|_{L^2(0,T;H_x^{1})}
\lesssim&\; \sqrt{T}\|\nabla_x\phi^{\e}\|_{L^\infty(0,T;H_x^{1})}
\|\rho^{\e}\|_{L^{\infty}(0,T;H_x^2)}
+\frac{1}{\e}\|(\mathbf{I}-\mathbf{P})f^{\e}\|_{L^2(0,T;H_{x}^2L_{v}^2(\nu))}\\
\lesssim &\;C\sqrt{T}+C
  \end{split}
\end{equation}
for any $T>0$ and $0<\e\leq1$.
On the other hand, we directly derive from \eqref{limit:9} that
\begin{equation}\label{jixian:u:jie}
  \|\mathcal{P}u^{\e}\|_{L^\infty(0,T;H_x^{2})}\leq C
\end{equation}
for all $T>0$ and $0<\e\leq1$.
Noticing that
\begin{equation}\label{jixian:qianru}
 H^2(\mathbb{R}^3_x){\hookrightarrow\hookrightarrow} H_{loc}^{1}(\mathbb{R}^3_x)\overset{\text{continuous}}{\hookrightarrow}H_{loc}^{1}
 (\mathbb{R}^3_x),
\end{equation}
where the embedding of $H^2_x$ in $H^1_x$ is compact locally and the embedding of $H_{loc}^{1}$ in $H_{loc}^{1}$ is naturally continuous.
Then, employing Aubin--Lions--Simon Theorem \cite{BF13}, the uniform bounds (\ref{jixian:u:shijianjie}), (\ref{jixian:u:jie}) and
the embedding inequality  (\ref{jixian:qianru}), we deduce that there is a
$\tilde{u}\in L^\infty(\mathbb{R}^+;H^2(\mathbb{R}^3_x))\cap C(\mathbb{R}^+;H_{loc}^{1}(\mathbb{R}^3_x))$ such that
\begin{equation*}
\mathcal{P}u^{\e}\rightarrow\tilde{u} \;\; \text{strongly in}\;  C(0,T;H_{loc}^{1}(\mathbb{R}^3_x))\;\;\text{as}\;  \e \to 0 \;\;\text{for any } \; T>0.
\end{equation*}
Further, using \eqref{limit:10} and \eqref{limit:13},  we know that
$$
 \mathcal{P}u^{\e} \rightarrow \mathcal{P}u=u \;\;\text{in the sense of distributions} \; \text{as}\;  \e\to 0.
$$
As a result, we obtain
\begin{equation}\label{jixian:zongjie-u}
\mathcal{P}u^{\e}\rightarrow u  \;\; \text{strongly in}\;  C(0,T;H_{loc}^{1}(\mathbb{R}^3_x))\;\;\text{as}\;  \e \to 0 \;\;\text{for any } \; T>0,
\end{equation}
 where $u\in L^\infty(\mathbb{R}^+;H^2(\mathbb{R}^3_x))\cap C(\mathbb{R}^+;H_{loc}^{1}(\mathbb{R}^3_x))$.
Similarly, we  are able to deduce
\begin{align}\label{jixian:zongjie-u+}
 &\mathcal{P}^\perp u^{\e}\rightarrow0 \;\;\text{weakly-}* {\text{for}\; t\geq 0},\;  \text{weakly in}\; H^2(\mathbb{R}^3_x), \;
  \text{strongly in}\; H^1_{loc}(\mathbb{R}^3_x)\;\text{as}\; \e \to 0.
\end{align}

For the convergence of $\frac{3}{2}\theta^{\e}-\rho^{\e}$,
we deduce from \eqref{limit:8-1:1}   and~H\"{o}lder~inequality that
\bals
  \bsp
\Big\|\pt_t\Big(\frac{3}{2}\theta^{\e}-\rho^{\e}\Big)\Big\|_{H_x^{1}}
=&\;\Big\|- u^{\e} \cdot\nabla_{x}\phi^\e-
\frac{1}{2\e}\int_{\mathbb{R}^3}B(v)\cdot\nabla_x
(\mathbf{I}-\mathbf{P})f^{\e}\sqrt{\mu}\mathrm{d}v \Big\|_{H_x^{1}} \\
\lesssim&\; \|\nabla_{x}\phi^\e\|_{H_x^{1}}\|\nabla_x u^{\e}\|_{H_x^{1}}
+\frac{1}{\e}\|(\mathbf{I}-\mathbf{P})f^{\e}\|_{H_{x}^2L^2_v(\nu)},
  \esp
\eals
which implies from the uniform   bound \eqref{limit:1}   that
\begin{equation}\label{jixian:midu:wendu:shijianjie}
\begin{split}
\Big\|\pt_t\Big(\frac{3}{2}\theta^{\e}-\rho^{\e}\Big)\Big\|_{L^2(0,T;H_x^{1})}
\lesssim&\;\sqrt{T}\|\nabla_x\phi^{\e}\|_{L^\infty(0,T;H_x^{1})}
\|u^{\e}\|_{L^{\infty}(0,T;H_x^2)}+\Big(\int_0^\infty
[\mathbf{D}_{\ell}^\mathbf{s}(t)]^2\d t\Big)^\frac{1}{2}\\
\leq &\;C\sqrt{T}+C
\end{split}
\end{equation}
for any $T>0$ and $0<\e\leq1$.
On the other hand, we derive directly from \eqref{limit:9} that
\begin{equation}\label{jixian:midu:wendu:jie}
  \Big\|\frac{3}{2}\theta^{\e}-\rho^{\e}\Big\|_{L^\infty(0,T;H^{2}(\mathbb{R}^3_x))} \lesssim C
\end{equation}
for all $T>0$ and $0<\e\leq1$.
Thus, from Aubin--Lions--Simon Theorem \cite{BF13},
the bounds (\ref{jixian:midu:wendu:shijianjie}), (\ref{jixian:midu:wendu:jie}) and the embedding (\ref{jixian:qianru}),
we deduce that there is
$\tilde{\theta}\in L^\infty(\mathbb{R}^+;H^2(\mathbb{R}^3_x))\cap C(\mathbb{R}^+;H_{loc}^{1}(\mathbb{R}^3_x))$,
 such that
\begin{equation*}
\frac{3}{2}\theta^{\e}-\rho^{\e}\longrightarrow\tilde{\theta}
  \;\; \text{strongly in}\;  C(0,T; H_{loc}^{1}(\mathbb{R}^3_x))\;\;\text{as}\;  \e \to 0 \;\text{for any } \; T>0.
\end{equation*}
Combining  with the convergence \eqref{limit:10}, we deduce that
$\tilde{\theta}=\frac{3}{2}\theta-\rho,$ which implies
\begin{equation}\label{jixian:zongjie-th}
 \frac{3}{2}\theta^{\e}-\rho^{\e}\rightarrow\frac{3}{2}\theta-\rho
   \;\; \text{strongly in}\;  C(\mathbb{R}^+;H_{loc}^{1}(\mathbb{R}^3_x))\;\;\text{as}\;  \e \to 0,
\end{equation}
 where $\frac{3}{2}\theta-\rho\in L^\infty(\mathbb{R}^+;H^2(\mathbb{R}^3_x))\cap C(\mathbb{R}^+;H_{loc}^{1}(\mathbb{R}^3_x))$.

\emph{Step 2.2.4. \  Limit equation from \eqref{jixian:u:zongjie} and \eqref{jixian:th:zongjie}.}  \
For any $T>0$, choose a vector-valued test function $\psi(t,x)\in C^1(0,T;C_0^\infty(\mathbb{R}^3))$ with $\nabla_x\cdot\psi=0$, $\psi(0,x)=\psi_0(x)\in C_0^\infty(\mathbb{R}^3)$, and $\psi(t,x)=0$ for $t\geq T'$ with some $T'<T$.
Multiplying
(\ref{jixian:u:zongjie}) by $\psi(t,x)$ and integrating by parts over $(t,x)\in[0,T]\times\mathbb{R}^3$, we have
\begin{equation*}
\begin{split}
\int_0^T\int_{\mathbb{R}^3}\pt_t\mathcal{P}u^{\e}\cdot\psi\d x\d t
=&\;-\int_{\mathbb{R}^3}\mathcal{P}u^{\e}(0,x)\cdot\psi(0,x)\d x-\int_0^T\int_{\mathbb{R}^3}\mathcal{P}u^{\e}\cdot\pt_t\psi\d x\d t\\
=&\;-\int_{\mathbb{R}^3}\mathcal{P}\langle f^{\e}_{0},v\sqrt{\mu}\rangle_{L_v^2}\cdot\psi_0(x)\d x-\int_0^T\int_{\mathbb{R}^3}\mathcal{P}u^{\e}\cdot\pt_t\psi(t,x)\d x\d t.
\end{split}
\end{equation*}
From the initial condition \eqref{limit:initial:hard} in Theorem \ref{mainth2} and the convergence (\ref{jixian:zongjie-u}), we
deduce that
\begin{align*}
 \int_{\mathbb{R}^3}\mathcal{P}\langle f^{\e}_{0},v\sqrt{\mu}\rangle_{L_v^2}\cdot\psi_0(x)\d x
 \rightarrow
 &\int_{\mathbb{R}^3}\mathcal{P}\langle f_{0},v\sqrt{\mu}\rangle_{L_v^2}\cdot\psi_0(x)\d x
 =\int_{\mathbb{R}^3}\mathcal{P}u_{0}\cdot\psi_0(x)\d x
\end{align*}
and
\begin{align*}
  \int_0^T\int_{\mathbb{R}^3}\mathcal{P}u^{\e}\cdot\pt_t\psi(t,x)\d x\d t
  \rightarrow\int_0^T\int_{\mathbb{R}^3}u\cdot\pt_t\psi(t,x)\d x\d t
\end{align*}
as $\e\rightarrow0$.
As a consequence, we have
\begin{equation}\label{jixian:tu:}
  \int_0^T\int_{\mathbb{R}^3}\pt_t\mathcal{P}u^{\e}\cdot\psi(t,x)\d x\d t
  \rightarrow-\int_{\mathbb{R}^3}\mathcal{P}u_{0}\cdot\psi_0(x)\d x-\int_0^T\int_{\mathbb{R}^3}u\cdot\pt_t\psi(t,x)\d x\d t
\end{equation}
as $\e\rightarrow0$.
Further, with the aid of  (\ref{jixian:zongjie-u}), (\ref{jixian:zongjie-u+}) and (\ref{jixian:zongjie-th}), we  show the following convergences
\beq
 \left\{
\begin{array}{ll}\label{jixian:u1}
  \mathcal{P}\nabla_x\cdot(\mathcal{P}u^{\e}\otimes\mathcal{P}u^{\e})\rightarrow
  \mathcal{P}\nabla_x\cdot(u\otimes u) \quad\text{strongly in} \; C(\mathbb{R}^+;L_{loc}^{2}(\mathbb{R}^3_x)),\\[2mm]
  \lambda\Delta_x\mathcal{P}u^{\e}\rightarrow
 \lambda\Delta_x u \quad\qquad\qquad\qquad\quad\quad\quad\text{in the  sense of distributions},\\[2mm]
  \mathcal{P}(\rho^{\e}\nabla_x\phi^{\e})\rightarrow
  \mathcal{P}(\rho\nabla_x\phi)  \quad\quad\quad\quad\;\;\,\quad\quad\;\text{strongly in}\; C(\mathbb{R}^+;H_{loc}^{1}(\mathbb{R}^3_x)),\\[2mm]
 R_{\e,u}\rightarrow 0  \qquad\qquad\qquad\qquad\qquad\qquad\quad\quad\text{in the  sense of distributions},
\end{array}
\right.
\eeq
as $\e\rightarrow0$.
Therefore,  collecting (\ref{jixian:tu:}) and  (\ref{jixian:u1}),  and using \eqref{limit:16-1}, we find that $u$ satisfies
\begin{align}
 \begin{split}\label{limit:u}
 \pt_t u+\mathcal{P}\nabla_x\cdot(u\otimes u)-\lambda\Delta_x u&=\mathcal{P}( \rho \nabla_x\th )
 \end{split}
\end{align}
with the initial data $u(0,x)=\mathcal{P}u_{0}(x)$.

Next, we derive the  limit equation for $(\frac{3}{2}\theta-\rho)$ from \eqref{jixian:th:zongjie}.
Indeed, for any $T>0$, choose a test function $\xi(t,x)\in C^1(0,T;C_0^{\infty}(\mathbb{R}^3))$
 with $\xi(0,x)=\xi_0(x)\in C_0^{\infty}(\mathbb{R}^3)$ and $\xi(0,x)=0$ for $t\geq T'$
 with some $T'<T$. Then from the initial condition \eqref{limit:initial:hard} in Theorem \ref{mainth2} and the convergence (\ref{jixian:zongjie-th}), we
find that
\begin{align}\label{jixian:t:hongguan}
  \begin{split}
   &\int_{0}^{T}\int_{\mathbb{R}^3}\pt_t\Big(\frac{3}{2}\theta^{\e}-\rho^{\e}\Big) \xi(t,x) \d x \d t
   \\
    =&-\int_{\mathbb{R}^3}\Big\langle   f^{\e}_{0}, \;\;\Big[\frac{3}{2}\Big(\frac{|v|^2}{3}-1\Big)-1\Big]\sqrt{\mu} \Big\rangle_{L_v^2}\xi_0(x) \d x
   -\int_{0}^{T}\int_{\mathbb{R}^3} \Big(\frac{3}{2}\theta^{\e}-\rho^{\e}\Big)\pt_t\xi(t,x)  \d x \d t\\
  \to &-\int_{\mathbb{R}^3}(\frac{3}{5}\theta_{0}-\frac{2}{5}\rho_{0})\xi_0(x) \d x
    -\int_{0}^{T}\int_{\mathbb{R}^3}  \Big(\frac{3}{2}\theta-\rho\Big)\pt_t\xi(t,x)  \d x \d t
  \end{split}
\end{align}
as $\e\rightarrow0$.
Further, with the help of (\ref{jixian:zongjie-u}), (\ref{jixian:zongjie-u+}) and (\ref{jixian:zongjie-th}), we
 establish the following convergences
\beq
 \left\{
\begin{array}{ll}\label{jixian:hongguan:1}
\displaystyle\frac{5}{2}\nabla_x\cdot(\mathcal{P}u^{\e}\theta^{\e})\rightarrow
 \frac{5}{2} \nabla_x\cdot(u\theta)\qquad\qquad\quad\quad\,\,\text{strongly in} \; C(\mathbb{R}^+;L_{loc}^{2}(\mathbb{R}^3_x)),
  \\[2mm]
  \displaystyle\frac{5}{2}\kappa\Delta_x\theta^{\e}\rightarrow\frac{5}{2}\kappa\Delta_x\theta
  \qquad\qquad\;\qquad\quad\quad\quad\quad\;\,\,\,\text{in the sense of distributions},\\[2mm]
  \displaystyle u^{\e}\cdot\nabla_x\phi^{\e}\rightarrow
  u\cdot\nabla_x\phi \quad\quad\quad\quad\qquad\qquad\quad\quad\;\;\text{strongly in}\; C(\mathbb{R}^+;H_{loc}^{1}(\mathbb{R}^3_x)),\\[2mm]
 \displaystyle\frac{5}{2}\nabla_x\cdot (\mathcal{P}^\perp u^{\e}\theta^{\e}) +\nabla_x\cdot (R_{\e,B})\rightarrow 0  \qquad\quad\quad\text{in the sense of distributions},
\end{array}
\right.
\eeq
as $\e\rightarrow0$.
Therefore, collecting the limits (\ref{jixian:t:hongguan}), ~(\ref{jixian:hongguan:1}) and using \eqref{limit:13} and \eqref{limit:16-1}, we find
that $(\frac{3}{2}\theta-\rho)$  satisfies
\begin{align}
 \begin{split}\label{limit:th}
 \pt_t\Big(\dfrac{3}{2}\th-\rho\Big)+u\cdot\nabla_x\Big(\dfrac{3}{2}\th-\rho\Big)-\dfrac{5}{2}\kappa
\Delta_x\th=0
 \end{split}
\end{align}
with the initial data
$\th(0,x)=\th_0(x)$ and $\rho(0,x)=\rho_0(x)$.

In summary, combining the limit equations \eqref{limit:13}, \eqref{limit:16}, \eqref{limit:u} and \eqref{limit:th}, we  conclude that
\bals
(\rho, u, \th, \nabla_x\phi)\in
L^\infty(\mathbb{R}^+; H^{2}(\bbR^3_x))
\cap C(\mathbb{R}^+; H^{1}_{loc}(\bbR^3_x))
 \eals
and its forms a solution to the incompressible NSFP system \eqref{INSFP}.
The proof of    Theorem \ref{mainth1} is completed.
\end{proof}

%

\subsection{Time Decay Estimate  of Energy  for Hard Potentials}
\hspace*{\fill}

To close the energy estimate under the {\em{a priori}} assumptions \eqref{eq:assumption:1}, \eqref{eq:assumption:2} and \eqref{energy-assumptition-hard}, we turn to deduce the time decay estimate of $\mathbf{E}^\mathbf{h}(t)$  and  $\widetilde{\mathbf{E}}^\mathbf{h}(t)$. For this,
denote
\bal
\bsp\label{1:hard:estimate:decay:10}
\mathbf{\Lambda}(t):=\;&\sup_{0\le\tau\le t}\Big\{(1+\tau)^{\frac{3}{2}} \mathbf{E}^\mathbf{h}(\tau) \Big\}
 +\sup_{0\le\tau\le t}\Big\{(1+\tau)^{\frac{5}{2}} \widetilde{\mathbf{E}}^\mathbf{h}(\tau)
\Big\}\\
\;&+\sup_{0\le\tau\le t}\Big\{\e^3(1+\tau)^{\frac{5}{2}} \| h^\e(\tau)\|_{W_{x,v}^{2,\infty}}^2 \Big\}.
\esp
\eal
Then we have the  result about the VPB system \eqref{eq:f}  for hard potentials $0\leq \gamma\leq 1$.

\begin{proposition}\label{decay:es:nonlinear:result:hard:1}
Let $0\leq\gamma\leq 1.$
Assume that
$\iint_{\mathbb{R}^3\times\mathbb{R}^3}\sqrt{\mu}f^\e_0\d x\d v=0$
and   the \emph{a priori} assumptions  \eqref{eq:assumption:1}, \eqref{eq:assumption:2} and \eqref{energy-assumptition-hard} hold true
for some small $\delta>0$.
Then,  any strong solution $(f^\e, \nabla_x\phi^\e)$ to the VPB system \eqref{eq:f} defined on $0\leq t\leq T$ with $0 <T \leq \infty$ satisfies
\bal\label{result-decay-hard}
\mathbf{\Lambda}(t)\leq C{\mathbf{\widetilde{\kappa}}}_0^2
+C\mathbf{\Lambda}^2(t)\;\;\; \text{on}\;\; 0\leq t\leq T,
\eal
where $C$ is a positive constant independent of $T$ and $\widetilde{\kappa}_0$ is defined by
\beqs
\bsp
{\widetilde{{\kappa}}}_0:=\;&
\|  f^\e_0\|_{H^2_{x,v}}+\| w f^\e_0\|_{H^1_{x}L^2_{v}}
 +\big\|\big(1+|x|\big)f^\e_0\big\|_{L^2_vL^1_x}.
\esp
\eeqs
\end{proposition}

To prove Proposition \ref{decay:es:nonlinear:result:hard:1}, we should first find  the optimal
 temporal decay estimate for the linearized  system \eqref{linear:esti:f:decay:1}.

\begin{lemma}\label{result:linear:estimate:decay:hard}
Let $0\leq\gamma\leq 1.$
Assume
$\iint_{\mathbb{R}^3\times\mathbb{R}^3}\sqrt{\mu}f^\e_0\d x\d v=0$
and
\beqs
\||x|a_0^\e\|_{L^1_x}+\|f^\e_0\|_{L^2_vL^1_x}
+\| \partial^\a_x f^\e_0\|_{L^2_{x,v}}<\infty.
\eeqs
Then the solution to the linearized VPB system \eqref{linear:esti:f:decay:1}  satisfies
\bal
\bsp\label{result:linear:estimate:decay-hard:3}
 & \| \partial^\a_x e^{{t B^{\e}}}f_0^\e\|_{L_{x,v}^2}^2
 + \| \partial^\a_x\nabla_x\Delta_x^{-1} \mathbf{P}_0e^{{t B^{\e}}}f_0^\e\|_{L_{x}^2}^2
 \\
\lesssim\;& (1+t)^{-\sigma_{|\a|}}
\Big(
\||x|a_0^\e\|_{L^1_x}^2+\|f^\e_0\|_{L^2_vL^1_x}^2
+\| \partial^\a_x f^\e_0\|_{L^2_{x,v}}^2
\Big)\\
\;&+\e^{2}\int_0^t (1+t-\tau)^{-2(\frac{3}{4}+\frac{|\a|}{2})}\Big( \|\nu^{-\frac{1}{2}} h^\e(\tau)\|_{L^2_vL^1_x}^2
+\| \nu^{-\frac{1}{2}}\partial^\a_x  h^\e(\tau)\|_{L^2_{x,v} }^2\Big)\d \tau
\esp
\eal
 for any $t\geq 0$.
\end{lemma}

\begin{proof} \   First of all, as in the soft potential case, it is easy to verify  that
there is a time frequency functional ${E}(\hat{f^{\e}})$ satisfying
\bals
\bsp
{E}(\hat{f^{\e}})\thicksim\;&\|\hat{f^{\e}}\|_{L^2_v}^2+
\frac{|\hat{a}^\e|^2}{|k|^2}
\esp
\eals
  for any $t>0$ and $k\in \mathbb{R}^3$, and there also holds
\bals
\bsp
\frac{\d}{\d t}  {E}(\hat{f^{\e}})+\frac{\sigma_0 |k|^2}{1+|k|^2}{E}(\hat{f^{\e}})
\lesssim
\e^{2}\|{\nu}^{-\frac{1}{2}}\hat{h}^\e\|_{L^2_v}^2.
\esp
\eals
Then adopting the same method as in \cite{DYZ2002}  by  the
high-low frequency decomposition, we   derive (\ref{result:linear:estimate:decay-hard:3})
under the condition $\iint_{\mathbb{R}^3}f^\e_0\sqrt{\mu}\d x\d v=0$. The details are omitted here for brevity.
\end{proof}

Secondly,   we need an additional lemma concerning the time decay estimates on the macroscopic quantity $\|\mathbf{P}f^\e\|_{H^1_{x}L^2_{v}}$ and $\|\nabla_x\phi^\e\|_{H^1_{x}}$ in terms of the initial data and $\mathbf{\Lambda}(t)$.

\begin{lemma}\label{decay:es:nonlinear:hard}
Under the assumptions of Proposition \ref{decay:es:nonlinear:result:hard:1}, there hold
\bal\label{hongguan:decay:result:hard:1}
\|\mathbf{P}f^\e(t)\|_{L^2_{x,v}}^2+\|\nabla_x\phi^\e(t)\|_{L^2_{x}}^2
\lesssim (1+t)^{-\frac{3}{2}}\left({\mathbf{\widetilde{\kappa}}}_0^2
+\left[\mathbf{\Lambda}(t)\right]^2\right)
\eal
and
\bal\label{hongguan:decay:result:hard:2}
\|\nabla_x\mathbf{P}f^\e(t)\|_{L^2_{x,v}}^2+\|\nabla_x^2\phi^\e(t)\|_{L^2_{x}}^2
\lesssim (1+t)^{-\frac{5}{2}}\left({\mathbf{\widetilde{\kappa}}}_0^2
+\left[\mathbf{\Lambda}(t)\right]^2\right)
\eal
for any $0\leq t\leq T.$
\end{lemma}
\begin{proof} \
Applying Lemma \ref{result:linear:estimate:decay:hard} to \eqref{hongguan:decay:equation:1}, we have
\bal
\bsp\label{hongguan:decay:proof:hard:1}
 & \| \mathbf{P}f^\e(t)\|_{L_{x,v}^2}^2
 + \| \nabla_x\phi^\e(t)\|_{L_{x}^2}^2
\\
\lesssim\;& (1+t)^{-\frac{3}{2}}
\left(
\||x|a_0^\e\|_{L^1_x}^2+\|f^\e_0\|_{L^2_vL^1_x}^2
+\|  f^\e_0\|_{L^2_{x,v}}^2
\right)\\
\;&+\e^{2}\int_0^t (1+t-\tau)^{-\frac{3}{2}}\left( \|\nu^{-\frac{1}{2}} h_1^\e(\tau)\|_{L^2_vL^1_x}^2
+\| \nu^{-\frac{1}{2}}  h_1^\e(\tau)\|_{L^2_{x,v} }^2\right)\d \tau\\
\;&+\left[\int_0^t (1+t-\tau)^{-\frac{3}{4}}\left( \| h_2^\e(\tau)\|_{L^2_vL^1_x}
+\|  h_2^\e(\tau)\|_{L^2_{x,v} }\right)\d \tau\right]^2.
\esp
\eal
On  one hand, we derive from the property of $\Gamma$ and the  H\"{o}lder inequality that
\bal
\bsp\label{hongguan:decay:proof:hard:2}
\e\|  \nu^{-\frac{1}{2}}h_1^\e\|_{L^2_vL^1_x}
+\e\|   \nu^{-\frac{1}{2}}h_1^\e\|_{L^2_{x,v}}
\lesssim \;&\left\|\|\nu f^\e\|_{L^2_v} \|f^\e\|_{L^2_v}\right\|_{L^1_x}+\left\|\|\nu f^\e\|_{L^2_v} \|f^\e\|_{L^2_v}\right\|_{L^2_x}\\
\lesssim\;& \|w(v)f^\e\|_{L^2_{x,v}} \|f^\e\|_{L^2_{x,v}}
+\|w(v)f^\e\|_{L^2_{x,v}} \|f^\e\|_{H^2_{x}L^2_v}\\
\lesssim \;&(1+t)^{-\frac{3}{2}}\mathbf{\Lambda}(t).
\esp
\mathcal{}\eal
On the other hand, from the expression \eqref{hongguan:decay:equation:2} of $h_2^\e$  and exploiting the H\"{o}lder inequality and the definition \eqref{1:hard:estimate:decay:10} of $\mathbf{\Lambda}(t)$, we find
\bal
\bsp\label{hongguan:decay:proof:hard:3}
\;&\| h_2^\e\|_{L^2_vL^1_x}
+\|  h_2^\e\|_{L^2_{x,v}}\\
\lesssim \;&
\|\nabla_x \phi^\e\|_{L^2_x}\left(\|\nabla_v f^\e\|_{L^2_{x,v}}+\|w f^\e\|_{L^2_{x,v}}\right)
+\|\nabla_x \phi^\e\|_{L^\infty_x}\left(\|\nabla_v f^\e\|_{L^2_{x,v}}+\|w f^\e\|_{L^2_{x,v}}\right)
\\
\lesssim \;&
\mathbf{E}^\mathbf{h}(t)
+ \left[\mathbf{E}^\mathbf{h}(t)\right]^\frac{1}{2}
\big[\widetilde{\mathbf{E}}^\mathbf{h}(t)\big]^\frac{1}{2}\\
\lesssim \;&(1+t)^{-\frac{3}{2}}\mathbf{\Lambda}(t).
\esp
\eal
Therefore, substituting  \eqref{hongguan:decay:proof:hard:2} and \eqref{hongguan:decay:proof:hard:3} into \eqref{hongguan:decay:proof:hard:1}
and using the basic inequalities
\bals
\int_0^t(1+t-\tau)^{-\frac{3}{4}}(1+\tau)^{-\frac{3}{2}}\d \tau \lesssim(1+t)^{-\frac{3}{4}},\quad
\int_0^t(1+t-\tau)^{-\frac{3}{2}}(1+\tau)^{-3}\d \tau \lesssim(1+t)^{-\frac{3}{2}},
\eals
we  derive \eqref{hongguan:decay:result:hard:1} immediately.

Similar to that of deducing \eqref{hongguan:decay:proof:hard:1} from
Lemma \ref{result:linear:estimate:decay:hard}, we derive from \eqref{hongguan:decay:equation:1} that
\bal
\bsp\label{hongguan:decay:proof:hard:5}
 & \|\nabla_x \mathbf{P}f^\e(t)\|_{L_{x,v}^2}^2
 + \| \nabla_x^2\phi^\e(t)\|_{L_{x}^2}^2
\\
\lesssim\;& (1+t)^{-\frac{5}{2}}
\left(
\||x|a_0^\e\|_{L^1_x}^2+\|f^\e_0\|_{L^2_vL^1_x}^2
+\| \nabla_x f^\e_0\|_{L^2_{x,v}}^2
\right)\\
\;&+\e^{2}\int_0^t (1+t-\tau)^{-\frac{5}{2}}\Big( \|\nu^{-\frac{1}{2}} h_1^\e(\tau)\|_{L^2_vL^1_x}^2
+\| \nu^{-\frac{1}{2}} \nabla_x h_1^\e(\tau)\|_{L^2_{x,v} }^2\Big)\d \tau\\
\;&+\Big[\int_0^t (1+t-\tau)^{-\frac{5}{4}}\Big( \| h_2^\e(\tau)\|_{L^2_vL^1_x}
+\| \nabla_x h_2^\e(\tau)\|_{L^2_{x,v} }\Big)\d \tau\Big]^2.
\esp
\eal
Adopting a similar technique as   the estimate of \eqref{hongguan:decay:proof:hard:2} and \eqref{hongguan:decay:proof:hard:3}, we have
\bals
\bsp
\;& \e^{2}\| h_1^\e\|_{L^2_vL^1_x}^2
+\e^{2}\| \nabla_x h_1^\e\|_{L^2_{x,v}}^2\\
\lesssim\;& \|w(v)f^\e\|_{L^2_{x,v}}^2 \|f^\e\|_{L^2_{x,v}}^2
+\|w(v)\nabla_xf^\e\|_{L^2_{x,v}}^2 \|f^\e\|_{L^\infty_{x}L^2_v}^2
+\|w(v)f^\e\|_{L^4_{x}L^2_v}^2 \|\nabla_xf^\e\|_{L^4_{x}L^2_v}^2\\
\lesssim \;&(1+t)^{-3}\left[\mathbf{\Lambda}(t)\right]^2,
\esp
\eals
and
\bals
\bsp
\;&\| h_2^\e\|_{L^2_vL^1_x}
+\| \nabla_x h_2^\e\|_{L^2_{x,v}}\\
\lesssim \;&
\|\nabla_x \phi^\e\|_{L^2_x}\left(\|\nabla_v f^\e\|_{L^2_{x,v}}+\|w(v) f^\e\|_{L^2_{x,v}}\right)
+\|\nabla_x ^2\phi^\e\|_{L^\infty_x}\left(\|\nabla_v f^\e\|_{L^2_{x,v}}+\|w(v) f^\e\|_{L^2_{x,v}}\right)
\\
\;&+\|\nabla_x \phi^\e\|_{L^\infty_x}\left(\|\nabla_x\nabla_v f^\e\|_{L^2_{x,v}}+\|w(v) \nabla_xf^\e\|_{L^2_{x,v}}\right)\\
\lesssim \;&(1+t)^{-\frac{3}{2}}\mathbf{\Lambda}(t).
\esp
\eals
Thus, inserting the above  two estimates  into \eqref{hongguan:decay:proof:hard:5} and using  the basic inequalities
\bals
\int_0^t(1+t-\tau)^{-\frac{5}{4}}(1+\tau)^{-\frac{3}{2}}\d \tau \lesssim(1+t)^{-\frac{5}{4}},\quad
\int_0^t(1+t-\tau)^{-\frac{5}{2}}(1+\tau)^{-3}\d \tau \lesssim(1+t)^{-\frac{5}{2}},
\eals
we  conclude \eqref{hongguan:decay:result:hard:2}.
This completes the proof of Lemma \ref{decay:es:nonlinear:hard}.
\end{proof}

Finally, with the aid of Lemma \ref{decay:es:nonlinear:hard}, we  give the proof of Proposition \ref{decay:es:nonlinear:result:hard:1}.
\begin{proof}[\textbf{Proof of Proposition \ref{decay:es:nonlinear:result:hard:1}}] \
Firstly,  we deduce the optimal time decay estimate on $\mathbf{E}^\mathbf{h}(t)$.
In fact, from \eqref{energy-nonlinear-result-3} in Proposition \ref{main-weighted-energy-estimate-2} and the \emph{a priori} assumption \eqref{eq:assumption:2}, we have
  \bal\label{main-result:hard:1}
  \frac{\d}{\d t}\mathbf{E}^\mathbf{h}(t)+\mathbf{D}^\mathbf{h}(t)
  \lesssim0.
  \eal
Comparing the expressions of $\mathbf{E}^\mathbf{h}(t)$ and $\mathbf{D}^\mathbf{h}(t)$ in \eqref{def-energy-R-hard-3} and  \eqref{def-energy-R-hard}, we find that
  \bal\label{main-result:hard:2}
 \mathbf{D}^\mathbf{h}(t)+\|\mathbf{P}f^\e(t)\|_{L^2_{x,v}}^2
 +\|\nabla_x\phi^\e(t)\|_{L^2_x}^2\gtrsim\mathbf{E}^\mathbf{h}(t).
  \eal
Then using the inequalities \eqref{main-result:hard:2} and \eqref{main-result:hard:1}, we have
  \bals
  \frac{\d}{\d t}\mathbf{E}^\mathbf{h}(t)+\mathbf{E}^\mathbf{h}(t)
  \lesssim \|\mathbf{P}f^\e(t)\|_{L^2_{x,v}}^2+\|\nabla_x\phi^\e(t)\|_{L^2_x}^2.
  \eals
Further, applying the Gr\"{o}nwall inequality to the above inequality  and utilizing \eqref{hongguan:decay:result:hard:1},  we have
\bals
  \mathbf{E}^\mathbf{h}(t)
  \lesssim (1+t)^{-\frac{3}{2}}\left(\widetilde{{\mathbf{\kappa}}}_0
+\mathbf{\Lambda}^2(t)\right).
  \eals
This means
\bal\label{main-result:hard:4}
  \sup_{0\leq\tau \leq t}\left\{(1+\tau)^{\frac{3}{2}}\mathbf{E}^\mathbf{h}(\tau)\right\}
  \lesssim {\widetilde{\mathbf{\kappa}}_0}
+\mathbf{\Lambda}^2(t).
  \eal

Secondly, we turn to estimate the second term in \eqref{1:hard:estimate:decay:10}.
To this end,
comparing the expressions of $\widetilde{\mathbf{E}}^\mathbf{h}(t)$ and $\mathbf{D}^\mathbf{h}(t)$ in \eqref{def-energy-R-hard-2} and  \eqref{def-energy-R-hard-3}, we find that
  \bals
 \mathbf{D}^\mathbf{h}(t)\gtrsim\widetilde{\mathbf{E}}^\mathbf{h}(t)
 .
  \eals
Then  employing this inequality, \eqref{eq:assumption:2} and  \eqref{energy-nonlinear-result-4}, we get
  \bals
  \frac{\d}{\d t}\widetilde{\mathbf{E}}^\mathbf{h}(t)+\widetilde{\mathbf{E}}^\mathbf{h}(t)
  \lesssim \|\nabla_x\mathbf{P}f^\e(t)\|_{L^2_{x,v}}^2.
  \eals
Moreover, applying the Gr\"{o}nwall inequality to the  inequality above  and utilizing \eqref{hongguan:decay:result:hard:2}, we have
\bal\label{main-result:hard:6}
  \sup_{0\leq\tau \leq t}\left\{(1+\tau)^{\frac{5}{2}}\widetilde{\mathbf{E}}^\mathbf{h}(\tau)\right\}
  \lesssim {\mathbf{\widetilde{\kappa}}}_0
+\mathbf{\Lambda}^2(t).
  \eal

Finally, 
we  derive from  Proposition \ref{result:W2:wuqiong-hard} and \eqref{eq:assumption:soft:1-2} that
\beq
\bsp
\label{result:W2:wuqiong:hard:1:decay}
\;&\sup_{0\leq \tau\leq t}\Big\{\e(1+\tau)^{\frac{5}{2}}
\| h ^\e(\tau)\|_{W_{x,v}^{1,\infty}}^2\Big\}
+\sup_{0\leq \tau\leq t}\Big\{\sum_{|\a|+|\b|= 2}\e^{3}(1+\tau)^{\frac{5}{2}}
\| \partial_{\b}^{\a}h ^\e(\tau)\|_{L_{x,v}^{\infty}}^2\Big\}\\
\lesssim \;&\e\| h^\e_0\|_{W^{2,\infty}_{x,v}}^2
+\sup_{0\le\tau\le t}\Big\{(1+\tau)^{\frac{5}{2}} \widetilde{\mathbf{E}}^\mathbf{h}(\tau)\Big\}.
\esp
\eeq

In summary,  taking a suitable linear combination of \eqref{main-result:hard:4}, \eqref{main-result:hard:6}
and \eqref{result:W2:wuqiong:hard:1:decay}, we prove \eqref{result-decay-hard}. This  completes the proof
of Proposition \ref{decay:es:nonlinear:result:hard:1}.
\end{proof}

\subsection{Proof of the  Main Theorem \ref{mainth2}}
\hspace*{\fill}

With
 Proposition \ref{decay:es:nonlinear:result:hard:1} in hand, we are ready to give the proof of Theorem \ref{mainth2} and verify the {\emph{a priori}} assumptions  \eqref{eq:assumption:1}, \eqref{eq:assumption:2} and \eqref{energy-assumptition-hard}.
\begin{proof}[\textbf{Proof of Theorem \ref{mainth2}}] \
The proof of Theorem \ref{mainth2} follows from a similar fasion as  the proof of Theorem \ref{mainth1}, but the hard potential case here needs some different treatments.

In fact,
assume  that $\eta$ and $\widetilde{\kappa}_0>0$ are chosen sufficiently small constants,
we derive from the
 estimate \eqref{result-decay-hard} that
\bal\label{energy:result:hard:1}
\mathbf{\Lambda}(t)\leq C\mathbf{\widetilde{\kappa}}_0^2.
\eal
Then,  the local solvability  and the global existence of strong  solution $(f^\e, \nabla_x\phi^\e)$ to   the VPB system  (\ref{eq:f})  for    hard potentials can be obtained with the help of the energy functional
$\interleave f^\e (t)\interleave^\mathbf{h}$ given by \eqref{def-energy-R-hard-sum},  \eqref{energy:result:hard:1} and the argument used in \cite{ guo2002cpam},
and the details are omitted for simplicity.
Moreover, combining  \eqref{eq:assumption:hard:1-1}, \eqref{1:hard:estimate:decay:10} and \eqref{energy:result:hard:1}, we find that
the \emph{a priori}
assumptions  \eqref{eq:assumption:1}, \eqref{eq:assumption:2} and \eqref{energy-assumptition-hard}  can be closed.

Next, we derive the incompressible NSFP limit from VPB system  (\ref{eq:f}) for hard potentials.
To this end, based on a uniform global energy estimate \eqref{main-result:hard:1},
there exists a positive constant $C$, independent of $\e$, such that
\bal\label{limit:hard:1}
\sup_{t\ge 0} \left\{\| f^{\e}(t)\|^2_{H^2_{x,v}}+ \| \nabla_x\phi^{\e}(t)\|^2_{H^2_{x}}\right\}\leq  C
\eal
and
\bal\label{limit:hard:2}
\int_0^\infty\| (\mathbf{I}-\mathbf{P}){f}^{\e}(\tau)\|^2_{H^2_{x,v}}\d \tau\lesssim\int_0^\infty\| (\mathbf{I}-\mathbf{P}){f}^{\e}(\tau)\|^2_{H^2_{x,v}(\nu)}\d \tau\lesssim C \e^2.
\eal
The uniform energy bound \eqref{limit:hard:1} and the dissipation bound \eqref{limit:hard:2}
allow  us to  obtain the incompressible NSFP  limit
using exactly the same methods as the proof of Theorem \ref{mainth1}. The details are omitted here
 for brevity. The proof of  Theorem \ref{mainth2} is thus completed.
\end{proof}

\noindent\textbf{Acknowledgements.}
The research is supported by NSFC under the grant number 12271179, NSFC key project under the grant number 11831003,
Basic and Applied Basic Research Project of Guangdong under the grant number 2022A1515012097,
and Basic and Applied Basic Research Project of Guangzhou under the grant number SL2022A04J01496.


\begin{thebibliography}{9}
\bibitem{AS2019}  Ars\'{e}nio, D.; Saint-Raymond, L. {\it From the Vlasov--Maxwell--Boltzmann System to Incompressible Viscous
Electro-Magneto-Hydrodynamics, Monographs in Mathematics.} European Mathematical Society, Z\"urich, 2019.
\bibitem{BGL91} Bardos, C.; Golse, F.; Levermore, D. {\it Fluid
dynamic limits of kinetic equations I: Formal derivations.} {J. Statist. Phys.} {\bf 63} (1991), no. 1-2, 323--344. 


\bibitem{BGL93} Bardos, C.; Golse, F.; Levermore, D. {\it Fluid
dynamic limits of kinetic equations II: Convergence proofs for the Boltzmann
equation.} {Comm. Pure Appl. Math.} {\bf 46} (1993),  no. 5, 667--753.


\bibitem{BU91} Bardos, C.; Ukai, S. {\it The classical incompressible Navier--Stokes limit of the Boltzmann equation.}  Math. Models
Methods Appl. Sci. \textbf{1}  (1991), no. 2, 235--257.


\bibitem{BF13} Boyer F.;   Fabrie P. {\it Mathematical tools for the study of the incompressible Navier--Stokes equations and related models.} Appl. Math. Sci., 183, Springer, New York, 2013.

\bibitem{Ca} Caflisch, R. {\it The fluid dynamic limit of the
nonlinear Boltzmann equation.} {Comm. Pure Appl. Math.}
{\bf 33} (1980), no. 5,  651--666.


\bibitem{CKL2019}Cao, Y.; Kim, C.; Lee, D. {\it Global strong solutions of the Vlasov--Poisson--Boltzmann system in bounded domains.} Arch. Ration. Mech.
Anal. {\bf 233} (2019), no. 3, 1027--1130.


\bibitem{CIP}
Cercignani, C.; Illner, R.; Pulvirenti, M. {\it The Mathematical Theory of Dilute Gases.}
Springer, New York, 1994.

%
%

 \bibitem{DEL}De Masi, A.; Esposito, R.; Lebowitz, J. {\it Incompressible Navier--Stokes and Euler limits of the Boltzmann
equation.} {Comm. Pure Appl. Math.} {\bf 42} (1989), no. 8, 1189--1214.



\bibitem{DD23}Deng D.; Duan R. {\it Low regularity solutions for the Vlasov-Poisson-Landau/Boltzmann system.}
 {Nonlinearity.} {\bf 36 } (2023), no. 5, 2193--2248.





\bibitem{DS11}Duan R.;  Strain R. {\it Optimal time decay of the Vlasov--Poisson--Boltzmann
system in $\mathbb{R}^3$.}  {Arch. Ration. Mech. Anal.} {\bf 199 } (2011), no. 1, 291--328.



\bibitem{DY10} Duan R.;  Yang, T. {\it Stability of the one-species Vlasov--Poisson--Boltzmann system.} SIAM J. Math. Anal., {\bf 41}
(2009), no. 6, 2353--2387.



\bibitem{DYZ2002}Duan R.;  Yang, T.; Zhao H. {\it The Vlasov--Poisson--Boltzmann system in the whole space: the
hard potential case.} J. Differential Equations. {\bf 252} (2012), no. 12, 6356--6386.

\bibitem{DL-13} Duan, R.; Liu, S. {\it The Vlasov--Poisson--Boltzmann system without angular cutoff.} {Comm. Math. Phys.} {\bf 324} (2013), no. 1, 1--45.

\bibitem{DYZ2003}Duan R.;  Yang, T.; Zhao H. {\it The Vlasov--Poisson--Boltzmann system for soft potentials.} {Math. Models Methods Appl. Sci.} {\bf 23} (2013),
     no. 6, 979--1028.








\bibitem{Esposito2018} Esposito, R.; Guo, Y.; Kim, C.; Marra, R.
{\it Stationary solutions to the Boltzmann equation in the hydrodynamic limit.} Ann. PDE. {\bf 4} (2018), no. 1,  Paper No. 1, 119 pp.



%




\bibitem{GS} Golse, F.; Saint-Raymond, L.  {\it The Navier--Stokes limit of the Boltzmann equation for bounded collision kernels.} Invent. Math. {\bf 155} (2004), no. 1, 81--161.


\bibitem{GZW-2021} Gong, W.; Zhou, F.; Wu, W. {\it Global strong solution and incompressible Navier--Stokes--Fourier--Poisson limit of the Vlasov--Poisson--Boltzmann system.}
    SIAM J. Math. Anal.  {\bf 53} (2021), no. 6, 6424--6470.

\bibitem{GJL2020} Guo, M.; Jiang, N.;  Luo, Y.
 {\it From Vlasov--Poisson--Boltzmann system to incompressible Navier--Stokes--Fourier--Poisson system: convergence for classical solutions.} arXiv:2006.16514.

\bibitem{Guo2001} Guo, Y. {\it The Vlasov--Poisson--Boltzmann system near vacuum.} Comm. Math. Phys. {\bf218} (2001), no. 2, 293--313.

\bibitem{guo2002cpam} Guo, Y. {\it The Vlasov--Poisson--Boltzmann system near Maxwellians.} {Comm. Pure Appl. Math.} {\bf55} (2002),  no. 9, 1104--1135.

\bibitem{Guo2003} Guo, Y. {\it The Vlasov--Maxwell--Boltzmann system near Maxwellians.} Invent. Math. {\bf153} (2003), no. 3,  593--630.


\bibitem{Guo2003-SOFT} Guo Y. {\it Classical solutions to the Boltzmann equation for molecules with an angular cutoff.} Arch.
 Ration. Mech. Anal. {\bf169} (2003),  no. 4, 305--353.



\bibitem{guo2006} Guo, Y. {\it Boltzmann diffusive limit beyond the Navier--Stokes approximation.} {Comm. Pure Appl. Math.} {\bf 59} (2006), no. 5,  626--687.

\bibitem{guo2010arma} Guo, Y. {\it Decay and continuity of the Boltzmann equation in bounded domains.} Arch. Ration. Mech. Anal. {\bf 197} (2010),  no. 3, 713--809.

\bibitem{guo2010cmp} Guo, Y.; Jang, J.  {\it Global Hilbert expansion for the
Vlasov--Poisson--Boltzmann system.} Comm. Math. Phys. {\bf 299} (2010), no. 2, 469--501.



\bibitem{GJJ} Guo, Y.; Jang, J.; Jiang, N. {\it Acoustic limit for the Boltzmann
equation in optimal scaling.} Comm. Pure Appl. Math. {\bf 63} (2010), no. 3, 337--361.



\bibitem{GHW21} Guo, Y.;  Huang, F.; Wang, Y. {\it Hilbert expansion of the Boltzmann equation with specular boundary condition in half-space.} Arch. Ration. Mech. Anal. {\bf 241} (2021), no. 1, 231--309.

\bibitem{GX2020} Guo, Y.;  Xiao, Q.  {\it Global Hilbert expansion  for the relativistic Vlasov--Maxwell--Boltzmann system.} Comm. Math. Phys. {\bf 384} (2021),  no. 1, 341--401.









%
%
\bibitem{J2009} Jang, J. {\it  Vlasov--Maxwell--Boltzmann diffusive limit.} Arch. Ration. Mech. Anal. {\bf 194} (2009), no. 2, 531--584.


\bibitem{JL2019} Jiang, N.; Luo, Y. {\it From Vlasov--Maxwell--Boltzmann system to two-fluid incompressible Navier--Stokes--Fourier--Maxwell system with Ohm's law: convergence for classical solutions.} Ann. PDE {\bf 8} (2022), no. 1, Paper No. 4, 126 pp.


\bibitem{LW2021}Li,F.; Wang, Y.; {\it Global strong solutions to the
Vlasov--Poisson--Boltzmann system with soft potential in
a bounded domain.} {J. Differential Equations}  {\bf 305} (2021), 143--205.


\bibitem{LYZ2016}Li, H.-L.; Yang, T.; Zhong, M. {\it  Spectrum analysis and optimal decay rates of the bipolar Vlasov--Poisson--Boltzmann equations.} Indiana Univ. Math. J. {\bf 65} (2016), no. 2, 665--725.


\bibitem{LYZ2020}Li, H.-L.; Yang, T.; Zhong, M. {\it  Diffusion limit of the Vlasov--Poisson--Boltzmann system.}  Kinet. Relat. Models  {\bf 14}  (2021),
    no. 2,  211--255.


\bibitem{Lion1993}Lions, P.-L. {\it On kinetic equations.} Proceedings of the International Congress of Mathematicians Kyoto,  Math. Soc. Japan (1991), 1173--1185.

\bibitem{LM} Lions, P.-L.; Masmoudi, N.  {\it From the Boltzmann equations to the equations of incompressible fluid mechanics. I, II.} Arch. Ration. Mech. Anal. {\bf158} (2001), 173--193, 195--211.





\bibitem{Mischler2000} Mischler, S. {\it On the initial boundary value problem for the Vlasov--Poisson--Boltzmann system.} {Comm. Math. Phys.} {\bf210} (2000), no. 2, 447--466.



\bibitem{S} Saint-Raymond, L. {\it Convergence of solutions to the Boltzmann equation in the incompressible Euler limit.}
{Arch. Ration. Mech. Anal.} {\bf 166} (2003),  no. 1, 47--80.







%
%
%
\bibitem{V} Villani, C. {\it A review of mathematical problems in collisional Kinetic theory.} Handbook of Fluid Mechanics, Vol. I, Amsterdam: North-Holland, (2002), 71--305.



%


\bibitem{wang2013JDE}Wang, Y. {\it Decay of the two-species Vlasov--Poisson--Boltzmann system.} J. Differential Equations. {\bf 254} (2013), no. 5,  2304--2340.


 \bibitem{WZL} Wu, W.; Zhou, F.; Li, Y. {\it Incompressible Euler limit of the Boltzmann equation in the whole space and a periodic box.}   Nonlinear Differ. Equ. Appl. {\bf 26} (2019),  no. 5, Paper No. 35, 21 pp.




\bibitem{WZL-2} Wu, W.; Zhou, F.; Li, Y. {\it Incompressible Euler--Poisson limit of the Vlasov--Poisson--Boltzmann system.}  J. Math. Phys. {\bf 63} (2022),  no. 8, Paper No. 081502, 34 pp.


\bibitem{XXZ2013} Xiao, Q.; Xiong, L.; Zhao, H. {\it The Vlasov--Poisson--Boltzmann system with angular cutoff for soft potentials.}
J. Differential Equations {\bf 255} (2013),  no. 6, 1196--1232.


\bibitem{XXZ2017}Xiao, Q.; Xiong, L.; Zhao, H. {\it The Vlasov--Poisson--Boltzmann system for the whole range of
cutoff soft potentials.} J. Funct. Anal. {\bf 272} (2017), no. 1,  166--226.




\bibitem{Yangyu2011CMP}
Yang, T.;  Yu, H. {\it Optimal convergence rates of classical solutions for Vlasov--Poisson--Boltzmann system.} Comm. Math. Phys. {\bf 301} (2011), no. 2, 319--355.


\bibitem{YYZ-06} Yang, T.;  Yu, H.; Zhao, H. {\it Cauchy problem for the Vlasov--Poisson--Boltzmann system.} Arch. Ration. Mech.
Anal., {\bf 182}, (2006),  no. 3,  415--470.

\bibitem{YZ-06} Yang, T.; Zhao, H. {\it Global existence of classical solutions to the Vlasov--Poisson--Boltzmann system.} Comm.
Math. Phys., {\bf 268}, (2006), no. 3, 569--605.




\end{thebibliography}
\end{document}